%% file: dp.tex
\documentclass[11pt]{amsart}
\usepackage{pb-diagram,amssymb,epic,eepic,verbatim,transparent}
\usepackage{tikz-cd,braket}
\usepackage{hyperref,graphics,epsfig,psfrag}
\usepackage[inline]{enumitem}
\hypersetup{colorlinks}
\usepackage{color}
\definecolor{darkred}{rgb}{0.5,0,0}
\definecolor{darkgreen}{rgb}{0,0.5,0}
\definecolor{darkblue}{rgb}{0,0,0.5}
\hypersetup{ colorlinks,
linkcolor=darkblue,
filecolor=darkgreen,
urlcolor=darkred,
citecolor=darkblue }
\newtheorem{theorem}{Theorem}[section]

\newtheorem{corollary}[theorem]{Corollary}

\newtheorem{construction}[theorem]{Construction}

\newtheorem{proposition}[theorem]{Proposition}
\newtheorem{lemma}[theorem]{Lemma}

\theoremstyle{definition}
\newtheorem{definition}[theorem]{Definition}
\theoremstyle{remark}
\newtheorem{remark}[theorem]{Remark}
\newtheorem{notation}[theorem]{Notation}
\newtheorem{example}[theorem]{Example}
\newtheorem{claim}[theorem]{Claim}
\newtheorem*{claim*}{Claim}

\newcommand\A{\mathcal{A}}
\newcommand\cA{\mathcal{A}}
\newcommand\cS{\mathcal{S}}

\newcommand\M{\mathcal{M}}

\renewcommand\M{\mathcal{M}}

\newcommand\bGam{\mathbb{\Gamma}}
\newcommand\bGamma{\bGam}
\newcommand\tGam{{\tilde{\Gamma}}}
\newcommand\Gam{{\Gamma}}

\newcommand\XX{\mathcal{X}}

\newcommand\YY{\mathcal{Y}}

\renewcommand\S{\mathcal{S}}

\newcommand{\W}{\mathcal{W}}

\newcommand{\J}{\mathcal{J}}

\newcommand{\N}{\mathbb{N}}
\newcommand{\R}{\mathbb{R}}

\newcommand{\C}{\mathbb{C}}

\newcommand{\Z}{\mathbb{Z}}
\newcommand{\Q}{\mathbb{Q}}

\newcommand{\ddt}{\frac{d}{dt}}

\newcommand{\dds}{\frac{d}{ds}}

\renewcommand{\P}{\mathbb{P}}

\newcommand{\PP}{\mathcal{P}}

\newcommand\lie[1]{\mathfrak{#1}}

\newcommand{\m}{\lie{m}}

\renewcommand{\t}{\lie{t}}

\newcommand{\on}{\operatorname}
\newcommand{\ainfty}{{$A_\infty$\ }}

\newcommand{\br}{{\on{br}}}
\newcommand{\ann}{{\on{ann}}}
\newcommand{\foc}{{\on{foc}}}
\newcommand{\rel}{{\on{rel}}}

\newcommand{\Crit}{\on{Crit}}
\newcommand{\Critval}{\on{Critval}}

\newcommand{\sm}{{\on{sm}}}
\newcommand{\fr}{{\on{fr}}}

\newcommand{\dual}{\vee}

\newcommand{\Newt}{\on{Newt}}

\newcommand{\Edge}{\on{Edge}}

\newcommand{\Ver}{\on{Vert}}

\newcommand{\unfr}{{\on{unfr}}}
\newcommand{\Disc}{\on{Disc}}

\newcommand{\Aut}{ \on{Aut} }

\newcommand{\muout}{\mu_{\on{out}}}

\newcommand{\Hol}{ \on{Hol} } 
\newcommand{\Rep}{\on{Rep}}

\newcommand{\Hom}{ \on{Hom}}

\newcommand{\Vol}{  \on{Vol}}

\newcommand{\trop}{{\on{trop}}}
\newcommand{\pert}{{\on{pert}}}

\newcommand{\codim}{\on{codim}}

\newcommand\dirac{/\kern-1.2ex\partial} 
\newcommand\qu{/\kern-.7ex/} 
\newcommand\hqu{/\kern-.7ex/\kern-.7ex/\kern-.7ex/}
 
\newcommand\white{{\includegraphics[width=.05in]{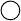}}}
\newcommand\black{{\includegraphics[width=.05in]{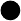}}}
\newcommand\lqu{\backslash \kern-.7ex \backslash} 
\newcommand\bs{\backslash}
\newcommand\dr{r_+ \kern-.7ex - \kern-.7ex r_-}
 
\newcommand{\labell}\label

\usepackage[bbgreekl]{mathbbol}

\newcommand{\hra}{\hookrightarrow}

\renewcommand{\d}{{\on{d}}}
\newcommand{\ol}{\overline}

\newcommand\Phinv{\Phi^{-1}}

\newcommand\lam{\lambda}
\newcommand\Lam{\Lambda}

\newcommand\sig{\sigma}
\newcommand\eps{\epsilon}

\newcommand\om{\omega}

\newcommand{\f}{\frac}
\newcommand{\lan}{\langle}
\newcommand{\ran}{\rangle}
\newcommand\bran[1]{ \lan {#1} \ran}
\newcommand{\hh}{{\f{1}{2}}}

\newcommand{\ti}{\tilde}

\newcommand\cT{\mathcal{T}}

\newcommand\cI{\mathcal{I}}

\newcommand\cP{\mathcal{P}}

\newcommand\ev{\on{ev}}

\newcommand\Vect{\on{Vect}}
\newcommand\ul{\underline}

\newcommand\bdefn{\begin{definition}}
\newcommand\edefn{\end{definition}}
\newcommand\bea{\begin{eqnarray*}}
\newcommand\eea{\end{eqnarray*}}
\newcommand\bcv{\left[ \begin{array}{r} }
\newcommand\ecv{\end{array} \right] }

\newcommand\bma{\left[ \begin{array}{l} }
\newcommand\ema{\end{array} \right]}
\newcommand\ben{\begin{enumerate}}
\newcommand\een{\end{enumerate}}
\newcommand\beq{\begin{equation}}
\newcommand\eeq{\end{equation}}
\newcommand\bex{\begin{example}}
\newcommand\bsj{\left\{ \begin{array}{rrr} }
\newcommand\esj{\end{array} \right\}}

\newcommand\Id{\on{Id}}

\newcommand\Fuk{\on{Fuk}}
\newcommand\Jac{\on{Jac}}

\newcommand\eex{\end{example}}
\newcommand\crit{{\on{crit}}}

\newcommand\sx{*\kern-.5ex_X}

\newcommand{\Bl}{\on{Bl}}

\def\mathunderaccent#1{\let\theaccent#1\mathpalette\putaccentunder}
\def\putaccentunder#1#2{\oalign{$#1#2$\crcr\hidewidth \vbox
to.2ex{\hbox{$#1\theaccent{}$}\vss}\hidewidth}}

\usepackage{xr}
\begin{document}

\title[Disk potentials of almost toric manifolds]{Tropical disk potentials \\ for almost toric manifolds}

\author{S. Venugopalan and C. Woodward}

\begin{abstract}  Using our previous work \cite{vw:trop},
we give a tropical formula for disk potentials for Lagrangian tori in almost toric four-manifolds, that is, fibrations by Lagrangian tori with only  toric and focus-focus singularities, generalizing results of Mikhalkin \cite{mikhalkin} for holomorphic spheres in the projective plane.      As examples, we directly compute potentials for Lagrangian tori  in del Pezzo surfaces equipped with monotone symplectic forms.  These formulas were established in the monotone case by different methods in Pascaleff-Tonkonog \cite{pasc}, and investigated from the point of view of the Gross-Siebert program in Carl-Pumperla-Siebert \cite{cps:trop}, Bardwell-Evans--Cheung--Hong--Lin \cite{bardwellevans} and 
also Lau-Lee-Lin \cite{lau:syz}.
\end{abstract}

\maketitle

\tableofcontents

\section{Introduction}
 
 Mikhalkin \cite{mikhalkin} has given a tropical formula for the counts of holomorphic spheres in the projective plane,
 which was generalized to toric varieties of arbitrary dimension by Nishinou-Siebert \cite{ns}.    In this paper, we generalize Mikhalkin's formula to the case of holomorphic spheres or disks with boundary in moment fibers of almost toric symplectic four-manifolds.  The formula is a special case of our previous work on the behavior of holomorphic curves under degeneration \cite{vw:trop}; the results extend, partially, to almost toric manifolds of dimension greater than four. The formula overlaps with the formulas for the potentials established in Pascaleff-Tonkonog \cite{pasc}, who used the behavior of the potentials under mutation.  It also overlaps with work of Bardwell-Evans--Cheung--Hong--Lin \cite{bardwellevans}, who give a tropical method for computing the potentials of toric moment fibers for del Pezzo surfaces that have semi-Fano degeneration; see also Lau-Lee-Lin \cite{lau:syz} for connection to the Strominger-Yau-Zaslow point of view for mirror symmetry.    Gromov-Witten invariants of blow-ups of projective spaces are computed tropically in Parker \cite{parker:blowups}, from a somewhat different viewpoint.  The techniques here extend to a formula for counting disks with boundary in Lagrangians such as those corresponding to simply laced root systems on del Pezzo surfaces, as we will explain in a sequel \cite{vw:tl} where we extend the results to tropical Lagrangians.

 Recall that an almost toric structure on a symplectic manifold as studied in Leung and Symington \cite{leungsym} is a Lagrangian fibration with {\em elliptic singularities}  occurring in the moment map for a torus action on a symplectic four-manifold or {\em focus-focus singularities} allowed, in which the 
torus fiber is allowed to develop a node.   The base of the fibration, equipped with the location of the singular fibers, is called an {\em almost toric diagram}.   An example of an almost toric diagram for the five-times blow-up of the projective plane (the del Pezzo surface of degree four) is shown in 
Figure \ref{fig:b5p2_at}, with the images of the focus-focus singularities marked with $\times$'s.  Vianna \cite{vianna:dp} has given explicit constructions of almost toric manifolds on del Pezzo surfaces.

\begin{figure}[ht]
\begin{center} 
\scalebox{.8}{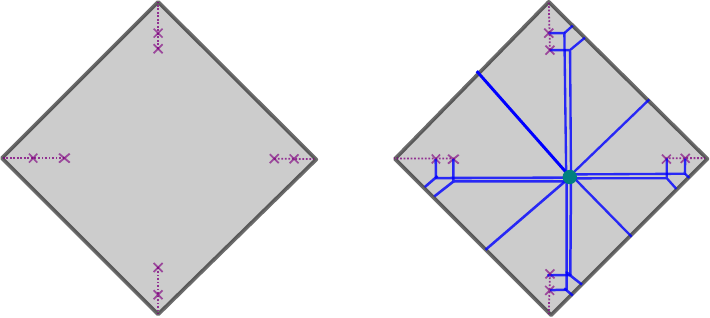}
\end{center} 
\caption{An almost toric diagram for the del Pezzo of degree four and the twelve tropical disks contributing to the potential of the monotone torus} 
\label{fig:b5p2_at}
\end{figure}

Holomorphic disks in such symplectic four-manifolds have tropical descriptions in terms of graphs in the base.   Bardwell-Evans--Cheung--Hong--Lin \cite{bardwellevans} and Gr\"{a}fnitz \cite{gr:corresp}, \cite{graef:trop} prove  tropical correspondences in the case
of toric moment fibers for del Pezzo surfaces
that have a semi-Fano degeneration. 
 The existence of such a formula was
suggested in Vianna's thesis \cite{vianna:dp}, and in previous work by  Carl-Pumperla-Siebert \cite{cps:trop}, 
The tropical computations reproduce the maximally mutable Laurent polynomials found in mirror symmetry; see, for example, Akhtar et al. \cite{akhtar} on the basis of {\em maximal mutability}.  The results in Bardwell-Evans et al. \cite{bardwellevans} take a detour through algebraic geometry.   In this paper, we apply the degeneration (that is, tropical) techniques we developed in Venugopalan-Woodward \cite{vw:trop} to directly compute various disk counts in almost toric manifolds.   In principle, these techniques apply to symplectic manifolds of any dimension, but potentials for four-dimensional almost-toric manifolds have a particularly nice formula.  As examples we give tropical computations of the disk potentials of Lagrangian torus fibers of monotone del Pezzo surfaces, which were first computed in Pascaleff-Tonkonog \cite{pasc}, and from a tropical point of view in Bardwell-Evans--Cheung--Hong--Lin \cite{bardwellevans} for del Pezzo surfaces of degree at least three. In particular,  for the monotone del Pezzo surfaces we give a combinatorial proof that the potential of the Lagrangian torus is mutable, and so equivalent to one of the mutable Laurent polynomials described in Akhtar et al. \cite{akhtar} (following Pascaleff-Tonkonog \cite{pasc}.)
The right-hand part of Figure \ref{fig:b5p2_at} shows the tropical disks in the five-times blow-up.    The monotone Lagrangian torus with its trivial local system
gives a summand in the Fukaya category with $w = 12$ (and so 
$12$ is an eigenvalue of quantum multiplication by the first Chern class; see 
\cite[Theorem 1.7]{vwx} and note that the leading order term in the open-closed map for the point class is non-zero).  

We describe in more detail the disk counts that we wish to compute.  Let $X$ be a compact symplectic manifold and $L \subset X$ be a compact Lagrangian
equipped with a relative spin structure and a local system.  Let $\Lambda_{\Q,\ge 0} = \Q[[q^\R]]$ be the Novikov field in a formal variable $q$ with non-negative real exponents and rational coefficients, and 
$\Lambda = \Lambda_\Q \otimes_\Q \C$ the version with complex coefficients.   The space of Morse chains $CF(L,L)$ on $L$ with coefficients in $\Lambda$ has a natural \ainfty structure introduced by Fukaya, giving rise to a space of projective Maurer-Cartan solutions $MC(X,L)$ with a {\em disk potential}
\begin{equation} \label{eq:mcl} W_{X,L}:  MC(X,L) \to \Lambda_{\ge 0} \end{equation}
counting holomorphic disks with boundary in $L$, weighted by powers of the formal variable with areas as exponents.   The same statement holds over the rationals $\Lambda_\Q$, but the complex version allows for more critical points and so more non-trivial brane structures.  Up to equivalence,
the potential $W_{X,L}$ is invariant of all choices.   We will be mainly interested in the monotone case, 
in which one typically computes the simpler version consisting of a count of disks passing through a generic point on the Lagrangian, weighted by holonomies of a given local system.  The results of this paper extend more generally to the non-monotone case, with the caveat that in this one is computing a function on the Maurer-Cartan space
up to isomorphism, rather than a well-defined potential that is a Laurent polynomial independent of all choices.   In the first purpose of the paper we use degeneration techniques to count such disks.   Let $T$ be a two-dimensional
torus with Lie algebra $\t$, integral lattice $\t_\Z \subset \t$, 
and dual $\t^\dual$. Let $X$
be a compact almost toric four-manifold with an almost toric moment map $\Phi:X \to B$,
where $B$ is an affine integral manifold modelled on $\t^\dual$,
and $L  = \Phinv(\lambda)$ a Lagrangian torus fiber. 
Denote by 
\[ B^\foc \subset B \]  the set of focus-focus values, so that each fiber
$\Phinv(b), b \in B^\foc$ is a nodal torus (with one or more nodes.)  \footnote{More generally, one could allow affine manifolds with singularities in the sense of \cite{vw:split}.}   Let
\[ \cP = \{ P \subset B \} \] 
be a polyhedral decomposition of $B$. Quotienting by the circle actions
normal to the facets gives a degeneration denoted $\XX$.  We will
choose particularly simple decompositions in order to compute disk potentials:

\begin{definition} \label{def:elementary} 
  The decomposition $\cP$ is  {\em elementary} if
  one of the possibilities holds exclusively: Each $P \in \cP$ either
  \begin{enumerate}
    \item  contains a (possibly empty) collection of 
  focus-focus values $b \in B^{\foc}$ all of which lie on the same branch cut in the base diagram,  and in which case $P$ is equivalent
  to the polytope shown in Figure \ref{fig:ff-general}, 
  up to the action of $GL(2,\Z)$;
  \item  contains the projection $\lambda = \Phi(L)$ of the Lagrangian
    $L$;
    \item intersects the boundary of
  $\Phi(X)$ and contains at most one vertex of $\Phi(X)$.
\end{enumerate}
\end{definition}

The broken manifold is equipped with the datum of a dual complex and a cutting
datum in the sense of Definition T-\ref{T-def:cut-datum}, meaning a collection of polytopes $P^\dual$ for each $P \in \PP$ with inclusions $P^\dual \to Q^\dual$ for each face $Q \subset P$.  Given a cutting datum, the dual complex is the union 
\begin{equation} \label{eq:sim}
B^\dual = \cup_{P\in \PP} P^\dual / \sim,\end{equation}
where $\sim$ is the identification of faces. See \cite[Section 3.3]{vw:trop} for detailed definitions.  
For a polyhedral decomposition of an almost toric manifold, the interior of the dual
complex has the structure of a singular affine manifold 
modeled on $\t$. The singular points in the affine manifold $B^\dual$
correspond to the polytopes $P \in \cP^{\foc}$ containing the focus-focus values $b \in B^{\foc}$.   We call the complement of these points
\begin{equation} \label{eq:smoothloc} B^{\dual,\sm} = B^\dual - \bigcup_{P \in \cP^{\foc}} P^\dual \end{equation}
the {\em smooth locus} in $B^\dual$.   In the neck
stretching limit the holomorphic disks with boundary in $L$ degenerate to
broken maps associated to some set $\cT$ of tropical graphs $\Gamma$
in $B^\dual$.  The
notion of {\em tropical graphs} in $B^\dual$ is well-defined; each
edge of a tropical graph $\Gamma$ in $B^\dual$ is an affine linear
segment, and the edges are joined at vertices where a balancing
condition holds as in Mikhalkin \cite{mikhalkin}. (See Section \ref{subsec:bmap} for details.)
The balancing
condition holds at all vertices $v$ except if $v$ is a disk vertex, or
if $P(v)$ intersects the boundary $\partial \Phi(X)$ or if
$P(v)$ contains focus-focus singular values $b \in B^{\foc}$. In these three
cases, we require vertices to be univalent, this condition is laid
down in the following definition.

\begin{definition}
\label{def:coll-interior}
  A tropical graph $\Gamma$  in $B^\dual$ has {\em collisions in the interior}
  if for any vertex of $\Gamma$ for which the polytope $P(v)$ intersects $\partial \Phi(X)$,
  \begin{enumerate}
  \item the valence of $v$ is $|v| = 1$, and 
  \item $P(v)$ intersects exactly a single facet $Q$ of $\Phi(X)$ and the direction $\cT(e) \in \t$ of the adjacent edge $e$ is the primitive normal to $Q$.  
  \end{enumerate}
  A tropical graph $\Gamma$  in $B^\dual$ has collisions only at {\em generic points} if for each vertex $v$ such the polytope $P(v)$ contains a focus-focus value $b \in B^{\foc}$,  $v$ is an open vertex and the valence of $v$ is $|v| = 1$; otherwise, the valence is most $|v| = 3$.
\end{definition}

For a large class of manifolds and Lagrangians, including in the monotone case, there exist subdivisions $\cP$ for which all graphs $\Gamma$ have collisions at 
interior points, see  Proposition \ref{prop:interior} below.

The direction of the edge emanating from a disk vertex  is called
the {\em initial direction} of the tropical graph. 
For univalent vertices $v$ mapping to polytopes
containing the focus-focus values, the directions $\cT(e)$ of the adjacent edges $e \in \Edge(\Gamma)$ are constrained to lie in the shear directions as
explained in Section \ref{sec:at}.

We assign orientations to edges of the tropical graphs
whose collisions are only at generic interior points.  
By Lemma  \ref{lem:orient}, 
there is a unique assignment of edge orientations with the following property:
Given a rigid map $u$ with tropical graph $\Gamma$, cutting an edge $e=(v_+, v_-)$ pointing towards $v_-$ produces maps $u_{v_+}$, $u_{v_-}$
that are rigid if the constraint on the  map component $u_{v_-}$ corresponding to the vertex on which $e$ is incident.

We add some additional data to the tropical graph:  The degeneration associated
to the polyhedron $\cP$ is a union of {\em cut spaces} $X_P, P \in \cP$.
We prove later that if a map $u_v$ lies in a piece $X_{P(v)}$
containing focus-focus singularities, $u_v$ meets exactly one
focus-focus singularity, which we denote by $x_v \in X_{P(v)}$. To the
tropical graph $\Gamma$ underlying $u$, we add the data of $x_v$ to
every such univalent vertex $v$.  The additional data plays a role in
the formula \eqref{eq:w-formula} in the following theorem: a graph
isomorphism $\phi$ on $\Gamma$ is in $\Aut(\Gamma)$ exactly if
$P(v)=P(\phi(v))$ for all vertices $v$, $\phi$ maps
$\Ver_\white(\Gamma)$ to itself, and for univalent vertices $v$ as
above, $x_v=x_{\phi(v)}$.

\begin{theorem} \label{thm:potthm} 
  Let $X$ be a compact almost toric manifold of dimension four, $L$ a
  Lagrangian torus fiber of $\Phi$ and $\cP$  an elementary polyhedral decomposition, for which any disk contributing to the potential $W_{\XX,L}$ has a tropical graph with all collisions 
in generic interior  points. Then
  the count of  rigid Maslov index two disks  with boundary in $L$ is a count of tropical graphs
  \begin{equation}
    \label{eq:w-formula}
W_{X,L} = \sum_{\Gamma \in \cT} \frac{1}{ \# \Aut(\Gamma)}  \prod_{v \in \Ver(\Gamma)} 
m(v)     
  \end{equation}
where the multiplicities $m(v)$ are as in Definition \ref{def:mvs} below, and 
$\Aut(\Gamma)$ denotes the group of automorphisms of $\Gamma$.
The same holds for counts of holomorphic spheres in $X$ with constraints.
\end{theorem}

The multiplicities in the Theorem combine Mikhalkin's multiplicity formula \cite{mikhalkin}, the Bryan-Pandharipande multiple cover formula \cite{bryan:lgw}, and the Cho-Oh \cite{chooh:fano} counts of disks in toric varieties meeting the interior.  The combination of these formulas appears in Gr\"afnitz \cite{graef:trop} for spheres, Lin
in the case of disks in K3 surfaces, appears in Lin \cite[Definition 3.3]{lin:trop}; also compare with Theorem 1.1 in Cheung-Mandel \cite{cm:dt}.   In each picture, the  graph (drawn in blue) represents $\Gamma$, the point (drawn in cyan) represents the Lagrangian image, if present, the purple point a focus-focus value, a solid segment a part of the image of the divisor at infinity $D$ of $X$, a dotted line (drawn in black) a divisor $X_Q$ created by the multiple cut, and a dotted line (drawn in purple) a branch cut.

 \begin{definition}\label{def:mvs}  The multiplicities $m(v)$ of vertices $ v \in \Ver(\Gamma)$ of valence at most three   are defined as follows:
\begin{enumerate}
    \item \label{mv:cyl} {\rm (Holomorphic cylinders)}  The multiplicity 
    \[ m(v)= 1 \] 
    for bivalent vertices $v \in \Ver_\black(\Gamma)$ such that the polytope $P(v)$ has no focus-focus values $b \in B^{\foc}$, does not meet the boundary of the moment polytope $\Phi(X)$ and $v$ has adjacent edges $e_1,e_2 \in \Edge_\black(\Gamma)$ in the same directions ${\cT}(e_1) = {\cT}(e_2)$.  
    \begin{center} 
\scalebox{.9}{\includegraphics{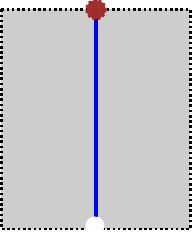}}
\end{center}

    \item \label{mv:pants} {\rm (Holomorphic pairs of pants)}  The Mikhalkin multiplicity 
    \[ m(v)= |\det({\cT}(e_1) {\cT}(e_2))| \] 
    for trivalent vertices $v \in \Ver_\black(\Gamma)$ such that the polytope $P(v)$ contains no focus-focus values $b \in B^{\foc}$, and edge directions 
    \[ {\cT}(e_1),{\cT}(e_2),{\cT}(e_3) \in \Z^2  \] 
    (exactly two of which are incoming) satisfying the balancing condition 
\[ {\cT}(e_1) + {\cT}(e_2) + {\cT}(e_3) = 0 ;\] 
    otherwise, $m(v) = 0$.
  \begin{center} 
\scalebox{.9}{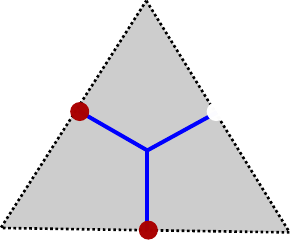}
\end{center}
\item \label{mv:cylhit} {\rm (Collisions of cylinders with boundary)}  The multiplicity  
\[ m(v) = 1 \] 
if $v \in \Ver_\black(\Gamma)$ is a univalent vertex 
such that $P(v)$ contains no focus-focus values $b \in B^{\foc}$ and intersects a facet $Q$ of the toric boundary $\Delta = \Phi(X)$ with  direction $\mu \in \R^2$ the primitive generator of the normal direction to $TQ$.

\begin{center} 
\scalebox{.5}{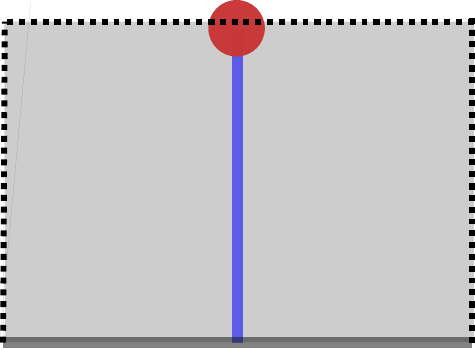}
\end{center}

\item \label{mv:multcov} {\rm (Multiple covers near focus-focus singularities)}  Denote by $\ell(\mu) \in \Z_{> 0}$  the lattice length of the 
direction $\mu = {\cT}(e)$ of the adjacent edge $e \in \Edge(\Gamma)$.
The Bryan-Pandharipande multiplicity is
\[ m(v) = (-1)^{\ell(\mu)-1}/\ell(\mu)^2 \] 
for the univalent closed vertices
$v \in \Ver_\black(\Gamma)$ such that the polytope $P(v)$ contains a single focus-focus value $b \in B^{\foc}$.

\begin{center} 
\scalebox{.7}{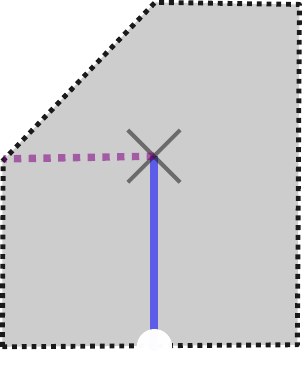}
\end{center}

\item \label{mv:toricdisk}  {\rm (Disks with boundary in the Lagrangian)}  The Cho-Oh multiplicity  
\[ m(v) = 1 \] 
if $v \in \Ver(\Gamma)$ is a univalent vertex such that the polytope $P(v)$ contains the  Lagrangian $\lambda = \Phi(L)$ and the adjacent edge $e \in \Edge(\Gamma)$
has direction ${\cT}(e)$ that is a primitive lattice vector. 
\end{enumerate}
\end{definition}

\begin{example} {\rm (Potential for the degree four del Pezzo)}  Consider the almost toric diagram for the del Pezzo surface
of degree four shown in Figure \ref{fig:b5p2_at}.    Twelve tropical graphs representing Maslov index two disks are shown.  The first four 
$\Gamma_1,\ldots,\Gamma_4$ have two vertices, each univalent, while the last eight have two three univalent vertices and one trivalent vertex.  Each of the multiplicities $m(v)$ of the vertices is equal to one, by the Definition above.   The initial directions of the tropical graphs are $(\pm 1, \pm 1)$, each with multiplicity one, and $(\pm 1,0),(0,\pm 1)$, each with multiplicity two, since there are two choices of focus-focus singularity to interact with in each of these directions.  It follows that 
\[ W_L(y_1,y_2) =   (1+y_1)^2 (1 + y_2)^2 /(y_1y_2)  - 4. \]
\end{example}

\begin{figure}[ht]
\begin{center} 
\scalebox{.8}{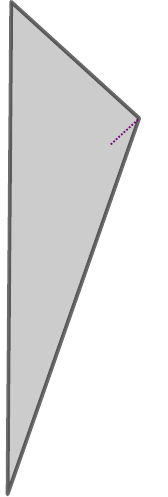}
\end{center} 
\caption{An almost toric diagram for the Chekanov torus and the four tropical disks contributing to the potential} 
\label{fig:chek}
\end{figure}

\begin{example} \label{ex:chek} {\rm (Potential for the Chekanov torus)}  An almost toric diagram for $\P^2$ is  shown in Figure \ref{fig:chek}, which has a single focus-focus singularity and three facets 
with normal vectors $(1,0), (-1,-1)$, and $(-3,1)$.  The diagram is obtained by mutation from the standard toric diagram of $\P^2$, as shown in Figure
\ref{fig:signs}. The corresponding  monotone torus fiber is known as the Chekanov torus, as studied in Vianna \cite{vianna:inf}.   There are three graphs $\Gamma_1,\Gamma_2,\Gamma_3$ with two bivalent vertices each, and both multiplicities are one in this case.  The final graph $\Gamma_4$ has initial direction $(2,0)$, and has a trivalent vertex with incoming edges $(2,0)$ and $(1,1)$ and so Mikhalkin multiplicity two.   The contribution of this graph to the potential is therefore $2 y_1^2$, with a total potential of 
\[ W_L(y_1,y_2) = 1/y_1 + y_1 y_2 + 2 y_1^2 + y_1^3/y_2 .\]
There are three critical values of $W_L$, corresponding to three local systems on $L$ which taken together split-generate the Fukaya category of the complex projective plane $\P^2$.
\end{example}

\begin{example} \rm{ (The number of exceptional curves in the degree one del Pezzo is the same as the number of roots of the $E_8$ root system)}  The same formula holds for rigid holomorphic spheres in 
del Pezzos, as explained in an Appendix \ref{sec:spheres}, simply by disallowing open vertices in item \eqref{mv:toricdisk}, in which case the formula is similar to one in Gr\"afnitz \cite[Proposition 4.48]{graef:trop}.   For example, it is well-known that a del Pezzo of degree one 
contains $240$ lines, that is, rigid embedded holomorphic spheres, which is the same as the number of roots of the $E_8$ root system connected to this del Pezzo by Manin \cite{manin:cubic}; see \cite{serg:adj}, \cite{testa:thesis}.  One can see these $240$ spheres
as follows:
Figure \ref{fig:b8p2_exc} shows 252 cartoon diagrams (approximate moment images) for stable genus zero maps to the del Pezzo of degree one.  Of these, exactly $12$ have self-intersections:  
\begin{enumerate} 
\item the $9$ curves of the $81$ curves of the last type that intersect the same focus-focus singularity twice, and 
\item the $3$ curves corresponding to tropical graphs with the trivalent vertex with determinant $3$; see the self-intersection computation in Mikhalkin \cite{mikhalkin}.
\end{enumerate}
Thus, there are $240 = 252 - 12$ embedded such curves.   Testa \cite[Lemma 2.3]{testa:thesis} explains the difference between the number in terms of the $12$ curves in the anti-canonical linear system. (These are half of the $24$ nodal fibers in the standard realization of a $K3$ surface.) 
\end{example}

\begin{figure}[ht]
  \begin{center}
    \scalebox{.8}{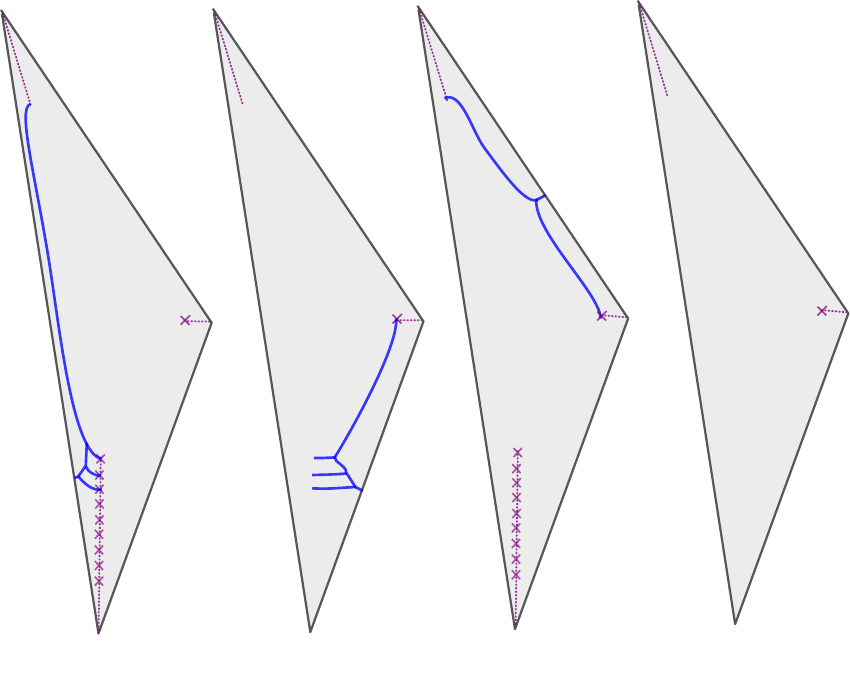}
  \end{center}
  \caption{Cartoon diagrams for the 252 degree one curves in $\Bl^8 \P^2$}
  \label{fig:b8p2_exc}
\end{figure}

It remains to define multiplicities for vertices with valence four or more. To define these multiplicities, we require a perturbation of the tropical graph which we describe next.
\begin{definition} \label{def:vpert}
  Let $v$ be a spherical vertex in a tropical graph $\Gamma$, and let $\Gamma_v$ be the subgraph consisting of $v$ and its incident edges. Suppose  $e_1,\dots,e_k \in \Edge(\Gamma)$ are incoming edges at $v$ where $k \geq 3$, and $e_0$ is the outgoing edge.
A {\em perturbation} $\Gamma_v^\pert$ of $\Gamma_v$ is a graph in $B^\dual$,
each of whose incoming edges $e_i'$, $1 \leq i \leq k$ is the edge $e_i$ translated by an amount $\ell_i$, where
\begin{equation}
  \label{eq:ell-order}
  \ell_1 \gg \ell_2 \gg \dots \gg \ell_k,  
\end{equation}
and the vector $(\ell_1,\dots \ell_k)$ is generic so that all sphere vertices in $\Gamma_v^\pert$ are trivalent, and a disk vertex, if present, is monovalent.
\end{definition}

\begin{figure}[ht]\begin{center}
    \scalebox{.6}{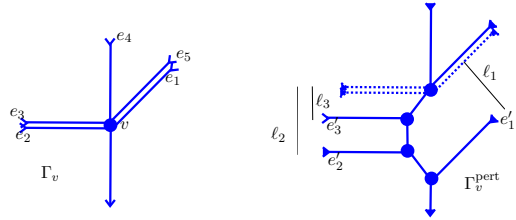}
  \end{center}
  \caption{Perturbing incoming edges of $v$ with parameters $\ell_1>\ell_2>\ell_3$. The pairs $e_1$, $e_5$
  and $e_2$, $e_3$ are coincident in $\Gamma$. }
  \label{fig:pert-edge}
\end{figure}

\begin{figure}[ht]\begin{center}
    \scalebox{.6}{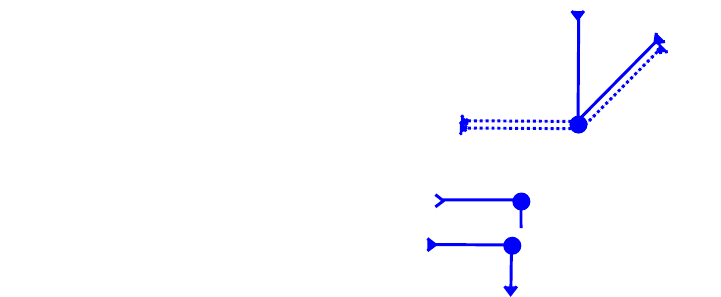}
  \end{center}
  \caption{Perturbing $\Gamma$ in Figure \ref{fig:pert-edge} with parameters $\ell_2>\ell_3>\ell_1$ produces two perturbed graphs.}
  \label{fig:pert-edge2}
\end{figure}

Note that for sphere vertices we may consider perturbations where two of the  parameters $\ell_{k-1}$, $\ell_k$ may  
be taken to be zero, assuming that the directions of the edges $e_{k-1}$, $e_k$ are not parallel. Indeed,
two non-parallel edges in general position will intersect, and the effect of the perturbation is to rule out more than two edges intersecting at a point.

\begin{definition}
  {\rm(Multiplicity for higher valent vertices)}
\label{def:higherval}
  Let $\Gamma$ be a tropical graph as in Theorem \ref{thm:potthm}, and let $v$ be a vertex of $\Gamma$ with valence more than three. The multiplicity of $v$ is defined by 
  is
  \[m(v):=\sum_{\Gamma_v^\pert} \left(\prod_{v_1 \in \Ver(\Gamma_v^\pert)}m(v_1)\right), \]
  where the sum ranges over all perturbed graphs of $\Gamma_v$ that respect a fixed ordering \eqref{eq:ell-order} of incoming edges. (We will see later that different orderings \eqref{eq:ell-order} of the incoming edges produce the same multiplicity. )
\end{definition}

Both the dual complex and tropical graphs can be viewed in a simplified way,
for judiciously chosen polyhedral decompositions, which aids the counting of tropical graphs in Theorem \ref{thm:potthm},  especially in the monotone case.  These simplified tropical graphs called {\em $\A$-tropical graphs} lie in a dual affine manifold:  Given an almost toric four manifold $(X,\Phi)$, the 
  associated {\em dual affine manifold} $\A(X)$ 
  is a singular affine integral manifold 
  modelled over $\t$ 
  that is complete and without boundary equipped with 
  \begin{enumerate}
  \item 
  a singular point $b^\dual \in \A(X)$ corresponding to each focus-focus values
  $b \in B^\foc$; the monodromy of the tangent space $T_\Z\A(X)$
  around the singular point $b$ is a shear matrix  $A_b \in GL(2,\Z)$
  whose primitive eigendirections are
   \begin{equation}
   \label{eq:mub}
   \mu_b^+, \mu_b^- \in T_{x_b}\A(X),
 \end{equation}
with singular points $b_1$, $b_2 \in B^\foc$ lying on the same branch cut allowed to coincide;
 %
%
\item  and a distinguished (smooth point) $P_0^\dual = \{ \lambda \}$ corresponding to the Lagrangian fiber. 
    \end{enumerate}
The dual affine manifold arises naturally from the dual complexes of
our polyhedral decompositions as explained in \eqref{eq:ax}.   There is a bounded open set $S \subset \A(X)$ such that the complement
\[\A_\ann(X):=\A(X) \bs S\]
is an affine annulus with no singular points. We assume in this introduction that the divisor at infinity $D \subset X$ is irreducible; in this case there is a distinguished vector field 
\[ \muout \in \Vect(\cA_\ann(X))_\Z \]
called the {\em primitive outward direction} that is flat with
respect to the affine structure, and is shown in Figure \ref{fig:dual1chart}.  Going one around the annulus produces a monodromy 
\[ h \in \Aut(T_x \A(X)) \simeq GL(2,\Z) \] 
of tangent spaces
$T\A(X)$ that is equal to the composition of shears at all the
focus-focus values. The outward vector $\muout(x)$ is fixed by $h$.
The dual affine manifold $\A(X)$ can be read off from the moment polytope $\Phi(X)$ up to equivalence as described in Remark \ref{rem:ax}. For examples of dual affine manifolds, see Figures \ref{fig:affineex} and  \ref{fig:b8p2-affine}. 

Tropical graphs in the dual complex $B^\dual$
 give rise to 
 {\em $\A$-tropical graphs} in the dual affine manifold
 $\A(X)$ by forgetting bivalent vertices, and the data of the polytopes $P(v)$ 
corresponding to vertices $v$.   The graphs are defined as follows:
\begin{definition} \label{def:a-graph}
  {\rm($\A$-tropical graph)} Let $\A(X)$ be an affine manifold corresponding to a manifold $X$ (as in \eqref{eq:ax}).  Let $\lam \in \A(X)$ correspond to the Lagrangian torus $L \subset X$.  An {\em $\A$-tropical graph} $\cT$ for $X$ consists of a graph 
  \[ \Gamma=(\Ver(\Gamma), \Edge(\Gamma)) \] 
  whose vertices are partitioned into disk and sphere vertices
  \[ \Ver(\Gamma)=\Ver_\white(\Gamma) \cup \Ver_\black(\Gamma), \]
  a {\em tropical embedding} 
  \[ \cT : \Gamma \to \A(X) \] 
  into the affine manifold $\A(X)$ consisting of
  \begin{enumerate}
  \item a tropical position $\cT(v) \in \A(X)$ for every vertex $v$
    such that $\cT(v)=\lam$ for all disk vertices
    $v \in \Ver_\white(\Gamma)$,
  \item and for any edge $e=(v_+,v_-)$, a map from $e$ to an affine
    linear segment in $B^\dual$ of rational direction $\cT(e)$
    connecting $\cT(v_+)$, $\cT(v_-)$ satisfying a {\em balancing
      condition}, namely that the sum of directions of edges at any
    vertex $v$ with valence $|v| \geq 3$ is zero:
    \[ \sum_{e \ni v} \cT(e) = 0, \]
    where $\cT(e)$ is viewed as an element in $T_{\cT(v), \Z}\A(X)$.
    \footnote{In general, in an $\A$-tropical graph, the direction
      $\cT(e)$ of an edge $e$ is the primitive slope of the segment
      $e$ multiplied by a positive integer called the
      multiplicity. When an affine chart is given, $\cT(e)$ is an
      element in $\t_Z$.}
    In the case that an edge $e$ is incident on a single vertex $v$,
    $\cT$ maps $e$ to a semi-infinite affine linear segment
    originating at $\cT(v)$.
  \end{enumerate}
\end{definition}

\begin{definition} {\rm($\A$-tropical tree, index two trees)}
    An $\A$-tropical graph $(\Gamma,\cT)$ is an {\em $\A$-tropical tree} if $\Gamma$ is a tree, and
    there is a single disk vertex.
An $\A$-tropical tree $(\Gamma,\cT)$ in $\A(X)$ has {\em index two} if the edges can be assigned orientations such that each vertex $v \in \Ver(T)$ has exactly one outgoing edge $e \ni v$, and one of the vertices $v \in \Ver(T)$ has a semi-infinite outgoing edge $e_{\on{out}}$ (that is, an edge incident on a single vertex), and 
    
  \begin{enumerate}
  \item {\rm(Tree)} a tree $T$ with directed edges $\Edge(T)$ so that 
\item {\rm(Leaves)} Each univalent vertex $v$ maps either to 
\begin{enumerate}
    \item the point $\lam \in \A(X)$ corresponding to the Lagrangian torus, in which case the adjacent edge $e$ to $v$ is called the {\em input leaf}
\item to a singular point 
$b^\dual, b \in B^{\foc}$,
and in this case the direction of the edge incident on $v$ is required to be a $\Z_+$-multiple of either $\mu_b^+$ or $\mu_b^-$.
\end{enumerate}
  \item {\rm(Final direction)} The direction of the outgoing semi-infinite outgoing edge $e_{\on{out}}$ is $\muout$.
  \end{enumerate}
\end{definition}
We later show that $\A$-tropical trees of index two correspond to broken disks of Maslov index two, justifying the terminology. 
 
\begin{figure}[ht]\begin{center}
    \scalebox{.8}{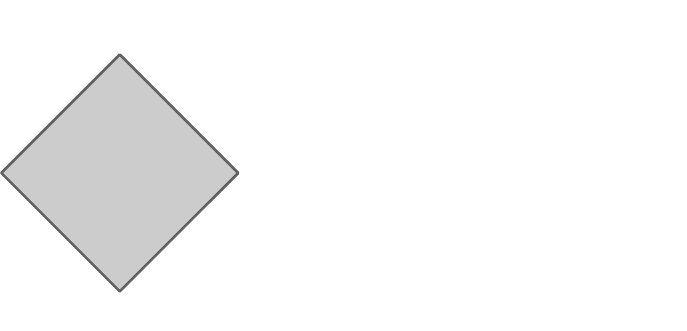}
  \end{center}
  \caption{Left: an almost toric diagram for $\Bl^5 \P^2$.  Right: The dual affine manifold, and the blue dotted lines indicate the direction $\mu_b$ in which a tropical curve emanates from a singular point $b$.}
  \label{fig:affineex}
\end{figure}

Theorem \ref{thm:potthm}  gives a method to compute the disk potentials of del Pezzo surfaces as conjectured by Akhtar et al. \cite{akhtar}.   
Examples of admissible graphs are given in Section \ref{sec:atinfinity}, and shown in Figure \ref{fig:b5p2_at} for the del Pezzo of degree four.  Theorem \ref{thm:potthm-T} is a simpler version of Theorem \ref{thm:potthm} for judiciously chosen polyhedral decompositions. For such decompositions all collisions occur in the annulus $\A_\ann(X)$.

The techniques of this paper can also be used, in good cases, to compute disk potentials of almost toric manifolds in dimensions higher than four in which case the focus-focus values form subsets of positive dimension,
but in this paper we focus on the case of almost toric structures on del Pezzo surfaces.   Although the techniques here are not restricted to the monotone case, it is difficult to formulate clean theorems in the non-monotone cases as the disk potential in the monotone case is only defined as a function on the space of projective Maurer-Cartan solutions up to equivalence.   But the space of such solutions is unknown, even for toric varieties.

\begin{theorem} \label{thm:potthm-T} Suppose $X$ is a symplectic
  monotone four-manifold equipped with an almost toric structure with
  base $B$ and let $L$ be a monotone torus fiber of the projection
  $\Phi:X \to B$.  The formula in Theorem \ref{thm:potthm} for the
  disk potential may be taken to be a sum over $\A$-tropical graphs in
  the affine manifold $\A(X)$.
\end{theorem}

We give some examples of the computations in Section \ref{sec:compute}. Theorem \ref{thm:potthm-T} also implies fairly easily
that the Newton polygon of the potential of an almost toric manifold is the dual of the moment polytope.

The polynomials described in \cite{akhtar} were previously
determined based on the {\em maximal mutability} property; we also show that Theorem \ref{thm:potthm} implies the mutation formula for monotone tori in del Pezzo surfaces, originally proved
in Pascaleff-Tonkonog \cite{pasc}:

\begin{theorem} \label{thm:potthm3} Let $X$ be a compact monotone symplectic four-manifold with almost toric diagrams $\Delta,\Delta'$  related
by a mutation, that is, nodal slide and transferring the cut as in Definition \ref{def:mutate}, and let $L,L' \subset X$ denote the corresponding monotone torus fibers.  Then their potentials $W_L, W_{L'}$ are related by the mutation formula, as in Definition \ref{def:mutform} below.  
\end{theorem}

Combining Theorems \ref{thm:potthm} and \ref{thm:potthm3} we give a combinatorial proof that the potentials are as claimed in the monotone case:

\begin{corollary} (Pascaleff-Tonkonog \cite{pt:wall}.)
The potential $W_L$ of the monotone Lagrangian torus fiber $L \subset X$ of a del Pezzo surface $X$ is given by a tropical disk count in $\A(X)$.   For the almost toric 
structures described by Vianna \cite{vianna:dp}, the potential is equal to the formulas in Akhtar et al \cite{akhtar} reproduced in Table \ref{table:pptable} below.    
\end{corollary}

\begin{table}
\[ \begin{array}{l|l|l}
\text{del Pezzo surface} & \text{Manin root system} & \text{disk potential} \\ 
\hline
\P^2  & & y_1 + y_2 +
1/y_1 y_2  \\ 
\P^1 \times \P^1 & A_1  & y_1 + 1/y_1 + y_2 + 1/y_2 \\
\Bl^1 \P^2 & & y_1 + y_2 + 1/(y_1 y_2) + y_1 y_2 \\
\Bl^2 \P^2 & A_1 & (1 + y_1 + y_2)(1 + 1/(y_1y_2)) -1 \\
\Bl^3 \P^2 & A_1 \oplus A_2 & (1 + y_1)(1 + y_2)(1 + 1/(y_1y_2)) -2 \\
\Bl^4 \P^2 & A_4 & (1 + y_1 + y_2)(1 + 1/y_1)(1 + 1/y_2) -3 \\ 
\Bl^5 \P^2 & D_5 & 
(1 + y_1)^2(1 + y_2)^2/(y_1y_2) - 4 \\ 
\Bl^6 \P^2 & E_6 & 
(1 + y_1 + y_2)^3/(y_1y_2) - 6  \\ 
\Bl^7 \P^2 &  E_7 & (1 + y_1 + y_2)^4/(y_1y_2) - 12  \\
\Bl^8 \P^2 & E_8 & 
(1 + y_1 + y_2)^6/(y_1y_2^2) - 60
\end{array} \] 
\caption{Potentials of del Pezzo surfaces}
\label{table:pptable}
\end{table}

The relationship of these potentials with the rational elliptic surfaces 
considered to be mirror to the del Pezzos is studied in Lutz \cite{lutz:dp}.
Each critical value of the potential is an eigenvalue of quantum multiplication by
the first Chern class, as explained, for example, in Sheridan \cite{sheridan:hypersurface}.
Note that depending on the almost toric structure, not all eigenvalues can appear
as critical values, hence the $\subseteq$ in the 
column heading.

In a sequel to this work, we identify split generators for the Fukaya category of monotone del Pezzos using the almost toric diagrams and tropical Lagrangians.


\section{Almost toric manifolds}
\label{sec:at}

\subsection{Focus-focus singularities}

An {\em almost toric structure} on a symplectic four-manifold is a
pair of functions that act like a moment map for a completely
integrable Hamiltonian two-torus action except for a finite number of
{\em focus-focus singularities} where the torus fibers acquire a node.  
By the Arnold-Liouville theorem, that the image of such a moment map has an
affine integral structure, which we define following
Symington \cite{sym} and 
Gross \cite{gross:trop}.  We use here a definition that is a simplified
one given in \cite{vw:split}, where we assumed a stratification on the singular set.

\begin{definition}{\rm(Singular affine integral manifold)} 
\begin{enumerate}
    \item 
An {\em integral affine structure} on a topological $n$-manifold $B$  is an
  equivalence class of an atlases, where each atlas
  $\{(U_i,\phi_i)\}_{i \in I}$ in the equivalence class has the property that for any $i,j \in I$
  the coordinate change map $\phi_j \circ \phi_i^{-1}$
  is given by the action of an element $g$ in semi-direct product
  $\R^n \times SL_n(\Z)$ in the sense that 
  \[ \phi_j \circ \phi_i^{-1}(x) = gx , \forall x \in \phi_i(U_i \cap U_j) \] 
    and two atlases are equivalent if their union is an atlas of this type. 
  \item
A {\em tropical affine manifold with singularities} (c.f. 
Gross \cite[Definition 1.24]{gross:trop}) is a 
topological manifold $B$ with an integral affine
structure on the complement of a finite union 
\[ B^{\foc} := \cup_{k  \in K} Z_k \]
of codimension at least two topological submanifolds  $Z_k \subset B$.
\end{enumerate}
\end{definition}

The integral affine manifolds we encounter in this paper may be obtained by gluing together open subsets of plane at a subset of {\em branch cuts}.    
Each branch cut $C_i \subset \R^2$ is a ray whose intersection 
$C_i \cap \Delta$ is a line segment whose end-points 
$ \partial C_i = \{ \zeta, b_i \} \subset \Delta $ 
consist of a vertex $\zeta$ of $\Delta$ and a singular point $b_i$.
The singular affine manifold $B$ is  obtained by gluing the 
trivial affine structure on $B \bs \{C_1,\dots,C_k\}$
along the branch cut $C_i$ with a  {\em shear}  conjugate to a matrix of the form 
\begin{equation} \label{eq:shear} s_i = 
\begin{pmatrix}
  1 & k_i.\\
  0&1
\end{pmatrix}
 \in GL(2,\Z) .\end{equation}
Of course, a given affine manifold may be represented by different polytopes that are related to each other by ``transferring the cut'' operations From Definition \ref{def:mutate}.   In the pictures, we draw only the line segments $C_i \cap \Delta$ with dotted lines.

\begin{definition} \label{def:at} An {\em almost toric moment map}  for a symplectic four-manifold $(X,\omega)$ is a smooth surjective map $\Phi: X \to B$
  to a two-dimensional singular affine integral manifold $B$ satisfying the following: 
  For any $x \in X$ 
  there 
is 
a Darboux chart $(q_1,p_1,q_2,p_2)$ on a neighborhood $U \subset X$ of $x$, and a representation of $B$ as a polytope $\Delta \subset \R^2$ so that
$\Phi|U$ is a product of maps of the form
  \begin{enumerate}
  \item {\rm (Non-singular)} $\Phi_i(q_i,p_i)=p_i$;
  \item {\rm (Elliptic)} $\Phi_i(q_i,p_i) = \hh (q_i^2 + p_i^2) $; and
  \item {\rm (Focus-focus)} 
    $\Phi(q_i,p_i, q_{i+1},p_{i+1}) = (q_i p_i + q_{i+1} p_{i+1}, q_i
    p_{i+1} - q_{i+1} p_i )$.
  \end{enumerate}
%
  The triple $(X,\omega,\Phi)$ is called an {\em almost toric
    manifold} and the map $\Phi$ is an {\em almost toric moment
    map}. The last singularities at the origin in the second and third
  cases are called {\em elliptic} and {\em focus-focus singularities}
  respectively. The set of focus-focus singularities is denoted by $X^\foc \subset X$.
\end{definition}

The focus-focus singularities represent nodal singularities in the fibers; see 
Figure \ref{fig:ffsing}. We often fix the representation of the affine manifold $B$ as a polytope $\Delta \subset \R^2$ with branch cuts. Then, the map
\[\Phi : X \to \R^2\]
is an honest moment map on
the complement 
$\Phinv(B \bs \cup_i C_i)$
of the preimage of the branch cuts $C_i$. The data $\Delta \subset B$ with branch cuts $(C_1,\dots, C_k)$ and focus-focus values $B^\foc \subset \Delta $
is called the {\em base diagram} of the almost toric manifold.

\begin{figure}[ht]\begin{center} 
    \scalebox{.7}{\includegraphics{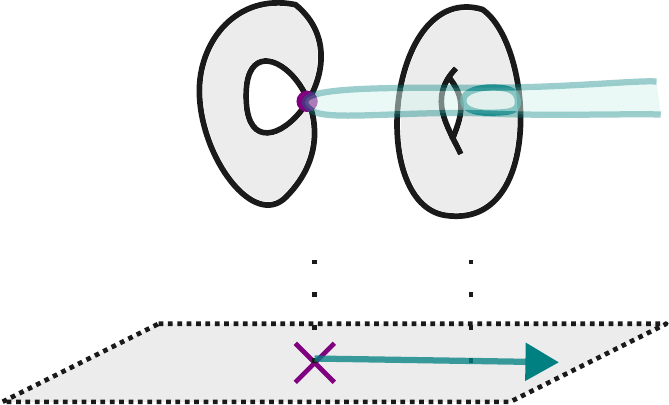}}
  \end{center}
  \caption{A focus-focus singularity}
  \label{fig:ffsing}
\end{figure}
The reader may wish to keep in mind the following example:
\begin{example} \label{ex:pend} {\rm (Spherical pendulum, see e.g. Duistermaat \cite{duist:global})}  A completely 
integrable structure on the cotangent bundle $X = T^* S^2 $ is given by the energy
of the spherical pendulum and its angular momentum.  The energy is defined by 
a sum of kinetic and potential terms, 
\[ \phi: X \to \R, \quad v \mapsto 
\hh \Vert v \Vert^2 + \lan \pi(v), \xi \ran \]
where $\pi: T^* S^2 \to S^2 $ is the projection, $\xi \in \R^3$ is the direction of gravity and $\Vert v \Vert$
is the norm defined using the standard invariant metric.  The 
angular momentum is defined by 
\[ \psi: X \to \R, \quad v \mapsto \lan \iota(v), \xi \ran \] 
where $\iota: T^* S^2 \to \R^3$ is the moment map for the $SO(3)$-action.
Combining the two functions gives rise to a map 
\[ \Psi := (\phi,\psi): X \to \R^2  .\]
Its moment image has boundary given by the set of points corresponding to the ``horizontal motions'' on the sphere
\cite[3.6]{duist:global}.
%
The map $\Psi$ has a single interior critical point $x_+ \in \Crit(\Psi)$ which is a focus-focus singularity corresponding to the unstable equilibrium of the spherical pendulum, and an elliptic critical point $x_- \in \Crit(\Psi)$ corresponding to the stable equilibrium that lies over the boundary. An application of the Arnold-Liouville theorem gives an almost toric structure for $X$ with a moment map $\Phi$ with
the following properties: a single
focus-focus singularity, a moment image equal to the cone on the vectors
$(-1,1)$ and $(1,1)$, with the stable equilibrium $x_-$ mapping to $(0,0)$
and the unstable equilibrium $x_+$ mapping to $(0,1)$.   See 
Figure \ref{pendulum}, which also has the disks contributing to the disk potential drawn. 
 This ends the Example.
\end{example}


%

%
%
%
%


 The following proposition is a well-known consequence of the Arnold-Liouville theorem 
 \cite[Chapter 2]{evans:lec}:
  
 \begin{proposition} \label{prop:affine} 
 Let $\Phi: X \to B$ be an almost toric moment map
 in the sense of Definition \ref{def:at}. Then $B$ has the structure of an tropical affine manifold with singularities the set $B^{\foc}$ of images of the focus-focus critical points.  
 \end{proposition}


The inverse image of the boundary of the moment polytope is an immersed symplectic submanifold of codimension two.  We call this inverse image 
\[ D = \Phinv(\partial \Phi(X)) \] 
of the boundary
$\partial \Phi(X)$ the {\em divisor at
  infinity}, borrowing terminology from algebraic geometry.  The following is immediate from the list of allowable singularities in Definition \ref{def:at}:

  \begin{lemma}   The divisor at infinity $D \subset X$ is an immersed
    symplectic submanifold, whose 
self-intersections are the points $x$ in $X$ where the moment map $\Phi$
has
rank $0$, but $x$ is not a focus-focus point.
\end{lemma}

That is, the self-intersection points are the elliptic singularities
of rank two, which would be the torus fixed points in the case of a
toric surface.  The irreducible components of the divisor at infinity 
\[ D_1,\ldots, D_k \subset D \subset X \] 
are called the {\em boundary divisors}, of in algebraic geometry language, the
prime components of the divisor at infinity $D$.  For example, if $X$ is a toric manifold then the $D_1,\ldots, D_k$ are the prime 
invariant divisors in $X$, which are the inverse images of the facets 
\[ Q_1,\ldots, Q_k \subset \Phi(X) \] 
under the moment map $\Phi$.    The vertices that are the images of elliptic singularities may be characterized alternatively as follows:

\begin{proposition}  A vertex $b$ of $\Phi(X)$ is the image of an elliptic singularity 
if $b$ is not contained in any branch cut, or if $b$ is contained in a branch cut $C_i$, the adjacent edges $e_-,e_+$ to $b$ are not parallel after applying
  the shear $s_i$ to $e_+.$
  \end{proposition}

  \begin{proof} After gluing the affine structures on either side of the branch cut, if it exists, the condition for $b$ to be the image of an elliptic singularity is exactly that $b$ is contained in two facets of the moment image in a neighborhood of $b$.
  \end{proof}

The following result will be used to facilitate the computation of Maslov indices of holomorphic disks by counts of intersection numbers with the divisor at infinity.  By a result of Symington \cite{sym}, see also Li-Ming-Ning \cite{li2023toric}:

\begin{definition} 
A {\em symplectic pair} $(X,D)$ consisting of a symplectic manifold $X$ and a finite union of codimension two symplectic submanifolds $D$.  The pair $(X,D)$ is {\em log Calabi-Yau} if $D$ represents the first Chern class of $X$ in the sense that 
\[ c_1(X) = [D]^\dual \in H^2(X). \] 
\end{definition} 

By a result of Symington \cite[Proposition 8.2]{sym}, see also Li-Ming-Ning \cite{li2023toric}, the pair consisting of an almost toric manifold and its divisor at infinity is log Calabi-Yau. 
The same identity holds in the relative cohomology $H^2(X,L)$ for any Lagrangian fiber fiber $L$ over an interior point.  

\begin{example}  Let $X$ be almost toric four-manifold with moment image $\Delta = \Phi(X)$.  
The {\em boundary divisor} $D = \Phinv(\partial \Phi(X))$ can be made into an embedded symplectic two-torus $D \cong T^2$ by performing nodal trades at all torus fixed points as explained below in Section \ref{fig:modsec}.  In particular, the first Chern class is represented by an embedded symplectic torus. 
\end{example}

\begin{definition}
An almost toric manifold $X$ is called {\em monotone} if the
symplectic class $[\omega ] \in H^2(X)$ is a positive multiple of the first Chern class
$c_1(X) \in H^2(X)$:
\[ \exists \tau > 0, \quad [\omega] = \tau c_1(X) . \]
\end{definition}

\begin{example} By a {\em del Pezzo surface} we mean a non-singular projective algebraic surface with ample anti-canonical divisor class,
whose symplectic form is taken to lie in the anti-canonical class. 
Uniqueness of these symplectic structures up to symplectomorphic follows from McDuff \cite{mcduff:isotopy}.  According to Ohta-Ono \cite{ohtaono}, any monotone symplectic four-manifold is of this form.  Vianna \cite{vianna:dp} has given examples of almost toric structures on del Pezzo surface, shown in Figure \ref{fig:dpall}, and shown that these contain monotone torus fibers:  The diagrams constructed there are obtained by a combination of monotone blow-ups and nodal trades and slides, starting from the standard toric diagram for the complex projective plane.  These operations preserve monotonicity.
\end{example}

\begin{figure}[ht]\begin{center} 
\scalebox{.4}{\includegraphics{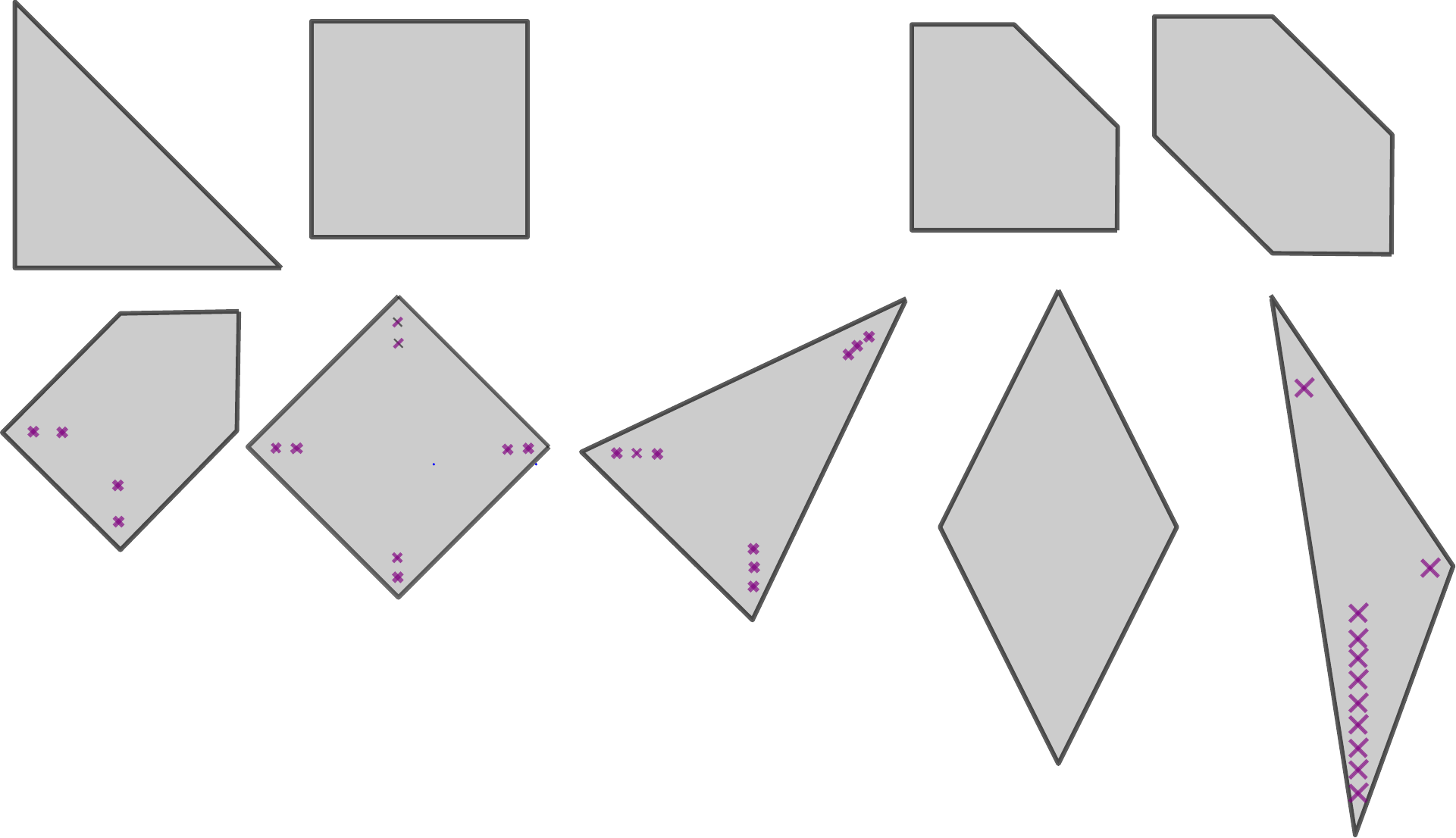}}
\end{center} 
\caption{Almost toric diagrams for the del Pezzo surfaces} 
\label{fig:dpall}
\end{figure}

We are particularly interested in almost toric diagrams for del Pezzo surfaces
where the focus-focus values are close to the vertices.



\begin{definition} \label{fig:viannatype}
  A base diagram $(B,\Delta,\ul{C}, B^{\foc})$ of a monotone almost toric manifold $X$ is {\em of Vianna type} if the cut loci 
  $C_i$ in $\ul{C}$ do not intersect, and the image of the monotone fiber $\lambda = \Phi(L)$ does not lie on any branch cut.
\end{definition}

Later we wil need the following relationship between 
the volume of the the moment polytope 
and the degree of the del Pezzo surface:

\begin{lemma} \label{lem:degree} Let $X$ be a compact connected almost toric four-manifold.  The moment polytope $\Delta$ determines, by its volume, the square of the first Chern class via the relation
\[ \int_X c_1(X)^2  = \int_X 2 \exp(\omega) =  2 \Vol(\Delta) \in \Z .\] 
\end{lemma} 

\begin{proof} Indeed, the  Duistermaat-Heckman measure is Lebesgue measure for toric varieties, by the local model for Lagrangian torus fibers.  See for example Guillemin \cite[Theorem 2.10]{gui:mom} for the toric case; the almost toric case has a similar proof. 
\end{proof}

\subsection{Modifications of almost toric structures}
\label{fig:modsec}

The following are operations on almost toric manifolds of dimension four, as shown 
in Figure \ref{fig:ops}.  Let $X$ be an almost toric symplectic four manifold with symplectic form $\omega$.

\begin{figure}[ht]\begin{center} 
\scalebox{.8}{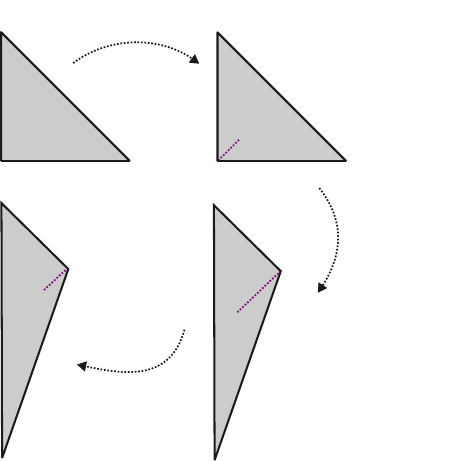}
\end{center} 
\caption{Operations on almost toric diagrams} 
\label{fig:ops}
\end{figure}

\begin{definition} \label{def:mutate}  Suppose an almost toric structure on 
a compact symplectic four-manifold $X$ with base diagram $\Delta$ is given. 
  \begin{enumerate}
  \item {\em Transferring the cut} is an operation which produces a
    new almost toric structure  on $X$ with the same number $| B^{\foc}|$ of focus-focus values as follows. Suppose the base diagram $\Delta'$ has cuts $(C_1,\ldots, c_k)$.  The new cuts are 
    \[ (C_1',\ldots, C_k') = (C_1,\ldots, C_{i-1}, C_i', s_i C_{i+1}, \ldots, s_i
      C_k) \]
    where $C_i'$ is obtained as follows.  Given a branch cut
    starting at $b$ in the direction of $v$,
    \[ C_i = \{b + t v, t \in \R_{\ge 0} \} \]
define a new cut in the opposite direction
    \[ C_i' = \{ b - tv, t \in \R_{\ge 0]} \} .\]
     The cuts $C_i,C_i'$ divide $\Delta$ into
    regions $\Delta_-,\Delta_+$, that is, 
    \[ \Delta - (\ol{C}_i \cup \ol{C}_i') = \Delta_- \sqcup \Delta_+. \]
    The base $\Delta'$ of the new almost toric structure is the union 
\[ \Delta' = \Delta_- \cup s_i(\Delta_+) \] 
of ``half'' of the base  $\Delta_-$  with the sheared ``other half''
$s_i(\Delta_+)$ where    $s_i \in GL(2,\R)$ is a shear along $C_i$.  The new base $\Delta'$ has a new vertex at the
    point of intersection $C_i' \cap \partial \Delta$. In the example in Figure \ref{fig:transf} the shear $s_2$ is
\[ s_2 = \begin{pmatrix}
      0&1\\-1&2
    \end{pmatrix}. \] 
  \item A {\em nodal trade} modifies the base diagram $\Delta$ so that the number of focus-focus values $B^{\foc}$ increases by one, as follows.  
  Suppose the base diagram  $\Delta$ has 
    an elliptic critical point $x \in X$ over a vertex $b \in \Delta$
    adjacent to facets $Q_1,Q_2 \subset \Delta$ whose normal vectors
    are $\nu_1,\nu_2$.  A nodal trade adds a  focus-focus singularity $x' \in X^{\foc}$ whose
    image is at a point $b' \in B^{\foc} $ on the ray in the direction
    $\nu_1 + \nu_2$ from the vertex $b$, and a new cut $C_{k+1}$ from $b'$ to
    $b$.
  \item A {\em nodal slide} modifies the base diagram $\Delta$ by changing
    the position of a focus-focus value $b \in B^{\foc}$  in the direction of the adjacent cut
    $C_i \subset \Delta$, and adjusting the length of the cut $C_i$ if necessary so that it ends at the focus-focus value $b$. 
  \item \label{part:mutate}
    A {\em mutation} produces a new base diagram $\Delta'$ which is a
    combination of transferring the cut and a nodal slide in which one
    of the focus-focus values $b_i \in B^{\foc}$ is moved in the direction of the branch cut until it is close to the other intersection point of the line containing the branch cut with the
    moment polytope $\Phi(X)$, and performing an accompanying shear %
    \[ \Delta' = \Delta_- \cup s_i \Delta_+ \] 
    of
    the moment polytope.
  \end{enumerate}
\end{definition}

\begin{figure}[ht]\begin{center}
      \scalebox{.8}{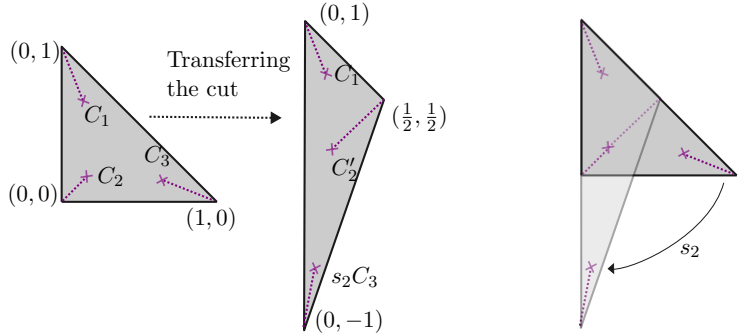}
  \end{center}
  \caption{Transferring a cut}
  \label{fig:transf}
\end{figure}

The Hamiltonian isotopy class of a Lagrangian fiber of an almost toric structure is invariant under modifications of almost toric
  structures  except for a nodal slide that crosses a focus-focus fiber: Let
  \[ \Phi_t: X \to \R^2, \quad t \in [0,1] \]
  be a smooth family of almost toric structures on a simply-connected
  four-manifold $X$, $\mu \in \R^2$, and
  \[ L_t = \Phinv(\mu_t) \]
  a regular monotone Lagrangian fiber of  $\Phi_t$ for each $t \in [0,1]$. 
  
  \begin{lemma}   \label{lem:indep} Suppose that $X,\Phi_t,L_t$ is a family 
  of smooth almost toric fibers with fixed moment image, so that $\Phi_t(X)$
  and $\Phi_t(L_t)$ are independent of $t \in [0,1]$ as above. The Hamiltonian isotopy class of $L_t$ is independent of $t$.
  \end{lemma}

  \begin{proof}  It suffices to check that the isotopy is exact in the sense of,
  for example,   Weinstein \cite{we:remove}.    By assumption $X$ is simply-connected, each loop $\gamma_t: S^1 \to L_t$ bounds a disk $u_t$.   We may assume that $u_t$ is smoothly varying in $t \in [0,1]$ since the isotopy is smooth.     By continuity, $u_t$  has   constant Maslov index $I(u_t)$ and so constant area  $A(u_t)$ by monotonicity.    The deformation $L_t$ defines a family of closed one-forms $\alpha_{t,s} \in \Omega^1(L_t)$,  so that $L_{t + s}$ is the graph of $\alpha_{t,s}$ in some Weinstein neighborhood of $L_t$   and the deformation one-forms $\alpha_t := \dds \alpha_{t,s} |_{s = 0}$ are independent of the choice of neighborhood.  By Stokes' theorem  
  \[ \ddt \int_\gamma \alpha_{t,s} = \ddt \int_D u_t^* \omega .\]
  Since the areas are constant, the periods of $\alpha_t$ vanish.  Hence the one-forms $\alpha_t$ are exact, which implies the deformation $L_t$   is exact. 
  \end{proof}

On the other hand, the fibers of a family of focus-focus values which cross the monotone value gives new isotopy classes of monotone Lagrangian tori.  Using this operation,  Vianna \cite{vianna:dp} shows that every monotone del Pezzo surface has the structure of an almost toric manifold (even with a triangular base), shown in Figure \ref{fig:dpall}.   An example of an almost toric diagram for the blow-up of the projective plane at five points is shown in Figure \ref{fig:b5p2_at}. The diagrams with triangular base corresponds to solutions to Markov-type equations, as explained in Vianna \cite{vianna:dp}.  Diagrams for $\Bl^k \P^2, k = 5,7,8$ are shown in Figure \ref{fig:triang}. 

\begin{figure}[ht]\begin{center} 
\scalebox{.8}{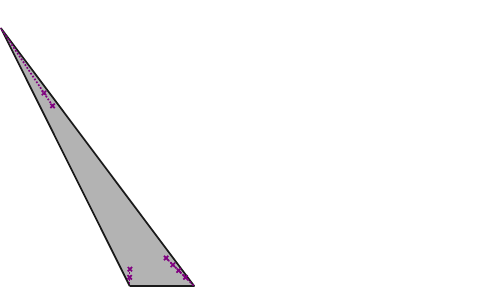}
\end{center} 
\caption{Triangular almost toric diagrams} 
\label{fig:triang}
\end{figure}

\section{Tropical limit theorems}
\label{sec:troplim}

We describe the degeneration techniques from Venugopalan-Woodward \cite{vw:trop} which allow the computation of disk potentials for Lagrangians submanifolds.   In \cite{vw:trop} these tropical limit theorems were stated for treed disks, while here we are concerned only with disks passing through a generic point on the Lagrangian.

\subsection{Disk potentials}

The disk potential of a monotone Lagrangian brane is a count of Maslov-index-two holomorphic disks with the Lagrangian as boundary condition.   Let $(X,\omega)$ be a compact symplectic manifold.   By a {\em Lagrangian brane} we mean a compact  relatively-spin graded Lagrangian $L \subset X$.   For such a brane (at least, under rationality constraints discussed in \cite{cw:traj}) the potential in a Morse model is defined as follows. Let
\[ f: L \to \R \]
be a Morse function.   Let 
\[ \Lambda_\Q  = \Q((q^\R)), \Lambda = \Lambda_\Q \otimes_\Q \C  \] 
denote the Novikov fields in a  formal variable $q$, and $\Lambda_{ \ge 0} \subset \Lambda$ the Novikov ring  
generated by elements with non-negative $q$-valuation.   The group of {\em Floer cochains} 
\[ CF(L) = \bigoplus_{x \in \cI(L)} \Lambda_{\ge 0} x  \]
is the free group generated by the set of critical points
\[ \cI(L) := \crit(f) .\]
By the construction in Charest-Woodward \cite{flips}, 
the group $CF(L)$ has an \ainfty structure with composition maps 
\[ m_d: CF(L)^{\otimes d} \to CF(L)[2-d], d \ge 0 \] 
defined
given by weighted counts of treed holomorphic disks $u: C \to X$ with boundary in $L$.   Here ``treed'' refers to the definition of $C$ as the union of a surface part $S$ and a tree part $T$.  The map $u$ is required to be pseudoholomorphic on the surface part
 on $S$ (with respect to some domain-dependent almost complex structure, after perturbation) and 
a gradient flow on the tree part $T$ (with respect to some domain-dependent function, after perturbation).    The details will not concern us very much, 
since we will be interested in a simplified version in which we count holomorphic disks with boundary in $L$ with point constraints.   Strict units
\[ 1_L \in CF(L,L) \]
are provided by a homotopy unit construction. 
The resulting Fukaya category $\Fuk(X)$ may be curved, meaning that the zero-th composition map $m_0$ may be non-zero.   
Let  $CF^{\on{odd}}(L)_+ \subset CF(L)$ denote the space generated by  
elements with positive $q$-valuation.  Let $MC(X,L)$ denote the space
of elements $b \in CF^{\on{odd}}(L)$ of solutions to the {\em projective Maurer-Cartan equation}
\[ m_0(1) + m_1(b) + m_2(b,b) + \ldots \in \on{span}(1_L) .\]
The coefficient in front of the identity gives rise to the {\em disk potential}
\begin{equation}
  \label{eq:diskpot-mc}
  W_{X,L}:  MC(X,L) \to \Lambda_{\ge 0}   
\end{equation}
counting holomorphic disks with boundary in $L$.  The pair $(MC(X,L),W_{X,L})$ is independent of all choices up to gauge transformation.     In the monotone case, there is a simplified definition of the disk potential which avoids the machinery of \ainfty algebras.  Namely, we restrict to the case $b = 0$
and obtain a function also denoted 
\[ W_{X,L}: \Rep(L) \to \C \] 
 on the space
of local systems 
\[ \Rep(L) := \Hom(\pi_1(L), \C^\times)  \]
which counts disks passing through a generic point. 
The arguments in Oh \cite{oh:lag} show that the potential is independent of the choice of almost complex structure, and depends only on the Hamiltonian isotopy class of $L$.    In this paper we will focus on the monotone case.  However, many of the techniques apply also to the computation of potentials for non-monotone Lagrangians; for example, the Lagrangians near the anti-canonical divisor considered in \cite{grz:proper}.

The regularization scheme described in \cite{oh:lag} to define the potential uses a generic domain-independent almost complex structure.    In order to make contact
with the theory of broken maps in \cite{vw:trop} we will use domain-dependent almost complex structures as in Cieliebak-Mohnke \cite{cm:trans}.   These moduli spaces
have additional interior markings corresponding to the intersection of the maps 
with a Donaldson hypersurface $D \subset X$ of large degree, disjoint from the Lagrangian, and so that the Lagrangian $L$ is exact in the complement of $D$.
We denote by  the moduli space of maps satisfying the given constraints, and $\M_d(X,L)$ the union over all types with $d$ semi-infinite boundary edges.

To each holomorphic disk, we associate a map type as follows.  The
vertices $\Ver(\bGam)$ of the type $\bGam$ correspond to components of
the domain, while the edges $\Edge(\bGam)$ correspond to either nodes
or markings.  Edges $e \in \Edge_{\to}(\bGam)$ corresponding to
markings are incident on one vertex each, and those
$e \in \Edge_{-}(\bGam)$ corresponding to nodes are incident on two
vertices.  The set of edges by definition have a partition
\[\Edge(\bGam)=\Edge_\white(\bGam) \cup \Edge_\black(\bGam)\]
into an open and closed subsets.   The closed leaves correspond to intersections with the Donaldson divisor.  The vertices $v \in \Ver(\bGam)$ are decorated with the homotopy classes $[u_v]$ of the corresponding map $u_v$.  For a boundary marking corresponding to
$e \in \Edge_{\white,\to}(\bGam)$ we have an evaluation map
\[ \ev_e:  \M(X,L) \to L \]
evaluating the map at the corresponding boundary marking $z_e \in C$.
The underlying {\em domain type} $\Gamma$ is obtained by forgetting
the labelling of the vertices by homotopy classes.

\begin{definition} \label{def:dp}
  {\rm(Disk potential for monotone Lagrangians)} Given a Lagrangian
  brane $L$, the disk potential of $L$ is the count of Maslov-index-two disks with a
  single point constraint on the boundary $\ul{Y} = (p \in L)$ defined by the formula 
  \begin{equation}
    \label{eq:wxl}
  W_{X,L} = \sum_{u \in \M_{\bGam}(X,L,\ul{Y})_0} \frac{
    \Hol_L([\partial u]) \eps(u) }{d_\black(\bGam)!},  
  \end{equation}
  where 
  \begin{itemize}
  \item $\bGam$ ranges over types of maps with no boundary input
  markings and a single output marking that is constrained to map to a
  point $p \in L$, that is, $\ul Y=(p)$;
  \item $\Hol_L([\partial u]) \in \Lam^\times$ is the evaluation of
    the local system $\Hol_L \in \on{Rep}(L)$ on the homotopy class of
    loops $[\partial u] \in \pi_1(L)$ defined by going around the
    boundary of each disk component in the treed disk once,
  \item $d_\black(\bGam) \in \Z_{\ge 0 }$ is the number of interior
    markings,
  \item and $\eps(u) \in \{\pm 1\}$ is the orientation sign of $u$
    defined using the relative spin structure. For almost toric moment
    fibers we assume that the relative spin structure is the standard
    torus-invariant spin structure.
  \end{itemize}
\end{definition}

\begin{remark}
 For a monotone pair $(X,L)$, the disk potential is independent of choices, by an argument similar to that in Theorem \ref{thm:coates} below.  In the general case, the disk potential is defined as a function on the Maurer-Cartan moduli space as in \eqref{eq:diskpot-mc}. 
 \end{remark}
 
\subsection{Broken Maps}\label{subsec:bmap}

Broken maps arise from multi-directional neck stretching.  We review the theory from Venugopalan-Woodward \cite{vw:trop}, which is a version of Brett Parker's work for closed maps. 

\begin{notation} 
Let $X$ be a compact symplectic four-manifold with an almost toric structure with codomain $B$, containing the image of the  almost toric moment map $\Delta = \Phi(X)$, and focus-focus image $B^{\foc} \subset B$. The locus $B - B^{\foc}$ is equipped with an affine structure modelled on $\t^\dual$,
where $T \simeq (S^1)^2$.   A subset $P \subset B$ is a {\em polyhedron} if in a coordinate chart $U \subset B - B^{\foc}$ homeomorphic to an open subset of $\R^n$, the intersection $P \cap U$ is the inverse image of a finite intersection of half-spaces $H_1,\dots, H_k$ in $\R^n$.
The polyhedron $P$ is {\em rational} if the half spaces $H_i$
have rational normal vectors.   Let $\Delta \subset B$ be a polyhedron; we have in mind
the case that $\Delta = \Phi(X)$ is the moment image of an almost toric moment map. 
\end{notation}

\begin{definition} \label{def:poly}  A {\em polyhedral decomposition} of $\Delta \subset B$ is a collection 
$\cP$ of rational polyhedra $P \subset B$ such that 
\begin{enumerate}
\item $B$ is equal to the union $ \cup_{P \in \cP} P  ;$ 
\item The intersection of any two polytopes $P_0,P_1 \in \cP$ is either empty or a face of each; 
\item each focus-focus value $b \in B^{\foc}$
is contained in the interior of some polyhedron $P \in \cP$;  and
\item  any polyhedron $P \in \PP$ intersects the boundary $\partial \Delta$ transversally.
\end{enumerate} 
Given such $\cP$, for any polytope $P \in \PP$, there is a cut space 
\[ X_P = \Phinv(P^\circ)/T_P \]
where $T_P \subset T$ is the torus whose Lie algebra is the annihilator
of the span of $TP \subset \t^\dual$, and $P^\circ$ is the complement of the faces $Q \subset P, Q \in \PP$ of $P$.  The {\em broken manifold} associated to $\cP$ is the disjoint union 
\[ \XX = \bigcup_{P \in \cP} \XX_P \]
where $\XX_P$ is the {\em thickening} of $X_P$ defined by 
\[ \XX_P = \Phi^{-1}(P^\circ) \times \t_P^\dual \]
This ends the definition. \end{definition}

We will consider in this paper only the case that the Lagrangian is a fiber 
of the almost toric structure.  In particular, we assume that the Lagrangian submanifold $L \subset X$ is disjoint from the cuts so that $L$ lies in a component $\XX_P$ of the broken manifold $\XX$ where $P \in \PP$ is a top-dimensional polytope.

\begin{example} The simplest examples of polyhedral decompositions 
are those cut out by a collection of lines with rational directions.
If $B$ is realized as a subset of $\R^2$ with cut loci $C_i \subset B$, then 
any line $H$ in $B$ not passing through the cut loci defines a polyhedral decomposition with three pieces. On the other hand, if $H$ passes through a cut $C_i$ then the complement of $H$ is not locally a convex polytope (as the half-spaces undergo shears through the cut loci).  Instead, changing the direction of $H$ after it passes through $C_i$ results in an allowable cut.    See Figures \ref{fig:poly}, \ref{fig:2p}, \ref{fig:b4p2dual} for examples. 
\end{example}

As described in  \cite{vw:trop}, in the neck-stretching limit holomorphic maps converge to broken maps, which is a version of Parker's exploded maps in \cite{parker:blowups}.  A broken map is equipped with a tropical structure, which we recall here. 

\begin{definition}
    A {\em tropical structure} on a graph 
\[ \Gamma = (\Ver(\Gamma), \Edge(\Gamma))  \]
is a collection of polytope assignments $P(v) \in \PP$ for vertices
$v \in \Ver(\Gamma)$ and 
{\em edge directions}
\[ {\cT}(e) \in \t_{P(e),\Z}  , \] 
with $P(e)$ defined by 
\[ P(e) := P(v_+) \cap P(v_-)  \] 
so that the graph $\Gamma$  is  realizable 
in the dual complex $B^\dual$ of the neck-stretching in the following sense: 
There exist a collection of {\em tropical positions} of
the vertices in the dual complex
\[\cT : \Ver(\Gamma) \to B^\dual, \quad \cT(v) \in P(v)^\dual\]
that satisfy 
\begin{equation} \label{eq:dircond} \cT(e) \in \R_{\ge 0} ( \cT(v_+) -
  \cT(v_-) ).
\end{equation}
for any edge $e=(v_+,v_-)$.
The image of $\cT(\Gamma)$ under the map to $B^\dual$ induced by
$v \mapsto \cT(v)$ is called 
a  {\em realization of a tropical graph}, and the underlying graph equipped with the edge directions $\{\cT(e)\}_e$ and vertex polytopes $\{P(v)\}_v$
is called a {\em tropical graph}.     This ends the Definition.  \end{definition}

A broken map is a collection of holomorphic maps on
punctured curves associated to the vertices of a tropical graph. 
We equip $X$ with an almost complex structure that is cylindrical near each cut; this induces
a collection of almost complex structures on $\XX_P$ that are invariant with respect to the 
translation action of $\t_P$.  

\begin{definition}
    Given such a graph $\Gamma$
a broken map with domain type $\Gamma$ is a collection
\[ u_v:C_v^\circ \to \XX_{P(v)}, v \in \Ver(\Gamma) \]
of pseudoholomorphic maps satisfying certain matching conditions explained below on the lifts of nodal points. Each of the domain components $C_v^\circ$ is an irreducible curve component $C_v \subset C$
(possibly with boundary) punctured at interior nodal points, that is,
\[C_v^\circ :=C_v \bs \{\text{interior nodes}\}.\]
The matching conditions are described as follows.
Suppose that 
$w_e^\pm \in C_{v_\pm}$ are the lifts of a node $w_e$ corresponding to an edge $e=(v_+,v_-)$. 
At the cylindrical or strip-like
ends of $C_{v_\pm} \bs w_e^\pm$, the map $u_{v_\pm}$ is asymptotic to a map given by the 
action of a one-parameter subgroup corresponding to some 
{\em direction} $\cT(e) \in \t_{P(e),\Z}$.
The matching condition at the node $w_e$ 
is that the map $(\pi_{\cT(e)}^\perp \circ u_{v_\pm})$ has a removable
singularity at the node $w_e^\pm$, and
\begin{equation}
  \label{eq:match-proj}
  (\pi_{\cT(e)}^\perp \circ u_{v_+})(w_e^+)=(\pi_{\cT(e)}^\perp \circ u_{v_-})(w_e^-) \in \XX_{P(e)}/T_{\cT(e),\C}.
\end{equation}
The quantities in the left-hand side and right-hand side of
\eqref{eq:match-proj} are called {\em projected tropical evaluations}.
   In case $Q$ is a facet so that $X_Q$ is a divisor of the compactified spaces $\ol X_{P(v_\pm)}$, this condition is
 simply that $\ol{u}_{v_\pm}(z_\pm) \in X_Q$ are equal but if $Q$ is
 higher codimension the condition says roughly that the ratios of
 derivatives match.
 In the latter case, we may consider a different compactification $\ol{\XX}_{P(v_\pm), \cT(e)}$ of $X_{P(v_\pm)}$ for which the space
\begin{equation}
   \label{eq:ycte}
   Y_{\cT(e)}:= \{[x] = T_{\cT(e),\C}x: x \in \ol{\XX}_{P(v_\pm)} \}  
 \end{equation}
 of orbit closures is a divisor in $\ol{\XX}_{P(v_\pm), \cT(e)}$ ; then \eqref{eq:match-proj} says that the points $\ol{u}_{v_\pm}(z_\pm) \in Y_{\cT(e)}$ are equal in  the compactifications $\ol{\XX}_{P(v_\pm), \cT(e)}$. 
 The object defined so far is an {\em unframed
   broken map}.

A {\em broken map} consists of an unframed broken map and a {\em framing}, which is a
linear isomorphism
\begin{equation} \label{eq:framingeq} \fr_e : T_{w^+_e} C_{v_+}
  \otimes T_{w^-_e} C_{v_-} \to \C,
\end{equation}
such that any pair of holomorphic coordinates $z_+$, $z_-$ in the
neighborhood of $w_e^+$, $w_e^-$ satisfying
\begin{equation}
  \label{eq:framingcoord}
  \d z_+(w^+_e) \otimes \d z_-(w^-_e)=\fr_e
\end{equation}
are {\em matching coordinates} in the sense that the following 
\begin{equation} \label{eq:nodematch-intro}
  \text{(Matching condition)} \quad \lim_{z_+ \to  0}z_+^{-\cT(e)}u(z_+) 
  =\lim_{z_- \to 0}z_-^{-\cT(e)}u(z_-)
\end{equation}
holds. 
 The quantities in the left-hand side and right-hand
side of \eqref{eq:nodematch-intro} are called the {\em tropical
  evaluations} at $w_e^+$ and $w_e^-$ respectively. 
  This ends the Definition.
\end{definition}
  
\begin{remark}
    An
unframed broken map has a finite number of framings as explained in
Remark T-\ref{T-rem:nfr}.
\end{remark}  

 A sample polyhedral degeneration  of the almost toric diagram of the degree four del Pezzo surface is shown in Figure  \ref{fig:poly}.   Let $u = (u_v)_{ v\ \in \Ver(\Gamma)}$
 be a broken map.  The graph $\Gamma$ with the additional data of the homotopy 
 classes of the maps $[u_v]$ is called the {\em type} of the broken map $u$ and denoted
 \[ \bGam = (\Gamma, ( [u_v])_{v \in \Ver(\Gamma)}) . \] 
 Denote by $\M_{\bGam}(\XX,L)$ the moduli space of broken maps of type $\bGam$.

\begin{figure}[ht]\begin{center} 
\scalebox{.5}{\includegraphics{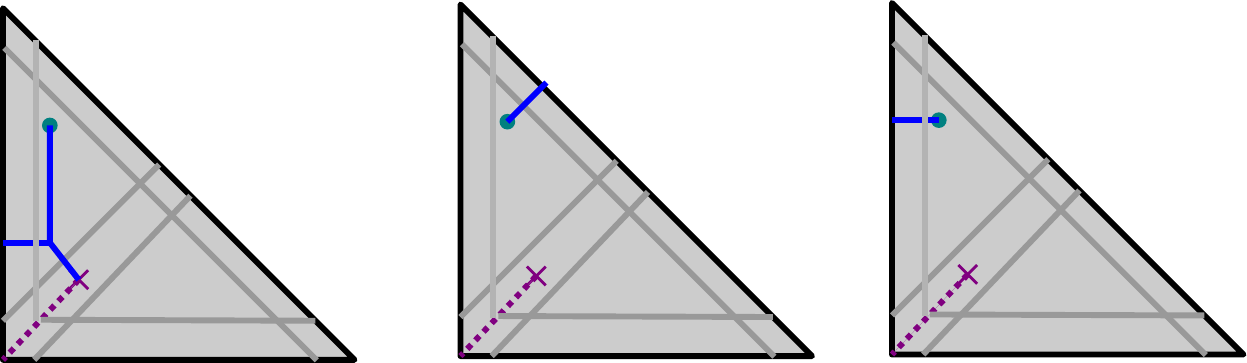}}
\end{center} 
\caption{A polyhedral decomposition of an almost toric diagram
and three Maslov-two broken disks}
\label{fig:poly}
\end{figure}

The disk potential in a broken manifold is defined similarly to the unbroken case, the only difference being that one counts broken disks. Assuming that $L \subset \XX$ is a monotone Lagrangian brane contained in a top-dimensional component $\XX_P \subset \XX$, the {\em disk potential} is 
 \begin{equation} \label{eq:bdp}  W_{\XX,L} = \sum_{u \in \M_{\bGam}(\XX,L,\ul{Y})_0} \frac{
     \Hol_L([\partial u]) \eps(u)}{d_\black(\bGam)!}, \end{equation}
  where $\bGam$ ranges over types of broken disks with no boundary input
  markings and a single output marking that is constrained to map to a
  point $p \in L$, that is, $\ul Y=(p)$. Furthermore,
  $\Hol_L([\partial u])$, $d_\black(\bGam)$,   $\eps(u)$ are as in \eqref{eq:wxl}; and the area  of a broken map $u$ is the sum of the areas of the projections of its components, that is, 
    \begin{equation} \label{eq:au} A(u) = \sum_{v \in \Ver(\Gamma)} A(\pi_{P(v)} \circ u_v), \end{equation}
    where $\pi_{P(v)} : \XX_{P(v)} \to X_{P(v)}$ is the projection to
    the base of the fibration. When $(X,L)$ is monotone, the fact that
    $W_{\XX,L}$ is independent of all choices will be a consequence of
    the tropical limit theorem in the next section.

\subsection{Tropical limit theorems for moduli spaces}
In \cite{vw:trop}, we showed that the Fukaya algebra of a Lagrangian is equivalent to 
Fukaya algebra defined by counting broken maps.  In particular, the corresponding disk potentials 
\[ W_{X,L}: MC(X,L) \to \Lambda, \quad W_{\XX,L}: MC(\XX,L) \to \Lambda \] 
are equivalent up to a gauge transformation on $MC(X,L) \cong MC(\XX,L).$  Returning to the monotone case, one has equality of the broken and unbroken potentials:

\begin{theorem} \label{thm:cobord} Suppose $L \subset X $
is monotone and $Y = \{ p \} $ consists of a point constraint at a single  boundary point $p \in L$.  There exists a compact oriented one-dimensional cobordism $\ti{\M}(X,L,\ul{Y})_1$ whose boundary 
\[ \partial \widetilde{\M}(X,L,\ul{Y})_1 = \M(X,L,\ul{Y})_0 \cup \M(\XX,L,\ul{Y})^-_0  \]
(here the superscript $-$ indicates reversed orientation) is the union of the original moduli space of Maslov-index-two disks $\M(X,L,\ul{Y})_0$ and the broken moduli space
of disks $\M(\XX,L,\ul{Y})_0$ that is a 
compact one-manifold with boundary.
\end{theorem}

\begin{proof} We assume the reader is familiar with the neck-stretching framework in our previous paper  \cite{vw:trop}.   Let 
\[  \ul{J} = (J_t \in \J(X), t \in [0,\infty)) \]
be a family of almost complex structures corresponding to neck-stretching along the inverse images of the proper faces in the polyhedral decomposition $\cP$ as in \cite{vw:trop}. 
Denote by $\M(X,L,\ul{J})$  denote the moduli space of pairs $(t,u)$ where
$u$ is $J_t$ holomorphic.  In general, $\M(X,L,\ul{J})$ has
boundary components  arising from disk bubbling. 

The situation is dramatically simplified
if $L$ is monotone:  Suppose, by way of contradiction, that there exist a $J_t$-holomorphic map 
$u: C \to X$ that lies in a  boundary component of  $\M(X,L,\ul{Y})$
has an edge $T_e\subset C$ of infinite length $\ell(e) = \infty$.  
Cutting at this edge produces maps 
\[ u_\pm: C_\pm \to X \] 
with boundary in $L$ with say the disk $C_+$ containing the point constraint on $L$.  Stability
forces each component $u_\pm$ to have positive area, since each has at most two semi-infinite edges.   Monotonicity implies that $I(u_-) \ge 2$. 
The total Maslov index is then 
\[ I(u)= I(u_-)+ I(u_+) \ge 4 \] 
which is a contradiction.   Therefore, 
the only boundary configurations are broken maps and configurations that are holomorphic with respect to the original complex structure.
\end{proof}

In particular, any broken disk is the end of a family of unbroken disks, and we define
the Maslov index of the broken disk to be the Maslov index of the corresponding family of unbroken disks.   The same holds for the area.  Since the Lagrangian is assumed monotone, 
the monotonicity relation holds for broken disks as well, so the Maslov index two broken disks have a fixed area.  By the Gromov compactness result in \cite{vw:trop}, there are finitely many such broken disks.  It follows that the potential  $W_{\XX,L}$ has finitely many terms.
The existence of the cobordism from Theorem \ref{thm:cobord} 
implies that  the disk potential may be computed by degeneration:

\begin{corollary}\label{cor:w} Let $L \subset X$ be a monotone Lagrangian contained in the interior 
of some piece $\Phinv(P) \subset X$. The potentials 
\[ W_L :=W_{X,L} = W_{\XX,L}: \Rep(L) \to \C^\times \] 
are equal as functions on $\Rep(L)$. 
\end{corollary} 

\begin{proof}    The statement of the Corollary follows from the fact that the signed count of ends of the cobordism in Theorem \ref{thm:cobord} is equal to zero.  
\end{proof}

\subsection{Relative maps}
\label{sec:relmap}

In this section we introduce terminology for a {\em relative map},
which is a version of a broken map obtained by cutting some tropical
edges. 

\begin{construction}
Let $u:C \to \XX$ be a broken map with domain type $\Gamma$. 
Cutting a tropical edge $T_e \subset C$ in a broken map $u$
produces two relative maps $u_+$, $u_-$, both of which contain a
cylindrical end corresponding to the edge $e$.  The graph $\Gamma_\pm$
of $u_\pm$ contains an edge $e_\pm$, called a {\em relative edge}
resulting from the cutting of $e$ that is incident on a single vertex
in $\Gamma_\pm$.
\end{construction}

\begin{definition}
\label{def:relmap}
\begin{enumerate}
    \item 
  A {\em type of relative map} is
  a type $\bGam$ a broken map with a distinguished 
 subset $\Edge_\rel(\bGam) \subset \Edge_{\to, \black}$ 
  of {\em relative semi-infinite edges }, and for each relative edge $e$,a polytope $P(e) \in \PP$ and a slope $\cT(e) \in \t_{P(e),\Z}$ as in the case with broken maps at the nodes.   
  \item 
  An {\em unframed relative map} of type $\bGam$ is a map $u:C \to \XX$  so that for each $e \in \Edge_\rel(\bGam) $, the  cylindrical end of $C_v$ corresponding to $e$ the map $u_v$ is asymptotic to a trivial cylinder of slope $\cT(e)$.  
  \item
  A {\em framed relative map} is an unframed relative map with the additional data of a  framing at each 
such  cylindrical end.
\end{enumerate}
  \end{definition}

  Let $\M_\bGam(\XX,L)$ resp. $\M_\bGam^{\fr}(\XX,L)$  denote the moduli space of unframed resp. framed relative maps  of type $\bGam$.
  Evaluation at the cylindrical ends corresponding to relative edges defines a  {\em evaluation map} as follows:  Given a relative map let $z$ be a 
  local holomorphic coordinate corresponding to the cylindrical end, compatible with the given framing.  The evaluation map for relative maps is then 
\[\ev_e^\fr: 
\M_\bGam^\fr(\XX,L) \to \XX_{P(e)} \quad u \mapsto 
\lim_{z\to 0 }z_+^{-\cT(e)}u(z) .\] 
It induces a {\em projected evaluation map }
\[ \ev_e :\M_\bGam(\XX,L) \to \XX_{P(e)}/T_{\cT(e),\C},
\quad 
u \mapsto (\pi_{\cT(e)}^\perp \circ \ev_e \]
where $(\pi_{\cT(e)}^\perp $ is the projection.

We now reconsider the situation considered at the beginning of this section in which we break an edge to obtain two relative types. We use the evaluation maps on relative markings 
to introduce moduli spaces of relative maps with constraints. 

\begin{definition}  \label{def:constraints} A {\em system of constraints} for a type $\bGam$
  consists of, for every relative edge
  $e \in \Edge_\rel(\bGam)$,
  \begin{itemize}
  \item {\rm(Map constraint)} a submanifold $\YY_e \subset 
  \XX_{P(e)}$ and
    %
  \item {\rm(Tropical constraint)}  an affine subspace $\Upsilon_e \subset P(e)^\dual$
that is parallel to the vector subspace  $\t^\rel_e  \subset \t_{P(e)}/\ \bran{\cT(e)}$.
  \end{itemize}
  Let
\begin{equation} \label{eq:con} 
\ul{\YY} = (\YY_e)_{e \in \Edge_\rel(\bGam)}, \end{equation} 
be a system of constraints for which $\bGam$ has a realization $\cT : \Ver(\bGam) \to B^\dual$ where for any relative edge $e$ incident on a vertex $v_e$ satisfies the tropical constraint
\begin{equation}
  \label{eq:tropcon}
\cT(v_e) \in \Upsilon_e.  
\end{equation}
The moduli space
of relative maps with constraints $\ul {\YY}$ is 
\[  \M_\bGam(\XX,L,\YY) = \{ u \in  \M_\bGam(\XX,L) | \ev_e(u) \in \YY_e,
  \quad \forall e \in \Edge_-(\bGam)  \} .\]
The {\em tropical symmetry groups} of a relative map with constraints is 
a collection 
  \[\ul g =(g_v)_{v \in \Ver(\Gamma)}\]
  that satisfies
  \begin{equation}
    \label{eq:unfr-symm}
  g_{v_+}g_{v_-}^{-1} \in \{z^{\cT(e)} : z \in \C^\times\} \quad \forall e=(v_+,v_-) \in \Edge_\black(\Gamma).  
  \end{equation}  
  so that the translated maps $(g_v u_v)_{v \in \Ver(\Gamma)}$ satisfy the given constraints.
  This ends the Definition.
\end{definition}

In four-dimensional examples, $\YY_e$ is either a point or the entire space
$\XX_{P(e)}$ which we call an {\em trivial constraint}. In the case of a point constraint, the 
tropical constraint $\YY^\cT_e$ is a line in the direction $\cT(e)$, and in the case of an trivial constraint, the tropical constraint is also trivial, that is, $\YY^\cT_e=P(e)^\dual$. 

In general, the counts of maps with constraints depend on the choice
of almost complex structure, in the same way that counts of
holomorphic disks depend on this choice in the closed case as well.
However, the count is invariant if all non-trivial constraints at
relative markings are point constraints, and the type $\bGam_v$ of any 
disk component $v \in \Ver_\white(\bGam)$ is {\em primitive} in the
sense that it is not decomposable into types $\bGam_1$ and $\bGam_2$ of positive area.
For example, if $(X_{P(v)}, L)$ is toric, and torus-invariant divisors
of $X_{P(v)}$ are relative divisors, then $\bGam_v$ is primitive
exactly if $v$ has exactly one tropical edge or relative marking.

\begin{proposition} \label{prop:inv}
  Let $\bGamma$ be a type of a relative broken map to $(\XX, L)$,
  whose glued type is either a disk or a sphere, and the type
  $\bGam_v$ of any disk vertex $v \in \Ver(\bGam)$ is primitive. Let
  $\YY= (\YY_e)_{e \in \Edge_\rel(\bGamma)}$
  be a collection of
  constraints each of which is either an trivial constraint or a point
  constraint, and so that the moduli space $\M_\bGam(L,\ul \YY)$ of maps of
  type $\bGamma$ with constraints $Y$, and a single point constraint on the boundary, has expected dimension zero.
  Then the count $m(\bGamma)$ of maps of type $\bGamma$ is independent
  of the choice of almost complex structure,  perturbation data, and generic point constraints $\ul \YY$. Moreover, for any two constraint data $\ul \YY_0$, $\ul \YY_1$, there is a cobordism between the moduli spaces $\M_\bGam(L,\ul \YY_0)$, $\M_\bGam(L,\ul \YY_1)$. 
   The tropical constraints $\{\Upsilon_e\}_e$ are assumed to be fixed.  
\end{proposition}

\begin{proof}
The invariance statement in the Proposition is a standard cobordism argument: Given two   data $(\ul{P}_b, \ul \YY_b)$ for $b \in \{ 0,1 \}$ let
  $(\ul{P}_t, \ul \YY_t), t \in [0,1]$ be a generic homotopy between
  them.  Let $\ti{\M}_{\bGamma}(L,Y)$ be the parametrized moduli space
  for the family.  We obtain a compact oriented cobordism between the
  moduli spaces for $(\ul{P}_b, \ul \YY_b)$ for $b \in \{ 0,1 \}$,
  because bubbling is ruled out as follows: The primitiveness
  assumption ensures that there is no disk bubbling. Formation of
  non-tropical interior nodes cuts down the expected dimension by two.  Therefore, formation of interior nodes does not occur if $(\ul{P}_t, \ul \YY_t)$ is generically chosen. Since the index of the maps in the parametrized moduli space
  is $1$, by Remark \ref{rem:troprel} \eqref{rem:troprel2},
  there cannot be a limit map with an additional tropical node.
\end{proof}

\begin{remark}{\rm(On tropical graphs of relative maps)} \label{rem:troprel}
  \begin{enumerate}
  \item  If a relative map type $\bGam$ has associated moduli space
  $\M_\bGamma(\XX,L)$ of 
    expected dimension zero, then its tropical symmetry group
  is finite, as explained in Lemma 4.39 of \cite{vw:trop}.  
  Consequently the tropical graph of $\bGam$ is {\em rigid}. That is, 
  the tropical graph of $\bGam$ has a unique set of vertex positions that satisfy the
  tropical constraint \eqref{eq:tropcon} at relative markings.
\item \label{rem:troprel2} The tropical behavior in the limit of relative maps is similar to the case of broken maps without relative markings.
  Suppose a sequence $\{u_\nu\}_\nu$ of relative maps (all whose non-trivial constraints are point constraints) with a tropical graph $\Gamma$ converges to a limit $u_\infty$ with a tropical graph $\Gamma' \neq \Gamma$, then, the dimension of the tropical symmetry group satisfies 
  \[ \dim(T_\trop(\Gamma')) > \dim(T_\trop(\Gamma)) .\]
  The proof is the same as the non-relative case in \cite[Theorem 8.3]{vw:trop}, with the additional feature that both
  $\Gamma$ and $\Gamma'$ satisfy tropical point constraints.
  In particular, if $\Gamma$ is rigid, $\dim(T_\trop(\Gamma')) \geq 2$. Since the tropical symmetry group acts freely on the moduli space of maps, $\dim \M_{\bGam'}(L, \ul \YY) \geq 2$, and since formation of tropical nodes does not affect the expected dimension of moduli spaces, we conclude that $\dim \M_{\bGam}(L, \ul \YY) \geq 2$.
  \footnote{The tropical constraint \eqref{eq:tropcon} takes a more complicated form when a non-trivial constraint $\YY_e \subset \XX_{P(e)}/\bran{\cT(e)}$ has positive dimension, and the constraint needs to be valid for maps in the compactification $\ol \M(L,\ul \YY)$. Indeed, for maps in the compactification $\ol \M(L,\ul \YY)$, the vertex containing the relative  marking $z_e$ may lie in $P^\dual$ for some $P \subset P(v_e)$. A positive dimensional constraint $\YY_e$ is required to be a manifold with cylindrical ends,
    and the tropical constraint is defined not just in $P(e)^\dual$ but also in polytopes $P^\dual \supset P(e)^\dual$ where $\YY_e$ has a cylindrical end in the $P$-cylindrical end of $\XX_{P(e)}$. We do not need to consider such constraints in this paper.}
    \end{enumerate}
\end{remark}
In the following result, we show that the moduli space of broken maps
associated to a polyhedral decomposition of a symplectic four-manifold
is honestly a product of moduli spaces associated to the pieces,
rather than a fiber product.  Such a product description is special to
the genus zero, dimension four case; in general such a result can be
expected only after further degeneration as explained in our previous
work \cite{vw:split}.  Denote by $L_{P(v)}$ the intersection
$L \cap \Phinv(P(v))$, which is either trivial or equal to $L$.  The
result involves the automorphism group of the map type which we recall
here:

\begin{definition}
An {\em automorphism of a map type} $\bGam$ is a graph
automorphism $\phi : \bGam \to \bGam$ for which $P(v)=P(\phi(v))$ for
all vertices, the homotopy classes of the maps at $v$ and $\phi(v)$
are the same, and $\phi$ preserves the ordering of boundary markings
$e \in \Edge_{\to, \white}(\bGam)$ and interior markings
$\Edge_{\to, \black}(\bGam)$. The group of automorphisms of $\bGam$ is
denoted by $\Aut(\bGam)$.  
\end{definition}

\begin{theorem}  \label{thm:split}
{\rm(Distribution of constraints)}
Let $\ul{Y} = (p \in L)$ be a point constraint, and let $\bGam$ be a rigid map type for the constraint $\ul{Y}$ so that every subtype $\bGam_v$ consisting of a vertex $v$ and the adjacent 
edges $e$ is primitive.   There exists a system of constraints (see Definition \ref{def:constraints}) 
 \[ \ul \YY = (\ul \YY_v, \quad v \in \Ver(\bGam)) \] 
for the graphs $\bGam(v) \subset \bGam$ associated to the vertices $v$ so that the moduli space $\M_\bGam(\XX,L,\ul{Y})$ of broken maps of type $\bGam$ admits a compact oriented cobordism 
\[ \M_\bGam(\XX,L,\ul{Y}) \sim n(\bGam) \left(\prod_{e \in \Edge_\trop(\bGam)}\mu_e \right) \left(\prod_{v \in \Ver(\bGam)} \M_{\bGam(v)}(\XX_{P(v)},L_{P(v)},\ul \YY_v) \right)\]
to the product of moduli spaces over vertices $v$ of $\bGam$, 
where $\mu_e \in \Z_+$ is the lattice length of the edge slope $\cT(e)$, and
\[n(\bGam):=\tfrac {d_\black(\bGam)!} {|\Aut(\bGam)| \left( \prod_v d_\black(v)! \right)}  \]
Any tropical edge $e$ has one end-point $v_+$ where $\ul \YY_{v_+}$ has a point constraint at the lift of the node $w_e$, and another end-point $v_-$ where the constraint $\ul \YY_{v_-}$  at the node $w_e$ is trivial.
\end{theorem}
\begin{proof}[Proof of Theorem \ref{thm:split}]   The proof is an inductive argument based on the fact that for each edge, the moduli space of relative maps on one side of the edge must be rigid. 
  First, we consider the case of replacing the matching condition at a single edge $e$ with a constraint at one of the lifts of the node $w_e$. Assuming that cutting the edge $e$ in $\bGam$ produces relative map types $\bGam_+$, $\bGam_-$,
  the moduli space of broken maps $\M_\bGam(\XX,L,\ul{Y})$ is the inverse image $\ev_e^{-1}(\Delta)$ of the diagonal where
  \[\ev_e=(\ev_{e_+}, \ev_{e_-}), \quad \ev_{e_\pm}: \M_{\bGam_\pm}(\XX,L,\ul{Y}) \to \XX_{P(e)}/\bran{\cT(e)} \]
  is the projected evaluation map.
  The moduli spaces $\M_{\bGam_\pm}(\XX,L,\ul{Y})$ are even dimensional, since they either involve spheres, or disks with their boundary on an oriented Lagrangian, and matching conditions are defined on even-dimensional manifolds. 
  Therefore, one of them, say, $\M_{\bGam_+}(\XX,L,\ul{Y})$ is two-dimensional and $\M_{\bGam_-}(\XX,L,\ul{Y})$
  is zero-dimensional.
  Then, we claim that the moduli space of broken maps is bijective to a product
  \[\M_\bGam(\XX,L,\ul{Y}) \simeq \mu_e n(\bGam_+,\bGam_-) \M_{\bGam_+}(\XX,L,\YY_+) \times \M_{\bGam_-}(\XX,L,\YY_-),  \]
  where  $\YY_\pm$ are obtained from $\ul{Y}_\pm$ by adding $\YY_{e_\pm}$ at the broken edges,  $\YY_{e_-}$ is the trivial constraint, and $\YY_{e_+}$ is a generic point constraint. To prove the claim, consider a map
  \[ (u_+,u_-) \in \M_{\bGam_+}(\XX,L, \ul{\YY}_+) \times \M_{\bGam_-}(\XX,L,\ul{\YY}_-) . \] 
  Via the cobordism in Proposition \ref{prop:inv},
  we may replace the point constraint $\YY_{e_+}$ by the point $\ev_{e_-}(u_-)$, and under the cobordism
  \[ \M_{\bGam_+}(\XX,L,Y_{e_+}) \sim \M_{\bGam_+}(\XX,L,\ev_{e_-}(u_-)) \] 
  we may replace the map $u_+$ with a map  a map $u_+'$. Then, $(u_+',u_-)$ gives rise to $n(\bGam_+, \bGam_-)$ unframed broken maps, and each of these has
  $\mu_e$ framings. 
  Theorem \ref{thm:split} follows by applying   the above reasoning inductively, cutting one edge at a time. The product of the factors $n(\bGam_+,\bGam_-)$ from all the steps is equal to 
$n(\bGam)$, which is the number of ways of distributing $d_\black$ markings onto vertices $v \in \Ver(\bGam)$ containing $d_\black(v)$ markings; two distributions are considered the same if they are related by an automorphism of $\bGam$.
\end{proof}

\begin{example}  We consider the polyhedral decomposition 
in Figure \ref{fig:poly}, and one of the Maslov-index-two disks contributing to the potential, which is an rigid element of the moduli space of broken maps with a point constraint on the Lagrangian.  The constraints are distributed as shown in Figure \ref{fig:distribconstraints}:
the moduli space for that type is a product of moduli spaces on the pieces with an additional constraint at one of the cylindrical ends on the piece containing the Lagrangian, and another additional constraint on moduli space for the piece meeting the toric boundary.
\end{example}

\begin{figure}[ht]\begin{center} 
\scalebox{.3}{\includegraphics{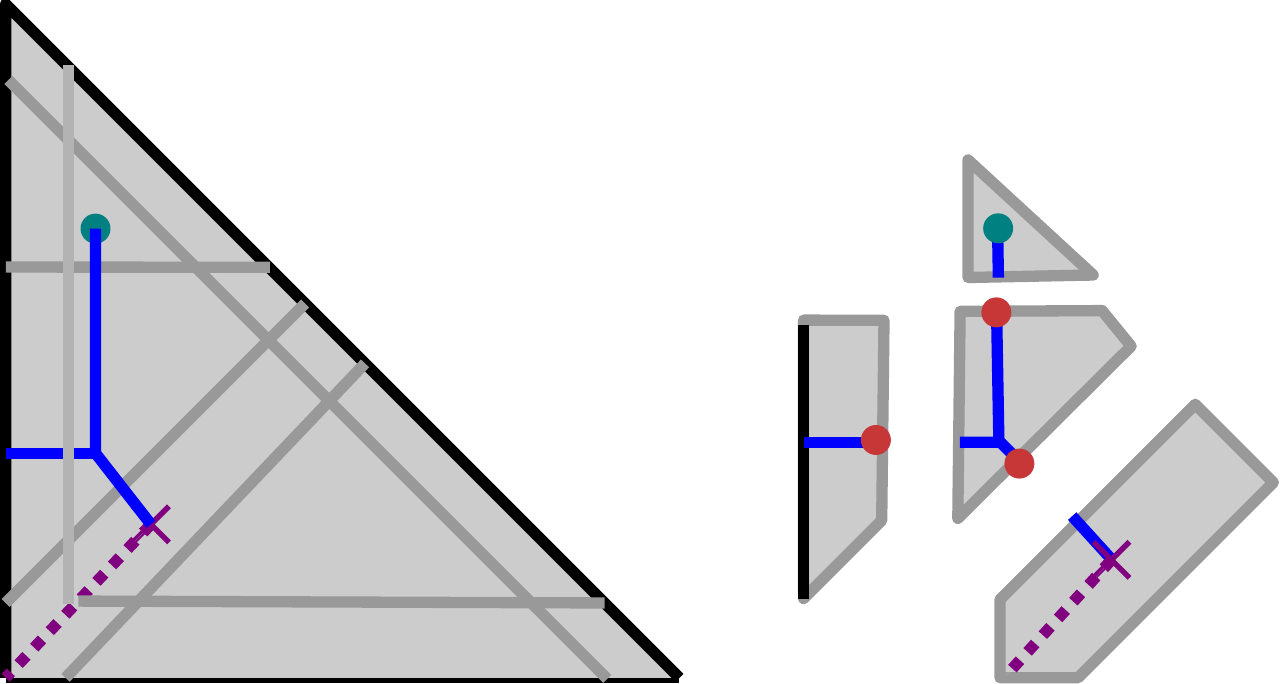}}
\end{center} 
\caption{The moduli space as a product of moduli spaces for its pieces} 
\label{fig:distribconstraints}
\end{figure}

The distribution of constraints give  orientations on the edges of the tropical graph. In the statement of Theorem \ref{thm:split}, for an edge $e=(v_+,v_-)$, if an end-point $v_+$ has a point constraint and the end-point $v_-$ has an trivial constraint, we orient the edge $e$ from $v_-$ to $v_+$. This system of orientations on the tropical graph is called the {\em constraint orientation} and it satisfies the following property.

\begin{lemma}\label{lem:orient}
  Let $\Gamma$ be the tropical graph of a rigid disk.
  At any vertex $v$ of $\Gamma$ with $P(v) \notin \PP_\partial$,
    there is exactly one outgoing edge $e \in \Edge(\Gamma)$ with respect to the constraint orientation. 
\end{lemma}
\begin{proof}
  The result follows from index considerations. The Maslov index of a
  disk or sphere is equal to twice the number of intersections with
  toric divisors, so for a rigid disk, exactly one of these
  intersections is unconstrained.
\end{proof}

\subsection{Intersecting polyhedral decompositions}

Two transversely intersecting polyhedral decompositions combine to
give a new polyhedral decomposition.  The combined decomposition has a
family of dual complexes. We outline this construction from
\cite[Example 3.30]{vw:trop}:

\begin{definition} 
\begin{enumerate}
    \item A pair of polyhedral decompositions
$\PP_0$, $\PP_1$ of $\t^\dual$ intersect {\em transversally} if any
pair of polytopes $P_0 \in \PP_0$, $P_1 \in \PP_1$ intersect
transversely. 
\item For a transversally intersecting pair, we define a {\em combined} polyhedral
decomposition
\[\PP:=\PP_0 \cap \PP_1:=\{P_{01}:=P_0 \cap P_1 : P_0 \in \PP_0, P_1 \in \PP_1\}. \]
\item 
For any minimal dimensional polytope $P_{01}$ and a constant
$\rho>0$, there is a family of dual polytopes
$(P_{01}^\dual)_\rho:=P_0^\dual \times \delta P_1^\dual$, which glue
(as in \eqref{eq:sim}) to yield a family of dual complexes
\begin{equation}
  \label{eq:Bdual-rho}
B^\dual_\rho:= (\cup_{P \in \PP}P^\dual_\rho)/\sim 
\end{equation}
\item  By assumption, each $b \in B^{\foc}$ is contained in the interior of a polytope 
$P \in \PP$ of maximal dimension, and we denote by $b^\dual := P_b^\dual$ the corresponding {\em dual focus-focus singularity} in $B^\dual$.
\end{enumerate}
\end{definition}

\begin{example}
  The simplest instance of this construction is a multiple cut consisting of two single cuts shown in Figure \ref{fig:rect}. The dual complex is a rectangle, and we obtain a family of possible complexes by varying the ratio of the sides of the rectangle.
  \end{example}

\begin{figure}[ht]\begin{center} 
\scalebox{.8}{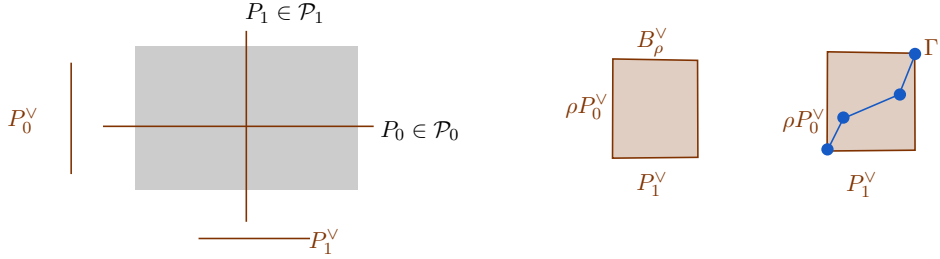}
\end{center} 
\caption{The combination $\PP$ of the polyhedral decompositions $\PP_0$, $\PP_1$ has a family of dual complexes $\{B^\dual_\rho\}_\rho$. The tropical graph $\Gamma$ is not realizable in $B^\dual_\rho$ if $\rho$ is large enough.}
\label{fig:rect}
\end{figure}

If the parameter $\rho$ is large enough, we may view the combined multiple cut $\PP$ as performing the cut $\PP_0$ followed by $\PP_1$. This is equivalent to saying that any tropical graph $\Gamma$ in $B^\dual_\rho$ can be transformed to a tropical graph in $B^\dual_0$ by collapsing
all edges whose directions have a non-zero in the $\t_{P_1}$-direction for any $P_1 \in \PP_1$.
The edges left uncollapsed are edges in $B^\dual_0$ and are defined as follows. 

\begin{definition}
  Let $\PP:=\PP_0 \cap \PP_1$ be the combination of two transversely
  intersecting multiple cuts $\PP_0$ and $\PP_1$. An edge $e$ of a
  tropical graph $\Gamma$ in $B^\dual_\rho$ (from \eqref{eq:Bdual-rho}) is a {\em $B^\dual_0$-edge} if
  $P(e)=P_0 \cap P_1$, $P_i \in \PP_i$, and
  $\cT(e)=(\cT_0,0) \in \t_{P_0} \times \t_{P_1}$. 
\end{definition}

Note that the condition in the Definition is automatically satisfied in the case when $P_1$ is top-dimensional and so, $\t_{P_1}=\{0\}$.

\begin{proposition}\label{prop:b0edge}
  Given an area bound $E_0$, there exists $\rho_0$ for which the
  following holds.  For a map type $\bGam$ with area $\leq E_0$ whose
  tropical graph $\Gamma$ is realizable in $B^\dual_\rho$ (from \eqref{eq:Bdual-rho}) for some 
  $\rho \geq \rho_0$,
  \begin{enumerate}
  \item
\label{part:b0edge1}
    $\Gamma$ is realizable in $B^\dual_\rho$ for all
    $\rho \geq \rho_0$, and 
  \item \label{part:b0edge2}
  there is a tropical graph $\Gamma_0$ in
  $B_0^\dual$ and an edge collapse morphism $\Gamma \to \Gamma_0$ that
  collapses all edges that are not $B_0^\dual$-edges.
  \end{enumerate}
\end{proposition}
\begin{example}  The graph $\Gamma$ in Figure \ref{fig:rect} does not have
any $B_0^\dual$ edge in the path from $v_0$ to $v_1$. Since $v_0$ and
$v_1$ lie in different polytopes of $\PP_0$, the proposition implies
that the graph is not realizable in $B^\dual_\rho$ for large enough
$\rho$. This conclusion is easy to deduce from the figure.
\end{example}

\begin{proof}[Proof of Proposition \ref{prop:b0edge}]
We claim that  there are a finite number of tropical graphs $\Gamma$ corresponding to broken
  maps $u$ satisfying a given area bound $A(u) \leq E_0$, and that are realizable in some
  $B^\dual_\rho$.  This claim follows from the proof of Proposition
  T-\ref{T-prop:finno}, which bounds the number of tropical graphs in
  a fixed dual complex.  Indeed, the proof of Proposition
  T-\ref{T-prop:finno} applies even if the parameter $\rho$ varies,
  because the proof does not rely on realizations of tropical graphs. The proof of part
  \eqref{part:b0edge1} now follows, because the set of realizations of a given tropical graph $\Gamma$ is convex.  Therefore the set of $\rho$ for which $\Gamma$ is 
 representable in $B^\dual_\rho$ is  an interval. 

 For part \eqref{part:b0edge2},  consider a tropical graph $\Gamma$ that is representable in
 $B^\dual_\rho$ for $\rho \geq \rho_0$.    Consider the continuous family of complexes $\{\frac 1 {\rho} B^\dual_{\rho}\}_\rho$ which limits to $B^\dual_0$.  By assumption, there is a 
  sequence of realizations 
  \[ \iota_\nu : \Gamma \to \frac 1 {\rho_\nu} B^\dual_{\rho_\nu} .\] 
  By compactness, after passing to a subsequence, we may assume
  that $\iota\nu$ converges to a limit $\iota_\infty : \Gamma \to B^\dual_0$. In the limit $\iota_\infty$, any edge $e$ of $\Gamma$ that is not a $B^\dual_0$-edge is mapped to a point. Indeed, suppose that 
  $P(e)=P_0 \cap P_1$, $P_0 \in \PP_0$, $P_1 \in \PP_1$ so that $\t_{P(e)}=\t_{P_0} \times \t_{P_1}$, and the $\t_{P_1}$-component of $\iota_\nu(e)$ is non-zero. 
The length of $\iota_\nu(e)$
is bounded by $\frac c {\rho_\nu}$ for a uniform constant $c$.  Therefore, 
\[ \lim_{\nu \to \infty} \ell( \iota_\nu(e)) = 0  . \] 
Thus, if $\Gamma_0$ is defined by collapsing all the $B_0^\dual$-edges in $\Gamma$, $\iota_\infty$ is a realization of $\Gamma_0$ in $B_0^\dual$. 
\end{proof}

\section{Curves in the elementary pieces}
\label{sec:piecesec}

By an {\em elementary piece}, we mean a piece in the broken symplectic
manifold that is either a toric manifold possibly containing the
Lagrangian torus fiber, or is an almost toric manifold with a
collection of focus-focus singularities along the same branch cut.
We will show in Section \ref{sec:atinfinity} that almost toric del Pezzos
possess polyhedral decompositions with the following property:  For any broken disk contributing to the potential the tropical graph has ``collisions in the interior''  in the sense of Definition \ref{def:coll-interior}.
In the case of holomorphic spheres in toric varieties, the sphere with two intersections with the boundary divisors is called a {\em holomorphic cylinder}, and the one with three intersections is called a {\em holomorphic pair of pants}. 

The disk potential is a sum over the set of rigid tropical graphs, and each summand is a product of
the curve counts corresponding to each of the vertices. 
We precisely define the weighted curve count $m(v)$ on each vertex.  On a broken map
type $\bGam$ we apply the distribution of constraints result (Theorem
\ref{thm:split}) so that the moduli space is a (weighted) product of
the moduli spaces of relative maps
$\M_{\bGam(v)}(\XX_{P(V)}, L_{P(v)}, \ul \YY_v)$. 

\begin{definition}For any vertex
$v \in \Ver(\bGam)$, define, for some particular perturbation data, the {\em count of unframed relative maps} as
\[m(v)^\unfr:= \sum_{u \in  \M_{\bGamma(v)}(\XX_{P(v)},L_{P(v)},\ul \YY_v)}
\frac{    \Hol_L([\partial u_v]) \eps(u)}{d_\black(\Gamma(v))!},\]
where $ \Hol_L([\partial u_v]) $ is from \eqref{eq:wxl} and is equal to $1$ if $u$ is a sphere. 

We adopt the convention of absorbing the framing symmetry factor of an edge $e$
into the vertex $v$ where there is a point constraint on $u_v$ at the edge $e$, 
and define the {\em count of framed relative maps}
\begin{equation}
  \label{eq:mv}
m(v):= \left(\prod_e |\cT(e)| \right) m(v)^\unfr, 
\end{equation}
where the product is over edges $e$ incident on $v$, and for which $\YY_{v,e}$ is a point constraint. 
    \end{definition}

    \begin{remark}   Under a primitivity assumption, the counts $m(v)$
    are independent of the choice of perturbation by the argument of Proposition \ref{prop:inv}.
    \end{remark}

\begin{proposition}The disk potential of a monotone Lagrangian $L$ compatible
with the polyhedral decomposition $\cP$ above 
\[W_{X,L} = \sum_\bGam\tfrac 1 {|\Aut(\bGam)|}(\prod_v m(v)), \]
where the sum is over all rigid map types $\bGam$.
\end{proposition}

\begin{proof}  By monotonicity, the area of any rigid broken disk contributing to the potential is $1$, and therefore the exponent of $q$ is $1$.  The statement of the Proposition is immediate from Theorem \ref{thm:split}.
\end{proof}

Our goal is now to compute the multiplicities of the vertices.   In theory, all of the multiplicities already appear in the literature, but in various foundational schemes.  We give proofs for each case, since none are particularly long.

\subsection{Holomorphic spheres in toric varieties}
\label{sec:spheresintoric}

We first count holomorphic spheres in toric varieties that intersect boundary divisors at isolated points.
The holomorphic spheres we consider form a component of a broken map,
and are therefore modelled on a graph $\bGam_v$ single vertex $v$  and a set $\Edge_\rel(\bGam_v)$ of relative edges.

\begin{notation}  In the notation of Section \ref{subsec:bmap}, such maps are defined on punctured domain curves
\[u_v : C_v^\circ \to \XX_{P(v)}, \quad C_v^\circ = \P^1 \bs \{z_e\}_{e \in \Edge_\rel(\bGam_v)}. \]
We suppose that $P(v)$ is a polytope not containing any focus-focus singularity, 
or the image of the Lagrangian. If $P(v)$ does not intersect $\partial \Delta$ then
$\XX_{P(v)} \simeq (\C^\times)^n$; otherwise it additionally contains a
divisor at infinity corresponding to $X_{P(v)} \cap \Phinv(\partial \Delta)$. When $u_v$ is viewed as a map to the toric compactification $\ol \XX_{P(v)}$, $u$ has removable singularities at punctures. Equivalently, in the neighborhood of a relative marking $z_e$, the map $u$ is asymptotically close at the puncture to a trivial cylinder
\[z \mapsto (c_1(z-z_e)^{\cT(e)_1}, \dots, c_n(z-z_e)^{\cT(e)_n})\]
for some $c_i \in \C^\times$, where 
\[ \cT(e):=(\cT(e)_1,\dots, \cT(e_n)) \in \Z^n \] 
is the {\em direction} of the relative edge $e$.
\end{notation}

The directions $\{\cT(e)\}_e$ are part of the data of the map type $\bGam$ of the map $u$, and they satisfy a {\em balancing condition} as follows:

\begin{lemma} 
  Let $u_v : \P^1 \bs \{z_e\}_e \to (\C^\times)^n$ be a relative map with a single component in a toric piece $\XX_{P(v)} \simeq (\C^\times)^n$ that does not intersect the boundary divisor of $X$.  The edge directions $\cT_e \in \Z^n$ at the relative markings $z_e$ satisfy
  \begin{equation}
    \label{eq:balance}
    \sum_e \cT(e)=0.    
  \end{equation}
\end{lemma}

\begin{proof}  The statement of the lemma is a computation with fundamental groups. 
  Let $\gamma_e$ be a loop around $z_e$. The product of the loops
  $\prod_e[\gamma_e]$ is the identity in $\pi_1(\P^1 \bs \{z_e\}_e)$,
  and therefore, 
  \[ (u_v)_* \left(\prod_e[\gamma_e] \right)=\Id \in \pi_1((\C^\times)^n) . \] 
  Since
  \[ (u_v)_*([\gamma_e])=\cT_e \in \pi_1((\C^\times)^n) \]
  the balancing condition follows.
\end{proof}

\begin{remark}
  We also consider holomorphic cylinders intersecting the
  boundary. Such a map lies on a piece $\XX_P$ that intersects the
  boundary divisor $\Phinv(\partial \Delta)$ of $X$, has a single
  relative marking $z_e$ of slope $\cT(e)$ and an intersection with
  the divisor $\Phinv(\partial \Delta) \cap \XX_P$. The argument in
  the proof of the balancing condition implies that $\cT(e)$ is the
  outward normal vector of $\Phinv(\partial \Delta) \cap \XX_P$.
\end{remark}



\begin{proposition}
  \label{prop:spheres}
  Let $X_{P(v)}$ be a projective toric variety with $\dim(X)=4$. Let $\bGam$ be a relative map type with a single vertex $v$ with a collection of constraints $\ul Y$ that are one of the following:
  \begin{enumerate}
  \item {\rm(Holomorphic cylinder in the interior)}  The vertex $v$ has
    two relative edges $e$, $e'$ with slopes $\cT(e)$, $-\cT(e)$, $\XX_{P(v)}$ does not intersect the boundary of $\Phinv(\partial \Delta)$,  and $\ul \YY$ consists of a point constraint at the relative marking $z_e$, or
  \item {\rm(Holomorphic cylinder intersecting the boundary)}
    the vertex $v$ has a relative edge of primitive slope $\cT(e)$, and a single intersection with a boundary toric divisor $\XX_{P(v)} \cap \Phinv(\partial \Delta)$
    that is not a relative divisor, and $\ul \YY$ consists of a point constraint at the relative marking $z_e$, or
  \item {\rm(Holomorphic pair of pants in the interior)} the vertex $v$ has three relative markings $z_1$, $z_2$, $z_3$, whose slopes $\cT_1$, $\cT_2$, $\cT_3$ sum to zero, $\XX_{P(v)}$ does not intersect the boundary of $\Phinv(\partial \Delta)$, 
    and $\ul \YY$ consists of point constraints at two of the markings $z_1$, $z_2$.
  \end{enumerate}
  In all the cases, the moduli spaces are zero-dimensional, and 
  \begin{equation}
    \label{eq:mv-sphere}
    m(v)=
    \begin{cases}
      1, &\text{$\bGam_v$ is a holomorphic cylinder,}\\
      |\det(\cT_1, \cT_2)|, &\text{$\bGam_v$ is a holomorphic pair of pants.}
    \end{cases}
  \end{equation}
\end{proposition}

\begin{proof}
  For a toric variety equipped with the standard alost complex structure, the moduli spaces of spheres intersecting boundary divisors at isolated points is transversely cut out, and the evaluation map at a single marking is submersive. By a similar proof, the evaluation maps at two markings with non-parallel slopes is also submersive.
  
  Suppose that the almost complex structure is standard (and
  domain-independent).  For a holomorphic cylinder with $|\cT_e|=d$,
  the map $u_v$ is a $d$-fold cover of its image.  On any such curve, markings can be
  labelled in $\frac{d_\black(v)}{d}$ distinct ways, and therefore,
  every curve contributes $\frac 1 d$ to $m(v)^\unfr$.  Thus 
  \[ m(v)=|\cT(e)| m(v)^\unfr=1 . \] 

  The curve count for  holomorphic pairs of pants is
  similar to the counts in Mikhalkin \cite[Theorem 1]{mikhalkin} and 
  Nishinou-Siebert \cite[Proposition 8.8]{ns}. We outline the proof in
  Lemma 4.4 in Venugopalan-Woodward \cite{vw:trop}. For a map type $\bGam$
  of holomorphic pair of pants with constraints $\ul \YY$ on two of the markings, 
  any two maps in the moduli space 
  $\M_\bGam(X, \ul \YY)$ are related by an
  element $g$ of the tropical symmetry group
  $T_\trop(\bGam) \simeq (\C^\times)^2$ that does not alter the
  projected evaluation at the relative markings,
  A tropical symmetry $g \in T_\trop(\bGam)$ does not alter the projected evaluations at
   $z_1$, $z_2$ exactly if 
   $g \in T_{\cT_j, \C}$ for $j=1,2$.
  That is, there exist $z_j \in \C^\times$, $j=1,2$ such that
  \begin{equation}
    \label{eq:gz}
    g=z_1^{\cT_1}=z_2^{\cT_2}.   
  \end{equation}
  By the calculation in the proof of \cite[Lemma 4.4]{vw:trop},
  \eqref{eq:gz} has $\det(\cT_1, \cT_2)$ solutions for
  $(g,z_1,z_2)$. Given a solution $(g,z_1,z_2)$,
  the pairi $(g,\theta_1 z_1,\theta_2 z_2)$ is also a solution where
  $\theta_j^{|\cT_j|}=1$ for $j=1,2$, and so, there are
  $\frac {|\det(\cT_1, \cT_2)|}{|\cT_1||\cT_2|}$ distinct values of
  $g \in T_{\cT_1, \C} \cap T_{\cT_2,\C}$, Therefore,
  \[ m(v)^\unfr=\frac {|\det(\cT_1, \cT_2)|}{|\cT_1||\cT_2|}, \quad m(v)=|\det(\cT_1, \cT_2)| . \] 
\end{proof}

\subsection{Disks in toric varieties}

Next we deal with pieces of the degeneration containing the Lagrangian.
Let $X$ be a compact toric variety  and $L \subset X$ a Lagrangian torus orbit.  Let $Y = \{ y \} \subset L$ be a generic point constraint.  

\begin{lemma} Suppose $L$ is a moment fiber in a toric manifold $X$
equipped with its standard complex structure, and $D_i \subset X$
be a boundary divisor.  The number of rigid disks $u:  C \to X$ with boundary in $L$ passing through a generic point $\ul{Y} = \{ Y \}, 
Y = \{ y \in L \} $ and  with a single intersection with  $D_i$ is $m(X,\ul{Y}) = 1$.
\end{lemma}

\begin{proof}
The proof is similar to that for the statements about spheres in the previous subsection, 
but now using the description of holomorphic disks as  Blaschke products 
as described in Cho-Oh \cite{chooh:fano}. 
Suppose that $X$ is given as the git quotient of a vector space $\ti{X} = \C^k$ by the linear action of a torus $G \subset (\C^\times)^k$.   
Any disk in $X$ with boundary in $L$ lifts to a disk in $\ti{X}$ with boundary in
some torus $\ti{L} = (S^1)^k$ mapping to $L$ under the quotient map. 
Such a holomorphic disk is given by a product
\begin{equation} \label{fig:blaschke} 
u(z) = \left(   c_i \prod_{j=1}^{d_i} \frac{z - a_{i,j}}{1 - \ol{a}_{i,j} z}   \right)_{i=1}^k .\end{equation}
By \cite[Theorem 5.1]{chooh:fano}, the Maslov index of such a product is the sum 
\[ I(u) = \sum_{i=1}^{2k}  2d_i .\]
The only rigid such disks meeting $D_i$ are those Blaschke products
with $d_i = 1$.  This justifies \eqref{mv:toricdisk}.
\end{proof}

This completes the proof of Theorem \ref{thm:potthm}.

\begin{figure}[ht]\begin{center} 
\scalebox{.7}{\includegraphics{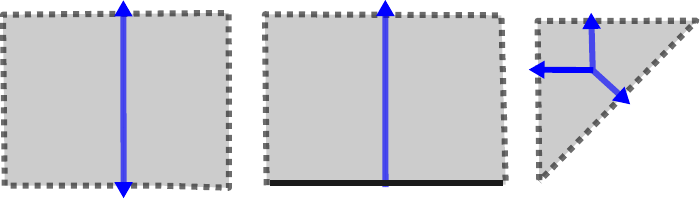}}
\end{center} 
\caption{Holomorphic curves in the toric local models} 
\label{fig:p_cyl}
\end{figure}

\begin{proposition}  \label{prop:noconstant} Suppose that $X$ is a compact symplectic manifold  of arbitrary dimension, $L$ is contained in a piece $\XX_{P_0}$ of $\XX$ which is a toric variety, that is, 
$P_0 \in \cP$ is a  polytope of maximal dimension and 
$\Phi(L)$ is contained in $P_0$, and $L$ is a toric moment fiber.  Then the disk potential $W_L$ has no constant term, that is, the coefficient of $y^0$
in $W$ vanishes. 
\end{proposition}   

\begin{proof}  The proof of the statement of the Proposition 
follows from the non-existence of Blaschke products with trivial boundary class. 
That is, let $u: C \to X_{P_0}$ be a holomorphic in the toric variety $X_{P_0}$ with boundary in a Lagrangian torus fiber.  By the classification in Cho-Oh \cite{chooh:fano}, the boundary class $[ u(\partial C)] \in H_1(L)$ is non-trivial.
\end{proof}

\subsection{Spheres meeting the focus-focus singularities}

 In this section we give the Bryan-Pandharipande formula from \cite{bryan:lgw} for the holomorphic spheres in pieces containing the focus-focus singularities.   In elementary polyhedral decompositions, the piece $\XX_P$ containing a focus-focus singularity has moment polytope of the form shown in Figure \ref{fig:ff-general} or Figure \ref{fig:ffd}.   The compactification $\ol \XX_P$ obtained by adding boundary divisors corresponding to the facets of $\Phi(\XX_P)$ is a singular fibration
 \begin{equation}
   \label{eq:sing-fib}
   f : \ol \XX_P \to \P^1.  
 \end{equation}
 The fibers $f^{-1}(0)$ and $f^{-1}(\infty)$ are orbifold spheres. For a focus-focus singularity $x \in \XX_P$, the fiber $f^{-1}(f(x))$ is a nodal sphere with a node at $x$. Any other fiber of $f$ is a $\P^1$. We denote by $S_0$ and $S_\infty$ the sections of $f$ that are inverse images of boundary divisors of $\ol \XX_P$.   A special case of such a cut space occurs when $f^{-1}(0)$ and $f^{-1}(\infty)$ are smooth spheres, and one of the sections $S_0$ is a $(-1)$-sphere. In this case $\ol \XX_P$ is isomorphic to a point blow-up of $\P^1 \times\P^1$. 
 
 \begin{figure}[ht]
  \begin{center}
    \scalebox{.8}{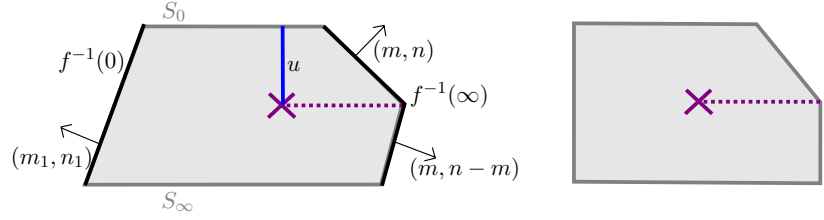}
  \end{center}
  \caption{Left: Moment image of a cut space $X_P$ containing a focus-focus singularity, and a holomorphic curve $u$ in $X_P$. Right: The particular case when $\ol \XX_P=\Bl(\P^1 \times \P^1)$, and $1=m=-m_1$, $n-m=-n_1$.}
  \label{fig:ff-general}
\end{figure}

\begin{lemma}\label{lem:nodal-sph}
  Let $u : \P^1 \to \ol \XX_P$ be a holomorphic map with a single intersection point $u^{1}(\ol \XX_P - \XX_P)$ with the boundary divisor,
  given by the inverse images of the four facets in 
  Figure \ref{fig:ff-general}.  Then the image of $u$ lies on one of the nodal fibers of the map $f$ (in \eqref{eq:sing-fib}).
\end{lemma}

\begin{proof}  The proof is a computation involving intersection numbers with the fibers.    The projection $f \circ u : \P^1 \to \P^1$ is holomorphic and
  intersects at most one of $0, \infty \in \P^1$. Therefore,
  $f \circ u$ is constant, and the image of $u$ lies in a fiber of
  $f$.  Since regular fibers of $f$ intersect two of the boundary
  divisors of $\ol \XX_P$, the only possibility is that $u$ maps to one
  of the spheres in a nodal fiber.
\end{proof}

 \begin{figure}[ht]
  \begin{center}
    \scalebox{.8}{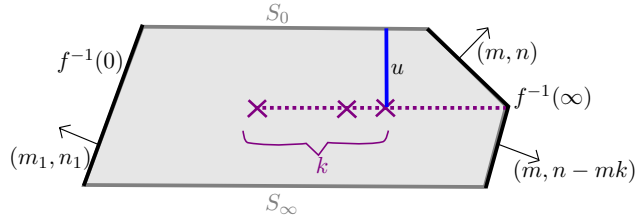}
  \end{center}
  \caption{Moment image of a cut space $X_P$ with multiple focus-focus singularities whose shear matrices have the same eigen-direction, and a holomorphic curve $u$ in $X_P$.}
  \label{fig:ffd}
\end{figure}

In the rest of the section, we count the number of holomorphic spheres in a focus-focus piece that have a single relative marking, or in other words, the sphere has a single intersection with boundary divisors. For simplicity of notation, we assume that $\XX_P$ has a single focus-focus singularity, but the proof in the general case is similar.  
Denote by 
\[ \beta_0 \text{ resp.}\  \beta_\infty \in H_2(\ol \XX_P) \] 
the homology class of the sphere in 
the nodal fiber that meets the section $S_0$ resp. $S_\infty$. We consider one of these, say $\beta_0$, since the other can be analyzed in a similar way. For any positive integer $d$, denote by 
\[ m(\XX_P, d\beta_0, \ul Y) \in \Q \]
the count of relative spheres in $\ol \XX_P$ in the class $\d\beta_0$ that have a single intersection (of maximum multiplicity) with the divisor $S_0$. The constraint $\ul Y$ has a single element $S_0$, indicating that the evaluation map at the relative marking is unconstrained. 



We first consider the special case when the cut space is the point blow-up of a product of projective lines and prove the following Proposition. Later, as part of the proof of Theorem \ref{thm:potthm} (part Definition \ref{def:mvs} \eqref{mv:multcov}) we show that the case of general $\ol \XX_P$ reduces to this special case.


\begin{proposition} \label{prop:nd} {\rm(Bryan-Pandharipande formula)}
  Let $X$ be the point blow-up of $\P^1 \times \P^1$ with exceptional
  divisor $E$ intersecting the divisor $Y$.  The count $m(X,\ul{Y})$
  of ${d}$-fold covers of $E$ meeting $Y$ with a single intersection
  of maximal tangency is
  \[ m(X,\ul{Y}) = \frac{(-1)^{{d}-1}}{{d}^2} \]
  as in Definition \ref{def:mvs} \eqref{mv:multcov}.
\end{proposition}

Proposition \ref{prop:nd} would be a special case of localization computation in 
Bryan-Pandharipande \cite[(13)]{bryan:lgw}, except that we are using a different foundational scheme for which virtual localization is not easily available.  For the sake 
of completeness, we give a different argument.  Note that the relevance of the Bryan-Pandharipande formula for disk counts near the focus-focus singularities appears also in Lin \cite{lin:open} and Gr\"{a}fnitz-Ruddat-Zaslow \cite{grz:proper}


 \begin{definition} \label{def:rel}
 For any $d \in \N$, define for short a relative Gromov-Witten invariant
\[n_d:=  \# \{ \text{$u: \P^1 \to E$ of class $d[E]$ with a tangency of order $d$ with $Y$} \}  \in \Q \]
where $\#$ denotes the signed, weighted count described in \cite{vw:trop}.  
\end{definition}
\noindent The results of \cite{vw:trop} imply that the number $n_d$
is independent of the choice of domain-dependent almost complex structure, using a Donaldson hypersurface to stabilize domains.  
Proposition \ref{prop:nd} claims that 
\[ n_d=(-1)^{d-1}/d^2 .\]
The proof uses the following recursive relation, for which we introduce some terminology. 

\begin{definition} A {\em partition} $\Theta$ of a positive integer $d$ is a decomposition 
\[ d=d_1+\dots+d_k \] 
for some $k \geq 1$, positive integers $d_i$ that are non-decreasing in $i$.  Set 
\[ \nu_j(\Theta):=\#\{d_i: d_i=j\} . \] 
\end{definition}

\begin{lemma}
  The curve counts $\{n_d\}_d \in \Q$ of \eqref{def:rel} satisfy a recursive relation
  \begin{equation}
    \label{eq:recrel}
    0=\sum_{\Theta \text{ is a partition of $d$}} \frac {\prod_i d_i n_{d_i}}{\prod_j \nu_j(\Theta)!}
  \end{equation}
  for all $d > 1$. 
\end{lemma}
\begin{proof}
The two sides of the expression in \eqref{eq:recrel} can be realized as curve counts for different almost complex structures. After perturbation, the curves counted by $m_d$ lie in a small
  neighborhood of $E$. So, we may assume that $X$ is the blow-up of $\P^2$ at a point $p$ and $Y$ is the strict transform of a line
  through $p$.  Denote the homology classes of $E, Y \subset X$ by
  \[e:=[E], \quad f:=[Y] \in H_2(X).\]
  Given $d+1$ generic points $\underline x=(x_0,\dots,x_d)$ in $X$, let 
  \[ N(de+f, \underline x) \in \Q \] 
  be the number of curves $u: \P^1 \to X$ of class $de+f \in H_2(X)$ passing
  through $x_0,\dots,x_d$. We claim that
  \[N(de+f, \underline x)=0\]
  if $d>1$. Indeed, we may perform the curve count using an almost
  complex structure $J$ for which $E$ is $J$-holomorphic.  Since
  \[ (de + f)\cdot e=1-d < 0 \] 
any $J$-holomorphic curve with class $de + f$ necessarily lies in $E$, which is impossible. 

To obtain the expression as a sum over partitions in \eqref{eq:recrel}, we perform the curve count in a carefully chosen broken manifold. We cut $X$
  into two pieces
  \[X_+:=\P^1 \times \P^1, \quad X_-=\on{Bl}_1 \P^2\]
  in such a way that all the point constraints $x_0,\dots,x_d$ lie in $X_+$. The
  exceptional divisor class $e$ splits into classes
  $e_\pm \in H_2(X_\pm)$. Any broken curve 
  \[ u=(u_-: C_- \to X_-,u_+: C_+ \to X_+ ) \] 
  counted by
  $N(de+f, \underline x)$ belongs to the class $f+de_+$ in $X_+$ and $d e_-$
  in $X_-$. The curve component $C_-$ may be disconnected, and the
  class of each connected component of $u_-$ is a multiple of $e_-$. Thus,
  $u_-$ determines a partition $\Theta$ of $d$.  The expression inside
  the summation sign in \eqref{eq:recrel} is the count of the curves
  $u_-$ with no constraint on the relative divisor $Y$, weighted by intersection multiplicities $d_i$ at the relative divisor $Y$   The denominator
  accounts for permutations of markings arising from the stabilizing
  divisor. To finish the proof, it remains to count the number of possibilities for the second component $u_+ $ for a fixed first component $u_+ $. The curve $u_+$ has degree $(1,d)$ with $d+1$ non-relative point constraints, and point constraints of order $d_1,\dots,d_k$ along the relative divisor. An explicit calculation 
  as Section \ref{sec:spheresintoric}    shows that there is exactly one curve satisfying these constraints,   and \eqref{eq:recrel} follows.
\end{proof}

\begin{proof}
  [Proof of Proposition \ref{prop:nd}] There is a unique sequence
  $\{n_d\}_{d \in \N}$ that satisfies $n_1=1$ and the recursive
  relation \eqref{eq:recrel} for all $d>1$. Therefore, it is enough to
  show that $n_d=(-1)^{d+1}/d^2$ satisfies \eqref{eq:recrel}.  Denote
  $m_d:=d n_d$.  Observe that the right-hand-side of \eqref{eq:recrel} is
equal to the coefficient of $x^d$ in the product 
\[
  \left( 1+m_1x +m_1^2 \frac{x^2}{2!} + m_1^3 \frac{x^3}{3!}+\dots \right) \left(1+m_2x^2 + m_2^2 \frac{x^4}{2!}
  + \dots \right)\dots,
\]
which is equal to
\begin{equation} \label{eq:equalto}
e^{(m_1 x + m_2 x^2 +\dots)}.\end{equation}
  The substitution $m_d=(-1)^{d+1}/d$ reduces the expression 
  \eqref{eq:equalto} to  
\[ e^{\ln(1+x)}=1+x .\]
Therefore, the expression \eqref{eq:recrel} vanishes for $d>1$,
proving the Proposition.
\end{proof}

\begin{proof}[Proof of Theorem \ref{thm:potthm}, part Definition \ref{def:mvs} \eqref{mv:multcov}]  
The computation reduces to the Bryan-Pandharipande formula as follows.
  Let $\XX_{P(v)}$ be a piece of the broken manifold containing focus-focus singularities. By assumption $\XX_{P(v)}$ does not contain the Lagrangian, and the moment image of $\XX_{P(v)}$ is as in  Figure \ref{fig:ff-general}. 
  We are considering tropical graphs with a single vertex $v$ mapping to $P(v)$; the corresponding component of the broken map has a single intersection with the divisor at infinity.  The space $\ol \XX_{P(v)}$ is a singular fibration over $\P^1$ (see \eqref{eq:sing-fib}).  By Lemma \ref{lem:nodal-sph}, the map $u_v$ maps to a sphere $E_0$ in a nodal fiber.  For a suitably chosen symplectic form on $X:=\Bl_p(\P^1 \times \P^1)$, a neighborhood $U$ of the exceptional divisor $E \subset X$ is symplectomorphic to a neighborhood $U_0$ of $E_0 \subset \ol \XX_{P(v)}$.
Via a deformation of the symplectic form on $\ol \XX_{P(v)}$, 
we may assume the symplectic area of the spheres $E_0$ (and $E$) are sufficiently small. By Gromov's monotonicity theorem, we conclude that holomorphic spheres of area $\om(E_0)$ that intersect $E_0 \subset X$ are contained in $U_0$. Therefore, the count  is 
\[ m(\ol \XX_{P(v)}, d[E_0], \ul Y) = m(X,d[E], \ul Y) = (-1)^{d-1}/d^2 \] 
by Proposition \ref{prop:nd}.  
\end{proof}

\subsection{Computing multiplicities via desingularization}
\label{subsec:desing}

So far we have proved the counting formula in Theorem \ref{thm:potthm} under the assumption that the tropical graphs contributing to the sum do not have valence higher
than three. In this Section, we extend the result to the case that higher valence
vertices occur.   The idea is that the relative curve counts are invariant under perturbations of the positions of the edges of the tropical graph, so that we may assume that all of the vertices are trivalent.    This is a special case of the 
splitting argument in our previous paper \cite{vw:split},  by splitting the edge
matching condition at all incident edges of higher valent vertices. The resulting split graphs only contain trivalent graphs and can be counted in the standard way.

 As in the Introduction, we define
%
\[m^\pert(v):= \sum_{\Gamma_{v,\pm}^\pert}m(\Gamma_{v,\pm}^\pert),\]
where the sum ranges over all the 
perturbations $\Gamma_{v,\pm}$ of $\Gamma_v$.

\begin{figure}[ht]
  \begin{center}
    \scalebox{.8}{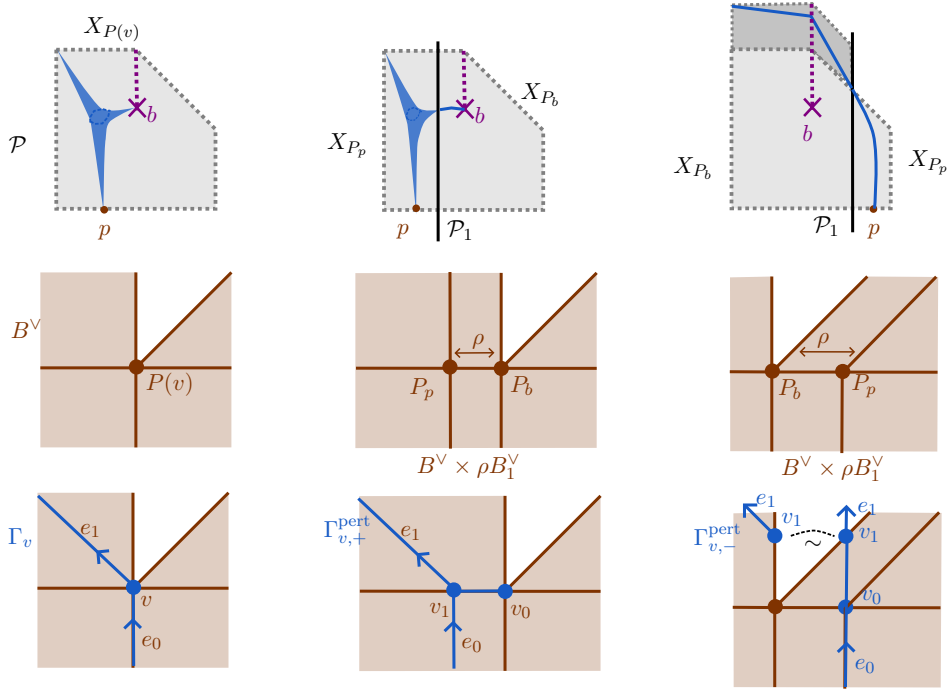}
  \end{center}
  \caption{
  Perturbations of $\Gamma_v$. In each case, the figure shows the space and the multiple cut, the dual complex, and the tropical graph.}
  \label{fig:crossff}
\end{figure}

\begin{proposition}
\label{prop:ffcross}
  In the above setting, $m(v) = m^\pert(v)$.
 \end{proposition}
 
 \begin{proof}
To prove the Proposition,  we recall the notion of {\em splitting the matching condition} from our earlier work \cite{vw:split}.  The splitting process produces a
version of a broken map, called a {\em split map} in which there is no
matching condition on a subset of tropical edges, called {\em split
  edges}. Defining a split map requires a choice of a non-zero vector,
called the {\em cone direction} $\eta_e \in \t/\bran{\cT(e)}$ at each
of split edge.  A {\em split map} with a single split edge
$e=(v_+,v_-)$ and a cone direction $\eta_e$ is defined as a broken map
$u = (u_v)$ except that the  matching condition at $e$ is not necessarily satisfied, 
and instead the tropical graph $\Gamma$ must satisfy a version of the Fulton-Sturmfels 
cone condition:  There exists $\eps>0$ such that
\[\Disc(\tGam)=[0,\eps)\eta,\]
where the {\em discrepancy cone}
\[\Disc(\tGam):=\{\pi^\perp_{\cT(e)}(\cT(v_+)-\cT(v_-)) \in \t/\bran{\cT(e)} \}\]
and $\cT_\tGam$ ranges over all vertex positions for the graph
$\tGam$.  The cone condition implies that the tropical graph $\tGam$
has one degree of freedom and $T_\trop(\tGam) \simeq \C^\times$. The Fulton-Sturmfels condition appeared in their study of the K\"unneth decomposition of the diagonal in toric varieties.  

Splitting an edge has the effect of dropping the matching condition at
the node, and increasing the dimension of the tropical symmetry group
by $2$.  In the four-dimensional case, suppose that cutting an edge $e$ splits a rigid tropical graph $\Gamma$ into $\Gamma_+$ and $\Gamma_-$, and the matching condition is a constraint
on $\M_{\bGam_+}(\XX)$ (as in Theorem \ref{thm:split}).
In the split tropical graph $\tGam$ obtained by splitting the edge $e$, 
the tropical symmetry group
$T_\trop(\tGam_+)$ is $\C^\times$, and $T_\trop(\tGam_-)=\Id$.

The cone condition in the case that there is more than one split edge
can be described as follows.  Suppose that there are $n$ split edges ordered as $e_1,\dots, e_n$ with deformation directions $\eta_i \in \t/\bran{\cT(e_i)}$.  The set of
discrepancies is a cone for each of the split edges, with the
discrepancy being much larger for $e_i$ compared to $e_j$ whenever
$i<j$. Specifically, the cone condition states that there are
constants $\eps_1,\dots,\eps_n >0$ such that the set of discrepancies
\[\Disc(\tGam) =\{(\pi_{\cT(e_i)}^\perp(\cT(v_{i,+})-\cT(v_{i,-})))_i\} \subset \oplus_i (\t/\bran{\cT(e_i))}\]
at the split edges is 
\[\Disc(\tGam) = \left\{\sum_i c_i \eta_i : c_i \in C \right\},\]
where
\[C=\{(x_1,\dots,x_n) \in (\R_{\geq 0})^n : x_1<\eps_1, x_i < \eps_i x_{i-1}\}.\]
The same moduli space of split maps is obtained if the splitting step is performed one split edge at a time in the order $e_1,\dots,e_n$.

\begin{figure}[ht]
  \begin{center}
    \scalebox{.8}{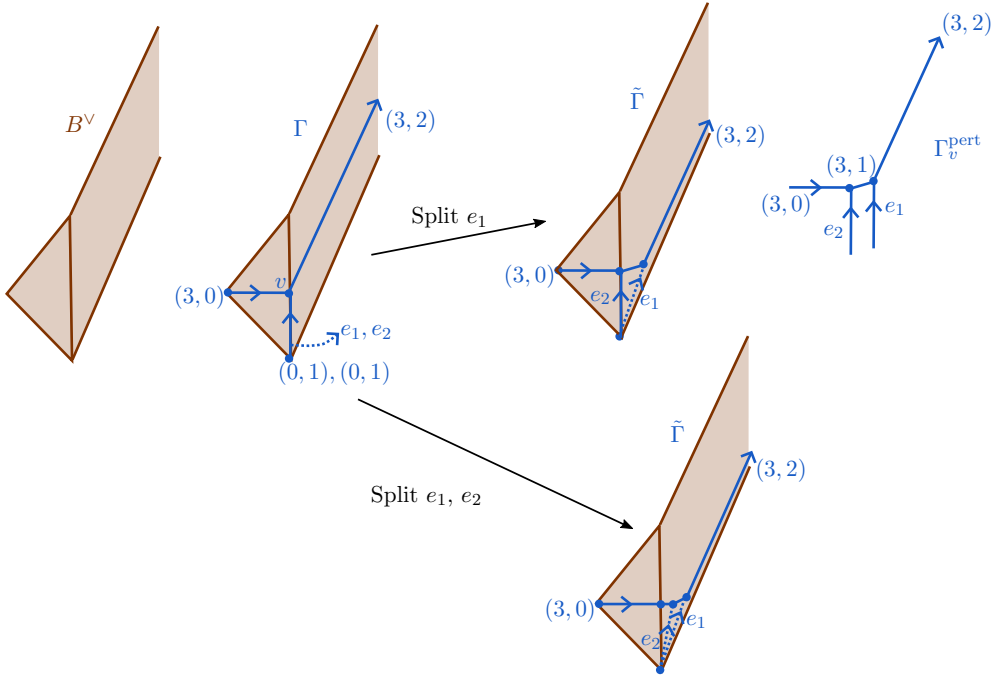}
  \end{center}
  \caption{The vertex $v$ in the tropical graph $\Gamma$ has two incoming edges in the direction $(0,1)$. Splitting $e_2$ produces the split graph $\tGam$. The corresponding perturbed graph at $v$ is $\Gamma_v^\pert$.}
  \label{fig:split2pert}
\end{figure}

  Let $\Gamma$ be rigid tropical graph contributing to the disk potential. 
  We split the matching condition at the following collection of nodes:
  Consider a spherical vertex $v \in \Ver_\black(\Gamma)$ with
  valence more than three.  Choose an ordering $e_1,\dots,e_k$ of
  incoming edges at $v$. We may assume that $e_{k-1}$ and $e_k$ are
  not parallel. We split the edges $e_1,\dots, e_{k-2}$ in that
  order. We carry out such splitting at all higher valent spherical
  vertices, in any order.    Given a rigid tropical graph $\Gamma$, a split tropical graph
  $\tGam \xrightarrow{\kappa} \Gamma$ is uniquely determined by a
  $\Gamma_v^\pert$ at every higher valent vertex $v$ as
  follows:  We obtain the perturbation $\Gamma_v^\pert$ by
  starting from a realization of the split graph $\tGam$, and drawing
  the incoming (split) edges $e_1, \dots e_{k-2}$ parallel to their
  original direction $\cT(e_i)$. See Figure \ref{fig:split2pert} for
  an example.

%
  Proposition
  S-\ref{S-prop:defmorph} gives a cobordism between the moduli space
  of 
  split maps and deformed maps with a large enough
  deformation parameter.
\end{proof}

Similarly, perturbations allows us to avoid the situation that vertices at focus-focus value have valence more than one, since for generic perturbations no edge of $\Gamma$
passes through a focus-focus value $b^\dual \in B^\dual, b \in B^{\foc}$.

\subsection{Counting by bunching}

We give a formula for counting curves where there are multiple edges emanating from a focus-focus singularity. A consequence is the binomial nature of the coefficients of the disk potential along edges of its Newton polytope.  We start by defining a bunched version of tropical graphs.

\begin{definition}
\label{def:bunch}
  \begin{enumerate}
  \item {\rm(Coincident edges)}
    Given a map type $\bGam$, a collection of edges $e_i=(v_i,w)$ is {\em coincident} if they share an end-point $w$, the other end-points  of each of the edges lie on the same focus-focus singularity in $X$, that is, for all $i$, $v_i$ is univalent,  $P(v_i)=P$, and the map $u_{v_i}$ is incident on a focus-focus singularity $b \in X_P$. 
\item  {\rm(Bunching of a tropical graph)}
  Given a broken map type $\bGam$, its {\em bunching} $\bGam_b:=b(\bGam)$ is given by replacing any set of coincident edges $e_1,\dots,e_k$ with a single edge $e$ with $\cT(e):=\sum_i \cT(e_i)$.  Thus a {\em bunched map type} $\bGam_b$ has the property that no two distinct edges are coincident. 
 \end{enumerate}

The {\em multiplicity of a bunched map type} $\bGam_b$ is the sum of the multiplicity of all map types whose bunching it is. That is,
\[m(\bGam_b):=\sum_{\bGam: b(\bGam)=\bGam_b} m(\bGam).\]
This ends the Definition.
\end{definition}

We describe a formula to compute multiplicity of bunched graphs under the assumption that
at any vertex there is at most one bunched incident edge.

\begin{definition}
    The {\em bunched multiplicity $m_b(v)$ of the vertex $v$} is equal to the ordinary multiplicity $m(v)$ 
\begin{enumerate}
    \item if there is no bunched edge incident on $v$, and 
    \item if $v$ has a bunched edge $e$ incident on it, $m_b(v)$ is defined as follows:  Let $\bGam_{b,e} \subset \bGam_b$ be the subgraph consisting of vertices of $v$, $w$ and the non-bunched incoming edge of $v$, denoted by $f_0$ and the outgoing edge of $v$ denoted by $f_1$. Let $\bGam_e$ be any subgraph whose bunching is $\bGam_{b,e}$. Thus, $\bGam_e$ consists of vertices $v$, $w_1,\dots,w_n$ where $e$ is the bunching of $e_i=(w_i,v)$ and $w_i$ is univalent. Then, we define
\begin{equation}
  \label{eq:mbv}
m_b(v):=\sum_{\bGam_e: b(\bGam_e)=
  \bGam_{e,b}} \tfrac 1 {|\Aut(\bGam_e)|} m(v) \prod_i m(w_i).   
\end{equation}
\end{enumerate}
\end{definition}

 The multiplicity $m(\bGam_b)$ can be written as a product of multiplicities over trivalent vertices 
\[m(\bGam_b)=\prod_{v : \on{val}(v)=3} m_b(v) .\]

\begin{proposition}\label{prop:bunchmult}
  Let $\bGam_b$ be a bunched map type, and let $v \in \Ver(\bGam_b)$ be a trivalent vertex with an incident bunched edge $e$, another incoming edge $f_0$ and an outgoing edge $f_1$. 
  Then,
\[m_b(v)=\binom k n,\]
where the direction $\cT(e)=n \mu_b$, $n \in \N$ 
and $\mu_b \in \t_\Z$ is primitive, and $k:=|\mu_b \times \cT(f_0)|$. 
\end{proposition}
\begin{proof}
  First we write down the contribution to the multiplicity of a
  subgraph $\bGam_e$ consisting of all edges that are bunched into $e$ in $\bGam_b$. That is, 
 $b(\bGam_e)=\bGam_{e,b}$. An equivalence class
 \footnote{The equivalence class in this context is given by permuting the
   set of edges which have the same multiplicity and which
   get bunched to $e$.}
 of a subgraph $\bGam_e$ whose bunching is $\bGam_{e,b}$ is determined by the
  number of occurrences $d_\nu$ of the input $\nu \mu_b$ for each $\nu \in \N$ in the graph $\bGam_e$.  We apply the perturbation described in Definition \ref{def:vpert} to the higher valent vertex $v$ in $\bGam_e$ to obtain a perturbed graph $(\bGam_e)_v^\pert$ where the vertex $v$ is split into trivalent vertices $v_1,\dots,v_d$, each with an incoming edge $e_i=(w_i,v_i)$, so that
  \[m(v)=\prod_i m(v_i)=\prod_i(k \nu_i) , \quad m(w_i)=\tfrac {(-1)^{\nu_i+1}}{\nu_i^2}.\] 
Furthermore, the automorphism group of $\bGam_e$ has order
\[ \# \Aut(\bGam_e)=\prod_{\nu \in \N} d_\nu!  \]
corresponding to permutations of edges $e_i$ going to the focus-focus value with the
same direction $\cT(e_i)$.   Thus the contribution of the graph $\bGam_e$ to the sum $m_b(v)$ in \eqref{eq:mbv}
 is   %
  \[ (\# \Aut(\bGam_e))^{-1} \prod_i m(v_i) m(w_i)=\prod_{\nu
      \in \N}\left(\frac {(-1)^{\nu+1} k }{\nu}\right)^{d_\nu}\frac 1 {d_\nu!}. \]
  The sum of contributions of all subgraphs $\bGam_e$ is the coefficient of $x^n$ in
  \begin{multline} \left(1+ k x+\frac{(k x)^2}{2!} + \dots)(1 -\frac{k x^2}{2} + \frac{(k x^2/2)^2}{3!} - \dots \right) \\ \left(1+ \frac{k x^3}{2} + \frac{(k x^3/3)^2}{3!} + \dots \right) \dots.  \end{multline}
  Using the power series expansion of $\ln(1+x)$ this is equal to
  \[e^{k x}e^{-k x^2/2}e^{k x^3/3}\dots=e^{k\ln(1+x)}=(1+x)^k.\]
  The coefficient of $x^n$ in the expression is $\binom{k}{n}$.
\end{proof}

\begin{figure}[ht]
  \begin{center}
    \scalebox{.8}{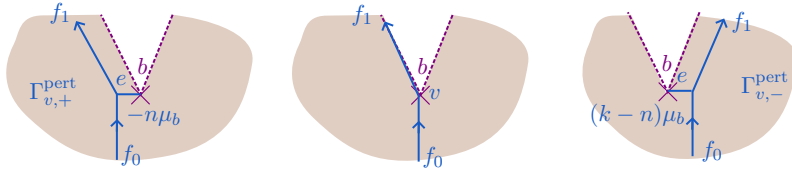}
  \end{center}
  \caption{Bunchings of one-sided perturbations of a vertex at a singular point $b$.}
  \label{fig:bunch}
\end{figure}

\begin{remark}
  Proposition \ref{prop:ffcross}
  on the equality of counts of $(+)$ and $(-)$-perturbations can alternately be proved by explicitly
  counting both sides via bunching. 
\end{remark}

 \section{Tropical graphs of rigid disks}
\label{sec:atinfinity}

In this Section we study special polyhedral decompositions for the almost toric manifolds of the type considered in Vianna \cite{vianna:dp}, which we call standard decompositions, for which all of the multiplicities of vertices appearing in tropical graphs corresponding to Maslov-index-two disks are those described in Definition \ref{def:mvs}.    Similar to Mikhalkin \cite{mikhalkin} and  Bardwell-Evans  et al \cite{bardwellevans}, we have a holomorphic-to-tropical correspondence in this case which takes place in affine manifold diffeomorphic to 
the punctured plane (with number of punctures equal to the number of corners in the moment polytope)  which we call the {\em dual affine manifold.}

\subsection{Standard decompositions}
\label{goodsec}

We begin by describing a standard polyhedral decomposition in the case of
a base diagram where the focus-focus values are close to the boundary.

\begin{proposition} \label{prop:interior}
  Let $X$ be a monotone almost toric four-manifold with moment map
  $\Phi : X \to \t^\dual$,
  $\Phi(X)$ is a convex polytope, and a
  monotone Lagrangian torus $L:=\Phinv(\ell)$ for $\ell \in
  \Phi(X)$. There is
  an elementary polyhedral decomposition (as in Definition
  \ref{def:elementary}) $\PP$ and cutting datum for which all collisions in rigid  tropical disks occur at interior generic points. Consequently,
  Theorem \ref{thm:potthm} applies to the triple $(X,\Phi, \PP)$.
\end{proposition}

The decomposition is constructed by taking the polytopes to be the intersections of polytopes in two polyhedral decompositions defined as follows.

\begin{definition}
  {\rm(Intersection of polyhedral decompositions)} Two decompositions $\PP_0$
  and $\PP_1$ of $\R^n$ {\em intersect transversally} if any pair
  $P_0 \in \PP_0$, $P_1 \in \PP_1$ intersects transversely.  The {\em intersection} 
  $\PP:=\PP_0 \cap \PP_1$ is defined as the polyhedral decomposition,
  each of whose  polytopes is an intersection
  $P_0 \cap P_1$ of $P_0 \in \PP_0$, $P_1 \in \PP_1$.
\end{definition}
  
\begin{definition} \label{def:good}
  For an almost toric manifold with no elliptic singularities, 
  define a {\em standard decomposition}
  to be a polyhedral decomposition of the form
  \[\PP:=\PP_\ann \cap \PP_{\on{in}}.\]
  where
  \begin{enumerate}
  \item $\PP_\ann$ is an {\em annular cut}, which is a single
    piece-wise linear path encircling $\lambda = \Phi(L)$ and all the
    focus-focus values $B^{\foc}$ and whose complement is the union of
    an inner ball and an outer annulus in $\Phi(X)$;
  \item and $\PP_{\on{in}}$ is an {\em inner cut} that separates each
    of the branch cuts $C_1,\ldots, C_k$ and the Lagrangian projection
    $\ell = \Phi(L)$ from each other 
    and is precisely described as follows: The one-dimensional part of
    $\PP_{\on{in}}$ consists of a polygon $P_0$ that bounds $\ell$,
    and semi-infinite rays $P_i,P_i' \in \PP_{\on{in}}$ emanating from the vertices, where $i$ is the index set for the vertices, pairs of  which, $P_i$, $P_i'$, are parallel to branch cuts $C_i$, and $C_i$
    lies between $P_i$, $P_i'$. See Figure \ref{fig:2p} for example.
  \end{enumerate}
  In the case when the almost toric structure has elliptic singularities, a standard decomposition is defined similarly, with the only difference that the annular cut is a piece-wise linear path with self-crossings near elliptic singularities.
  Each elliptic singularity $x$ corresponds to a zero-dimensional polytope $R_x$ which a self-crossing of the annular cut. The collection of these zero-dimensional polytopes is denoted by
  \[\PP^{\on{ell}}:=\{R_x \in \PP : \text{$x$ is an elliptic singularity}\}.\]
  See Figure \ref{fig:b4p2dual} for an example. This ends the Definition. 
\end{definition}

\begin{notation}
The polytope $P_0$ containing $\ell$ has facets consisting of {\em long facets}
\begin{equation}
  \label{eq:P0-long}
  F_\mu:=\{\bran{x,\mu}=1-\eps\}   
\end{equation}
parallel to each facet of $\Phi(X)$ and {\em short facets}
\begin{equation}
  \label{eq:P0-short}
  F_\nu:=\{\bran{x,\nu}=c_\nu\} 
\end{equation}
near the vertices of $\Phi(X)$, where the vector $\nu \in \t_\Z$ lies outside $\Phi(X)^\dual$ (Definition \ref{def:polydual} below).
The constants $\eps$, $c_\nu$ are chosen so that for a small enough $\eps_0>0$, $P_0$ contains $(1-\eps_0) \Phi(X)$, see Proposition \ref{prop:collg} in the following section.
\end{notation}
\begin{definition}\label{def:polydual}
  {\rm(Dual of a polytope)} Let $\Delta \subset \t^\dual$ be a
  polytope with rational edge directions. The {\em dual} of $\Delta$,
  denoted by $\Delta^\dual \subset \t$, is the convex hull of the
  primitive normals $\mu \in \t_\Z$ of the facets of $\Delta$.
\end{definition}

\begin{figure}[ht]
  \begin{center}
    \scalebox{.8}{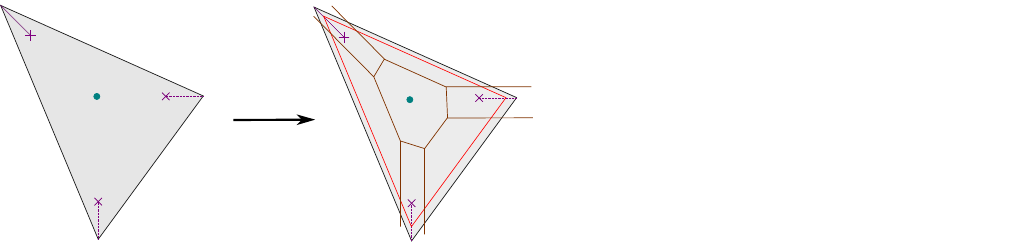}
  \end{center}
  \caption{A standard polyhedral decomposition $\PP$ on $\Bl_8\P^2$ defined as 
     the intersection of $\PP_{\on{ann}}$ and $\PP_{\on{in}}$.}
  \label{fig:2p}
\end{figure}

The dual complexes of $\PP_\ann$ and $\PP_{\on{in}}$ combine to give a one-parameter family of dual complexes
\begin{equation}
  \label{eq:dualrho}
  \{B^\dual(\rho)\}_{\rho>0}  
\end{equation}
for $\PP$ as in the following Lemma which is proved as part of
\cite[Example 3.30]{vw:trop}.  

\begin{lemma}\label{lem:union-cd}
  {\rm(A family of dual complexes)}
  Let $\PP_0$, $\PP_1$ be polyhedral decompositions of $\R^n$ with
  dual complexes $B_0^\dual$ and $B_1^\dual$ respectively, and both
  $\PP_0$ and $\PP_1$ have a cutting datum. Then, the intersection
  $\PP_0 \cap \PP_1$ has a family of dual complexes and cutting data
  parametrized by $\rho >0$, where for any polytope
  $P_{01} \in P_0 \cap P_1$ in $\PP$, where $P_0 \in \PP_0$ and $P_1 \in \PP_1$, the dual
  polytope $P_{01}^\dual(\rho)$ is the product $P_0^\dual \times \rho P_1^\dual$.
\end{lemma}


Via the following lemma, we show that for any $\rho>0$ the dual complex $B^\dual:=B^\dual(\rho)$ is  a singular integral affine manifold with corners whose singular set is 
\[B^{\dual,\foc} :=\{P^\dual: \codim(P)=0, \text{ $P$ contains a focus-focus value}\}.\]

\begin{lemma}
  {\rm(A singular affine manifold)}
  Let $\PP$ be a polyhedral decomposition defined on a two-dimensional base diagram $\Delta \subset \t^\dual$ of an almost toric manifold, such that any focus-focus value is contained in a top-dimensional polytope $P \in \PP$ and edges intersect branch cuts transversally. Then, the interior of the dual complex $B^\dual$ is a singular affine manifold modelled on $\t$ with singularities at $B^{\dual,\foc}$.
\end{lemma}
\begin{proof}
  Let $\PP^\br \subset \PP$ denote the subset of polytopes $P$ that intersect branch cuts. For any polytope $P \in \PP \bs \PP^\br$, the dual $P^\dual \subset \t_P \subset \t$.  An affine structure is induced by
  the embedding
  \begin{equation}
    \label{eq:embedt}
    i:\left(\bigsqcup_{P \in \PP \bs \PP^\br}P^\dual/\sim \right)\hra \t.
  \end{equation}
  Here, the equivalence relation is as in \eqref{eq:sim}.
 We remark that top-dimensional dual polytopes $Q^\dual$ are embedded by \eqref{eq:embedt}.
  If a one-dimensional polytope $P\in \PP^\br$ with end points
  $Q_0, Q_1 \in \PP$ intersects a branch cut $C_i$, the dual $P^\dual$
  has two embeddings $i_j : P^\dual \to \t$, $j=0,1$ whose images lie on the boundary of the image of $i$. That is, 
  $i_j(P^\dual) \subset i(\partial Q_j^\dual)$. The embeddings
  $i_0, i_1$ are related to each other by a shear transformation
  \eqref{eq:shear}. Thus, we obtain an affine structure on
  $B^\dual \bs B^{\dual,\foc}$ by gluing the affine structure on
  the domain of \eqref{eq:embedt} along the one-dimensional edges
  $i_j(P^\dual)$, $P\in \PP^\br$.
\end{proof}

\begin{example}
  Figure \ref{fig:dual1chart} shows a standard polyhedral decomposition of an almost toric moment polytope of $\Bl^8 \P^2$. The dual complex $B^\dual$ is an affine manifold obtained by gluing
  together the dual polytopes 
\[ \{P^\dual \subset \t  : P \in \PP, \dim(P)=0\} . \]
We choose coordinates arising from the moment polytope, that is,
the complement of the branch cut produces an chart of the dual complex that is a subset of $\t$. In particular, in this chart, the direction $\mu_{Q^\dual} \in \t$ of  any one-dimensional dual polytope $Q^\dual$ ($Q \in \PP$) is given by the normal of the facet $Q \subset \t^\dual$.
The dual complex $B^\dual$, which is an affine manifold is obtained by gluing three pairs of edges by shear maps, namely the outer green edges of $P_0^\dual$ and $P_1^\dual$ are glued by the shear map $A_1 \in GL(2,\Z)$, those of $P_2^\dual$ and $P_3^\dual$ are glued by $A_2$, and  those of $P_4^\dual$ and $P_5^\dual$ are glued by $A_3$. 
The shear maps $A_1$, $A_2$, $A_3 \in GL(2,\Z)$ are  given by 
  \[A_1=
  \begin{pmatrix}
    1&9\\
    0&1
  \end{pmatrix}, \quad
  A_2=
  \begin{pmatrix}
    -2&9\\
    -1&4
  \end{pmatrix}, \quad
  A_3=
  \begin{pmatrix}
    1&0\\
    -1&1
  \end{pmatrix}.
  \]
  The matrix $A_1$ has $(1,0)$ as an eigenvector with eigenvalue $1$, and sends $(3,-1)$ to $(-6,-1)$; $A_2$, $A_3$ can similarly be read off from Figure \ref{fig:dual1chart}. More precisely, the gluing is defined on neighborhoods of the outer green edges by a map
  \[ \t \to \t, \quad x \mapsto A_ix+ b,\]
  where the constant $b \in \t$ is chosen so that the respective edges
  are identified by the gluing.  Various tropical graphs corresponding
  to rigid broken maps are given in Figures \ref{fig:divcurve},
  \ref{fig:vertcurve} and \ref{fig:dp1full}.  This ends the Example.
\end{example}

\begin{figure}[ht]
  \begin{center}
    \scalebox{.8}{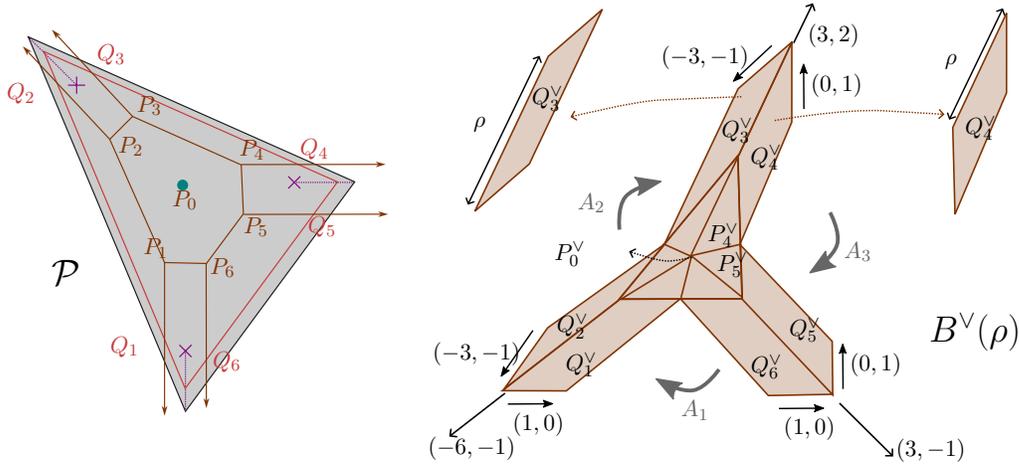}
  \end{center}
  \caption{Left: A polyhedral decomposition $\PP$ for $X:=\Bl^8\P^2$.
The polytope $P_0$ is top-dimensional, $P_i$, $Q_i$, $i=1,\dots,6$ are zero-dimensional. 
    Right: Dual complex $B^\dual(\rho)$. }
  \label{fig:dual1chart}
\end{figure}

\begin{figure}[ht]
  \begin{center}
    \scalebox{.8}{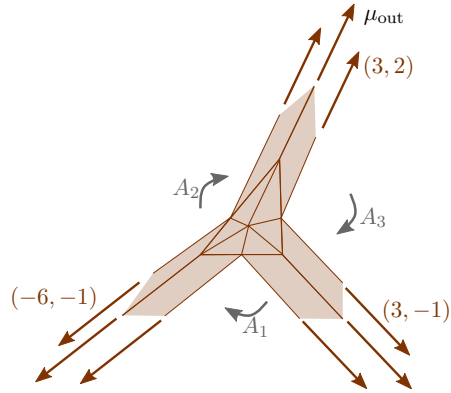}
  \end{center}
  \caption{Dual affine manifold corresponding to the polyhedral decomposition $\PP$ of  $\Bl_8\P^2$ in Figure \ref{fig:dual1chart}.}
  \label{fig:b8p2-affine}
\end{figure}

\begin{figure}[ht]
  \begin{center}
   \scalebox{.8}{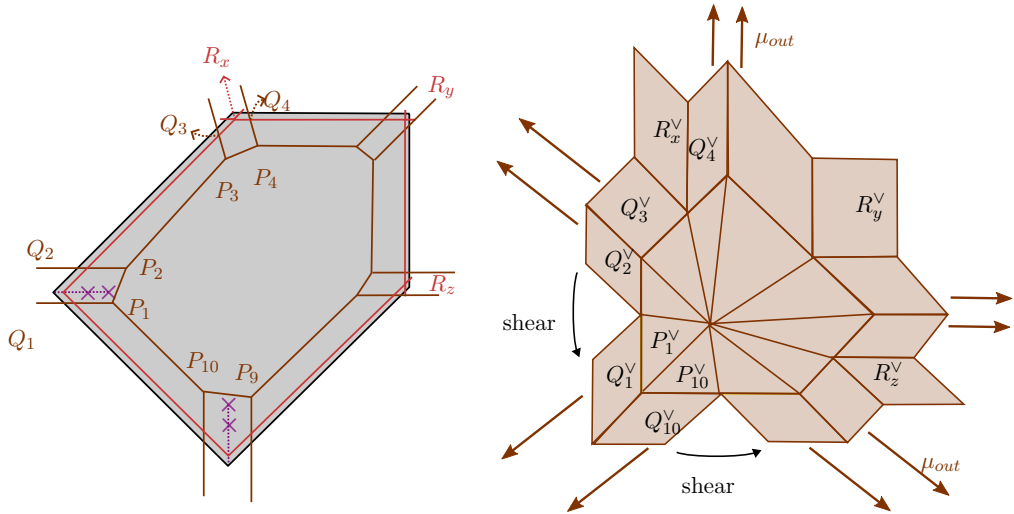}
  \end{center}
  \caption{Left: A multiple cut on an almost toric manifold with three elliptic singularities at vertices of the pentagon. Right: The dual complex. An outward pointing vector field $\mu_{out}$ is defined on $Q_i^\dual$, $i=1,\dots,10$.}
  \label{fig:b4p2dual}
\end{figure}

\begin{figure}[ht]
  \begin{center}
   \scalebox{.8}{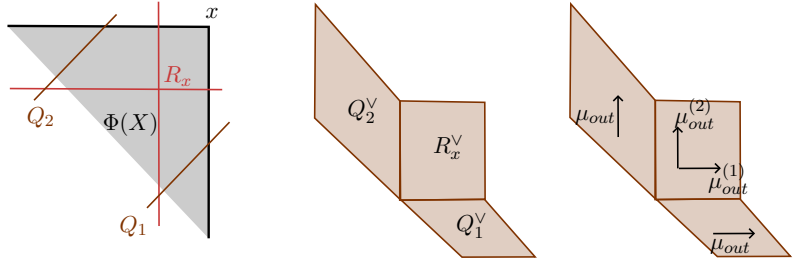}
  \end{center}
  \caption{Dual polytope $R_x^\dual$ corresponding to an elliptic singularity $x$, and the extensions of the outward normal vector $\muout$ to $R_x^\dual$.}
  \label{fig:ellip}
\end{figure}

\begin{notation}
  Let $\PP$ be a standard polyhedral decomposition of an almost toric
  manifold $X$, with dual complex $B^\dual$.
  \begin{enumerate}
  \item {\rm(Annular part of the dual complex)}
    The {\em annular part of the dual complex} is a subset of the dual complex $B^\dual$ defined as
    \begin{equation}
      \label{eq:bann-def}
      B^\dual_\ann:=\bigcup(P^\dual=P^\dual_{\on{in}} \times P^\dual_\ann ), \quad \dim(P^\dual_\ann) \geq 1.      
    \end{equation}
In the absence of elliptic singularities, the top-dimensional polytopes in $B^\dual_\ann$ are rectangles, where both $P^\dual_{\on{in}}$, $P^\dual_\ann$ are one-dimensional.
In case there are elliptic singularities, $B^\dual_\ann$ also consists of the dual polytope $R_x^\dual$ corresponding to each elliptic singularity $x$.

\item   {\rm(Normals to facets)} In the absence of elliptic singularities,
  the affine manifold $B^\dual_\ann \subset B^\dual$ has a parallel vector field
  \begin{equation}
    \label{eq:muout}
    \muout \in \Vect(B^\dual_\ann)    
  \end{equation}
  made up of primitive normals to facets. Such a vector field exists,
  since in the absence of elliptic singularities, the shear at a
  vertex maps the primitive normal of a facet to that of the adjacent
  one.  If there are elliptic singularities, $\muout$ is defined only
  on the complement
  $B^\dual_\ann \bs (\cup_{R \in \PP^{\on{ell}}}R^\dual)$ 
  of the dual polytopes corresponding to the elliptic singularities.
  In the dual polytope $R^\dual_x$ corresponding to an elliptic
  singularity $x$, that shares facets with rectangular dual polytopes
  $Q_1^\dual$ and $Q_2^\dual$, the parallel extension of
  $\muout|Q_i^\dual$, $i=1,2$, to $R_x^\dual$ is denoted by
  $\muout^{(i)}$, see Figure \ref{fig:ellip}.
  \end{enumerate}
\end{notation}

\begin{definition}\label{def:op}
  A vector $v \in T_bB^\dual_\ann$ is {\em outward pointing} if
\begin{enumerate}
  \item \label{op1} $b$ is in a rectangular dual polytope in
    $B^\dual_\ann$, and $v =n\mu_{\on{out}}$ where $n \in \Z_{>0}$ and
    $\mu_{\on{out}}$ is from \eqref{eq:muout}),
\item \label{op2} or $b$ lies in a polytope $R_x^\dual \in \PP^{\on{ell}}$
    corresponding to an elliptic singularity, and
    $v=n_1 \mu_{\on{out}}^{(1)} + n_2 \mu_{\on{out}}^{(2)}$, where
    $n_i \in \Z_{\geq 0}$ and $\mu_{\on{out}}^{(i)}$ is as in Figure
    \ref{fig:ellip}.
  \end{enumerate}
  For an outward pointing vector $v$, the {\em intersection number $n_\partial(v)$ with
    the boundary divisor} is $n$ resp. $n_1+n_2$ in 
  case \eqref{op1} resp. \eqref{op2}.
\end{definition}

\begin{remark}\label{rem:bal}
  {\rm(Balancing condition)}
  Recall that for a tropical graph $\Gamma$, 
  the standard balancing condition holds for vertices $v$ lying in zero-dimensional  polytopes,
  that is, $\dim(P(v))=0$. The balancing condition is that
  \begin{equation}
    \label{eq:bal0}
    \sum_{e \ni v} (-1)^{\sig_v(e)}\cT(e)=0,
  \end{equation}
  where $\sig_v(e)$ is $-1$ resp. $1$ if $e$ points towards resp. away from $v$. 
  We generalize the condition to vertices $v$ with
  $\dim(P(v))=1$.
  For an edge $e$ if $P(e) \notin \PP^\br$ then, we view $\cT(e)$ as an element in $\t$ via the embedding \eqref{eq:embedt}. If $P(e) \in \PP^\br$,
  there are two embeddings $i_0, i_1 : P(e)^\dual \to \t$
  related by a shear $A \in GL(2,\Z)$, 
  and we have $A(i_0(\cT(e)))=i_1(\cT(e))$.
  The following cases arise.
  \begin{description}
  \item[Case 1]   Suppose $P(v) \notin \PP^\br$ and $P(v)$ does not intersect $\partial \Phi(X)$. Then, the standard balancing condition holds.
  \item[Case 2] Suppose $P(v) \notin \PP^\br$ and $P(v)$ intersects a face $F_\nu$ of $\partial \Phi(X)$ with outward normal vector $\nu \in \t_\Z$. Then, 
       \begin{equation}
         \label{eq:bal1}
    \sum_{e \ni v} (-1)^{\sig_v(e)}\cT(e) + n_v \nu=0,
   \end{equation}
  where $n_v \in \Z_{\geq 0}$ is the intersection number of $u_v$ with $\Phinv(F_\nu)$. 
\item[Case 3] Suppose $P(v) \in \PP^\br$, with incident edges $\{e_i\}_i$ in the $i_0$-side and $\{f_j\}_j$ in the $i_1$-side. If $P(e) \in \PP^\br$, we choose a side arbitrarily. Then,
  the balancing condition is
\[A \left(\sum_i (-1)^{\sig_v(e_i)}\cT(e_i) \right) + \sum_j (-1)^{\sig_v(f_j)}\cT(f_j)=0.\] 
\end{description}
\end{remark}

\begin{remark}
  {\rm(Inner and annulus parts of the dual complex)}
\label{rem:B-ann}
  The dual complex
  partitions into an inner and an annulus part.  The inner part is the
  dual complex of $\PP_{\on{in}}$, and is independent of $\rho$.  For
  any $\rho$, define
  \[B^\dual_{\on{in}}:=\cup_{Q \in \PP_{\on{in}} : \dim(Q)=0}Q^\dual
    \subset B^\dual(\rho), \quad B^\dual_{\ann}(\rho):=B^\dual(\rho)
    \bs \on{int}(B^\dual_{\on{in}}(\rho)).\]
  For example, in Figure \ref{fig:dual1chart},
  $B^\dual_{\on{in}}=\cup_{1 \leq i \leq 6}P_i$ and
  $B^\dual_{\ann}=\cup_{1 \leq i \leq 6}Q_i$.  If the almost toric
  manifold has no elliptic singularities, the annulus part
  $B^\dual_\ann$ is obtained by gluing a set of rectangular dual
  complexes $Q_i^\dual$ via shears. Each $Q_i^\dual$ corresponds to a
  point of intersection between the decompositions $\PP_{\on{in}}$ and
  $\PP_\ann$.

  If the almost toric manifold has elliptic singularities, ``gluing by
  shear'' is replaced by additional polytopes. For every elliptic
  singularity $x \in X$, there is a top-dimensional dual polytope
  $R_x^\dual$ that shares facets with rectangles
  $Q_i^\dual$, $Q_j^\dual$. There is no shear going from $Q_i^\dual$
  to $R_x^\dual$ to $Q_j^\dual$. See Figures \ref{fig:b4p2dual} and \ref{fig:ellip}. 
  This ends the Remark.
\end{remark}

\begin{figure}[ht]
  \begin{center}
    \scalebox{.8}{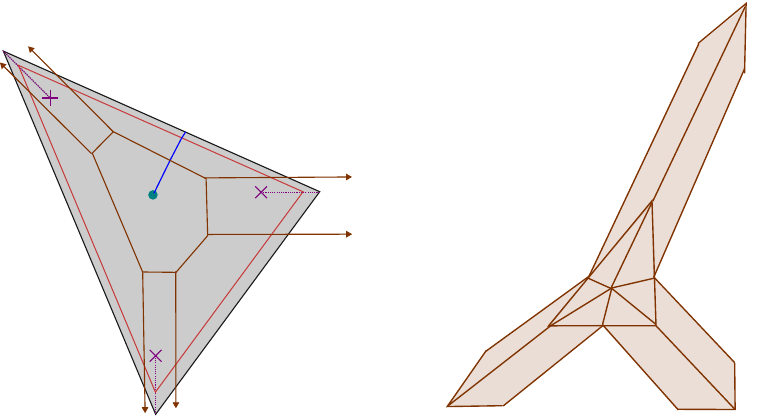}
  \end{center}
  \caption{Left: A broken map in $\Bl_8\P^2$ whose starting direction is a normal to a divisor. Right: Tropical curve in the dual complex. The polyhedral decomposition and dual complex are as in Figure \ref{fig:dual1chart}}
  \label{fig:divcurve}
\end{figure}

\begin{figure}[ht]
  \begin{center}
    \scalebox{.8}{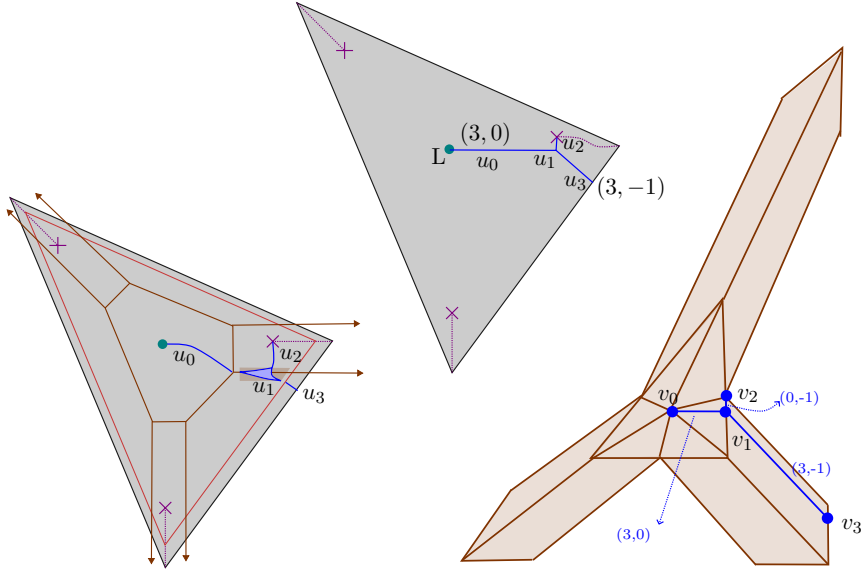}
  \end{center}
  \caption{Left: A broken map in $\Bl_8\P^2$. Middle: The corresponding cartoon tropical curve.    Right: Actual tropical graph in the dual complex. }
  \label{fig:vertcurve}
\end{figure}

\begin{figure}[ht]
  \begin{center}
  \scalebox{.8}{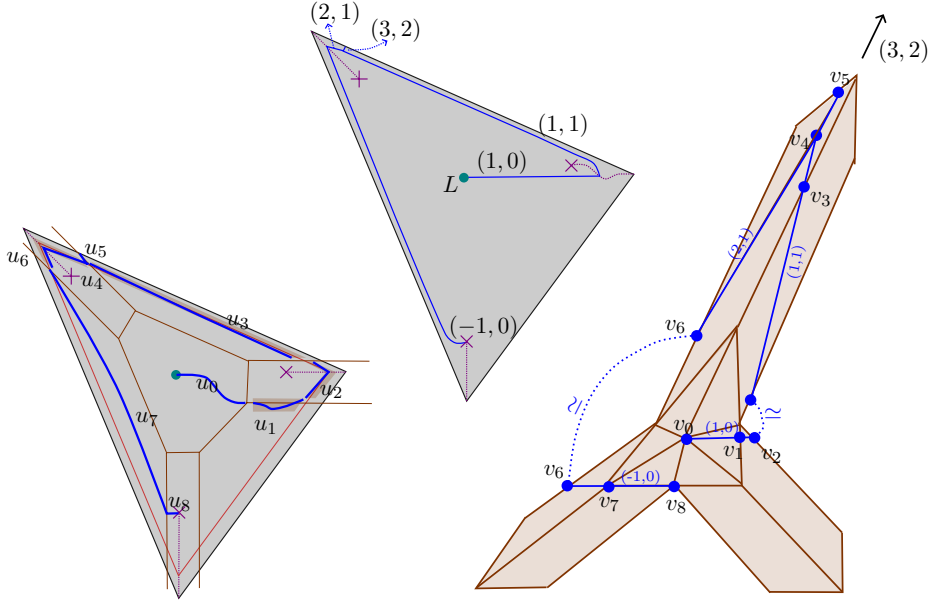}
  \end{center}
  \caption{Left: A broken map in $\Bl_8\P^2$ whose starting direction is a normal to a divisor. Middle: The corresponding cartoon tropical curve. 
    Right: Actual tropical graph in the dual complex. }
  \label{fig:dp1full}
\end{figure}

\subsection{Collisions at generic points}

In this section we describe the conditions on a standard polyhedral
decomposition that ensures that for broken disks of Maslov index $2$,
collisions occur at generic points in the sense of Definition
\ref{def:coll-interior}. The following is the precise statement:

\begin{proposition}{\rm(Collisions at generic points)}\label{prop:collg}
  There exists a constant $\eps>0$ such that the following holds: Suppose $\PP$ is a standard polyhedral decomposition $\PP$  satisfying 
  \begin{enumerate}
  \item  $P_0 \supset (1-\eps)\Phi(X)$ where $P_0 \in \PP$ is the polytope containing the Lagrangian fiber $\Phi(L)$, and 
  \item for any cut space $X_{P_b}$ ($P_b \in \PP$) that contains focus-focus singularities, the primitive normal $\nu \in \t_\Z$ to the facet $P_b \cap P_0$ does not lie in $\Phi(X)^\dual$.
  \end{enumerate}
  Then in a  broken disk of Maslov index two in $\XX_\PP$ with tropical graph $\Gamma$,
  \begin{enumerate}
  \item \label{part:collg1} the disk vertex $v_0$ of $\Gamma$ is univalent,
  \item \label{part:collg2} the initial slope of $\Gamma$ lies in $\Phi(X)^\dual$, and
  \item \label{part:collg3} if for a vertex $v$, $X_{P(v)}$ contains a focus-focus singularity, then, $v$ is univalent.
  \end{enumerate}
\end{proposition}

\begin{lemma}{\rm(Areas of disks in toric orbifolds)}
  \label{lem:disk-area}
Let  $X$ be a toric orbifold with moment map $\Phi : X \to \t^\dual$
  and suppose the moment polytope $\Delta:=\Phi(X)$ containing $0$ in the
  interior.
  Let $u: D\bs \{0\} \to X \bs \Phinv(\partial \Delta)$ be a holomorphic
  disk with boundary in the Lagrangian $L=\Phinv(0)$, with 
  $\lim_{z \to 0}u(z) \in \Phinv(\partial \Delta)$, and $u_*[\partial D]=\xi \in \t_\Z$. Then,
  \begin{equation}
    \label{eq:disk-area}
    \on{Area}(u)=c_\xi,  
  \end{equation}
  where $c_\xi >0$ is such that for the plane $F_\xi:=\{\bran{x,\xi}=c_\xi\} \subset \t^\dual$, the intersection $\Delta \cap F_\xi$ is non-empty and contained in $\partial \Delta$. 
\end{lemma}
\begin{proof}
  The proof is identical to that of Theorem 8.1 in Cho-Oh \cite{chooh:fano} which computes the area of disks that intersects the boundary divisor at a single smooth point. 
\end{proof}

The next result gives a lower bound on the constants $c_\xi$ occuring
in the area formula \eqref{eq:disk-area} when $\xi$ is not in the dual
of $\Phi(X)$.

\begin{lemma}\label{lem:cxi}
   Let $\Delta \subset \t^\dual \simeq \R^2$ be the moment polytope of a compact toric orbifold all whose faces are of the form $F_\mu=\{\bran{x,\mu}=1\}$ where $\mu \in \t_\Z$ is a primitive vector. 
  For any
  $\xi \in \t_\Z$ let $c_\xi$ be the constant such that for the
  hyperplane 
  \[ F_\xi:=\{\bran{x,\xi}=c_\xi\} \subset \t^\dual , \] 
  the
  intersection $\Delta \cap F_\xi$ is non-empty and contained in
  $\partial \Delta$.  There is a constant $C>1$ such that for any
  $\xi \in \t_\Z\bs \Phi(X)^\dual$,
  \[c_\xi \geq C.\]
\end{lemma}
\begin{proof}
  We consider the case when $\xi \notin \Phi(X)^\dual$ and $F_\xi$
  intersects $\Delta$ at a vertex $x_0$, the proof of the other cases
  is the same. Suppose $x_0$ is the intersection of facets
  $F_{\mu_i}=\{\bran{y,\mu_i}=1\}$, $i=1,2$ of $\Delta$. Then, $\xi$
  lies in the cone generated by $\mu_1$, $\mu_2$, that is,
    \[\xi=C_1(t \mu_1 + (1-t)\mu_2)\]
    for some $t \in [0,1]$, $C_1 \geq 0$.
We have $c_\xi=C_1$, because  $F_\xi$ passes through $F_{\mu_1} \cap F_{\mu_2}$.
    Since $\xi \notin \Delta^\dual$,
    $C_1>1$.  
      Define $C>1$ to be the least value for which
    the set $(C \Delta^\dual \bs \Delta^\dual) \cap \t_\Z$ is non-empty.
    We conclude $c_\xi=C_1 \geq C$ because $\xi \in \t_\Z \cap C_1 \Delta^\dual$. 
\end{proof}

\begin{proof}
  [Proof of Proposition \ref{prop:collg}] We first prove
  \eqref{part:collg1}, namely that the disk vertex is
  univalent. Suppose the disk vertex $v_0$ of $\Gamma$ is not
  univalent, and as incident edges $e_1,\dots, e_k$, $k \geq 2$. There
  is a decomposition of the homology class
  \[[u_0]=\sum_{i=1}^k [u_i] \in H_2(X_{P_0},L),\]
  where $u_i$ is disk in $X_{P_0}$ with a single relative node with
  direction $\cT(e_i)$.  For reasons of dimension of the moduli space,
  among the edges $e_1,\dots, e_k$, there is exactly one outgoing edge
  (with respect to the constraint orientation from Lemma \ref{lem:orient})  which we
  assume to be $e_1$.  We arrive at a contradiction by constructing a
  non-constant broken map whose area is smaller than that of $u$. Let
  $u_1'$ be the relative map consisting of all the sphere components
  of $u$ that are connected to $u_0$ via $e_1$. There is a relative
  map $u_1''$ homotopic to $u_1'$ such that $(u_1,u_1'')$ satisfy the
  matching condition at the edge $e_1$. Then, $(u_1,u_1'')$ is a
  broken map with positive area, but whose area is less than
  $\on{Area}(u)$ which is a contradiction, since $(X,L)$ is
  monotone. Therefore, the disk vertex of $\Gamma$ is univalent.

  To prove part \eqref{part:collg2}, we again assume the contrapositive, namely that the initial slope of the broken map does not lie in $\Phi(X)^\dual$. In other words, we assume that for the disk component
  $u_0:D \to X_{P_0}$ of $u$, the element 
  \[ \xi:=(u_0)_*[\partial D] \in \t_\Z \] 
  does not lie in $\Phi(X)^\dual$. By Lemma \ref{lem:disk-area},
  and the condition $P_0 \supset (1-\eps)\Phi(X)$, we get 
  $\on{Area}(u_0)=(1-\eps)c_\xi$. Here $c_\xi>0$ is defined by the condition that the plane $\{\bran{x,\xi}=c_\xi\}$ intersects $\Phi(X)$ at the boundary $\partial \Phi(X)$. By Lemma \ref{lem:cxi}, $c_\xi \geq C$, and thus,
  \[\on{Area}(u_0) \geq (1-\eps)C.\]
  We fix $\eps$ so that $(1-\eps)C>1$ so that the assumed case
  does not occur. That is, the initial slope of the broken map $u$
  lies in $\Phi(X)^\dual$.

  Next, we prove \eqref{part:collg3} which says that there are no
  collisions at singular points of the dual complex, that is, there
  cannot be a vertex $v$ with valence at least $2$ and $P(v)=P_b$ and
  $X_{P_b}$ contains a focus-focus singularity.  Lemma \ref{lem:no-interior} shows
  that tropical graphs in the annular part of the dual complex spiral
  outwards.   A collision at a singular point can occur
  only if the disk component $u_0:D \to X_{P_0}$ of $u$ intersects the
  divisor $X_{P_0} \cap X_{P_b}$.
  This can happen only if
  $\xi$ is equal to the primitive normal
  of the facet $P_0 \cap P_b$, which by assumption, does not lie in
  $\Phi(X)^\dual$. Therefore, by \eqref{part:collg3} such a disk does
  not exist.
\end{proof}

\subsection{Collisions in the interior}
\label{sec:interior}

In this Section, we prove that collisions happen in the interior.

\begin{proposition}\label{prop:col-int}
  {\rm(Collisions in the interior)}   Suppose that $X$ is a monotone symplectic four-manifold with an almost toric structure with no elliptic singularities.   For a large enough $\rho$, for a tropical graph $\Gamma$ corresponding to a disk of Maslov index $2$, there is a single vertex $v_\partial$ with $P(v_\partial) \in \PP_{\partial}$. Furthermore,
\begin{enumerate}
\item the map $u_{v_{\partial}}$ is a holomorphic cylinder intersecting a boundary divisor $\Phinv(F_\nu)$ where $F_\nu$ is a facet of $\Phi(X)$ with normal vector $\nu \in \t_\Z$, and 
\item  $v_\partial$ is univalent with incident edge $e_\partial$ and $\cT(e_\partial)=\nu$. 
\end{enumerate}
\end{proposition}

\begin{proof}
  First, we show that for a large $\rho$, any tropical graph has
  outward-pointing edges.  We recall from Definition \ref{def:good}
  that a good polyhedral decomposition is a combination of $\PP_\ann$
  and $\PP_{\on{in}}$. We apply Proposition \ref{prop:b0edge} with
  $\PP_0:=\PP_\ann$ and $\PP_1:=\PP_{\on{in}}$ and conclude the
  following: There exists $\rho_0$ such that the set of tropical
  graphs realizable in $B^\dual_\rho$ is the same for all
  $\rho \geq \rho_0$. Proposition \ref{prop:b0edge}
  \eqref{part:b0edge2} implies that in any such tropical graph
  $\Gamma$, any path from the disk vertex $v_0$ to a vertex in the
  boundary contains an outward-pointing edge in the sense of
  Definition \ref{def:op} (or a $B_0^\dual$-edge in the terminology of
  Proposition \ref{prop:b0edge}). Let $S \subset \Edge(\Gamma)$ denote
  the set of outward-pointing edges in $\Gamma$, and let
  $\Gamma_0 \subset \Gamma \bs S$ be the connected component
  containing the disk vertex $v_0$. By the terminology of Definition
  \ref{def:op}, for any edge $e$ in $S$, $n_\partial(\cT(e))$ is the
  intersection number of $\cT(e)$ with the boundary divisor. For the
  reader's convenience we recall that in the case with no elliptic
  singularity, the direction of an outward pointing edge $e$ is
  $n_\partial(\cT(e)) \mu_{\on{out}}$ and $P(e)^\dual$ is a rectangle.  The
  proof proceeds via the following claim.

  \begin{claim} For the subgraph $\Gamma_0 \subset \Gamma$ as above, 
  \begin{equation}
    \label{eq:agam0}
    \sum_{v \in \Ver(\Gamma_0)} A(u_v)= (1-\eps)\sum_{e \in S} n_\partial(\cT(e)), 
  \end{equation}
  where $n_\partial(\cT(e))$ is the intersection number of $\cT(e)$ with the boundary
  divisor (Definition \ref{def:op}).    
  \end{claim}
  \begin{proof}
    [Proof of Claim] We construct an augmentation of $u|\Gamma_0$
    denoted by $u_{\on{aug}}$ with tropical graph $\Gamma_{\on{aug}}$, by adding
     vertices to $\Gamma_0$ corresponding to each edge in $S$. 
    
    First, consider the case that an edge $e \in S$ that lies in a rectangular dual
    polytope $P(e)^\dual$. Suppose $e$ is incident on the vertex $v_e$
    in $\Gamma_0$. In a realization of $\Gamma$, 
    the edge $e$ exits $P(e)^\dual$ at an edge $Q^\dual \subset P(e)^\dual$. We 
    add vertex $v_e'$ as the other end-point of $e$,
    with $P(v_e'):=Q$, and let $u_{v_e'}$ be a
    holomorphic cylinder which is an $n_e$-fold cover
    of a fiber of the projection $\ol{X}_{P(e)^\dual} \to \ol{X}_Q$.
    Each fiber $F \cong \P^1$ itself is a toric variety, with moment polytope the fiber 
    $\Phi(F)$ of the projection from $P(e)^\dual$ to $Q$. As a special case
    of the Archimedes' formula \cite[Theorem 6.4]{gu:ka}
    the area of the fiber is $\eps$, up to a universal normalization constant, which 
    with our conventions is equal to $1$.   Indeed, as $\eps \to 1$, the broken map corresponding to a Maslov index two disk with direction $\cT(e)$ consists of a sphere with image $F$ and a disk with small area; it follows that the area of $F$ tends to $1$ as $\eps \to 1$.  Hence the area of $u_{v_e'}$ is $n_e \Vol(F)$. 
    
    Next, consider an edge $e \in S$ that lies in the dual polytope $R_x^\dual$ corresponding to an elliptic singularity $x$.   Suppose $\cT(e)=n_1 \mu_{\on{out}}^{(1)} + n_2 \mu_{\on{out}}^{(2)}$ as in Definition \ref{def:op} \eqref{op2}. If $n_1$ or $n_2$ vanishes, we add a single vertex
    $v_e'$ as the second endpoint of $e$ exactly as in the preceding case. If both
    $n_1$, $n_2$ are non-vanishing, 
    we add $3$ vertices as in Figure \ref{fig:corner}. In all cases, 
    the sum of the areas of the maps corresponding to the new vertices
    is $n_e \eps$.  The map $u_{\on{aug}}$ has $\sum_{e \in S} n_e$
    intersections with boundary divisors.  Therefore, the area of
    $u_{\on{aug}}$ is $\sum_{e \in S} n_e$.  The equation \eqref{eq:agam0}
    follows.
  \end{proof}
  
  Since the total area if $u$ is $1$, from \eqref{eq:agam0}, we conclude that $S$ has a single element $e$ and $n_e=1$. Any map $u$ other than $u_{\on{aug}}$ that contain $u|\Gamma_0$ has area more than that of $u_{\on{aug}}$. We conclude that $u=u_{\on{aug}}$, which proves the Proposition.
\end{proof}

\begin{figure}[ht]
 \begin{center}
  \scalebox{.8}{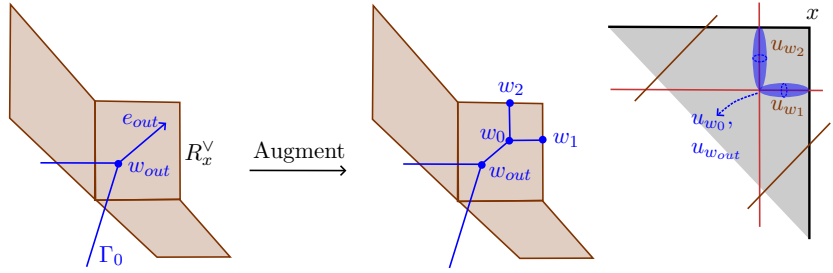}
 \end{center}
 \caption{Left: The subgraph of the tropical graph $\Gamma_0 \subset \Gamma$} 
 \label{fig:corner}
\end{figure}

\begin{figure}[ht]
  \begin{center}
    \scalebox{.8}{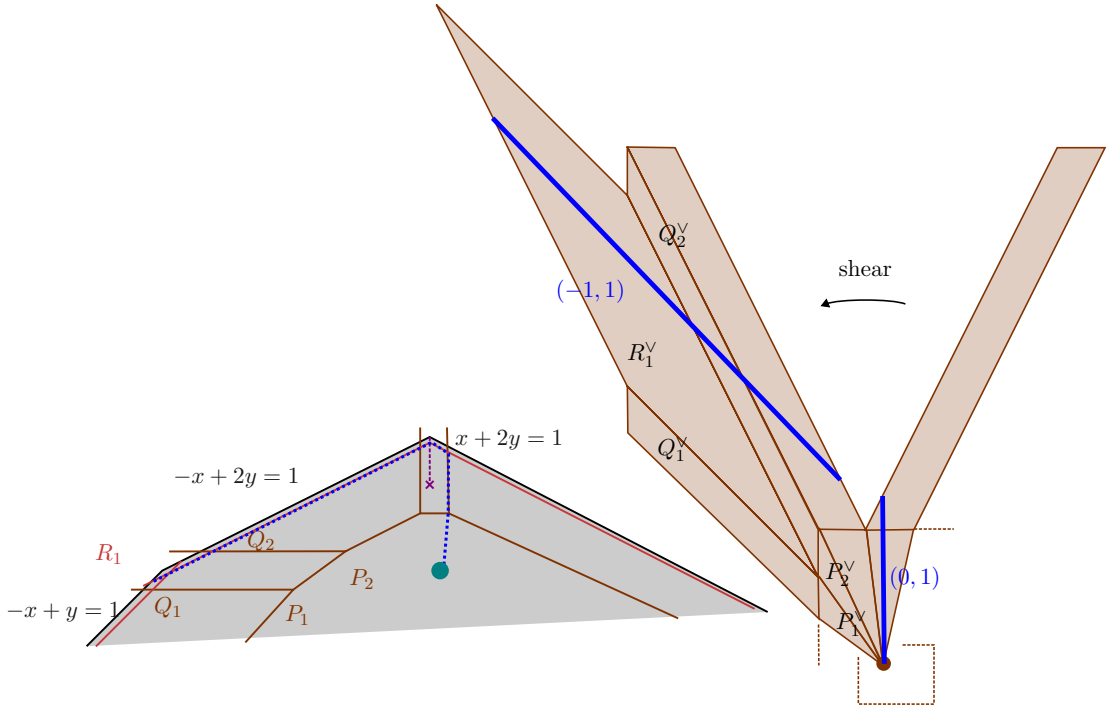}
  \end{center}
  \caption{Left: A portion of a moment polytope containing a $A_1/\Z_2$-singularity and an elliptic singularity. Right : A tropical graph of a Maslov index two disk.} 
  \label{fig:ff2ell}
\end{figure}

\subsection{Tropical graphs spiral outward}
\label{sec:spiral}

We now give a complete description of tropical graphs corresponding to rigid broken disks in standard decompositions, and prove Proposition \ref{prop:interior}.
We assume that edges in the tropical graph are oriented by the constraint orientation from Lemma \ref{lem:orient}. 

We define a notion of {\em outward-pointing} cones on the annulus part of the dual complex,
which will be used to show that the tropical graphs corresponding to broken disks of Maslov index two 
{\em spiral outward} towards
the boundary of the moment polytope.

\begin{definition}{\rm(Outward pointing cone)}
  \label{def:outcone}
  \begin{enumerate}
  \item {\rm(On the interior of a rectangle)}
    Let $P^\dual \subset B^\dual_\ann(\rho)$ be a rectangular dual polytope given by
    $P^\dual=P_0^\dual \times \rho P_1^\dual$. 
    For any $x \in \on{int}(P^\dual)$, the {\em outward pointing cone} is  a half space
\[H_x:=\{v \in T_x B^\dual : g_P(v,\mu_{\on{out}}) \geq 0\},\]
    where $g_P$ is the product metric on $P_0^\dual \times \rho P_1^\dual$. 
  \item {\rm(On the boundaries of rectangles)}
    Let $P^\dual, Q^\dual \subset B^\dual$ be rectangular polytopes whose intersection $P^\dual \subset Q^\dual$ is an edge. For any $x \in P^\dual \cap Q^\dual$,
    the outward pointing cone is
    \[H_x:=\{v \in T_x B^\dual : g_P(v,\mu_{\on{out}}), g_Q(v,\mu_{\on{out}}) \geq 0\}.\]
  \item {\rm(Near elliptic singularities)}
    Let $x \in X$ be an elliptic singularity and let $R_x^\dual \subset B^\dual$ be the corresponding dual polytope that shares edges with rectangular dual polytopes $P_0^\dual$, $P_1^\dual$. For any $y \in R_x^\dual$, the outward-pointing cone $H_y \subset T_yB^\dual$ is obtained by parallel transporting (with respect to the affine structure) the half planes in the interior of $P_0^\dual$ or $P_1^\dual$. Both give the same half planes.
    
  \end{enumerate}
\end{definition}

\begin{example}
  In Figure \ref{fig:dual1chart}, for $x \in Q_1^\dual$ the half plane
  $H_x$ is bounded by $(1,0)$, whereas for $x \in Q_2^\dual$ the half
  plane $H_x$ is bounded by $(-3,-1)$. The half-planes remain constant
  while passing over a shear.  For $x \in Q_1^\dual \cap Q_2^\dual$,
  $H_x$ is a cone $\R_{\geq 0}\bran{(1,0),(-3,-1)}$.
\end{example}

\begin{lemma} \label{lem:outward}
  Suppose $\ell \subset B^\dual_{\ann}$ (defined in Remark \ref{rem:B-ann}) is an oriented line segment starting at a point $x \in B^\dual_{\ann}$ and which is outward-pointing at $x$. Then, $\ell$ is outward-pointing everywhere. 
\end{lemma}

The proof is straightforward and left to the reader. 

\begin{lemma}\label{lem:no-interior}
  Let $\Gamma$ be a 
  tropical graph in $B^\dual(\rho)$ corresponding to a broken disk of Maslov index two, with disk vertex $v_0$ with incident edge $e_0$. 
  \begin{enumerate}
  \item Any edge $e \in \Edge(\Gamma) \bs \{e_0\}$, when equipped with the constraint orientation, has direction lying in the outward pointing cone. 
  \item All vertices
    $v \in \Ver(\Gamma)$ except for disk vertices
    lie in
 the annulus part  $B^\dual_\ann(\rho)$ of the dual complex. 
  \end{enumerate}
  
\end{lemma}

\begin{proof}[Proof of Lemma \ref{lem:no-interior}]
  A tropical graph $\Gamma$ has a path emanating from a disk vertex
  $v_0$ lying in $P_0^\dual$, and this path intersects the annulus
  $B^\dual_\ann$ in an outward-pointing direction.  In fact, in such a
  path $v_0$ is the only vertex not lying in the annulus
  $B^\dual_\ann$.  Any other path starting from a univalent vertex
  starts out at a piece $X_P$ where $P$ contains a focus-focus value
  and $P^\dual$ is a point in $\partial B^\dual_\ann$ and is outward
  pointing; therefore such a path is contained in
  $B^\dual_\ann$. Finally, when paths collide, their merger is a path
  with an outward-pointing direction, because the set of outward
  pointing directions is a cone at any point.

\end{proof}


%



\section{Newton polygons and mutations} 
\label{sec:compute}

In this section we prove various theorems on the potentials of almost toric four-manifolds manifolds; namely that potentials for mutations of almost toric diagrams are related by
algebraic mutation in Theorem \ref{thm:nowall}, that  the 
Newton polygon is the convex hull of the normal vectors to the facets in Theorem \ref{thm:mutthm}, and that the potential coefficients are binomial along edges in Theorem \ref{thm:bithm2}.

\subsection{The dual affine manifold}

Given a monotone almost toric manifold, the standard decomposition of
Section \ref{sec:atinfinity} gives a family of dual complexes, whose
limit is an dual affine manifold.  Up to equivalence, a dual affine
manifold can be constructed directly from the base diagram of the
almost toric manifold, without considering the polyhedral
decomposition.  We show that the potential is preserved by these
equivalences.  Thus we obtain a direct combinatorial way of computing
the potential from the base diagram of a monotone almost toric
manifold.  In the next section, we consider the question of
deformations of the dual affine manifold that involve crossing a wall,
causing the potential to mutate.

Given an almost toric manifold $X$ and a standard decomposition $\PP$, the
corresponding dual affine manifold $\A(X)$ is defined as the limit of
dual complexes:

\begin{notation}
A standard polyhedral decomposition $\PP$
of an almost toric manifold $X$ has a family of dual complexes $\{B^\dual(\rho)\}_{\rho>0}$
 (see \eqref{eq:dualrho}).
The parameter $\rho$ is the ratio of sides of certain rectangles that occur in the dual complex. 
By taking the parameter $\rho$ in the dual complex  to infinity, we obtain a singular affine manifold (complete, without boundary) as follows.
We have natural inclusions of singular integral affine manifolds 
\[ \on{int} B^\dual(\rho) \subset 
\on{int} B^\dual(\rho'), \quad \rho < \rho' \] 
arising from the canonical isomorphisms between the neighborhood of
$P_0^\dual$ in $B^\dual(\rho)$ and $B^\dual(\rho')$. 
(We recall that $P_0$ contains the Lagrangian fiber, and thus, $P_0^\dual$ is the point in the dual complex where the disk vertices lie.)
\end{notation}

\begin{definition}
    The union 
\begin{equation} \label{eq:ax}
\A(X) = \bigcup_{\rho \ge 0} \on{int} B^\dual(\rho) \end{equation}
is the {\em dual affine manifold} for $X$. 
The Lagrangian point is the element $\lam:=P_0^\dual$ where $P_0$ contains the image of the Lagrangian.
\end{definition}

 See Figure \ref{fig:b8p2-affine} for an example.
The dual affine manifold depends on the decomposition
$\PP$, but the dual affine manifolds corresponding to various standard decompositions have the same potential, see Remark \ref{rem:nowall}. 
 We denote by $T_{\Z}(\A(X)) \subset T \A(X)$ the
set of integral tangent directions.

\begin{remark} \label{rem:ax}
Given an almost toric manifold $(X,\Phi)$ that does not have 
elliptic singularities, the dual affine manifold $\A(X)$ is given by
\begin{equation}
  \label{eq:dualconst}
  \left(\Phi(X) \cup (\cup_i F_i \times \R_{\geq 0}) \right)/\sim   
\end{equation}
where $\Phi(X)$ is embedded in $\t$ by an arbitrary choice of a linear
identification $\t \simeq \t^\dual$, $F_i$ ranges over all facets of
$\Phi(X)$, $F_i \times \R_{\geq 0} \subset \t$ is a rectangle where
$R_{\geq 0}$ is in the outward normal direction of the facet $F_i$,
$F_i \times \{0\}$ is the facet $F_i \subset \Phi(X)$, and $\sim$
makes the following identification: For a vertex $v \in F_i \cap F_j$
of $\Phi(X)$, the line $\R_{\geq 0} \times \{v\}$ in the rectangles
$F_i \times \R_{\geq 0}$, $F_j \times \R_{\geq 0}$ are identified and
the affine structures in the neighborhoods are identified via a shear. 
See Figure \ref{fig:affineex}. 
\end{remark}

\begin{definition}
Given an almost toric diagram and gluing datum, the {\em tropical potential} of the dual affine manifold $\A$ with distinguished point $\lam$ is
\[\W_{(\A,\lam)}: \Hom(T_\lam \A, \C) \to \C, \quad y \mapsto \sum_\bGam y(\partial \bGam) m(\bGam),\]
where $\partial \bGam \in (T_\lam \A)_\Z$ is the initial slope of $\bGam$, and $\bGam$ ranges over all 
$\A$-tropical trees of index two, 
and $\m(\bGam) = \prod_{v \in \Ver(\bGam)} m(v)$ is defined in Definition \ref{def:mvs}.
\end{definition}

\begin{remark} 
By the definition of the dual affine manifold as the limit of dual complexes, the potential $W_{(\A,\lam)}$ is equal to the potential computed from broken map in a standard decomposition  in \eqref{eq:bdp}. 
\end{remark}

\begin{remark} By choosing an identification
\begin{equation}
  \label{eq:Tlam-r2}
  T_\lam \A \simeq \t \simeq \R^2  
\end{equation}
that sends $(T_\lam \A)_\Z$ to $\t_\Z$ and $\Z^2$, the potential is a Laurent polynomial
\[W_{(\A,\lam)} : \C^2 \to \C.\]
Transforming the identification by $A \in GL(2,\Z)$ has the effect of applying a multiplicative
transformation to each monomial in the potential $W$: the monomial $x^a y^b$ is transformed to $x^c y^d$ where $A^t \big(\begin{smallmatrix}
  a\\
  b
\end{smallmatrix}\big) = \big(\begin{smallmatrix}
  c\\
  d
\end{smallmatrix}\big)$. 
\end{remark}

The potential of a dual affine manifold is unchanged by homotopies
that do not contain a ``wall'' in the sense of the following
definition:

\begin{definition}
  {\rm(Wall configurations)}
  \begin{enumerate}
  \item An
  $\A$-tropical disk has {\em index zero} if it does not have any
  semi-infinite edges.
\item 
  A dual affine manifold $\A(X)$ with a
  Lagrangian point $\lam \in \A(X)$ is a {\em wall configuration} if
  there is an $\A$-tropical disk of index zero in $(\A(X),\lam)$. A wall configuration is {\em simple} if the $\A$-tropical disk has two vertices -- a disk vertex at
  $\lam$ and a sphere vertex at a singular point $b$. (More complicated wall configurations exist, but in this paper we only encounter simple walls.)


\item Dual affine manifolds $\A_0$, $\A_1$ are {\em equivalent}, if there is a continuous family of dual affine manifolds $\{A_t\}_{t \in [0,1]}$ that does not contain any wall configuration. 
  \end{enumerate}
\end{definition}

For example, two distinct singular points $b_0$, $b_1$ on the same branch cut in an affine manifold $\A_0$ may coalesce to the same point in an equivalent affine manifold $\A_1$ producing equivalent dual affine manifolds.  See Figure \ref{fig:moving-ff} for example.

\begin{figure}[ht]
 \begin{center}
   \scalebox{.8}{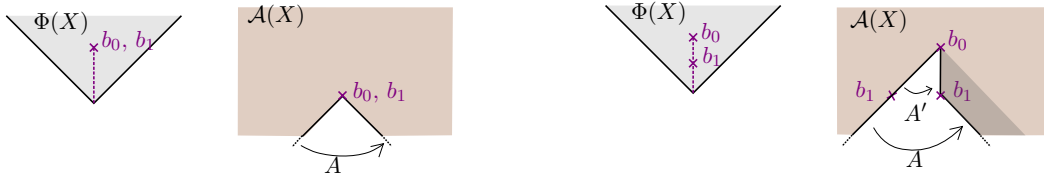}
 \end{center}
 \caption{Moving the singular point $b_1$ inward (towards the Lagrangian point $\lam$) removes a rectangle from the dual affine manifold, but does not change the equivalence class. }
 \label{fig:moving-ff}
\end{figure}

For an almost toric manifold $X$ and a monotone Lagrangian fiber, the
dual affine manifold can be determined up to equivalence by the
construction \eqref{eq:dualconst}, which produces a dual manifold
where all focus-focus singularities in a branch cut
coincide. Equivalent dual manifolds are obtained by moving the
singularities and $P_0^\dual$, while avoiding walls.

\begin{theorem}\label{thm:nowall}
  For a pair of equivalent dual affine manifolds $(\A_0, \lam_0)$, $(\A_1, \lam_1)$,
  $W_{(\A_0,\lam_0)}=W_{(\A_1,\lam_1)}$.
\end{theorem}


\begin{proof}
  Suppose, as in the assumption in the statement of the Theorem, that the affine manifolds $\A_0$, $\A_1$ are connected by a path
  $\{\A_t\}_t$ that does not contain a wall configuration. After a small
  perturbation, we may assume that the path $\{\A_t\}_t$ is generic in
  a sense pointed out later in the proof.

  We first list the possible reasons a family of tropical
  graphs
  ends at a point. In particular, 
  if a tree $\bGam$ with edge slopes $\{\cT(e)\}_e$ has an embedding $\cT_t$ in $\A_t$ for $t \in (t_0 -\eps, t_0)$ but does not have an embedding in $\A_{t_0}$, then one of the following occurs:
  \begin{enumerate}[label*=(\arabic*)]
  \item \label{part:e0} An edge length goes to zero, that is, for an edge $e=(v_-, v_+)$ of $\bGam$, 
  \[ \lim_{t \to t_0}\cT_t(v_-)=\lim_{t \to t_0}\cT_t(v_+); \] 
  \item \label{part:eff} or in the limit $t \to t_0$, the embedding of an edge $e$ passes through a singular point. 
  \end{enumerate}
  %

Consider case \ref{part:e0} 
  where the length of an edge $e=(v_-, v_+)$ goes to zero in the
  tropical embedding.
  Define the limit tropical graph 
  \[ \bGam_0:=\lim_{t \to t_0^-} (\bGam, \cT_t) \] 
  to be the tropical disk where the end-points of $e$, $v_+$ and $v_-$ are replaced by a single vertex $v$. Let $S_\pm$ be the set of index two tropical disks
  in $(\A_{t_0 \pm \eps}, \lam_{t_0 \pm \eps})$ for a small enough $\eps>0$ which deform to $\bGam_0$ as $\eps \to 0$.

  The curve count in $S_+$ and $S_-$ is the same since they are both
  smoothings of $\bGam_0$ as we now explain.  We observe that in
  $\bGam$ at most one end-point of $e$ has valence one, in which case,
  the univalent vertex is $v_-$ since the edge incident on a univalent
  vertex is outgoing. Therefore the following cases (see Figure \ref{fig:types}) arise:

\begin{figure}[ht]
 \begin{center}
   \scalebox{.8}{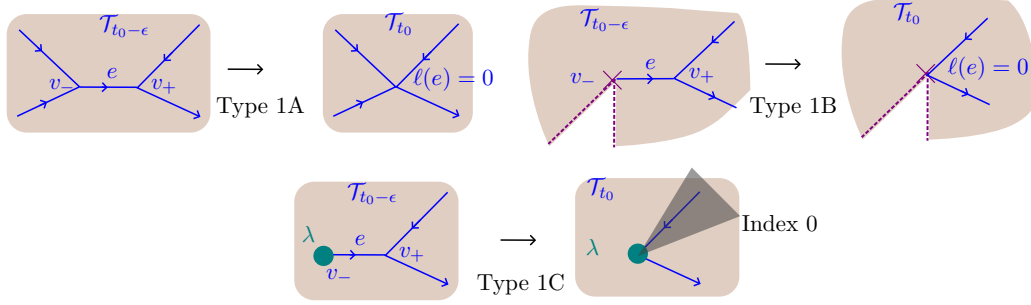}
 \end{center}
 \caption{Types of end-points for a one-dimensional family of tropical graphs.}
 \label{fig:types}
\end{figure}

  {\rm Case 1A:} First, consider the case that both $v_+$, $v_-$
  have valence at least $3$ in $\bGam$. Consider any tropical graph $\bGam_\pm \in S_\pm$, and let $\bGam_{\pm,v} \subset \bGam_\pm$ be the subgraph 
  consisting of the vertices that get collapsed to $v$ in $\bGam_0$. The positions of the incoming edges \footnote{By `position', we mean the starting point and direction.}
  of $\bGam_{\pm,v}$,
   collectively denoted by $\ell_\pm(v)$, 
  are fixed across all graphs $\bGam_\pm \in S_\pm$, and these positions are perturbations of the edge positions in $\bGam_0$. Thus $\{\bGam_{\pm,v} : \bGam_\pm \in S_\pm\}$ is the set
  of split perturbations (as in Definition \ref{def:vpert}) of $\bGam_{0,v} \subset \bGam_0$ with splitting data $\ell_\pm$. Therefore,
 by the proof of the last part of Theorem \ref{thm:potthm}
  in Section \ref{subsec:desing},
  \[ \sum_{\bGam_{+,v}} m(\bGam_{+,v}) = \sum_{\bGam_{-,v}} m(\bGam_{-,v}).\]
  For any $\bGam_\pm \in S_\pm$, the vertices in
  $\bGam_\pm \bs \bGam_{\pm,v}$, alongwith their incident edges, are
  the same as the corresponding vertices in $\bGam_0$. Therefore, the
  count of tropical graphs in $S_+$ and in $S_-$ is the same.

  {\rm Case 1B:} Next, consider the case that $v_-$ is univalent and
  maps to a singular point $b$ of $\A(X)$. In this case, $e$ may be
  one of the multiple coincident (as in Definition \ref{def:bunch})
  edges $e_1,\dots,e_k$ whose lengths go to zero simultaneously, all
  of which are incident on univalent vertices lying on the singular
  point $b$, and whose other end-point is $v_+$.  In $\bGam_0$, the
  vertex $v$ obtained by collapsing the edges $e_1,\dots, e_k$ lies on
  the singularity $b \in \A_{t_0}$. By the genericity of the path
  $\{\A_t\}_t$, we may assume the following: $b$ does not coincide
  with other singular points in $\A_t$ for $t$ close to $t_0$; $v$ has
  a single input edge; and there is no collision at the singular point
  for $t \neq t_0$ in the neighborhood of $t_0$. Then for any
  $\bGam_\pm \in S_\pm$, the subgraph $\bGam_{\pm,v}$ is a
  $(\pm)$-perturbation (or equivalently $(\mp)$-perturbation,
  depending on the sign convention) of $\bGam_{0,v}$. By Proposition
  \ref{prop:ffcross}, the counts of $(+)$ and $(-)$-perturbations of
  $(\bGam, \cT_{t_0})$ are equal to each other.  Therefore, the count
  of tropical graphs in $S_+$ is equal to the count in $S_-$.
 
  {\rm Case 1C:} Next, consider the case that $v_-$ is univalent and is
  a disk vertex. We rule out this case by showing that $\A_{t_0}$ is a
  wall configuration. Since $v_+$ has valence at least $3$ in $\bGam$,
  the vertex $v$ in $\bGam$ has valence at least $2$. The tropical
  disk $\bGam$ has one semi-infinite edge, and therefore the same is
  true of $\bGam_0$. Therefore, there is an edge $e_0 \ni v$ of
  $\bGam_0$, such that the subgraph $\bGam'_0 \subset \bGam_0$
  consisting of $v$ and all the vertices whose path to $v$ contains
  $e$, is a tropical disk of index $0$. Therefore, $\A_{t_0}$ is a
  wall configuration which is a contradiction.

  Finally, we consider the case \ref{part:eff} where an edge $e$
  passes through a focus-focus singularity $b$ in $\A_{t_0}$. We
  subdivide the edge, and add a vertex $v$ that lies on $b$. This
  graph is the same as the graph in case 1B and can be analyzed in the
  same way. See the top row of Figure \ref{fig:slide-wall} for an
  example.
\end{proof}

  \begin{figure}[ht]\begin{center}
    \scalebox{.8}{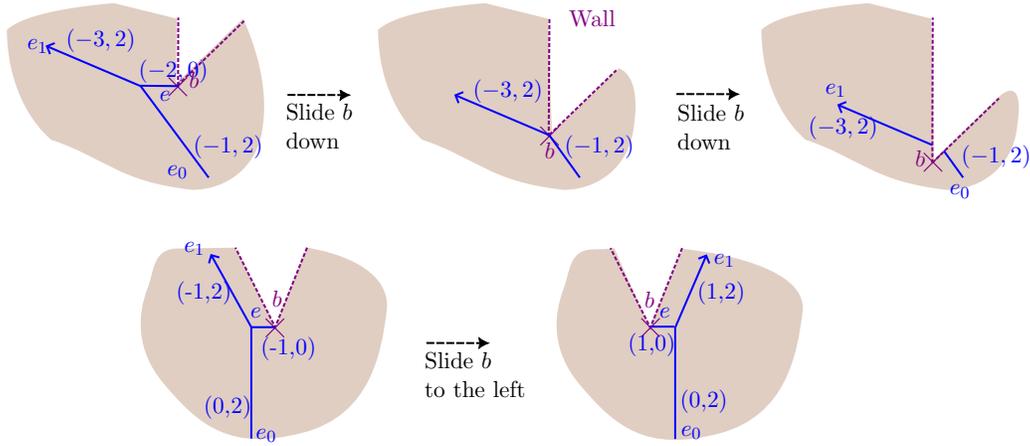}
  \end{center}
  \caption{Sliding a node does not change the count of these tropical graphs. The top left is a bunched tropical graph.}
  \label{fig:slide-wall}
\end{figure}

To apply Theorem \ref{thm:nowall}, we need to show that a path of dual affine manifolds
does not contain a wall. For this purpose, 
it is helpful to have a notion of an annular part of a
dual affine manifold, which is a slight variation of the corresponding
notion \eqref{eq:bann-def} for dual complexes: For our purposes we
consider dual affine manifolds $(\A,\lam)$ that are close to those
obtained from a base diagram $\Delta$ of an almost toric manifold and
a polyhedral decomposition $\PP$.  Therefore, we may assume that
for any $(\A,\lam)$ there is a related moment polytope $\Delta$.  In
such a case, we choose a bounded convex polygon $S \subset \A(X)$ that
contains all the singular points of $\A(X)$, $\lam$ lies in the
interior of $S$, and for any vertex $v$ of $\Delta$ from which a
branch cut emanates, the outermost singular point $b_v$ on the branch
cut is a vertex of $S$, and furthermore the line $\ell_{b_v}$ of slope
$\mu_{b_v}$ through $b_v$ intersects $S$ only at $b_v$.  We define the
{\em annular part of the dual manifold} to be
\begin{equation}
  \label{eq:ax-ann}
  \A(X)_\ann:= \A(X) \bs \on{int}(S).  
\end{equation}

\begin{remark}\label{rem:nowall}
  In this remark, we show that certain dual affine manifolds are not walls. 
  Using the definition of $\A(X)_\ann$,
  it can be seen in a
  straightforward way that for a dual affine manifold $(\A_0,\lam_0)$
  corresponding to a standard polyhedral decomposition, there are no walls
  nearby.  Therefore the potential is invariant under small
  deformations.
  \begin{enumerate}
  \item Firstly, in $(\A_0,\lam_0)$, all collisions involving inputs
    from singular points happen in the annulus, and consequently
    $(\A_0,\lam_0)$ is not a wall. Similarly, if $(\A_t,\lam_t)$ is a
    family of dual affine manifolds obtained from polyhedral
    decompositions $\PP_t$, the potential $W_{(\A_t,\lam_t)}$ is
    $t$-independent, since no element in the family $(\A_t,\lam_t)$ is
    a wall.
  \item By the definition of a standard polyhedral decomposition, all
    singular points on a branch cut coincide in $\A_0$. Let $\A_t$ be
    a family obtained by separating the singular points. Collisions
    involving inputs from singular points continue to happen in the
    annulus if the singular points are moved by small amounts, and
    therefore the potential is not altered.
  \end{enumerate}
\end{remark}

\subsection{Mutations and wall-crossing}
\label{mutsec}
By a mutation of the moment polytope of an almost toric manifold, we
mean a change in the almost toric structure by which one of the
focus-focus values is moved from one side of the polytope to the
other, as in Definition \ref{def:mutate}.  Lemma \ref{lem:indep}
implies that the disk potential is not altered by a nodal slide as
long as the focus-focus value does not cross the image of the monotone
torus in the moment polytope.  By Vianna's work \cite{vianna:inf},
such a crossing changes the Hamiltonian isotopy class of the monotone
Lagrangian torus fibers.
Pascaleff-Tonkonog \cite{pasc} proved that the change in potential is
given by a mutation of Laurent polynomials.
We prove the same result (Proposition \ref{prop:wallA}) using tropical methods. 
The tropical approach to
the change in potential also figures prominently in the Gross-Siebert
approach to mirror symmetry. See for example
Gr\"{a}fnitz-Ruddat-Zaslow \cite{grz:proper} (where the potential is
computed for a non-monotone torus close to the anticanonical divisor.)

\begin{definition} \label{def:mutform}
  {\rm(Mutation of a Laurent polynomial)}
  Given a vector 
  \[ \nu=(\nu_1,\nu_2) \in \Z^2 ,\] 
the {\em mutation coordinate change} associated to $\nu$ is the map 
\begin{equation}
  \label{eq:mut}
  \cS_\nu : (\C^\times)^2 \to (\C^\times)^2, \quad (y_1,y_2) \mapsto (y_1(1+y_1^{\nu_2}y_2^{-\nu_1})^{\nu_1}, y_2(1+y_1^{\nu_2}y_2^{-\nu_1})^{\nu_2}).
\end{equation}
  The {\em mutation} of a function $W: (\C^\times)^2 \to \C$ in the direction $\nu$ is defined as the pullback\footnote{More general mutations are considered in 
  Coates et al \cite{coates:max}.}
  \[ \cS_\nu^*W:  (\C^\times)^2 \to \C .\] 
  \end{definition} 
  
  \begin{remark} 
  Suppose that $W$ is a Laurent polynomial.  The mutation $\cS_\nu^*W$ is a rational 
  function obtained by replacing each monomial in $W$ as follows:
  \[y_1^{\mu_1}y_2^{\mu_2} \mapsto y_1^{\mu_1}y_2^{\mu_2}(1+y_1^{\nu_2}y_2^{-\nu_1})^{\mu_1\nu_1 + \mu_2 \nu_2}.\]
\end{remark}



\begin{figure}[ht]\begin{center} 
\scalebox{.8}{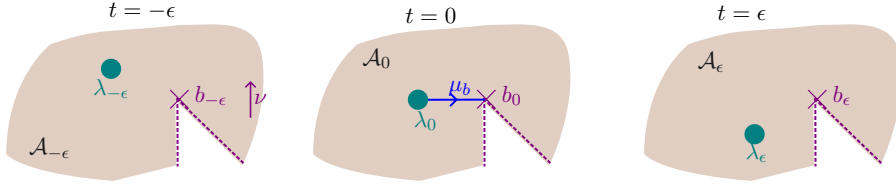}
\end{center} 
\caption{Crossing a wall in the direction $\nu=(0,1)$.} 
\label{fig:mute1}
\end{figure}

In the following result, we show that a mutation of an almost toric manifold
 (Definition \ref{def:mutate} \eqref{part:mutate}) 
corresponds to a simple wall crossing in the dual affine manifold. 

\begin{proposition}\label{prop:simple-wall}
  Suppose $X_0$, $X_1$ are almost toric manifolds related by a mutation 
  where a single focus-focus singularity $b \in X^\foc$
  is moved from one side of the polytope to another. 
  Let 
  $\Delta_0$, $\Delta_1$ be base diagrams of $X_0$, $X_1$  of Vianna type and
  let $(\A_0, \lam_0)$, $(\A_1, \lam_1)$ be dual affine manifolds obtained via standard decompositions applied to $\Delta_0$, $\Delta_1$, with singular points $b_0 \in \A_0$, $b_1 \in \A_1$ corresponding to
 the singularity $b$ in the almost toric manifold. 
  Then, $\A_0$, $\A_1$ are connected by a path of dual manifolds $(\A_t, \lam_t)_{t \in [0,1]}$ which contains a single wall $(\A_{t_0}, \lam_{t_0})$, $0<t_0<1$.
  Furthermore, $(\A_{t_0}, \lam_{t_0})$ is a simple wall, with an index zero tropical disk connecting $\lam_{t_0}$ to $b_{t_0}$ by a linear segment with slope $\mu_b$. 
\end{proposition}
\begin{proof}[Proof of Proposition \ref{prop:simple-wall}]
  The path of dual manifolds is given by moving the singular point
  $b_0$ to $b_1$ along a straight line, with a slight perturbation so that the
  Lagrangian point $\lam_t$ does not lie on it.  We rule out any other
  wall occuring in the path $\{\A_t\}_{t \in [0,1]}$ as follows:
  Moving the singular point $b_t$ inward has the effect of removing a
  rectangle from the dual affine manifold (see Figure \ref{fig:moving-ff}),
  therefore, the notion of
  outward-pointing cones from Definition \ref{def:outcone} is defined
  on the annulus part of $\A_t$.  In a tropical graph in $\A_t$, paths
  emanating from focus-focus singularities are in the outward pointing cone,\footnote{Note that this property is special to the particular kind of almost toric diagrams, in the monotone case, considered here.} and
  are contained in the annulus $\A_{t,\ann}$ (as in \eqref{eq:ax-ann})
  except if their starting point is
  $b_t$. Therefore, two such paths can merge only in the annulus, and
  a merged path can not be pass through the Lagrangian point
  $\lam_t$. Therefore, any wall in $\A_t$ is a simple wall with the
  index zero tropical disk emanating from $b_t$. The path $\{\A_t\}$ is
  such that a line of slope $\mu_b$ passing through $\lam$ can
  intersect the path $\{b_t\}$ at only one point $b_{t_0}$, and there
  is exactly one wall configuration in the path $\{\A_t\}_t$.
\end{proof}

\begin{notation}{\rm(Sign convention)}
  \label{note:sign}
  A mutation formula is defined for a simple wall-crossing satisfying
  a sign convention. Let $\{(\A_t, \lam_t)\}_{t \in [-\eps,\eps]}$ be
  a path of dual affine manifolds, which at $t=0$, crosses a simple
  wall at $t=0$ with index zero disk being the path $(\lam_0,b_0)$ in the primitive direction $\mu_b \in \Z^2$. A
  mutation formula is defined for this wall-crossing if the vector
  path $(\lam_t, b_t)$ rotates counter-clockwise (with respect to the
  $\R^2$ coordinates from the identification \eqref{eq:Tlam-r2} ) at
  time $t=0$, and the wall-crossing is in the direction $\nu = \mu_b^\perp$, where for a vector $(a,b) \in \R^2$,
  \begin{equation}
    \label{eq:perp}
    (a,b)^\perp:=(-b,a).
  \end{equation}
  For example, in Figure \ref{fig:mute1}, the index zero disk $(\lam_0,b_0)$ has primitive direction $\mu_b=(1,0)$, and $\nu=(0,1)$.
\end{notation}

 A simple wall-crossing respecting the above sign convention changes the potential of the dual affine manifold by a mutation.
\begin{proposition}\label{prop:wallA}
  Suppose the path $\{(\A_t, \lam_t)\}_{t \in [-\eps,\eps]}$ has no wall besides a simple wall at $t=0$. Suppose the wall-crossing at $t=0$ respects the sign convention from Notation \ref{note:sign} and is in the direction $\nu$. 
Then,
\begin{equation}
  \label{eq:wtransf}
  W_{(\A_\eps,\lam_\eps)}=S_\nu^* W_{(\A_{-\eps},\lam_{-\eps})}.
\end{equation}
\end{proposition}
\begin{figure}[ht]\begin{center} 
\scalebox{.8}{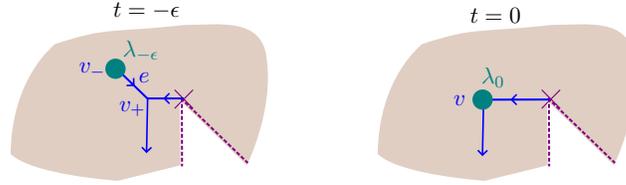}
\end{center} 
\caption{The edge length $\ell(e)$ goes to $0$ in $(\A_0, \lam_0)$.}
\label{fig:mute2}
\end{figure}

\begin{proof}[Proof of Proposition \ref{prop:wallA}]
  By the proof of Theorem \ref{thm:nowall}, the potential $W_{(\A_t,\lam_t)}$ is discontinuous at $t=t_0$ only if $\A_t$ has a tropical disk of type 1C, which can occur only if $\A_{t_0}$ has a wall configuration, so $t_0=0$. Furthermore, since $\A_0$ only has simple walls, the tropical disk of type 1C is of the form shown in Figure \ref{fig:mute2}.

  \begin{figure}[ht]\begin{center} 
\scalebox{.8}{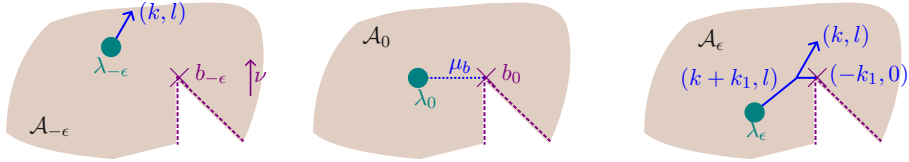}
\end{center} 
\caption{Corresponding to a disk with initial slope $(k,l)$, $l \geq 0$,  in $\A_{-\eps}$, there is a
bunched disk in $\A_\eps$ for each $k_1$. Here, the edge with direction $(-k_1,0)$ is a bunched edge.
}
\label{fig:mute3}
\end{figure}

We count all such contributions. By the $GL(2,\Z)$-equivariance of the
mutation formula \eqref{eq:mut} (see \cite[Remark 4.2]{pasc}), it is enough to consider the case
when $\nu=(0,1)$ and $\mu_b=\pm(1,0)$. Corresponding to any tropical
disk with initial slope $(k,l)$, $l \geq 0$ in $\A_t$, $t<0$, there is
a bunched tropical disk of the form shown in Figure \ref{fig:mute3} in
$\A_t$, $t>0$ for each $k_1 \in [0,k] \cap \Z$. Such a bunched
tropical disk has a multiplicity of $\binom l {k_1}$ by Proposition
\ref{prop:bunchmult}.  Thus, going from $\A_{-\eps}$ to $\A_\eps$
transforms the potential as
  \[x^ky^l, l \geq 0  \mapsto \sum_{k_1} x^{k+k_1} y^l \binom l {k_1} = x^k y^l (1+x)^l,\]
  which is the same as the mutation $S_\nu$. In a similar way,
  corresponding to a collection of terms
  $x^k y^{-l}(1+x)^l$, $l \geq 0$, in $W_{\A_\eps}$, there is a single term in $x^k y^{-l}$ in $W_{\A_{-\eps}}$. Therefore,  \eqref{eq:wtransf} follows. 
\end{proof}

\begin{theorem} \label{thm:mutthm} Let $X$ be a compact almost-toric
  manifold with polytope $\Delta = \Phi(X) \subset \R^2$ and monotone
  Lagrangian $L$.  Let $\Delta' \subset \R^2$ be a new almost-toric
  diagram obtained by a nodal slide that moves the focus-focus value
  $b \in \Delta$ across $\lambda = \Phi(L)$ in the direction
  $\nu \in \Z^2$ followed by transferring the cut as in Figure
  \ref{fig:signs},
\footnote{The Lagrangian fiber in the middle picture in Figure \ref{fig:signs} is depicted as perturbed, to show that in the affine dual, $\lam$ lies to the left of the line of motion of $b$, so the sign convention of Figure \ref{fig:mute1} is followed. In the third picture, the cut is transferred so that the perturbed $\lam$ lies in the part of the moment polytope whose embedding in $\R^2$ is left unchanged.}
  and let $X'$ be the corresponding almost toric
  manifold with monotone Lagrangian $L' \subset X'$.  The potentials
  of $L$ and $L'$ are related by a mutation
  $W_{L'} = \cS_{\nu}^* W_L$.
\end{theorem}
  \begin{figure}[ht]\begin{center} 
\scalebox{.8}{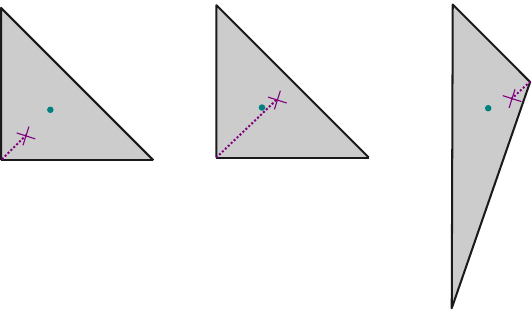}
\end{center} 
\caption{Nodal slide in the direction $\nu$ and transferring the cut respecting the sign convention of Notation \ref{note:sign}.
}
\label{fig:signs}
\end{figure}
\begin{proof}
  The potential in both cases is equal to the potential of the 
  corresponding dual affine manifold $\A$, $\A'$ as defined in
  \eqref{eq:ax}, and by Proposition \ref{prop:simple-wall}, the dual
  affine manifolds $\A$, $\A'$ are connected by a path that crosses a simple wall.
  By Proposition \ref{prop:wallA}, crossing a simple wall mutates the
  potential as claimed in the Theorem.
\end{proof}

\begin{theorem} \label{thm:bithm2} 
Let $X$ be a 
monotone almost toric four-manifold with a polytope $\Delta = \Phi(X) \subset \t^\dual$ whose facets have normal vectors $\mu(Q_1),\ldots, \mu(Q_k)$ with convex hull $N$.  Let $L$ be a (unique if it exists) monotone moment fiber in $X$.  The potential of $L$ has coefficients $1$ at each $\nu_i$ and has binomial
coefficients on the edges joining $\mu(Q_i)$ and $\mu(Q_{i+1})$
for each $i.$   Namely if $\zeta$ is the $k$-th lattice point between $\mu(Q_i)$ and $\mu(Q_{i+1})$ out of a total of $n+1$ then the coefficient of $y^\zeta$ in $W_L$ is $\binom{n+1}{k}$.
\end{theorem}

We first prove a lemma about tropical graphs contributing to the potential that are on the boundary of the moment polytope.    Let $b \in B^{\foc}$ be a focus-focus value in a base diagram of Vianna type, close to a corner of $\Delta = \Phi(X)$ that is the intersection of facets 
$Q_i$ and $Q_{i+1}$.  Let $\cP_b \subset \cP$ be the collection of those polytopes meeting the cone  $C_b$ on vectors $\nu_i,\nu_{i+1}$
normal to adjacent facets $Q_i,Q_{i+1}$ of $\Delta$, 
with vertex of the cone $C_b$  at $\lambda = \Phi(L)$,  that is, the polytopes in the corner ``near $b$''. 

\begin{lemma} \label{lem:pb}  Suppose that $\Gamma$ is a tropical graph contributing to the potential
with initial direction $\zeta$ that is in the convex hull of two vectors $\nu_i,\nu_{i+1}$
normal to adjacent facets $Q_i,Q_{i+1}$ of $\Delta$, meeting at a vertex $b \in \Delta$.  
Then $\Gamma$ is contained in the union of polytopes $P^\dual$ for $P \in \cP_b$.
\end{lemma}

\begin{proof} As before, we consider a polyhedral decomposition 
$\cP$ with $\Phi(L)$ contained in a ``big inner piece'' $P_0$ as in Figure
\ref{fig:2p}.  Consider a coefficient $\kappa_\nu$ of $y^\nu$ in $W_L$ where
$\nu$ is on the boundary of the Newton polygon $\Newt(W_L)$ of $W_L$.
Let $u: C \to \XX$ be a broken disk corresponding to the type $\bGam$ with inner piece $u_0:C_0 \to \XX_{\ol P_0}$.  Necessarily $u_0$ is a Blaschke disk with homology class
\[ [u_0] = \sum m_i [v_i] \in H_2(X_{P_0}, \Q)  \]
where $v_i$ is the disk with degree one in the $i$-th component
from \eqref{fig:blaschke} in the direction $\nu_{R_i}$.
The assumption that $\nu$ is in the convex hull of $\nu_i,\nu_{i+1}$ implies 
\[ m_i + m_{i+1} = 1 .\] 
Hence 
\[ A(u) - A(u_0) = 1 - (1 - \eps)(m_i + m_{i+1}) = \eps. \] 

We claim that there is not enough energy left for the tropical graph to escape the corner. For each face  $Q_i \subset \Phi(X)$, let $P_i \subset \Phi(X)$ be the parallel face of one of the outer trapezoids opposite the face $Q_i$
and let $\delta_i$ be the symplectic volume of the corresponding divisor $X_{P_i}$ in $\ol{X}_{\ol P _0}$.  We may assume that the decomposition $\cP$ is chosen 
with $P_i$ sufficiently close to $Q_i$ so that the symplectic volume $\delta_i$ of $X_{P_i}$
satisfies $\delta_i > \eps$.  Suppose $u$ has components $u_0$,
mapping to $\XX_{\ol P_0}$, and $u_1$ mapping to some $\XX_{\ol P_1}$.  Then 
\[ A(u) \ge A(u_0) + A(u_1) =  (1 - \eps) + \delta_i  > 1 \] 
which is a contradiction.   
\end{proof}

\begin{proof}[Proof of Theorem \ref{thm:bithm2}]  The proof is by reduction to case of potentials for cyclic quotient singularities computed previously.  Let $L$ be a monotone torus fiber. 
By Lemma \ref{lem:pb}, the tropical graphs $\Gam$ that contribute to
the $y^\nu$-coefficient in $W_L$ are contained in
the union $B_b^\dual$ of polytopes $P^\dual$ for $P \in \cP_b$.  The dual complex
$B_b^\dual$ is the dual complex considered in Proposition \ref{prop:bi}
and the claim follows from that Theorem.
\end{proof}

 \subsection{The Newton polygon of the potential} 

We first show that the Newton polygon as the dual of the moment
polytope, that is, the convex hull of the normal vectors to the
facets. 

 \begin{definition} Let $W: (\C^\times)^n \to \C$ be a Laurent polynomial 
 given as a sum of terms
\[ W(y) = \sum_{\nu \in \Z^n} \kappa_\nu y^\nu . \] 
The {\em Newton polytope} 
\[ \Newt(W_L) = \on{hull} \{ \nu | \kappa_\nu \neq 0 \} \subset \R^n \]
of $W_L$ is the convex hull of the exponents $\nu$ for non-zero coefficients $\kappa_\nu \neq 0.$
\end{definition}

\begin{theorem} \label{thm:newt}  Let $X$ be an almost toric
  four-manifold with moment polytope $\Phi(X) \subset \R^2$ and let
  $W_L$ be the disk potential of a monotone torus fiber $L \subset X$.  The Newton polytope $\Newt(W_L) \subset \Z^2$ is equal to the dual $\Phi(X)^\dual$ (as in Definition \ref{def:polydual}).
\end{theorem}

\begin{proof} The inclusion of $\Newt(W_L)$ in the convex hull of the
  vectors $\mu(Q_i)$ was shown in Proposition \ref{prop:collg} \eqref{part:collg2}.  To
  prove the reverse inclusion, we need to show that the primitive normal 
  $\mu_j \in \t_\Z$  to a facet of $\Phi(X)$ appears in the Newton polygon
  of 
  the potential $W_L$.  For
  this purpose, we use the multiple cut $\PP_{\on{in}}$ from Definition
  \ref{def:good} (see Figure \ref{fig:2p}) with the polytope $P_0$
  containing $\Phi(L)$ being large enough that it contains
  $(1-\eps) \Phi(X)$ for a small $\eps>0$ determined later in the
  proof. For a broken map $u$ with initial slope $\mu_j$, the disk
  component $u_0$ has Maslov index two and intersects a relative
  divisor of $X_{P_0}$ corresponding to a long facet $Q_j \subset P_0$.
  The broken map
  has a sphere component $u_1$
that shares a tropical node with the disk $u_0$ and 
that lies
  either in the neck piece $\XX_{Q_j}$ or in the piece $X_{P_j}$
  corresponding to the trapezoid $P_j \in \PP$ adjacent to the facet
  $Q_j$.  The piece $X_{P_j}$ has a fibration structure over $\P^1$
  and any sphere that projects to a non-constant sphere in $\P^1$ has
  area bounded below by some constant $\delta > 0 $ independent of
  $\eps$.  As a result, for $\eps$ sufficiently small, $u_1$ lies 
  in $X_{P_j}$ and must be a fiber. The sphere $u_1$ does not meet the other relative divisors
  of $X_{P_j}$, and thus the broken map $u$ just consists of $u_0$ and $u_1$. 
  It follows that the number
  of such broken maps is equal to $1$.  In particular, each
  normal vector $\nu_j$ to a facet appears in the Newton polygon.
\end{proof}

\subsection{Potentials of smoothings of toric singularities }

In this section we compute the disk potential for a class of non-compact toric manifolds 
which model del Pezzo surfaces locally.   These are the so-called {\em smoothings of cyclic quotient $T$-singularity}, which are the most general cyclic quotient surface singularities of algebraic varieties admitting a $\Q$-Gorenstein smoothing.  
These singularities also possess an almost toric smoothing in the following sense:
By an orbifold, we mean a Deligne-Mumford stack in the sense of \cite{le:stacks}.
The notions of symplectic orbifolds, toric orbifolds etc. are then obtained by 
obvious modifications of the usual definitions. 

\begin{definition} An almost toric manifold $X$ a {\em smoothing} of a
  toric orbifold $X_0$ if the almost toric diagram $\Delta$ for $X$ is
  the same as that $\Delta_0$ of $X_0$, with some collection of
  focus-focus values $b \in B^{\foc}$ with branch cut pointing towards
  the vertex of $\Delta$.
\end{definition}

We describe a toric model for cyclic surface singularities, and an
almost toric smoothing for $T$-singularities.

\begin{notation} 
\begin{enumerate}
\item   A {\em cyclic surface
singularity} is a quotient of $\C^2$ by a finite cyclic group and is
specified by a pair of coprime integers $p$, $q$. 
    \item The singularity
denoted by $\frac 1 p(1,q)$ is a quotient $\C^2/\Upsilon$ by the group
\[ \Upsilon := \{\zeta \in \C : \zeta^p=1\} \] 
with action given by
\[ \zeta \cdot (z_1,z_2)=(\zeta z_1, \zeta^q z_2) . \]
The function
\[ \mu(z_1,z_2):=(\hh |z_2|^2, \frac 1 {2p}(|z_1|^2 + q|z_2|^2)) \] 
descends to $\C/\Upsilon$ and is a moment map for a $(S^1)^2$-action,
and the moment polytope is the wedge
\[\mu(\C/\Upsilon)=\{(x,y) \in (\R_{\geq 0})^2: py \geq qx\}.\]
\item  A {\em cyclic quotient $T$-singularity} is a cyclic quotient singularity of the form $\frac {1} {dp^2}(1, dpq-1)$. Its smoothing is the almost toric manifold $X_0$ with the same moment polytope as $X$, along with $d$-focus-focus singularities lying on the line $\{\lam(p,q): \lam \geq 0\}$. Note that the base diagram of $X$
is the most general one possible.  Therefore, a cyclic surface
singularity has an almost toric smoothing exactly if it is a
$T$-singularity. 
\item  In the particular case when $p=1$, the $T$-singularity is an
$A_{d-1}$-singularity which we denote by $M_d$.  Furthermore, the
$T$-singularity is the $\Z_p$-quotient of an $A_{dp-1}$-singularity,
and the smoothing of example 7.8 in Evans \cite{evans:lec} is fiberwise a
$\Z_p$-quotient of a smoothing of an $A_{dp-1}$-singularity.
\end{enumerate}
\end{notation}

\begin{figure}[ht]\begin{center} 
\scalebox{.8}{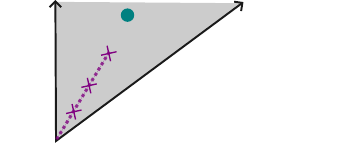}
\end{center} 
\caption{Moment polytope of a $T$-singularity $\frac 1 {dp^2}(1,dpq-1)$}
\label{fig:Tsing}
\end{figure}


The sense in which $X$ is a smoothing of $X_0$ is explained in Evans \cite[Chapter 7]{evans:lec}.   The following example is the important one for our purposes:

\begin{example}
Cyclic quotient $T$-singularities are quotients of $A_n$-singularities by finite groups.    For coprime positive integers $q <  p$ and an integer $d \ge 0$ let $X_0$  be a cyclic quotient of type $\frac{1}{dp^2}(1,dpq - 1)$ in the language of Evans \cite[Chapter 7]{evans:lec}.  The orbifold $X_0$ is equipped with a toric structure whose moment polytope is the cone 
\[ \Phi(X_0) = \on{Cone}((0,1),(dp^2, dpq - 1)) \]
on the vectors $(0,1)$ and $(dp^2,dpq - 1)$.  It has a smoothing $X$ with $d$ focus-focus singularities along a line in the $p,q$ direction.  The dual polytope is a line segment with vertices at $(-1,0)$ and $(dpq - 1, - dp^2)$.  It contains $dp$ lattice points 
\[ (-1,0), (-1 + q,-p), (-1 + 2q, - 2p), \ldots, (-1 + dpq, -dp^2). \]
In particular, the {\em smoothing of the $A_n$ singularity} is the  almost toric manifold with polytope  given by the cone on the rays $(0,1)$ and $(1,n+1)$ and 
$n+1$ focus-focus singularities along the ray $(0,1)$.
\end{example}

Convexity of the cylindrical end, as in Ritter-Smith \cite[Section 3]{rittersmith} implies that monotone Lagrangians in such smoothings have well-defined potentials. For any monotone Lagrangian torus $L$  the moduli space of holomorphic disks bound $L$ with any particular energy bound is compact.  As such $L$ has a well-defined  {\em local potential}
\[ W_L: \Rep(L) \to \C^\times \]
by counting Maslov index two holomorphic disks in $X$ with boundary in $L$ passing   
through a generic point in $L$.

\begin{figure}[ht]\begin{center} 
\scalebox{.8}{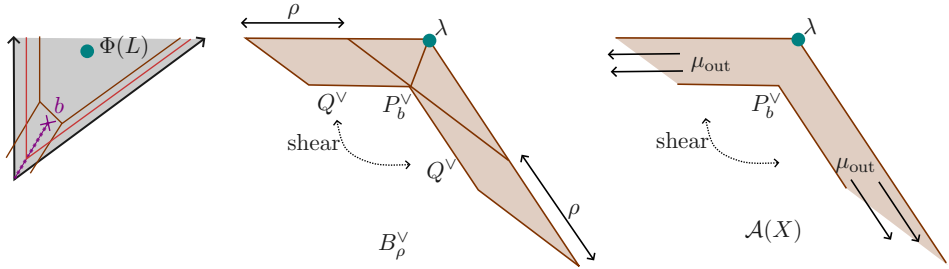}
\end{center} 
\caption{Multiple cut on the $T$-singularity, its dual complex $B^\dual_\rho$ and dual affine $\A(X)$.}
\label{fig:cutTsing}
\end{figure}

\begin{figure}[ht]\begin{center} 
\scalebox{.8}{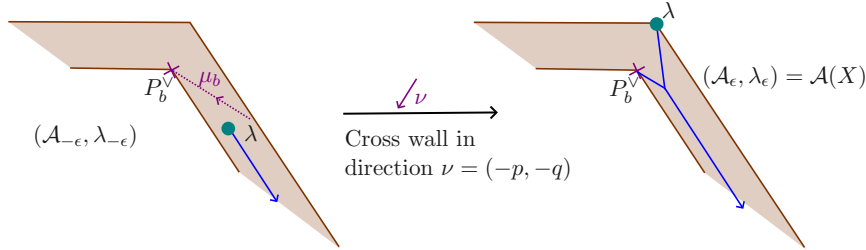}
\end{center} 
\caption{A tropical disk in a $T$-singularity, before and after crossing the wall indicated by $\mu_b$.}
\label{fig:curveT}
\end{figure}

The main result of this subsection is the computation of the potential of a smoothing of a cyclic quotient $T$-singularity:

\begin{proposition}  \label{prop:bi} Let $X$ be the smoothing of a 
cyclic quotient $T$-singularity $X_0$ with normal vectors to facets
$\mu_1 = (-1,0)$
and $\mu_2 = (-1 + dpq, -d p^2)$.  Let $L \subset X$ be a monotone Lagrangian
torus fiber over a point $\lambda \in \Phi(X)$.  Thus 
\[ \lan \lambda ,  \mu
\ran = \lan \ell, \mu_2 \ran  \] 
and $\ell$ is further away from the
vertex of $\Phi(X)$ than any of the focus-focus singularities. 
The potential 
\[ W_L: \Rep(L) \to \C^\times \] 
has coefficients given by binomial coefficients on the line segment
from $\mu_1$ to $\mu_2$; namely
\[ W_L(y_1,y_2) = y_1^{-1} ( 1 + y_1^q y_2^{-p})^{dp}. \] 
\end{proposition}

\begin{proof}
  [Proof of Proposition \ref{prop:bi}]
  The proof is by applying the mutation formula. Under a standard decomposition, the affine dual manifold $\A(X)$ corresponding to a $T$-singularity is as in Figure \ref{fig:cutTsing}.
  This space is the same as $(\A_\eps,\lam_\eps)$ in Figure \ref{fig:curveT}, which is obtained  $(\A_{-\eps},\lam_{-\eps})$ by $d$ simple wall-crossings. Each of the wall-crossings respects the sign convention of Notation \ref{note:sign}, and is  in the direction $\nu=(p,q)$. 
In $(\A_{-\eps},\lam_{-\eps})$, there is only one rigid tropical disk, and it has slope $(dpq-1,-dp^2)$.  
Then,
\begin{multline*}
  W=W_{(\A_\eps,\lam_\eps)} = (\S_\nu^*)^dW_{(\A_{-\eps}, \lam_{-\eps})}=(\S_{(-p,-q)}^*)^d(y_1^{dpq-1}y_2^{-dp^2}) \\
  =(y_1^{dpq-1}y_2^{-dp^2})(1+y_1^{-q} y_2^{p})^{dp}=y_1^{-1}(1+y_1^{q} y_2^{-p})^{dp}.
\end{multline*}

\end{proof}

\begin{figure}[ht]\begin{center} 
\scalebox{.3}{\includegraphics{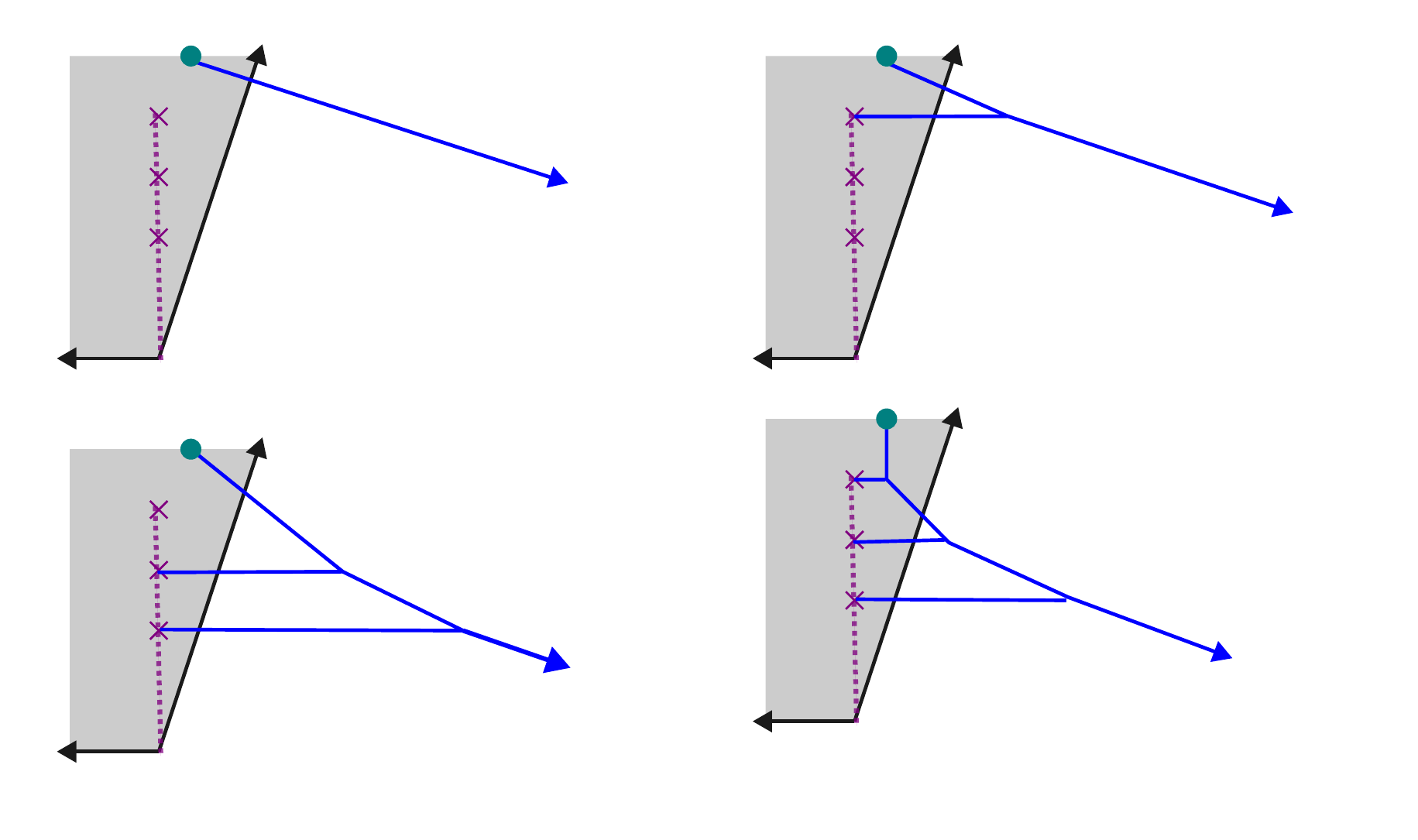}}
\end{center} 
\caption{Four of of the eight tropical disks for the smoothing of an $A_3$ singularity} 
\label{a3fig}
\end{figure}

\begin{example}  We describe the computation for the disk
potential of the spherical pendulum from Example \ref{ex:pend}.  The potential of the toric fiber
depicted in Theorem \ref{pendulum} on the left 
is 
\[ W_L(y_1,y_2) =       (y_1 + 2 + y_1^{-1}) /y_2 .\]
The potential of the moment fiber on the right is 
\[ W_{L'}(y_1',y_2') = (y_1'^{-1} + 1) (y_2')^{-1} .\]
The two potentials $W_L, W_{L'}$ are related by the {\em mutation formula}
\begin{equation} \label{mut1} 
y_1 = y_1' , \ y_2 = y_2' (1 + y_1^{-1})  \end{equation}
from Definition \ref{def:mutform}.
\end{example}

\begin{figure}[ht]\begin{center} 
\scalebox{.2}{\includegraphics{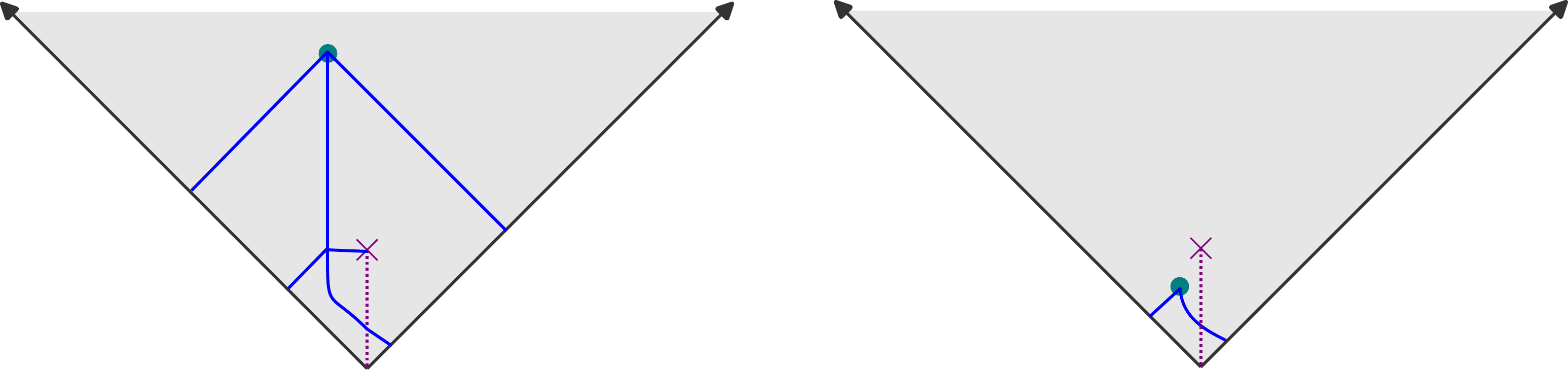}}
\end{center} 
\caption{Computing the disk potential of the spherical pendulum} 
\label{pendulum}
\end{figure}

The reader may deduce Proposition \ref{prop:bi} by taking a decomposition separating the focus-focus values and applying 
Proposition \ref{prop:bunchmult} to each focus-focus value.  However, we formally give the proof at the end of the next section after proving the mutation formula.

\subsection{The classification of potentials}

In this section we reprove the Pascaleff-Tonkonog result \cite{pasc} that the potentials are 
the mutable polynomials as expected, via a somewhat indirect argument involving their mutability.  A similar argument was used in Pascaleff-Tonkonog \cite{pasc},  but the argument here is entirely combinatorial.   We will need some definitions from 
Coates-Kasprzyk-Pitton-Tveiten \cite{coates:max}.

\begin{definition}
A convex polytope $P$ in $\R^2$ is {\em Fano} if  if it is of maximum dimension
$\dim(P) = 2 $, its dual $P^\dual$ contains the origin, and  the vertices $v \in \Ver(P)$ of  $P$ are primitive lattice vectors in $\Z^2$.
\end{definition}

We give a slightly more general definition of mutation that the one in \ref{def:mutform}, to follow the terminology in  \cite{coates:max}:

\begin{definition}  Given a 
vector $w \in \Z^2$ and a Laurent polynomial $a(y_1,y_2)$, the {\em mutation}  with respect
to $(w,a)$ is the map 
\[ \C(y_1,y_2) \to \C(y_1,y_2), \quad 
y^\mu \mapsto y^\mu a^{w(\mu)}.  \]
\end{definition}

\begin{definition} Given a Laurent polynomial $W$ the {\em mutation graph} $G(W)$ is the graph whose vertices $\Ver(G(W))$ are repeated mutations $W'$ of $W$ and edges $\Edge(G(W))$
connecting $W,W'$ are mutations of $W$ to $W'$.
\end{definition}

\begin{example}  Let $W(y_1,y_2) = y_1 + y_2 + (y_1 y_2)^{-1}$ be the standard potential for the projective plane.  The associated graph is then the graph of Markov triples, as explained in \cite[Section 3.2]{akhtar}.  Indeed, in this case the possible mutations are those corresponding to edges, and correspond to the mutations of almost toric structure described in Definition \ref{def:mutate}, see the proof of Proposition 3.9 in 
 \cite{coates:max}.
\end{example}

\begin{definition} 
A Laurent polynomial $W$  is 
 is {\em maximally mutable} if its Newton polygon $P = \Newt(W)$ is a Fano polytope, the constant term of $W$ is zero, and the mutation graph $G(W)$ of $W$
is maximal among all Laurent polynomials with Newton polytope $P$. 

Given a Fano polytope $P \subset \R^2$ and an edge $e$, denote by $v_e \in \Z^2$ the direction of the edge $e$ and $n_e \in \Z$ the {\em singularity content} of the cone over $e$ in the sense of 
Coates-Kasprzyk-Pitton-Tveiten \cite{coates:max}. This ends the Definition.
\end{definition}

\begin{remark}
In the cases considered here, the singularity content $n_e$ of an edge $e$ will simply be its lattice length. 
\end{remark}

We recall from  Coates-Kasprzyk-Pitton-Tveiten \cite{coates:max}:
 
 \begin{theorem}\label{thm:coates} (Proposition 3.9,  Theorem 3.12 of \cite{coates:max})  Let  $W$ be a Laurent polynomial in two variables with Fano Newton polygon $\Newt(W)$. Then the following are equivalent:   
Denote by $w_e \in \R^2$ the primitive inward pointing normal vector to $e$.   
 \begin{enumerate}
 \item $W$ is maximally mutable; 
 \item $W$ is mutable with respect to $(w_e, (1+ y^{v_e})^{n_e})$ for each edge $e$ of $\Newt(W)$;
\item   $W$ coincides with one of the polynomials in Table \ref{table:pptable}, up to mutation and $SL(2,\Z)$-equivalence.
\end{enumerate}
\end{theorem} 

To apply the Theorem, we first prove, similar to Pascaleff-Tonkonog \cite{pt:wall}:

\begin{proposition} \label{prop:maxmut} Let $X$ be a compact monotone almost-toric 
four-manifold with monotone Lagrangian torus fiber $L$.  The disk potential $W_L$ is maximally mutable. 
\end{proposition} 

\begin{proof}   Theorem \ref{thm:mutthm}
shows that the potential $W_{L'}$ of the  monotone moment fiber
in the mutated diagram is a mutation of $W_L.$  Since $L'$ is a disk potential for a monotone Lagrangian, $W_{L'}$ is a Laurent polynomial, and since it is an almost toric moment fiber, it has vanishing constant term by Proposition \ref{prop:noconstant}.  It follow that $W_L$ is mutable with respect to all three of the directions corresponding to the edges
of $\Newt(W_L)$.  By Theorem \ref{thm:coates},  the mutation graph of $W_L$ is maximal. 
\end{proof}

\begin{corollary}  \label{cor:ptable} Let $L$ be the smooth monotone fiber of an almost toric structure on a monotone del  Pezzo surface $X$. Then the disk potential $W_L$ is the mutation of the potentials in Table \ref{table:pptable} corresponding to the del Pezzo surface $X$.
\end{corollary}

\begin{proof}   By Theorem \ref{thm:coates} ( Proposition 3.9 of Coates-Kasprzyk-Pitton-Tveiten \cite{coates:max}) 
and Proposition \ref{prop:maxmut}, the disk potential $W_L$ must be one of the potentials in Table \ref{table:pptable}, up to mutation.   
It remains to check that in fact the potential $W_L$ is the potential for the del Pezzo surface $X$ containing $L$.      Thus the potential $W_L$ must be the potential $W$ for a del Pezzo surface $X'$ with an almost toric diagram $\Delta'$ with the same
volume as $\Delta$.  By Lemma \ref{lem:degree}, $X'$ has the same degree as $X$.    The only two del Pezzos with the same degree are $(\P^1)^2$ and $\Bl^1 \P^2$, which both have  $c_1^2$ equal to $8$.   We distinguish these  two cases by a cohomology argument.   Suppose that $W_L$ is equivalent by 
mutation to $W$ which is the standard potential for either 
$(\P^1)^2$ and $\Bl^1 \P^2$.   By mutating the original almost toric diagram, we obtain an almost toric diagram $\Delta'$ with Lagrangian fiber  $L'$ is equal to $W_{L'} = W.$  We claim that in fact the diagram $\Delta'$ has at most one focus-focus value at any vertex.  Indeed, if there is more than one focus-focus value at any corner then the normal vectors $v_i,v_{i+1}$ at a vertex would be related by the composition of at least two shears, and so $|\det(v_i v_{i+1})| \ge 2$ which is not the case.   After a nodal trade,  we obtain a base diagram without focus-focus values.  

After absorbing the focus-focus values into the vertices by nodal trades, one obtains a toric structure on $X$ with polytope dual to $\Newt(W)$. The two possible toric varieties $(\P^1)^2$ and $\Bl^1 \P^2$ are not diffeomorphic, since $(\P^1)^2$ has no homology classes of square $-1$.  Therefore,  $W_L$ must be mutation-equivalent to the toric potential $W_{L'}$ on $X$.   \end{proof}

\begin{remark} An inspection of the tables in Vianna shows that $\Newt(W_L)$ is the Newton polygon of the  potential in Corollary \ref{cor:ptable} for a suitable diagram. 
For most of these, the diagram giving the potential in \ref{cor:ptable} is easily read off.
For example, for the del Pezzo of degree $1$, the diagram is labelled $C2$ in 
Vianna's Figure 16 in \cite{vianna:dp}. For the del Pezzo of degree $2$, the diagram is labelled $B2$ in Figure 15 in \cite{vianna:dp}.
\end{remark} 

\section{Explicit computations for del Pezzo surfaces} 

We now discuss each of these potentials in turn, with the goal of illustrating the tropical formulas by computing various coefficients of the potentials.  For blow-ups at less than four points, the corresponding del Pezzo surfaces admit toric structures.    For convenience, we record the critical values of the potentials for del Pezzos in Table \ref{table:eigentable}.

\begin{table}
\[ 
\begin{array}{|l|l|l|}X   &  \text{Manin Root System} &  \approx \Critval(W) \\
\hline 
\P^2                 & & 3\alpha, \alpha^3 = 1 \\
\P^1 \times \P^1     & A_1 & 4, 0^{\oplus 2}, -4 \\ 
\Bl^1 \P^2           & & -0.33, 3.8, -2.23 \pm 1.94 I \\
\Bl^2 \P^2            & A_1 & (-1)^{\oplus 2}, 4.73, -2.86 \pm 0.94 I \\
\Bl^3 \P^2            & A_1 \oplus A_2  &  (-2)^{\oplus 3}, (-3)^{\oplus 2} ,6 \\
\Bl^4 \P^2             & A_4 & (-3)^{\oplus 5}, 8.09, -3.09 \\
\Bl^5 \P^2            & D_5 & (-4)^{\oplus 7}, 12 \\
\Bl^6 \P^2             & E_6 & (-6)^{\oplus 8}, 21 \\
\Bl^7 \P^2            & E_7 &  (-12)^{\oplus 9}, 52\\
\Bl^8 \P^2 & E_8 &  (-60)^{\oplus 10} , 372. 
\end{array} \]
\caption{Critical values of the potential for del Pezzos}
\label{table:eigentable}
\end{table}

\subsection{Potentials of toric del Pezzo surfaces}

 The disk potential of monotone toric manifolds was computed
in Cho-Oh \cite{chooh:fano}.  Suppose that a  monotone toric variety $X$ has convex Delzant polytope $P$.  Denote the primitive normal 
vectors  to the facets
\[ \nu_1,\ldots, \nu_k \in \Z^n.  \] 
The potential for monotone torus $L$ is the function
\[ W_L: (\C^\times)^n \to \C, \quad y \mapsto \sum_{i=1}^k y^{\nu_i}  \]
called the {\em Givental-Hori-Vafa potential}, see Givental 
\cite{giv:icm}.    (Even if $L$ is not monotone, the potential for $L$ in the sense
of \eqref{eq:mcl} is equivalent to the Givental-Hori-Vafa potential by the results of Fukaya-Oh-Ohta-Ono \cite{fooo:toric}.)

\subsubsection{The projective plane}
Let $X = \P^2$. The canonical toric structure has moment polytope $\Phi(X)$ a triangle as shown in Figure \ref{fig:p2fig}, where we have drawn the cartoon diagrams for the broken disks.  
By a {\em cartoon diagram}, we mean a collection of subsets $C_v \subset P(v)$ that are the moment images of some collection $\ol{u}'_v$ of smooth (not necessarily pseudoholomorphic) maps isotopic to $\ol{u}_v$.    The tropical graphs $\Gamma$ contributing to the potential $W_L$ each have a single edge and no vertex.  The critical values are of the form $3\alpha$ for $\alpha^3 = 1$.     Another example is computed in Example \ref{ex:chek}.

\begin{figure}[ht]\begin{center} \scalebox{.8}{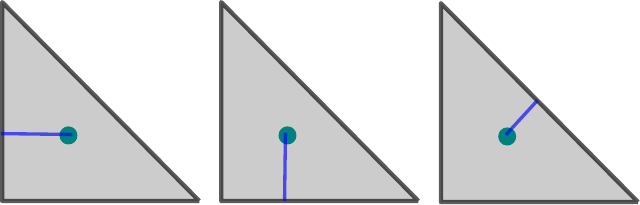}
\end{center} \caption{Cartoon diagrams of Maslov-index-two disks in $\P^2$}  \label{fig:p2fig}
\end{figure}

\subsubsection{The quadric surface}
Let $X=  \P^1 \times \P^1$.  The canonical
toric structure has polytope given by a product of intervals
\[ \Delta = \Phi(X) = [-1,1]^2 \] 
shown in Figure \ref{p1p1fig}.  The potential is the monotone torus fiber is 
\[ W_L(y_1,y_2) = y_1 + 1/y_1 + y_2 + 1/y_2 .\]
The critical points and values are easily found by hand to be
\[ \Crit(W_L) = \{ y_1 = \pm 1 ,y_2 = \pm 1 \}, \quad \Critval(W_L) = \{  \pm 4, 0  \}  .\]

\begin{figure}[ht]\begin{center} \scalebox{.5}{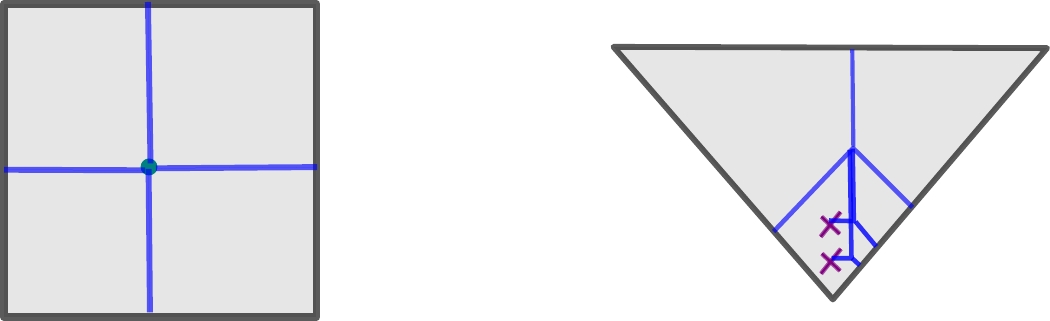}
\end{center} \caption{Two almost toric diagrams for $\P^1 \times \P^1$ and cartoon diagrams of Maslov-index-two disks}  \label{p1p1fig}
\end{figure}

A second almost toric diagram  of the quadric surface is shown in Figure \ref{p1p1fig}.
This diagram and  is related to the realization of $X$ as a compactification 
of the cotangent bundle of the two-sphere.   The moment polytope 
\[ \Delta' = \Phi' (X) = \on{hull} \{ (0,0), (1,1), (-1,1) \}  \] %
is a triangle, with 
the vertex at the bottom representing an $A_1$-singularity.  The potential for this almost toric structure is 
is 
\[ W_{L'}(y_1,y_2) = y_1 y_2 + y_2/y_1 + 1/y_2 + 2y_2 .  \]
The  critical values are easily found by hand to be
\[ \Critval(W_{L'}) = \{4, -4 \} . \]
Thus the monotone torus fiber $L'$ for the second moment polytope
is not Hamiltonian isotopic to the monotone fiber $L$ for the first polytope.
In fact this can be easily seen by readers familiar with the technology of open-closed maps:   The Lagrangian $L'$ is disjoint from the 
Lagrangian two-sphere $L''$ given by the zero-section, whose moment image connects the two 
focus-focus critical values near the bottom vertex of the diagram.  
The image of $HF(L') \cong H(S^2) \cong \C^2$ under the open closed map is two-dimensional, since the classical, leading order term in the open closed map is non-zero.  Hence $L''$ split generates the $0$-eigencategory, while $L'$ with two local systems generates the $\pm 4$-eigencategories.

\subsubsection{The del Pezzo of degree eight}
Let  $X = \Bl^1 \P^2$.  The toric structure has a polytope given by a trapezoid
with normal vectors $(1,0), (0,1), (0,-1), (1,-1).$
The potential of the monotone torus is 
\[ W_L(y_1,y_2) = y_1 + y_2 + 1/y_2 + y_2/y_1 . \]
We used Mathematica to compute the critical points and critical values in Table \ref{table:eigentable}. 

\subsubsection{The del Pezzo of degree seven}
Let $X = \Bl^2 \P^2$ be the twice blow-up of the projective plane.  The toric structure has polytope $\Phi(X)$ given by a square with a  corner cut off.   In Figure \ref{fig:b2p2_tori}, the cartoon diagrams $\Gamma$ are drawn below the approximate images of the curves in the moment polytope.    The tropical graphs associated to these curves can be reconstructed by examining the tangent vectors to the paths at the vertices. That is, to draw the tropical graph from the cartoon picture, one would simply straighten out the edges.  In order to give a complete description, one would also have to describe the dual complex of some polyhedral decomposition, which we find tedious as the answer is independent of which polyhedral decomposition one chooses.   The potential for the monotone torus may be computed as in Cho-Oh \cite{chooh:fano} as
\[ W_L(y_1,y_2) = y_1 + y_2 + 1/y_1 + 1/y_2 + 1/(y_1 y_2)  .\]
There is also an almost toric structure on $X$
 with two focus-focus singularities whose polytope  $\Delta$ is a quadrilateral,  is shown in Figure \ref{fig:b2p2_tori}.  The potential  $W_{L'}(y)$ has six terms.  Since the potentials $W_L, W_{L'}$ have five terms and six terms respectively, the tori $L,L'$ are not Hamiltonian isotopic. 

\begin{figure}[ht]\begin{center} \scalebox{.5}{\includegraphics{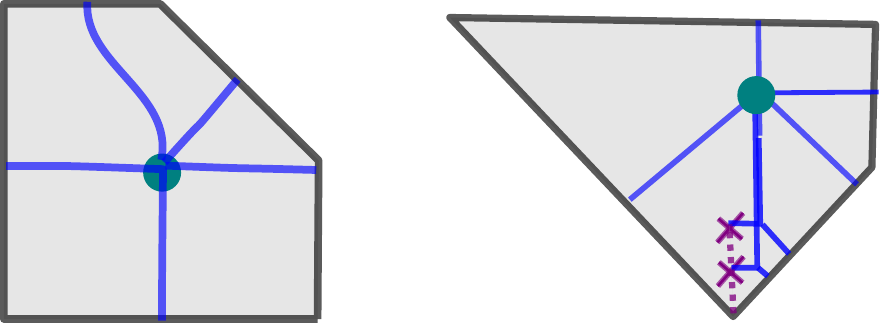}}
\end{center} \caption{Cartoon diagrams for Maslov-index-two disks in $\Bl^2 \P^2$}  \label{fig:b2p2_tori}
\end{figure}

%
%

\subsubsection{The del Pezzo of degree six}

The thrice-blow up $ X = \Bl^3 \P^2$ is toric with moment polytope $\Phi(X)$ a hexagon
as shown in Figure \ref{fig:b3p2_tori}.    The potential for the monotone torus fiber for this toric structure is 
\[ W_L(y_1,y_2) = y_1 + y_2 + y_1 y_2 + 1/y_1 + 1/y_2 + 1/(y_1y_2) .\]
We used Mathematica to compute the critical points and critical values:
\[ \Critval(W_L) = \{ - 2, -3, 6 \} .\]

\begin{figure}[ht] 
\begin{center} \scalebox{.6}{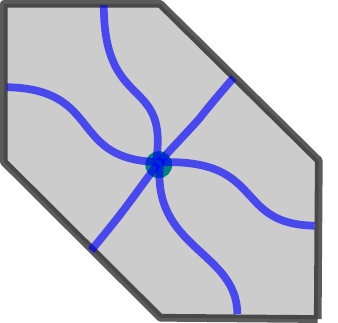}
\end{center} \caption{Cartoon diagrams for Maslov-index-two disks in $\Bl^3 \P^2$} 
\label{fig:b3p2_tori} 
\end{figure}

\subsection{Potentials of non-toric del Pezzo surfaces}

\subsubsection{The del Pezzo of degree five}
The four-times blow-up of the projective plane does not admit a toric 
monotone symplectic form.  An almost toric diagram for the monotone symplectic form with two focus-focus singularities is shown in Figure \ref{fig:b4p2_tori}.  The potential of the monotone Lagrangian torus is 
\[ W_L(y_1,y_2) = 2 y_1 + 2 y_2 + 1/y_1 + 1/y_2 + y_1 y_2 + y_1/y_2 + y_2/y_1 . \] 
Some of the cartoon diagrams contributing to the potential are shown in Figure \ref{fig:b4p2_tori}.  We used Mathematica to compute the critical points and critical values in Table \ref{table:eigentable}.

\begin{figure}[ht] \begin{center} \scalebox{.6}{\includegraphics{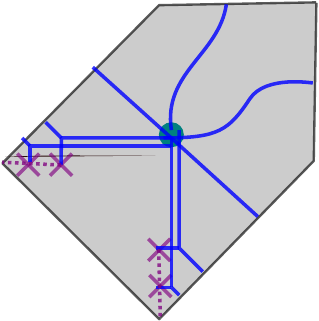}}
\end{center} 
\caption{Cartoon diagrams of disks contributing to the potentials for $\Bl^4 \P^2$} 
\label{fig:b4p2_tori} 
\end{figure}

\subsubsection{The del Pezzo of degree four}
The five-times blow up of the projective plane has an almost toric diagram shown in Figure \ref{fig:b5p2_at}, with corresponding potential
\[ W_L(y_1,y_2) = 2 y_1 + 2 y_2 + 2/y_1 + 2/y_2 + y_1 y_2 + y_2/y_1 + y_1/y_2  + 1/(y_1y_2) .\]
The cartoon diagrams for the disks are shown in Figure \ref{fig:b5p2_at}.
We used Mathematica to compute the critical points and critical values in Table \ref{table:eigentable}.  

\subsubsection{The cubic surface}

The six-times blow-up has an almost toric diagram shown in Figure \ref{fig:b6p2_tori}.  The potential is given by a count of disks in Figure \ref{fig:b6p2_tori} and is given by 
\[ W_L(y_1,y_2) =
  3 y_1 + 3 y_2 + y_1^2/y_2 + y_2^2/y_1 + 3 y_1/y_2 + 3 y_2/y_1 + 3/y_1 + 3/y_2 + 1/(y_1 y_2) .\] 
\begin{figure}[ht] \begin{center} \scalebox{.4}{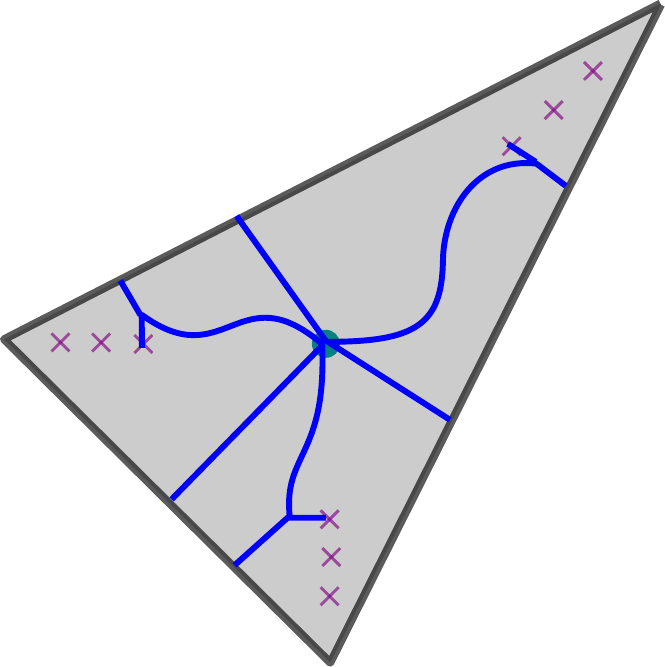}
\end{center} 
\caption{Cartoon diagrams for disks contributing to the potentials for $\Bl^6 \P^2$} 
\label{fig:b6p2_tori} 
\end{figure}

\subsubsection{The del Pezzo of degree two}
The seven-times blow-up has an almost toric base diagram
shown in Figure \ref{fig:b7p2_torus}.
Some of the cartoon diagrams for the disks contributing to the potential
are shown in Figure \ref{fig:b7p2_torus}.  The potential 
of the monotone torus fiber is 
\[ W_L(y_1,y_2) = 
(y_2 + 1/y_2)(1/y_1^2+4/y_1 + 6 + 4 y_1 + y_1^2) + (2/y_1^2 + 8/y_1 + 
    8 y_1 + 2 y_1^2) .\]
We used Mathematica to compute the critical points and critical values in Table \ref{table:eigentable}.

\begin{figure}[ht] \begin{center} \scalebox{.5}{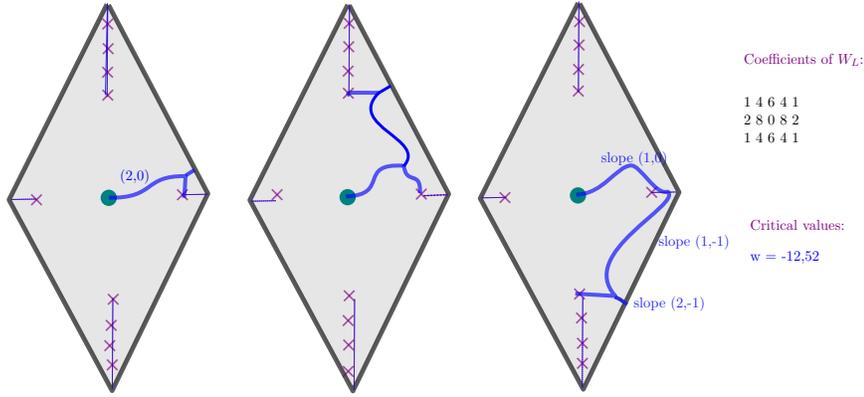}
\caption{Cartoon diagrams for disks contributing to the potential for the torus for $\Bl^7 \P^2$}
\label{fig:b7p2_torus}
\end{center} \end{figure}

\subsubsection{The del Pezzo of degree one}
The eight-times blow-up $\Bl^8 \P^2$ has an almost toric diagram
described by Vianna \cite{vianna:dp} whose polytope is the triangle
\[ \Delta = \on{hull} \{ (-1,3), (1,0), (0,-3) \}  \]
shown in Figure \ref{fig:b8p2_torus}.
Some of the disks contributing to the potential, which has $372$ terms, are shown in Figure \ref{fig:b8p2_torus}.  The potential $W_L$ in Figure \ref{fig:b8p2_torus} is that in Akhtar et al \cite[Figure 1]{akhtar}.
Note that the coefficients along each edge  of the Newton polygon are binomial, as 
justified in Theorem \ref{thm:bithm2}.  We used Mathematica to compute the critical points and critical values shown in Table \ref{table:eigentable}.

%

\begin{figure}[ht] \begin{center} 
\scalebox{1}{{\fontsize{1}{1}\selectfont  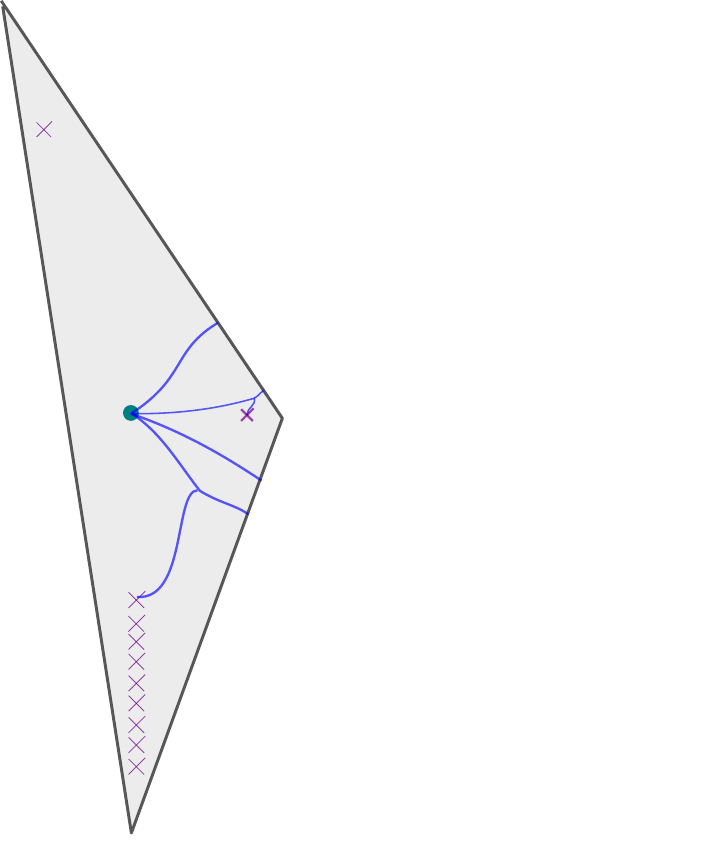}}
\caption{Cartoon diagrams for Maslov-index-two disks in $\Bl^8 \P^2$}
\label{fig:b8p2_torus}
\end{center} \end{figure}

\appendix

\section{Rigid spheres}
\label{sec:spheres}

This appendix contains a digression on counting rigid holomorphic stable maps of genus zero in almost toric four-manifolds; these are closely related
to the so-called {\em exceptional spheres} in del Pezzo surfaces considered
in the algebraic geometry literature, for example Testa \cite{testa:thesis}.

\begin{definition} A {\em rigid sphere} in an almost complex four-manifold $X$ is a holomorphic map $u:\P^1 \to X$ of Chern number $c_1(u) = 1$.
\end{definition}

\begin{lemma} Suppose $X$ is a compact symplectic manifold of
  dimension four.   Any rigid sphere  $u$ lies in a component of the moduli space $\M(X)$ of stable maps to $X$ of expected dimension zero.
\end{lemma}

\begin{proof}  The dimension of the moduli space is the index of the linearized operator.  By Riemann-Roch (see \cite[Appendix C]{ms:jh}) the dimension of $T_u \M_{\bGam}(X)$ at any rigid sphere $u$  is $\dim(X) + c_1(u) - 6
  = 0$, as claimed.
\end{proof}

Rigid spheres are closely related to the {\em exceptional spheres} in the theory of del Pezzo surfaces, which are defined to be embedded smooth genus zero curves of self-intersection minus one. 
These two definitions are not quite the same in the case of del Pezzo surfaces of degree one; see the discussion in Remark \ref{rem:exc}.   The same techniques as in the case of disks shows the following:

\begin{theorem} \label{thm:potthm4} The signed count $E(X) \in \Z$ of rigid spheres  is equal to a sum of multiplicities  $m(\Gamma) \in \Q$ over tropical graphs $\Gamma$ in $\A(X)$ with all incoming edges colinear with the lines $\Theta_b^\pm$ emanating out of focus-focus values $b \in B^{\foc}$, and outgoing edge normal to a facet of the moment polytope $\Phi(X)$.  After desingularization, each contribution $m(\Gamma)$ is given by the same product of multiplicities $m(v)$ over vertices
$v \in \Ver(\Gamma)$ as above in Definition \ref{def:mvs}. 
\end{theorem}

\begin{proof}  The proof is identical to that of Theorem \ref{thm:potthm} except for the fact that all incoming edges of the tropical graphs of the broken maps arising from the degeneration emanate from 
the focus-focus values $b \in B^{\foc}$.
\end{proof}

By enumeration of graphs we obtain the following:

\begin{theorem} (see Testa  \cite{testa:thesis})
\label{thm:ex}
The number $E(X)$ of rigid spheres $u: \P^1 \to X$ is given in the following table.
\[ 
\begin{array}{l|llllllll}
X & 
\Bl^1 \P^2 &
\Bl^2 \P^2 & 
\Bl^3 \P^2 & 
\Bl^4 \P^2 & 
\Bl^5 \P^2 & 
\Bl^6 \P^2 & 
\Bl^7 \P^2 & 
\Bl^8 \P^2 \\ \hline
E(X) & 1 & 3 & 6 & 10 & 16 & 27 & 56 & 252 .
\end{array} \]
\end{theorem}

\begin{remark} \label{rem:exc}  In all but the case of the degree one del
Pezzo, every stable map of non-zero lowest area has embedded image as explained in Testa \cite{testa:thesis}.  However, 
for the degree one case there are $240$ of $-1$ curves and  $252$ stable maps of Chern number one.  See Figure \ref{fig:b6p2_exc} for the cartoon diagram and tropical graph for one of the 27 lines on the cubic surface.
\end{remark}

\begin{figure}[ht]
  \begin{center}
    \scalebox{.4}{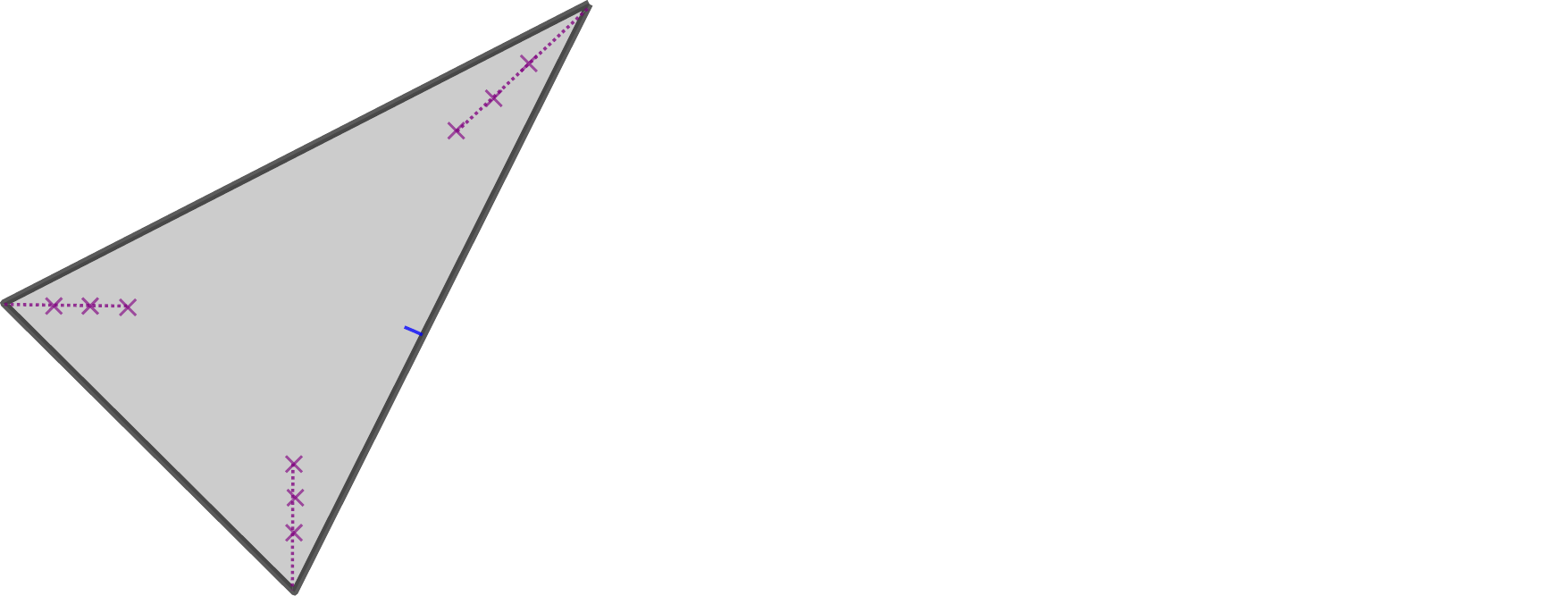}
  \end{center}
  \caption{Cartoon diagrams (left) and tropical graphs (right) for one of the 27 lines on the cubic surface}
  \label{fig:b6p2_exc}
\end{figure}

\section{Code for computing potentials and Jacobian rings}

In the course of writing the paper, we found it useful to have computer-assisted calculations of sets of critical values of potentials and subrings of the Jacobian ring. We include the code for these computations below.   We begin with the computation of the critical values for $\P^2$ using Mathematica:  In the code below, the variables have the following meaning:

\begin{center}
\begin{tabular}{l|l}
Critpoint & the set of critical points \\
Critval & the set of critical values \\
Crithess &  the set of determinants of the Hessian at each of the critical values \end{tabular} \end{center}  

\noindent For $X = \P^2$, all of the determinants in Crithess are non-zero which indicates that the Landau-Ginzburg potential $w$ is Morse in this case. 

\begin{verbatim}
w[y1_, y2_] = y1 + y2 + 1/(y1*y2);
g[y1_, y2_] = { y1* D[w[y1, y2], y1], y2* D[w[y1, y2], y2]};
h[y1_,y2_] = D[w[y1,y2],{{y1,y2}},{{y1,y2}}];
Critpoint = NSolve [ g[y1, y2] == 0, {y1, y2}] 
Critval  = w[y1, y2] /. Critpoint
Crithess = Det[h[y1,y2]] /. Critpoint
{{y1 -> -0.5 + 0.866025 I, 
  y2 -> -0.5 + 0.866025 I}, {y1 -> -0.5 - 0.866025 I, 
  y2 -> -0.5 - 0.866025 I}, {y1 -> 1., y2 -> 1.}}
{-1.5 + 2.59808 I, -1.5 - 2.59808 I, 3.}
{-1.5 + 2.59808 I, -1.5 - 2.59808 I, 3.}
\end{verbatim}

The ring generated by the potential $W$ in $\Jac(W)$ for $X = \P^2$ can be computed by hand, or (as practice for later more complicated example) verified using
the following code from Macaulay2:

\begin{verbatim}
i1 : r1 = QQ[y1,y2,Y1,Y2]/(y1*Y1 -1, y2*Y2-1, y1 - Y1*Y2, y2 - Y1*Y2)
i2 : r2 = QQ[z]
i4 : gen = map(r1,r2,{y1+y2+Y1*Y2})
i5 : i1= kernel gen
i6 : factor i1_0
o6 =  (z-3)(z^2+3z+9)
\end{verbatim}

Mathematica can be used to compute the critical points $\P^1 \times \P^1$ with its standard almost toric diagram.  For 
\begin{verbatim}
w[y1_, y2_] = y1 + 1/y1 + y2 + 1/y2
\end{verbatim}
we obtain critical values 

\begin{verbatim}
{4., 0., 0., -4.}
\end{verbatim}

For $X = \P^1 \times \P^1$ the potential for the 
almost toric diagram whose moment polytope is a triangle has potential
\begin{verbatim}
   w[y1_, y2_] = y1*y2 + y2/y1 + 1/y2 + 2*y2;
\end{verbatim}
The critical values are
\begin{verbatim}
{4, -4.}
\end{verbatim}
This gives an example of an almost toric diagram where the Fukaya category cannot be generated by monotone moment fibers, since the critical value $0$ does not occur.  

For $X = \Bl^1 \P^2$ we have using Mathematica for the toric potential:

\begin{verbatim}
w[y1_, y2_] = y1 + y2 + 1/y2 + y2/y1;
Critval = {-0.33, 3.8, -2.23 + 1.94 I, -2.23 - 1.94 I}
\end{verbatim}


For $X = \Bl^2 \P^2$ with the standard toric structure we have via Mathematica:

\begin{verbatim}
w[y1_, y2_] = y1 + y2 + 1/y1 + 1/y2 + 1/(y1*y2) ;
Critval = {-1., -1., 4.73, -2.86 + 0.94 I, -2.86 - 0.94 I}
\end{verbatim}

The ring generated by $W$  in $\Jac(W)$ for $X = \Bl^2 \P^2$ was calculated using 
Macaulay2.

\begin{verbatim}
r1 = QQ[y1,y2,Y1,Y2]/(y1*Y1 -1, y1*Y2-1, y1 - Y1 + y1*y2, y2 - Y2 + y1*y2);  
r2 = QQ[z];
gen = map(r1,r2,{y1+Y1+y1+Y2+y1*y2});
i1= kernel gen;
factor i1_0
o16 =  (z+1)(z^3 +z^2 -18z-43) 
\end{verbatim}
.  

The critical points for the toric structure on $X = \Bl^3 \P^2$ were computed via Mathematica:

\begin{verbatim}
w[y1_, y2_] = y1 + y2 + 1/y1 + 1/y2 + y1*y2 + 1/(y1*y2);
Critval = {-3., -3., -2., -2., 6., -2.}
\end{verbatim}

The ring generated by $W$ in $\Jac(W)$ for $X = \Bl^3 \P^2$ was computed using
Macaulay2.

\begin{verbatim}
 r1 = QQ[y1,y2,Y1,Y2]/(y1*Y1 -1, y*Y2-1); 
w = y1 + y2 + Y1 + Y2 + y1*y2 + Y1*Y2;
d1 = y1*diff(y1,w) - Y1*diff(Y1,w); d2= y2*diff(y2,w) - Y2*diff(Y2,w); 
r2=r1/(d1,d2); r3=QQ[z]; gen = map(r2,r3,{w}); 
i1= kernel gen; factor i1_0
o9 =  (z-6)(z+2)(z+3) 

\end{verbatim}

Mathematica gives the critical values for $X = \Bl^4  \P^2$:
\begin{verbatim}
w[y1_, y2_] = 2 y1 + 2 y2 + 1/y1 + 1/y2 + y1*y2 + y1/y2 + y2/y1 ;
Critval = {-3., -3., -3., -3., 8.09, -3.09, -3.}
\end{verbatim}
\noindent Note that the potential is not Morse, as the Hessian is degenerate for some of the 
$-3$ critical points. 

The critical values for $X = \Bl^5 \P^2$ are 

\begin{verbatim}
w[y1_, y2_] = 2 y1 + 2 y2 + 2/y1 + 2/y2 + y1*y2 + y1/y2 + y2/y1  + 1/(y1*y2);
Critval = {-4., -4., -4., -4., -4., -4., 12., -4.}
\end{verbatim}

\noindent In this case, all the $-4$ critical points are degenerate; mathematica gives a somewhat misleading answer
in this case since the critical set is not isolated.   The ring generated by $w$ in the Jacobian ring was computed using Macaulay.

\begin{verbatim}
r1 = QQ[y1,y2,Y1,Y2]/(y1*Y1 -1, y2*Y2-1); 
w = 2*y1 + 2*y2 + 2*Y1 + 2*Y2 + y1*Y2 + y2*Y1 + y1*y2 + Y1*Y2; 
d1 = y1*diff(y1,w) - Y1*diff(Y1,w); 
d2= y2*diff(y2,w) - Y2*diff(Y2,w); r2=r1/(d1,d2); 
 r3=QQ[z];  gen = map(r2,r3,{w}); i1= kernel gen; factor i1_0

o10 =  (z-12)(z+4) 

\end{verbatim}

\noindent This sub-ring is dimension two, as opposed to the three-dimension ring one expects.  Thus the Jacobian ring of the disk potential, and so not equal to the 
quantum cohomology; see Barrott \cite{barrott:explicit} for 
the potential that does have the correct Jacobian ring.

The code for computing the potential and its critical points in Mathematica for $X = \Bl^6 \P^2$ is:

\begin{verbatim}
w[y1_, y2_] =   3 y1 + 3 y2 + y1^2/y2 + y2^2/y1 + 3 y1/y2 + 
3 y2/y1 + 3/y1 + 3/y2 + 1/(y1*y2);
g[y1_, y2_] = {y1*D[w[y1, y2], y1], y2*D[w[y1, y2], y2]};
h[y1_,y2_] = D[w[y1,y2],{{y1,y2}},{{y1,y2}}];
Critpoint = NSolve [ g[y1, y2] == 0, {y1, y2}];
Critpointapprox = {Round[y1, .01], Round[y2, .01]} /. Critpoint
Critval = Round[w[y1, y2], .01] /. Critpoint
Crithess = Round[Det[h[y1, y2]], .01] /. Critpoint
{{2.12 + 0.22 I, -3.12 - 0.22 I}, {2.09 + 0.23 I, -3.09 - 
   0.23 I}, {-0.49 - 0.14 I, -0.51 + 0.14 I}, {-0.49 - 
   0.14 I, -0.51 + 0.14 I}, {-0.6 + 0.62 I, -0.4 - 0.62 I}, {-0.6 + 
   0.62 I, -0.4 - 0.62 I}, {-0.52 - 0.69 I, -0.48 + 0.69 I}, {-0.52 - 
   0.69 I, -0.48 + 0.69 I}, {1., 1.}}
{-6., -6., -6., -6., -6., -6., -6., -6., 21.}
{0., 0., 0., 0., 0., 0., 0., 0., 243.}
\end{verbatim}

There are only four critical points with critical values $w = -6$, and one
critical point with value $w = 21$.  The first Chern class obeys the relation $(c_1 + 6)^2 (c_1 - 21) =
0$, as a special case of the computations Givental \cite{gi:eq}, as
explained in Sheridan's appendix in \cite{sheridan:hypersurface}. 
In particular $c_1` - 21 $ is a generalized eigenvector of $c_1 + 6 $ rather than an actual eigenvector.  On other hand, one might compare the potential above to  the potential used in Sheridan \cite{sheridan:hypersurface}, which is of the correct dimension:

\begin{verbatim}
r1 = QQ[a,b,c,d]; w = a^3+b^3+c^3+d^3-a*b*c*d - 6 ;  
r2 = r1/(3*a^2 - b*c*d, 3*b^2 - a*c*d, 3*c^2 - a*b*d, 3*d^2 - a*b*c); r3=QQ[z];
gen = map(r2,r3,{w}); i1= kernel gen; factor i1_0
o7 =  (z-21)(z+6)^2
\end{verbatim}

The code for computing the critical points for $X = \Bl^7 \P^2$ via Mathematica is:

\begin{verbatim}
w[y1_, y2_] = (y2 + 1/y2)*(1/y1^2 + 4/y1 + 6 + 4 y1 + y1^2) + (2/y1^2 + 8/y1 + 
    8 y1 + 2 y1^2);
Critval = 
{-12., -12., -12., -12., -12., -12., -12., -12., 52., -12., -12., \
-12., -12., -12., -12., -12.}
\end{verbatim}

The ring generated by $c_1$ was computed using Macaulay2:

\begin{verbatim}
r1 = QQ[y1,Y1,y2,Y2]/(y1*Y1 -1, y2*Y2-1); 
w = ((1 + y1 + y2)^4 - 12 *y1 * y2)*Y1*Y2; 
d1 = y1*diff(y1,w) - Y1*diff(Y1,w); 
d2= y2*diff(y2,w) - Y2*diff(Y2,w);  
r2=r1/(d1,d2); r3=QQ[z]; 
gen = map(r2,r3,{w});  i1= kernel gen; factor i1_0
o9 =  (z-52)(z+12) 
\end{verbatim}

The following  Macaulay2 computation shows that the Jacobian ring of the disk potential
is not isomorphic to the quantum cohomology, since the sub-ring generated by $c_1(X)$
has rank two.

\begin{verbatim}
r1 = QQ[y1,Y1,y2,Y2]/(y1*Y1 -1, y2*Y2-1); 
w = (y2 + Y2)*(Y1^2 + 4*Y1 + 6 + 4 * y1 + y1^2) + (2*Y1^2 + 8*Y1 + 
    8 * y1 + 2 * y1^2);
d1 = y1*diff(y1,w) - Y1*diff(Y1,w); 
d2= y2*diff(y2,w) - Y2*diff(Y2,w);  r2=r1/(d1,d2); r3=QQ[z]; 
gen = map(r2,r3,{w});  i1= kernel gen; factor i1_0
o9 =  (z-52)(z+12) 
\end{verbatim}

The Hori-Vafa-style mirror, roughly adapted from \cite{horivafa}, has better properties.   It is given in Macaulay2 by the following code:

\begin{verbatim}
i66 : r1 = QQ[a,b,c,d]; w = a^4+b^4+c^4+d^4-a^2*b*c*d  ;
r2 = r1/(4*a^3 - 2*a*b*c*d, 4*b^3 - a^2*c*d, 4*c^3 - a^2*b*d, 4*d^3 - a^2*b*c); 
r3=QQ[z];gen = map(r2,r3,{w});
i1= kernel gen; factor i1_0
o72 =  (z)^2(z-64) 
\end{verbatim}

In particular, the sub-ring generated by $c_1(X) = w$ has dimension four, as expected.

The code for computing the critical points for $X = \Bl^8 \P^2$ via Mathematica is:

\begin{verbatim}
w[y1_, y2_] = (1 + y1)^9 / (y1^6 * y2) + (3 + 18 y1 + 45 y1^2 + 45 y1^4 + 
     18 y1^5 + 3 y1^6)/ y1^3 + (3 + 9 y1 + 9 y1^2 + 3 y1^3 )*y2 + y1^3*y2^2; 
Critval = 
{-60., -60., -60., -60., -60., -60., -60., -60., -60., -60., -60., \
-60., -60., -60., -60., -60., -60., -60., -60., -60., -60., -60., \
-60., -60., -60., 372.}
\end{verbatim}

The subring generated by $c_1$ was computed using Macaulay2:

\begin{verbatim}
i1 : r1 = QQ[y1,y2,Y1,Y2]/(y1*Y1 -1, y2*Y2-1)
i4 : w = (1 + y1)^9 * (Y1^6 * Y2) 
+ (3 + 18 * y1 + 45 * y1^2 + 45 * y1^4 + 18  * y1^5 + 3  * y1^6) * Y1^3 
+ (3 + 9  * y1 + 9 * y1^2 + 3 * y1^3 )*y2 + y1^3*y2^2; 
d1 = y1*diff(y1,w) - Y1*diff(X,w);  d2= y2*diff(y2,w) - Y2*diff(Y2,w); 
r2=r1/(d1,d2);  r3=QQ[z];  gen = map(r2,r3,{w});  i1= kernel gen; factor i1_0
o11 =  (z-372)(z+60) 
\end{verbatim}

The ring generated by 
$z = c_1 -60$ is given by the Macaulay2 code:

\begin{verbatim}
r1 = QQ[a,b,c,d]; w = a^6+b^6+c^6+d^6-a*b^2*c^3*d  ;
r2 = r1/(6*a^5 - b^2*c^3*d, 6*b^5 - a*2*b*c^3*d, 
6*c^5 - a*b^2*3*c^2*d, 6*d^5 - a*b^2*c^3); 
r3=QQ[z]; gen = map(r2,r3,{w}); i1= kernel gen; factor i1_0
o21 =  (z)^2(z-432) 
\end{verbatim}

\noindent Thus the Jacobian ring of the Hori-Vafa potential re-produces the correct sub-ring generated by the canonical class, while the disk potential of the monotone torus 
does not. 

\appendix

\bibliography{dp}{}
\bibliographystyle{plain}
\end{document}

%% file: b5p2_at.pdf_tex
\begingroup%
  \makeatletter%
  \providecommand\color[2][]{%
    \errmessage{(Inkscape) Color is used for the text in Inkscape, but the package 'color.sty' is not loaded}%
    \renewcommand\color[2][]{}%
  }%
  \providecommand\transparent[1]{%
    \errmessage{(Inkscape) Transparency is used (non-zero) for the text in Inkscape, but the package 'transparent.sty' is not loaded}%
    \renewcommand\transparent[1]{}%
  }%
  \providecommand\rotatebox[2]{#2}%
  \newcommand*\fsize{\dimexpr\f@size pt\relax}%
  \newcommand*\lineheight[1]{\fontsize{\fsize}{#1\fsize}\selectfont}%
  \ifx\svgwidth\undefined%
    \setlength{\unitlength}{340.44848271bp}%
    \ifx\svgscale\undefined%
      \relax%
    \else%
      \setlength{\unitlength}{\unitlength * \real{\svgscale}}%
    \fi%
  \else%
    \setlength{\unitlength}{\svgwidth}%
  \fi%
  \global\let\svgwidth\undefined%
  \global\let\svgscale\undefined%
  \makeatother%
  \begin{picture}(1,0.44703259)%
    \lineheight{1}%
    \setlength\tabcolsep{0pt}%
    \put(0,0){\includegraphics[width=\unitlength,page=1]{b5p2_at.pdf}}%
  \end{picture}%
\endgroup%

%% file: mv1.pdf_tex
\begingroup%
  \makeatletter%
  \providecommand\color[2][]{%
    \errmessage{(Inkscape) Color is used for the text in Inkscape, but the package 'color.sty' is not loaded}%
    \renewcommand\color[2][]{}%
  }%
  \providecommand\transparent[1]{%
    \errmessage{(Inkscape) Transparency is used (non-zero) for the text in Inkscape, but the package 'transparent.sty' is not loaded}%
    \renewcommand\transparent[1]{}%
  }%
  \providecommand\rotatebox[2]{#2}%
  \newcommand*\fsize{\dimexpr\f@size pt\relax}%
  \newcommand*\lineheight[1]{\fontsize{\fsize}{#1\fsize}\selectfont}%
  \ifx\svgwidth\undefined%
    \setlength{\unitlength}{138.12511826bp}%
    \ifx\svgscale\undefined%
      \relax%
    \else%
      \setlength{\unitlength}{\unitlength * \real{\svgscale}}%
    \fi%
  \else%
    \setlength{\unitlength}{\svgwidth}%
  \fi%
  \global\let\svgwidth\undefined%
  \global\let\svgscale\undefined%
  \makeatother%
  \begin{picture}(1,0.83121261)%
    \lineheight{1}%
    \setlength\tabcolsep{0pt}%
    \put(0,0){\includegraphics[width=\unitlength,page=1]{mv1.pdf}}%
    \put(0.26751777,0.30329169){\color[rgb]{0,0,0}\makebox(0,0)[lt]{\lineheight{1.25}\smash{\begin{tabular}[t]{l}$e_1$\end{tabular}}}}%
    \put(0.68301413,0.30611068){\color[rgb]{0,0,0}\makebox(0,0)[lt]{\lineheight{1.25}\smash{\begin{tabular}[t]{l}$e_2$\end{tabular}}}}%
    \put(0.58459937,0.09312353){\color[rgb]{0,0,0}\makebox(0,0)[lt]{\lineheight{1.25}\smash{\begin{tabular}[t]{l}$e_3$\end{tabular}}}}%
  \end{picture}%
\endgroup%

%% file: mv2.pdf_tex
\begingroup%
  \makeatletter%
  \providecommand\color[2][]{%
    \errmessage{(Inkscape) Color is used for the text in Inkscape, but the package 'color.sty' is not loaded}%
    \renewcommand\color[2][]{}%
  }%
  \providecommand\transparent[1]{%
    \errmessage{(Inkscape) Transparency is used (non-zero) for the text in Inkscape, but the package 'transparent.sty' is not loaded}%
    \renewcommand\transparent[1]{}%
  }%
  \providecommand\rotatebox[2]{#2}%
  \newcommand*\fsize{\dimexpr\f@size pt\relax}%
  \newcommand*\lineheight[1]{\fontsize{\fsize}{#1\fsize}\selectfont}%
  \ifx\svgwidth\undefined%
    \setlength{\unitlength}{227.92770506bp}%
    \ifx\svgscale\undefined%
      \relax%
    \else%
      \setlength{\unitlength}{\unitlength * \real{\svgscale}}%
    \fi%
  \else%
    \setlength{\unitlength}{\svgwidth}%
  \fi%
  \global\let\svgwidth\undefined%
  \global\let\svgscale\undefined%
  \makeatother%
  \begin{picture}(1,0.73383834)%
    \lineheight{1}%
    \setlength\tabcolsep{0pt}%
    \put(0,0){\includegraphics[width=\unitlength,page=1]{mv2.pdf}}%
  \end{picture}%
\endgroup%

%% file: mv3.pdf_tex
\begingroup%
  \makeatletter%
  \providecommand\color[2][]{%
    \errmessage{(Inkscape) Color is used for the text in Inkscape, but the package 'color.sty' is not loaded}%
    \renewcommand\color[2][]{}%
  }%
  \providecommand\transparent[1]{%
    \errmessage{(Inkscape) Transparency is used (non-zero) for the text in Inkscape, but the package 'transparent.sty' is not loaded}%
    \renewcommand\transparent[1]{}%
  }%
  \providecommand\rotatebox[2]{#2}%
  \newcommand*\fsize{\dimexpr\f@size pt\relax}%
  \newcommand*\lineheight[1]{\fontsize{\fsize}{#1\fsize}\selectfont}%
  \ifx\svgwidth\undefined%
    \setlength{\unitlength}{144.89644328bp}%
    \ifx\svgscale\undefined%
      \relax%
    \else%
      \setlength{\unitlength}{\unitlength * \real{\svgscale}}%
    \fi%
  \else%
    \setlength{\unitlength}{\svgwidth}%
  \fi%
  \global\let\svgwidth\undefined%
  \global\let\svgscale\undefined%
  \makeatother%
  \begin{picture}(1,1.20685066)%
    \lineheight{1}%
    \setlength\tabcolsep{0pt}%
    \put(0,0){\includegraphics[width=\unitlength,page=1]{mv3.pdf}}%
  \end{picture}%
\endgroup%

%% file: chek.pdf_tex
\begingroup%
  \makeatletter%
  \providecommand\color[2][]{%
    \errmessage{(Inkscape) Color is used for the text in Inkscape, but the package 'color.sty' is not loaded}%
    \renewcommand\color[2][]{}%
  }%
  \providecommand\transparent[1]{%
    \errmessage{(Inkscape) Transparency is used (non-zero) for the text in Inkscape, but the package 'transparent.sty' is not loaded}%
    \renewcommand\transparent[1]{}%
  }%
  \providecommand\rotatebox[2]{#2}%
  \newcommand*\fsize{\dimexpr\f@size pt\relax}%
  \newcommand*\lineheight[1]{\fontsize{\fsize}{#1\fsize}\selectfont}%
  \ifx\svgwidth\undefined%
    \setlength{\unitlength}{80.6023165bp}%
    \ifx\svgscale\undefined%
      \relax%
    \else%
      \setlength{\unitlength}{\unitlength * \real{\svgscale}}%
    \fi%
  \else%
    \setlength{\unitlength}{\svgwidth}%
  \fi%
  \global\let\svgwidth\undefined%
  \global\let\svgscale\undefined%
  \makeatother%
  \begin{picture}(1,2.95524565)%
    \lineheight{1}%
    \setlength\tabcolsep{0pt}%
    \put(0,0){\includegraphics[width=\unitlength,page=1]{chek.pdf}}%
    \put(-0.00490443,2.09946619){\color[rgb]{0,0,0}\makebox(0,0)[lt]{\lineheight{1.25}\smash{\begin{tabular}[t]{l}$(-1,0)$\end{tabular}}}}%
    \put(0.70679249,2.40419704){\color[rgb]{0,0,0}\makebox(0,0)[lt]{\lineheight{1.25}\smash{\begin{tabular}[t]{l}$(1,1)$\end{tabular}}}}%
    \put(0.75477862,1.82125537){\color[rgb]{0,0,0}\makebox(0,0)[lt]{\lineheight{1.25}\smash{\begin{tabular}[t]{l}$(3,1)$\end{tabular}}}}%
    \put(0.27486842,1.64512591){\color[rgb]{0,0,0}\makebox(0,0)[lt]{\lineheight{1.25}\smash{\begin{tabular}[t]{l}$(2,0)$\end{tabular}}}}%
    \put(0,0){\includegraphics[width=\unitlength,page=2]{chek.pdf}}%
  \end{picture}%
\endgroup%

%% file: b8p2_exc.pdf_tex
\begingroup%
  \makeatletter%
  \providecommand\color[2][]{%
    \errmessage{(Inkscape) Color is used for the text in Inkscape, but the package 'color.sty' is not loaded}%
    \renewcommand\color[2][]{}%
  }%
  \providecommand\transparent[1]{%
    \errmessage{(Inkscape) Transparency is used (non-zero) for the text in Inkscape, but the package 'transparent.sty' is not loaded}%
    \renewcommand\transparent[1]{}%
  }%
  \providecommand\rotatebox[2]{#2}%
  \newcommand*\fsize{\dimexpr\f@size pt\relax}%
  \newcommand*\lineheight[1]{\fontsize{\fsize}{#1\fsize}\selectfont}%
  \ifx\svgwidth\undefined%
    \setlength{\unitlength}{407.74041378bp}%
    \ifx\svgscale\undefined%
      \relax%
    \else%
      \setlength{\unitlength}{\unitlength * \real{\svgscale}}%
    \fi%
  \else%
    \setlength{\unitlength}{\svgwidth}%
  \fi%
  \global\let\svgwidth\undefined%
  \global\let\svgscale\undefined%
  \makeatother%
  \begin{picture}(1,0.79316564)%
    \lineheight{1}%
    \setlength\tabcolsep{0pt}%
    \put(0,0){\includegraphics[width=\unitlength,page=1]{b8p2_exc.pdf}}%
    \put(0.33402512,0.00022073){\color[rgb]{0,0,1}\transparent{0.76470602}\makebox(0,0)[lt]{\lineheight{1.25}\smash{\begin{tabular}[t]{l}84 + 84 +3 + 81 = 252\end{tabular}}}}%
    \put(0,0){\includegraphics[width=\unitlength,page=2]{b8p2_exc.pdf}}%
  \end{picture}%
\endgroup%

%% file: pert-edges.pdf_tex
\begingroup%
  \makeatletter%
  \providecommand\color[2][]{%
    \errmessage{(Inkscape) Color is used for the text in Inkscape, but the package 'color.sty' is not loaded}%
    \renewcommand\color[2][]{}%
  }%
  \providecommand\transparent[1]{%
    \errmessage{(Inkscape) Transparency is used (non-zero) for the text in Inkscape, but the package 'transparent.sty' is not loaded}%
    \renewcommand\transparent[1]{}%
  }%
  \providecommand\rotatebox[2]{#2}%
  \newcommand*\fsize{\dimexpr\f@size pt\relax}%
  \newcommand*\lineheight[1]{\fontsize{\fsize}{#1\fsize}\selectfont}%
  \ifx\svgwidth\undefined%
    \setlength{\unitlength}{328.64142951bp}%
    \ifx\svgscale\undefined%
      \relax%
    \else%
      \setlength{\unitlength}{\unitlength * \real{\svgscale}}%
    \fi%
  \else%
    \setlength{\unitlength}{\svgwidth}%
  \fi%
  \global\let\svgwidth\undefined%
  \global\let\svgscale\undefined%
  \makeatother%
  \begin{picture}(1,0.41113135)%
    \lineheight{1}%
    \setlength\tabcolsep{0pt}%
    \put(0,0){\includegraphics[width=\unitlength,page=1]{pert-edges.pdf}}%
    \put(0.93619029,0.1788473){\color[rgb]{0,0,0}\makebox(0,0)[lt]{\lineheight{1.25}\smash{\begin{tabular}[t]{l}$e_1'$\end{tabular}}}}%
    \put(0.60456364,0.16759392){\color[rgb]{0,0,0}\makebox(0,0)[lt]{\lineheight{1.25}\smash{\begin{tabular}[t]{l}$e_3'$\end{tabular}}}}%
    \put(0.60735185,0.10048664){\color[rgb]{0,0,0}\makebox(0,0)[lt]{\lineheight{1.25}\smash{\begin{tabular}[t]{l}$e_2'$\end{tabular}}}}%
    \put(0,0){\includegraphics[width=\unitlength,page=2]{pert-edges.pdf}}%
    \put(0.58434614,0.21368376){\color[rgb]{0,0,0}\makebox(0,0)[lt]{\lineheight{1.25}\smash{\begin{tabular}[t]{l}$\ell_3$\end{tabular}}}}%
    \put(0.49878748,0.14986376){\color[rgb]{0,0,0}\makebox(0,0)[lt]{\lineheight{1.25}\smash{\begin{tabular}[t]{l}$\ell_2$\end{tabular}}}}%
    \put(0.90328491,0.26840158){\color[rgb]{0,0,0}\makebox(0,0)[lt]{\lineheight{1.25}\smash{\begin{tabular}[t]{l}$\ell_1$\end{tabular}}}}%
    \put(0.86852655,0.05492561){\color[rgb]{0,0,0}\makebox(0,0)[lt]{\lineheight{1.25}\smash{\begin{tabular}[t]{l}$\Gamma_v^\pert$\end{tabular}}}}%
    \put(0,0){\includegraphics[width=\unitlength,page=3]{pert-edges.pdf}}%
    \put(0.29639274,0.25838365){\color[rgb]{0,0,0}\makebox(0,0)[lt]{\lineheight{1.25}\smash{\begin{tabular}[t]{l}$e_1$\end{tabular}}}}%
    \put(0.01034154,0.14996721){\color[rgb]{0,0,0}\makebox(0,0)[lt]{\lineheight{1.25}\smash{\begin{tabular}[t]{l}$e_2$\end{tabular}}}}%
    \put(-0.00098347,0.19386939){\color[rgb]{0,0,0}\makebox(0,0)[lt]{\lineheight{1.25}\smash{\begin{tabular}[t]{l}$e_3$\end{tabular}}}}%
    \put(0.20343963,0.34013326){\color[rgb]{0,0,0}\makebox(0,0)[lt]{\lineheight{1.25}\smash{\begin{tabular}[t]{l}$e_4$\end{tabular}}}}%
    \put(0.31720341,0.30513886){\color[rgb]{0,0,0}\makebox(0,0)[lt]{\lineheight{1.25}\smash{\begin{tabular}[t]{l}$e_5$\end{tabular}}}}%
    \put(0.20991673,0.16895505){\color[rgb]{0,0,0}\makebox(0,0)[lt]{\lineheight{1.25}\smash{\begin{tabular}[t]{l}$v$\end{tabular}}}}%
    \put(0.05897361,0.07509915){\color[rgb]{0,0,0}\makebox(0,0)[lt]{\lineheight{1.25}\smash{\begin{tabular}[t]{l}$\Gamma_v$\end{tabular}}}}%
  \end{picture}%
\endgroup%

%% file: pert-edges2.pdf_tex
\begingroup%
  \makeatletter%
  \providecommand\color[2][]{%
    \errmessage{(Inkscape) Color is used for the text in Inkscape, but the package 'color.sty' is not loaded}%
    \renewcommand\color[2][]{}%
  }%
  \providecommand\transparent[1]{%
    \errmessage{(Inkscape) Transparency is used (non-zero) for the text in Inkscape, but the package 'transparent.sty' is not loaded}%
    \renewcommand\transparent[1]{}%
  }%
  \providecommand\rotatebox[2]{#2}%
  \newcommand*\fsize{\dimexpr\f@size pt\relax}%
  \newcommand*\lineheight[1]{\fontsize{\fsize}{#1\fsize}\selectfont}%
  \ifx\svgwidth\undefined%
    \setlength{\unitlength}{345.88375782bp}%
    \ifx\svgscale\undefined%
      \relax%
    \else%
      \setlength{\unitlength}{\unitlength * \real{\svgscale}}%
    \fi%
  \else%
    \setlength{\unitlength}{\svgwidth}%
  \fi%
  \global\let\svgwidth\undefined%
  \global\let\svgscale\undefined%
  \makeatother%
  \begin{picture}(1,0.41188226)%
    \lineheight{1}%
    \setlength\tabcolsep{0pt}%
    \put(0,0){\includegraphics[width=\unitlength,page=1]{pert-edges2.pdf}}%
    \put(0.85921521,0.24980294){\color[rgb]{0,0,0}\makebox(0,0)[lt]{\lineheight{1.25}\smash{\begin{tabular}[t]{l}$e_1'$\end{tabular}}}}%
    \put(0.637156,0.14484927){\color[rgb]{0,0,0}\makebox(0,0)[lt]{\lineheight{1.25}\smash{\begin{tabular}[t]{l}$e_3'$\end{tabular}}}}%
    \put(0.62385396,0.054531){\color[rgb]{0,0,0}\makebox(0,0)[lt]{\lineheight{1.25}\smash{\begin{tabular}[t]{l}$e_2'$\end{tabular}}}}%
    \put(0,0){\includegraphics[width=\unitlength,page=2]{pert-edges2.pdf}}%
    \put(0.59585595,0.20824084){\color[rgb]{0,0,0}\makebox(0,0)[lt]{\lineheight{1.25}\smash{\begin{tabular}[t]{l}$\ell_3$\end{tabular}}}}%
    \put(0.51456242,0.14760228){\color[rgb]{0,0,0}\makebox(0,0)[lt]{\lineheight{1.25}\smash{\begin{tabular}[t]{l}$\ell_2$\end{tabular}}}}%
    \put(0,0){\includegraphics[width=\unitlength,page=3]{pert-edges2.pdf}}%
    \put(0.94972106,0.36216666){\color[rgb]{0,0,0}\makebox(0,0)[lt]{\lineheight{1.25}\smash{\begin{tabular}[t]{l}$\ell_1$\end{tabular}}}}%
    \put(0,0){\includegraphics[width=\unitlength,page=4]{pert-edges2.pdf}}%
    \put(0.38281563,0.30101313){\color[rgb]{0,0,0}\makebox(0,0)[lt]{\lineheight{1.25}\smash{\begin{tabular}[t]{l}$e_1'$\end{tabular}}}}%
    \put(0.115659,0.16451886){\color[rgb]{0,0,0}\makebox(0,0)[lt]{\lineheight{1.25}\smash{\begin{tabular}[t]{l}$e_3'$\end{tabular}}}}%
    \put(0.10833634,0.07605485){\color[rgb]{0,0,0}\makebox(0,0)[lt]{\lineheight{1.25}\smash{\begin{tabular}[t]{l}$e_2'$\end{tabular}}}}%
    \put(0,0){\includegraphics[width=\unitlength,page=5]{pert-edges2.pdf}}%
    \put(0.08035913,0.22427744){\color[rgb]{0,0,0}\makebox(0,0)[lt]{\lineheight{1.25}\smash{\begin{tabular}[t]{l}$\ell_3$\end{tabular}}}}%
    \put(-0.00093444,0.16363887){\color[rgb]{0,0,0}\makebox(0,0)[lt]{\lineheight{1.25}\smash{\begin{tabular}[t]{l}$\ell_2$\end{tabular}}}}%
    \put(0,0){\includegraphics[width=\unitlength,page=6]{pert-edges2.pdf}}%
    \put(0.43858932,0.37881719){\color[rgb]{0,0,0}\makebox(0,0)[lt]{\lineheight{1.25}\smash{\begin{tabular}[t]{l}$\ell_1$\end{tabular}}}}%
    \put(0,0){\includegraphics[width=\unitlength,page=7]{pert-edges2.pdf}}%
  \end{picture}%
\endgroup%

%% file: annulus.pdf_tex
\begingroup%
  \makeatletter%
  \providecommand\color[2][]{%
    \errmessage{(Inkscape) Color is used for the text in Inkscape, but the package 'color.sty' is not loaded}%
    \renewcommand\color[2][]{}%
  }%
  \providecommand\transparent[1]{%
    \errmessage{(Inkscape) Transparency is used (non-zero) for the text in Inkscape, but the package 'transparent.sty' is not loaded}%
    \renewcommand\transparent[1]{}%
  }%
  \providecommand\rotatebox[2]{#2}%
  \newcommand*\fsize{\dimexpr\f@size pt\relax}%
  \newcommand*\lineheight[1]{\fontsize{\fsize}{#1\fsize}\selectfont}%
  \ifx\svgwidth\undefined%
    \setlength{\unitlength}{329.62654442bp}%
    \ifx\svgscale\undefined%
      \relax%
    \else%
      \setlength{\unitlength}{\unitlength * \real{\svgscale}}%
    \fi%
  \else%
    \setlength{\unitlength}{\svgwidth}%
  \fi%
  \global\let\svgwidth\undefined%
  \global\let\svgscale\undefined%
  \makeatother%
  \begin{picture}(1,0.48113044)%
    \lineheight{1}%
    \setlength\tabcolsep{0pt}%
    \put(0,0){\includegraphics[width=\unitlength,page=1]{annulus.pdf}}%
    \put(0.01585433,0.35844466){\color[rgb]{0,0,0}\makebox(0,0)[lt]{\lineheight{1.25}\smash{\begin{tabular}[t]{l}$\Phi(X)$\end{tabular}}}}%
    \put(0,0){\includegraphics[width=\unitlength,page=2]{annulus.pdf}}%
    \put(0.73829763,0.47375238){\color[rgb]{0,0,0}\makebox(0,0)[lt]{\lineheight{1.25}\smash{\begin{tabular}[t]{l}shear\end{tabular}}}}%
    \put(0.97456508,0.23672851){\color[rgb]{0,0,0}\makebox(0,0)[lt]{\lineheight{1.25}\smash{\begin{tabular}[t]{l}shear\end{tabular}}}}%
    \put(0.71973123,0.00009708){\color[rgb]{0,0,0}\makebox(0,0)[lt]{\lineheight{1.25}\smash{\begin{tabular}[t]{l}shear\end{tabular}}}}%
    \put(0.45491127,0.25231728){\color[rgb]{0,0,0}\makebox(0,0)[lt]{\lineheight{1.25}\smash{\begin{tabular}[t]{l}shear\end{tabular}}}}%
    \put(0.70241427,0.24449674){\color[rgb]{0,0,0}\makebox(0,0)[lt]{\lineheight{1.25}\smash{\begin{tabular}[t]{l}$\A(X)$\end{tabular}}}}%
    \put(0,0){\includegraphics[width=\unitlength,page=3]{annulus.pdf}}%
  \end{picture}%
\endgroup%

%% file: ops.pdf_tex
\begingroup%
  \makeatletter%
  \providecommand\color[2][]{%
    \errmessage{(Inkscape) Color is used for the text in Inkscape, but the package 'color.sty' is not loaded}%
    \renewcommand\color[2][]{}%
  }%
  \providecommand\transparent[1]{%
    \errmessage{(Inkscape) Transparency is used (non-zero) for the text in Inkscape, but the package 'transparent.sty' is not loaded}%
    \renewcommand\transparent[1]{}%
  }%
  \providecommand\rotatebox[2]{#2}%
  \newcommand*\fsize{\dimexpr\f@size pt\relax}%
  \newcommand*\lineheight[1]{\fontsize{\fsize}{#1\fsize}\selectfont}%
  \ifx\svgwidth\undefined%
    \setlength{\unitlength}{226.85177334bp}%
    \ifx\svgscale\undefined%
      \relax%
    \else%
      \setlength{\unitlength}{\unitlength * \real{\svgscale}}%
    \fi%
  \else%
    \setlength{\unitlength}{\svgwidth}%
  \fi%
  \global\let\svgwidth\undefined%
  \global\let\svgscale\undefined%
  \makeatother%
  \begin{picture}(1,0.97438281)%
    \lineheight{1}%
    \setlength\tabcolsep{0pt}%
    \put(0,0){\includegraphics[width=\unitlength,page=1]{ops.pdf}}%
    \put(0.13852674,0.9518169){\color[rgb]{0,0,0}\makebox(0,0)[lt]{\lineheight{1.25}\smash{\begin{tabular}[t]{l}Nodal trade\end{tabular}}}}%
    \put(0.72949637,0.45176574){\color[rgb]{0,0,0}\makebox(0,0)[lt]{\lineheight{1.25}\smash{\begin{tabular}[t]{l}Transferring the cut\end{tabular}}}}%
    \put(0.1674878,0.11699651){\color[rgb]{0,0,0}\makebox(0,0)[lt]{\lineheight{1.25}\smash{\begin{tabular}[t]{l}Nodal slide\end{tabular}}}}%
    \put(0,0){\includegraphics[width=\unitlength,page=2]{ops.pdf}}%
  \end{picture}%
\endgroup%

%% file: transf.pdf_tex
\begingroup%
  \makeatletter%
  \providecommand\color[2][]{%
    \errmessage{(Inkscape) Color is used for the text in Inkscape, but the package 'color.sty' is not loaded}%
    \renewcommand\color[2][]{}%
  }%
  \providecommand\transparent[1]{%
    \errmessage{(Inkscape) Transparency is used (non-zero) for the text in Inkscape, but the package 'transparent.sty' is not loaded}%
    \renewcommand\transparent[1]{}%
  }%
  \providecommand\rotatebox[2]{#2}%
  \newcommand*\fsize{\dimexpr\f@size pt\relax}%
  \newcommand*\lineheight[1]{\fontsize{\fsize}{#1\fsize}\selectfont}%
  \ifx\svgwidth\undefined%
    \setlength{\unitlength}{343.48774452bp}%
    \ifx\svgscale\undefined%
      \relax%
    \else%
      \setlength{\unitlength}{\unitlength * \real{\svgscale}}%
    \fi%
  \else%
    \setlength{\unitlength}{\svgwidth}%
  \fi%
  \global\let\svgwidth\undefined%
  \global\let\svgscale\undefined%
  \makeatother%
  \begin{picture}(1,0.43965955)%
    \lineheight{1}%
    \setlength\tabcolsep{0pt}%
    \put(0,0){\includegraphics[width=\unitlength,page=1]{transf.pdf}}%
    \put(0.21742949,0.36251417){\color[rgb]{0,0,0}\makebox(0,0)[lt]{\lineheight{1.25}\smash{\begin{tabular}[t]{l}Transferring \\the cut\end{tabular}}}}%
    \put(0,0){\includegraphics[width=\unitlength,page=2]{transf.pdf}}%
    \put(0.92088665,0.10637635){\color[rgb]{0,0,0}\makebox(0,0)[lt]{\lineheight{1.25}\smash{\begin{tabular}[t]{l}$s_2$\end{tabular}}}}%
    \put(0,0){\includegraphics[width=\unitlength,page=3]{transf.pdf}}%
    \put(0.44340071,0.21421139){\color[rgb]{0,0,0}\makebox(0,0)[lt]{\lineheight{1.25}\smash{\begin{tabular}[t]{l}$C_2'$\end{tabular}}}}%
    \put(0.44707687,0.34348029){\color[rgb]{0,0,0}\makebox(0,0)[lt]{\lineheight{1.25}\smash{\begin{tabular}[t]{l}$C_1$\end{tabular}}}}%
    \put(0.44277544,0.06565878){\color[rgb]{0,0,0}\makebox(0,0)[lt]{\lineheight{1.25}\smash{\begin{tabular}[t]{l}$s_2 C_3$\end{tabular}}}}%
    \put(0.42619777,0.42673288){\color[rgb]{0,0,0}\makebox(0,0)[lt]{\lineheight{1.25}\smash{\begin{tabular}[t]{l}$(0,1)$\end{tabular}}}}%
    \put(0.41996001,0.0058503){\color[rgb]{0,0,0}\makebox(0,0)[lt]{\lineheight{1.25}\smash{\begin{tabular}[t]{l}$(0,-1)$\end{tabular}}}}%
    \put(0.52529698,0.28564138){\color[rgb]{0,0,0}\makebox(0,0)[lt]{\lineheight{1.25}\smash{\begin{tabular}[t]{l}$(\hh,\hh)$\end{tabular}}}}%
    \put(0,0){\includegraphics[width=\unitlength,page=4]{transf.pdf}}%
    \put(-0.00105593,0.17819922){\color[rgb]{0,0,0}\makebox(0,0)[lt]{\lineheight{1.25}\smash{\begin{tabular}[t]{l}$(0,0)$\end{tabular}}}}%
    \put(0.00197275,0.37558814){\color[rgb]{0,0,0}\makebox(0,0)[lt]{\lineheight{1.25}\smash{\begin{tabular}[t]{l}$(0,1)$\end{tabular}}}}%
    \put(0.24301638,0.1458649){\color[rgb]{0,0,0}\makebox(0,0)[lt]{\lineheight{1.25}\smash{\begin{tabular}[t]{l}$(1,0)$\end{tabular}}}}%
    \put(0.10424326,0.28711872){\color[rgb]{0,0,0}\makebox(0,0)[lt]{\lineheight{1.25}\smash{\begin{tabular}[t]{l}$C_1$\end{tabular}}}}%
    \put(0.12264216,0.20375355){\color[rgb]{0,0,0}\makebox(0,0)[lt]{\lineheight{1.25}\smash{\begin{tabular}[t]{l}$C_2$\end{tabular}}}}%
    \put(0.18641599,0.23203225){\color[rgb]{0,0,0}\makebox(0,0)[lt]{\lineheight{1.25}\smash{\begin{tabular}[t]{l}$C_3$\end{tabular}}}}%
    \put(0,0){\includegraphics[width=\unitlength,page=5]{transf.pdf}}%
  \end{picture}%
\endgroup%

%% file: triang.pdf_tex
\begingroup%
  \makeatletter%
  \providecommand\color[2][]{%
    \errmessage{(Inkscape) Color is used for the text in Inkscape, but the package 'color.sty' is not loaded}%
    \renewcommand\color[2][]{}%
  }%
  \providecommand\transparent[1]{%
    \errmessage{(Inkscape) Transparency is used (non-zero) for the text in Inkscape, but the package 'transparent.sty' is not loaded}%
    \renewcommand\transparent[1]{}%
  }%
  \providecommand\rotatebox[2]{#2}%
  \newcommand*\fsize{\dimexpr\f@size pt\relax}%
  \newcommand*\lineheight[1]{\fontsize{\fsize}{#1\fsize}\selectfont}%
  \ifx\svgwidth\undefined%
    \setlength{\unitlength}{229.12706124bp}%
    \ifx\svgscale\undefined%
      \relax%
    \else%
      \setlength{\unitlength}{\unitlength * \real{\svgscale}}%
    \fi%
  \else%
    \setlength{\unitlength}{\svgwidth}%
  \fi%
  \global\let\svgwidth\undefined%
  \global\let\svgscale\undefined%
  \makeatother%
  \begin{picture}(1,0.60088096)%
    \lineheight{1}%
    \setlength\tabcolsep{0pt}%
    \put(0,0){\includegraphics[width=\unitlength,page=1]{triang.pdf}}%
    \put(0.06755506,0.13693241){\color[rgb]{0,0,0}\makebox(0,0)[lt]{\lineheight{1.25}\smash{\begin{tabular}[t]{l}$D_5$\end{tabular}}}}%
    \put(0,0){\includegraphics[width=\unitlength,page=2]{triang.pdf}}%
    \put(0.54814635,0.41018195){\color[rgb]{0,0,0}\makebox(0,0)[lt]{\lineheight{1.25}\smash{\begin{tabular}[t]{l}$E_8$\end{tabular}}}}%
    \put(0,0){\includegraphics[width=\unitlength,page=3]{triang.pdf}}%
    \put(0.92337872,0.18668221){\color[rgb]{0,0,0}\makebox(0,0)[lt]{\lineheight{1.25}\smash{\begin{tabular}[t]{l}$E_7$\end{tabular}}}}%
    \put(0,0){\includegraphics[width=\unitlength,page=4]{triang.pdf}}%
  \end{picture}%
\endgroup%

%% file: rect.pdf_tex
\begingroup%
  \makeatletter%
  \providecommand\color[2][]{%
    \errmessage{(Inkscape) Color is used for the text in Inkscape, but the package 'color.sty' is not loaded}%
    \renewcommand\color[2][]{}%
  }%
  \providecommand\transparent[1]{%
    \errmessage{(Inkscape) Transparency is used (non-zero) for the text in Inkscape, but the package 'transparent.sty' is not loaded}%
    \renewcommand\transparent[1]{}%
  }%
  \providecommand\rotatebox[2]{#2}%
  \newcommand*\fsize{\dimexpr\f@size pt\relax}%
  \newcommand*\lineheight[1]{\fontsize{\fsize}{#1\fsize}\selectfont}%
  \ifx\svgwidth\undefined%
    \setlength{\unitlength}{441.88905484bp}%
    \ifx\svgscale\undefined%
      \relax%
    \else%
      \setlength{\unitlength}{\unitlength * \real{\svgscale}}%
    \fi%
  \else%
    \setlength{\unitlength}{\svgwidth}%
  \fi%
  \global\let\svgwidth\undefined%
  \global\let\svgscale\undefined%
  \makeatother%
  \begin{picture}(1,0.28004768)%
    \lineheight{1}%
    \setlength\tabcolsep{0pt}%
    \put(0,0){\includegraphics[width=\unitlength,page=1]{rect.pdf}}%
    \put(0.39609056,0.124937){\color[rgb]{0,0,0}\makebox(0,0)[lt]{\lineheight{1.25}\smash{\begin{tabular}[t]{l}$P_0 \in \PP_0$\end{tabular}}}}%
    \put(0.25182356,0.249841){\color[rgb]{0,0,0}\makebox(0,0)[lt]{\lineheight{1.25}\smash{\begin{tabular}[t]{l}$P_1 \in \PP_1$\end{tabular}}}}%
    \put(0,0){\includegraphics[width=\unitlength,page=2]{rect.pdf}}%
    \put(-0.00138261,0.13578069){\color[rgb]{0.50196078,0.2,0}\makebox(0,0)[lt]{\lineheight{1.25}\smash{\begin{tabular}[t]{l}$P_0^\dual$\end{tabular}}}}%
    \put(0.31971395,0.00612889){\color[rgb]{0.50196078,0.2,0}\makebox(0,0)[lt]{\lineheight{1.25}\smash{\begin{tabular}[t]{l}$P_1^\dual$\end{tabular}}}}%
    \put(0.59338615,0.15039589){\color[rgb]{0.50196078,0.2,0}\makebox(0,0)[lt]{\lineheight{1.25}\smash{\begin{tabular}[t]{l}$\rho P_0^\dual$\end{tabular}}}}%
    \put(0.66903452,0.06789037){\color[rgb]{0.50196078,0.2,0}\makebox(0,0)[lt]{\lineheight{1.25}\smash{\begin{tabular}[t]{l}$P_1^\dual$\end{tabular}}}}%
    \put(0,0){\includegraphics[width=\unitlength,page=3]{rect.pdf}}%
    \put(0.82485241,0.13491411){\color[rgb]{0.50196078,0.2,0}\makebox(0,0)[lt]{\lineheight{1.25}\smash{\begin{tabular}[t]{l}$\rho P_0^\dual$\end{tabular}}}}%
    \put(0.89266631,0.06702381){\color[rgb]{0.50196078,0.2,0}\makebox(0,0)[lt]{\lineheight{1.25}\smash{\begin{tabular}[t]{l}$P_1^\dual$\end{tabular}}}}%
    \put(0,0){\includegraphics[width=\unitlength,page=4]{rect.pdf}}%
    \put(0.97588264,0.21215735){\color[rgb]{0.50196078,0.2,0}\makebox(0,0)[lt]{\lineheight{1.25}\smash{\begin{tabular}[t]{l}$\Gamma$\end{tabular}}}}%
    \put(0.66903451,0.22064363){\color[rgb]{0.50196078,0.2,0}\makebox(0,0)[lt]{\lineheight{1.25}\smash{\begin{tabular}[t]{l}$B^\dual_\rho$\end{tabular}}}}%
  \end{picture}%
\endgroup%

%% file: ff-general.pdf_tex
\begingroup%
  \makeatletter%
  \providecommand\color[2][]{%
    \errmessage{(Inkscape) Color is used for the text in Inkscape, but the package 'color.sty' is not loaded}%
    \renewcommand\color[2][]{}%
  }%
  \providecommand\transparent[1]{%
    \errmessage{(Inkscape) Transparency is used (non-zero) for the text in Inkscape, but the package 'transparent.sty' is not loaded}%
    \renewcommand\transparent[1]{}%
  }%
  \providecommand\rotatebox[2]{#2}%
  \newcommand*\fsize{\dimexpr\f@size pt\relax}%
  \newcommand*\lineheight[1]{\fontsize{\fsize}{#1\fsize}\selectfont}%
  \ifx\svgwidth\undefined%
    \setlength{\unitlength}{381.96223094bp}%
    \ifx\svgscale\undefined%
      \relax%
    \else%
      \setlength{\unitlength}{\unitlength * \real{\svgscale}}%
    \fi%
  \else%
    \setlength{\unitlength}{\svgwidth}%
  \fi%
  \global\let\svgwidth\undefined%
  \global\let\svgscale\undefined%
  \makeatother%
  \begin{picture}(1,0.24859569)%
    \lineheight{1}%
    \setlength\tabcolsep{0pt}%
    \put(0,0){\includegraphics[width=\unitlength,page=1]{ff-general.pdf}}%
    \put(0.34007282,0.17193203){\color[rgb]{0,0,0}\makebox(0,0)[lt]{\lineheight{1.25}\smash{\begin{tabular}[t]{l}$u$\end{tabular}}}}%
    \put(0,0){\includegraphics[width=\unitlength,page=2]{ff-general.pdf}}%
    \put(0.05760892,0.17634424){\color[rgb]{0,0,0}\makebox(0,0)[lt]{\lineheight{1.25}\smash{\begin{tabular}[t]{l}$f^{-1}(0)$\end{tabular}}}}%
    \put(0.49128721,0.13420093){\color[rgb]{0,0,0}\makebox(0,0)[lt]{\lineheight{1.25}\smash{\begin{tabular}[t]{l}$f^{-1}(\infty)$\end{tabular}}}}%
    \put(0.1856073,0.23798827){\color[rgb]{0.50196078,0.50196078,0.50196078}\makebox(0,0)[lt]{\lineheight{1.25}\smash{\begin{tabular}[t]{l}$S_0$\end{tabular}}}}%
    \put(0.18501616,0.00335412){\color[rgb]{0.50196078,0.50196078,0.50196078}\makebox(0,0)[lt]{\lineheight{1.25}\smash{\begin{tabular}[t]{l}$S_\infty$\end{tabular}}}}%
    \put(0,0){\includegraphics[width=\unitlength,page=3]{ff-general.pdf}}%
    \put(-0.00086648,0.05784362){\color[rgb]{0,0,0}\makebox(0,0)[lt]{\lineheight{1.25}\smash{\begin{tabular}[t]{l}$(m_1,n_1)$\end{tabular}}}}%
    \put(0.44520956,0.18973958){\color[rgb]{0,0,0}\makebox(0,0)[lt]{\lineheight{1.25}\smash{\begin{tabular}[t]{l}$(m,n)$\end{tabular}}}}%
    \put(0.49060577,0.04370622){\color[rgb]{0,0,0}\makebox(0,0)[lt]{\lineheight{1.25}\smash{\begin{tabular}[t]{l}$(m,n-m)$\end{tabular}}}}%
    \put(0,0){\includegraphics[width=\unitlength,page=4]{ff-general.pdf}}%
  \end{picture}%
\endgroup%

%% file: ffd.pdf_tex
\begingroup%
  \makeatletter%
  \providecommand\color[2][]{%
    \errmessage{(Inkscape) Color is used for the text in Inkscape, but the package 'color.sty' is not loaded}%
    \renewcommand\color[2][]{}%
  }%
  \providecommand\transparent[1]{%
    \errmessage{(Inkscape) Transparency is used (non-zero) for the text in Inkscape, but the package 'transparent.sty' is not loaded}%
    \renewcommand\transparent[1]{}%
  }%
  \providecommand\rotatebox[2]{#2}%
  \newcommand*\fsize{\dimexpr\f@size pt\relax}%
  \newcommand*\lineheight[1]{\fontsize{\fsize}{#1\fsize}\selectfont}%
  \ifx\svgwidth\undefined%
    \setlength{\unitlength}{273.05304674bp}%
    \ifx\svgscale\undefined%
      \relax%
    \else%
      \setlength{\unitlength}{\unitlength * \real{\svgscale}}%
    \fi%
  \else%
    \setlength{\unitlength}{\svgwidth}%
  \fi%
  \global\let\svgwidth\undefined%
  \global\let\svgscale\undefined%
  \makeatother%
  \begin{picture}(1,0.34774988)%
    \lineheight{1}%
    \setlength\tabcolsep{0pt}%
    \put(0,0){\includegraphics[width=\unitlength,page=1]{ffd.pdf}}%
    \put(0.65707868,0.24050837){\color[rgb]{0,0,0}\makebox(0,0)[lt]{\lineheight{1.25}\smash{\begin{tabular}[t]{l}$u$\end{tabular}}}}%
    \put(0,0){\includegraphics[width=\unitlength,page=2]{ffd.pdf}}%
    \put(0.08058663,0.24992173){\color[rgb]{0,0,0}\makebox(0,0)[lt]{\lineheight{1.25}\smash{\begin{tabular}[t]{l}$f^{-1}(0)$\end{tabular}}}}%
    \put(0.86860615,0.18772786){\color[rgb]{0,0,0}\makebox(0,0)[lt]{\lineheight{1.25}\smash{\begin{tabular}[t]{l}$f^{-1}(\infty)$\end{tabular}}}}%
    \put(0.44100345,0.33291162){\color[rgb]{0.50196078,0.50196078,0.50196078}\makebox(0,0)[lt]{\lineheight{1.25}\smash{\begin{tabular}[t]{l}$S_0$\end{tabular}}}}%
    \put(0.44017657,0.00469194){\color[rgb]{0.50196078,0.50196078,0.50196078}\makebox(0,0)[lt]{\lineheight{1.25}\smash{\begin{tabular}[t]{l}$S_\infty$\end{tabular}}}}%
    \put(0,0){\includegraphics[width=\unitlength,page=3]{ffd.pdf}}%
    \put(-0.00121208,0.08415627){\color[rgb]{0,0,0}\makebox(0,0)[lt]{\lineheight{1.25}\smash{\begin{tabular}[t]{l}$(m_1,n_1)$\end{tabular}}}}%
    \put(0.80415,0.26541867){\color[rgb]{0,0,0}\makebox(0,0)[lt]{\lineheight{1.25}\smash{\begin{tabular}[t]{l}$(m,n)$\end{tabular}}}}%
    \put(0.86765278,0.06113876){\color[rgb]{0,0,0}\makebox(0,0)[lt]{\lineheight{1.25}\smash{\begin{tabular}[t]{l}$(m,n-mk)$\end{tabular}}}}%
    \put(0,0){\includegraphics[width=\unitlength,page=4]{ffd.pdf}}%
    \put(0.52773521,0.06804146){\color[rgb]{0.50196078,0,0.50196078}\makebox(0,0)[lt]{\lineheight{1.25}\smash{\begin{tabular}[t]{l}$k$\end{tabular}}}}%
  \end{picture}%
\endgroup%

%% file: crossff.pdf_tex
\begingroup%
  \makeatletter%
  \providecommand\color[2][]{%
    \errmessage{(Inkscape) Color is used for the text in Inkscape, but the package 'color.sty' is not loaded}%
    \renewcommand\color[2][]{}%
  }%
  \providecommand\transparent[1]{%
    \errmessage{(Inkscape) Transparency is used (non-zero) for the text in Inkscape, but the package 'transparent.sty' is not loaded}%
    \renewcommand\transparent[1]{}%
  }%
  \providecommand\rotatebox[2]{#2}%
  \newcommand*\fsize{\dimexpr\f@size pt\relax}%
  \newcommand*\lineheight[1]{\fontsize{\fsize}{#1\fsize}\selectfont}%
  \ifx\svgwidth\undefined%
    \setlength{\unitlength}{453.16181882bp}%
    \ifx\svgscale\undefined%
      \relax%
    \else%
      \setlength{\unitlength}{\unitlength * \real{\svgscale}}%
    \fi%
  \else%
    \setlength{\unitlength}{\svgwidth}%
  \fi%
  \global\let\svgwidth\undefined%
  \global\let\svgscale\undefined%
  \makeatother%
  \begin{picture}(1,0.71423601)%
    \lineheight{1}%
    \setlength\tabcolsep{0pt}%
    \put(0,0){\includegraphics[width=\unitlength,page=1]{crossff.pdf}}%
    \put(0.09366205,0.47221895){\color[rgb]{0.50196078,0.2,0}\makebox(0,0)[lt]{\lineheight{1.25}\smash{\begin{tabular}[t]{l}$p$\end{tabular}}}}%
    \put(0.1436132,0.58967767){\color[rgb]{0.50196078,0,0.50196078}\makebox(0,0)[lt]{\lineheight{1.25}\smash{\begin{tabular}[t]{l}$b$\end{tabular}}}}%
    \put(-0.00015598,0.55828854){\color[rgb]{0,0,0}\makebox(0,0)[lt]{\lineheight{1.25}\smash{\begin{tabular}[t]{l}$\PP$\end{tabular}}}}%
    \put(0.00199123,0.36526322){\color[rgb]{0.50196078,0.2,0}\makebox(0,0)[lt]{\lineheight{1.25}\smash{\begin{tabular}[t]{l}$B^\dual$\end{tabular}}}}%
    \put(0,0){\includegraphics[width=\unitlength,page=2]{crossff.pdf}}%
    \put(0.40351201,0.47268931){\color[rgb]{0.50196078,0.2,0}\makebox(0,0)[lt]{\lineheight{1.25}\smash{\begin{tabular}[t]{l}$p$\end{tabular}}}}%
    \put(0.48428617,0.58844687){\color[rgb]{0.50196078,0,0.50196078}\makebox(0,0)[lt]{\lineheight{1.25}\smash{\begin{tabular}[t]{l}$b$\end{tabular}}}}%
    \put(0,0){\includegraphics[width=\unitlength,page=3]{crossff.pdf}}%
    \put(0.45579111,0.47088187){\color[rgb]{0,0,0}\makebox(0,0)[lt]{\lineheight{1.25}\smash{\begin{tabular}[t]{l}$\PP_1$\end{tabular}}}}%
    \put(0,0){\includegraphics[width=\unitlength,page=4]{crossff.pdf}}%
    \put(0.42397008,0.22455648){\color[rgb]{0.50196078,0.2,0}\makebox(0,0)[lt]{\lineheight{1.25}\smash{\begin{tabular}[t]{l}$B^\dual \times \rho B_1^\dual$\end{tabular}}}}%
    \put(0,0){\includegraphics[width=\unitlength,page=5]{crossff.pdf}}%
    \put(0.13682656,0.08604384){\color[rgb]{0.50196078,0.2,0}\makebox(0,0)[lt]{\lineheight{1.25}\smash{\begin{tabular}[t]{l}$v$\end{tabular}}}}%
    \put(0.13774944,0.04461075){\color[rgb]{0.50196078,0.2,0}\makebox(0,0)[lt]{\lineheight{1.25}\smash{\begin{tabular}[t]{l}$e_0$\end{tabular}}}}%
    \put(0.07406253,0.16138654){\color[rgb]{0.50196078,0.2,0}\makebox(0,0)[lt]{\lineheight{1.25}\smash{\begin{tabular}[t]{l}$e_1$\end{tabular}}}}%
    \put(0.14241266,0.31120443){\color[rgb]{0.50196078,0.2,0}\makebox(0,0)[lt]{\lineheight{1.25}\smash{\begin{tabular}[t]{l}$P(v)$\end{tabular}}}}%
    \put(0,0){\includegraphics[width=\unitlength,page=6]{crossff.pdf}}%
    \put(0.47067933,0.04182683){\color[rgb]{0.50196078,0.2,0}\makebox(0,0)[lt]{\lineheight{1.25}\smash{\begin{tabular}[t]{l}$e_0$\end{tabular}}}}%
    \put(0.40699242,0.15860262){\color[rgb]{0.50196078,0.2,0}\makebox(0,0)[lt]{\lineheight{1.25}\smash{\begin{tabular}[t]{l}$e_1$\end{tabular}}}}%
    \put(0,0){\includegraphics[width=\unitlength,page=7]{crossff.pdf}}%
    \put(0.41659512,0.30687349){\color[rgb]{0.50196078,0.2,0}\makebox(0,0)[lt]{\lineheight{1.25}\smash{\begin{tabular}[t]{l}$P_p$\end{tabular}}}}%
    \put(0.52099937,0.30912515){\color[rgb]{0.50196078,0.2,0}\makebox(0,0)[lt]{\lineheight{1.25}\smash{\begin{tabular}[t]{l}$P_b$\end{tabular}}}}%
    \put(0,0){\includegraphics[width=\unitlength,page=8]{crossff.pdf}}%
    \put(0.8760716,0.31055872){\color[rgb]{0.50196078,0.2,0}\makebox(0,0)[lt]{\lineheight{1.25}\smash{\begin{tabular}[t]{l}$P_p$\end{tabular}}}}%
    \put(0.79964289,0.30508938){\color[rgb]{0.50196078,0.2,0}\makebox(0,0)[lt]{\lineheight{1.25}\smash{\begin{tabular}[t]{l}$P_b$\end{tabular}}}}%
    \put(0,0){\includegraphics[width=\unitlength,page=9]{crossff.pdf}}%
    \put(0.5213393,0.07667717){\color[rgb]{0.50196078,0.2,0}\makebox(0,0)[lt]{\lineheight{1.25}\smash{\begin{tabular}[t]{l}$v_0$\end{tabular}}}}%
    \put(0.43733004,0.07863433){\color[rgb]{0.50196078,0.2,0}\makebox(0,0)[lt]{\lineheight{1.25}\smash{\begin{tabular}[t]{l}$v_1$\end{tabular}}}}%
    \put(0.88348817,0.15634833){\color[rgb]{0.08627451,0.35294118,0.77254902}\makebox(0,0)[lt]{\lineheight{1.25}\smash{\begin{tabular}[t]{l}$v_1$\end{tabular}}}}%
    \put(0.80336726,0.17278863){\color[rgb]{0.08627451,0.35294118,0.77254902}\makebox(0,0)[lt]{\lineheight{1.25}\smash{\begin{tabular}[t]{l}$v_1$\end{tabular}}}}%
    \put(0.88214178,0.18427175){\color[rgb]{0.08627451,0.35294118,0.77254902}\makebox(0,0)[lt]{\lineheight{1.25}\smash{\begin{tabular}[t]{l}$e_1$\end{tabular}}}}%
    \put(0.77658329,0.19249744){\color[rgb]{0.08627451,0.35294118,0.77254902}\makebox(0,0)[lt]{\lineheight{1.25}\smash{\begin{tabular}[t]{l}$e_1$\end{tabular}}}}%
    \put(0,0){\includegraphics[width=\unitlength,page=10]{crossff.pdf}}%
    \put(0.82687643,0.14493492){\color[rgb]{0,0,0}\makebox(0,0)[lt]{\lineheight{1.25}\smash{\begin{tabular}[t]{l}$\sim$\end{tabular}}}}%
    \put(0,0){\includegraphics[width=\unitlength,page=11]{crossff.pdf}}%
    \put(0.87941532,0.02014263){\color[rgb]{0.08627451,0.35294118,0.77254902}\makebox(0,0)[lt]{\lineheight{1.25}\smash{\begin{tabular}[t]{l}$e_0$\end{tabular}}}}%
    \put(0.88592591,0.09188371){\color[rgb]{0.08627451,0.35294118,0.77254902}\makebox(0,0)[lt]{\lineheight{1.25}\smash{\begin{tabular}[t]{l}$v_0$\end{tabular}}}}%
    \put(0,0){\includegraphics[width=\unitlength,page=12]{crossff.pdf}}%
    \put(0.89321057,0.47299145){\color[rgb]{0.50196078,0.2,0}\makebox(0,0)[lt]{\lineheight{1.25}\smash{\begin{tabular}[t]{l}$p$\end{tabular}}}}%
    \put(0.82491334,0.56958087){\color[rgb]{0.50196078,0,0.50196078}\makebox(0,0)[lt]{\lineheight{1.25}\smash{\begin{tabular}[t]{l}$b$\end{tabular}}}}%
    \put(0,0){\includegraphics[width=\unitlength,page=13]{crossff.pdf}}%
    \put(0.83650356,0.4748072){\color[rgb]{0,0,0}\makebox(0,0)[lt]{\lineheight{1.25}\smash{\begin{tabular}[t]{l}$\PP_1$\end{tabular}}}}%
    \put(0,0){\includegraphics[width=\unitlength,page=14]{crossff.pdf}}%
    \put(0.48153312,0.35844183){\color[rgb]{0.50196078,0.2,0}\makebox(0,0)[lt]{\lineheight{1.25}\smash{\begin{tabular}[t]{l}$\rho$\end{tabular}}}}%
    \put(0,0){\includegraphics[width=\unitlength,page=15]{crossff.pdf}}%
    \put(0.84017199,0.35971792){\color[rgb]{0.50196078,0.2,0}\makebox(0,0)[lt]{\lineheight{1.25}\smash{\begin{tabular}[t]{l}$\rho$\end{tabular}}}}%
    \put(0.80285567,0.22251158){\color[rgb]{0.50196078,0.2,0}\makebox(0,0)[lt]{\lineheight{1.25}\smash{\begin{tabular}[t]{l}$B^\dual \times \rho B_1^\dual$\end{tabular}}}}%
    \put(-0.00056004,0.15092723){\color[rgb]{0.08627451,0.35294118,0.77254902}\makebox(0,0)[lt]{\lineheight{1.25}\smash{\begin{tabular}[t]{l}$\Gamma_v$\end{tabular}}}}%
    \put(0.32733824,0.15761399){\color[rgb]{0.08627451,0.35294118,0.77254902}\makebox(0,0)[lt]{\lineheight{1.25}\smash{\begin{tabular}[t]{l}$\Gamma_{v,+}^\pert$\end{tabular}}}}%
    \put(0.7121369,0.14942641){\color[rgb]{0.08627451,0.35294118,0.77254902}\makebox(0,0)[lt]{\lineheight{1.25}\smash{\begin{tabular}[t]{l}$\Gamma_{v,-}^\pert$\end{tabular}}}}%
    \put(0.07769413,0.68258683){\color[rgb]{0,0,0}\makebox(0,0)[lt]{\lineheight{1.25}\smash{\begin{tabular}[t]{l}$X_{P(v)}$\end{tabular}}}}%
    \put(0.33086513,0.56126561){\color[rgb]{0,0,0}\makebox(0,0)[lt]{\lineheight{1.25}\smash{\begin{tabular}[t]{l}$X_{P_p}$\end{tabular}}}}%
    \put(0.53224473,0.61297584){\color[rgb]{0,0,0}\makebox(0,0)[lt]{\lineheight{1.25}\smash{\begin{tabular}[t]{l}$X_{P_b}$\end{tabular}}}}%
    \put(0.69106796,0.53672127){\color[rgb]{0,0,0}\makebox(0,0)[lt]{\lineheight{1.25}\smash{\begin{tabular}[t]{l}$X_{P_b}$\end{tabular}}}}%
    \put(0.93497095,0.54362055){\color[rgb]{0,0,0}\makebox(0,0)[lt]{\lineheight{1.25}\smash{\begin{tabular}[t]{l}$X_{P_p}$\end{tabular}}}}%
  \end{picture}%
\endgroup%

%% file: split2pert.pdf_tex
\begingroup%
  \makeatletter%
  \providecommand\color[2][]{%
    \errmessage{(Inkscape) Color is used for the text in Inkscape, but the package 'color.sty' is not loaded}%
    \renewcommand\color[2][]{}%
  }%
  \providecommand\transparent[1]{%
    \errmessage{(Inkscape) Transparency is used (non-zero) for the text in Inkscape, but the package 'transparent.sty' is not loaded}%
    \renewcommand\transparent[1]{}%
  }%
  \providecommand\rotatebox[2]{#2}%
  \newcommand*\fsize{\dimexpr\f@size pt\relax}%
  \newcommand*\lineheight[1]{\fontsize{\fsize}{#1\fsize}\selectfont}%
  \ifx\svgwidth\undefined%
    \setlength{\unitlength}{476.52984023bp}%
    \ifx\svgscale\undefined%
      \relax%
    \else%
      \setlength{\unitlength}{\unitlength * \real{\svgscale}}%
    \fi%
  \else%
    \setlength{\unitlength}{\svgwidth}%
  \fi%
  \global\let\svgwidth\undefined%
  \global\let\svgscale\undefined%
  \makeatother%
  \begin{picture}(1,0.66538687)%
    \lineheight{1}%
    \setlength\tabcolsep{0pt}%
    \put(0,0){\includegraphics[width=\unitlength,page=1]{split2pert.pdf}}%
    \put(0.4018384,0.44616928){\color[rgb]{0,0,0}\makebox(0,0)[lt]{\lineheight{1.25}\smash{\begin{tabular}[t]{l}Split $e_1$\end{tabular}}}}%
    \put(0.3755614,0.54165615){\color[rgb]{0.08627451,0.35294118,0.77254902}\makebox(0,0)[lt]{\lineheight{1.25}\smash{\begin{tabular}[t]{l}$(3,2)$\end{tabular}}}}%
    \put(0.16682493,0.36764765){\color[rgb]{0.08627451,0.35294118,0.77254902}\makebox(0,0)[lt]{\lineheight{1.25}\smash{\begin{tabular}[t]{l}$(3,0)$\end{tabular}}}}%
    \put(0.26957821,0.29184283){\color[rgb]{0.08627451,0.35294118,0.77254902}\makebox(0,0)[lt]{\lineheight{1.25}\smash{\begin{tabular}[t]{l}$(0,1), (0,1)$\end{tabular}}}}%
    \put(0,0){\includegraphics[width=\unitlength,page=2]{split2pert.pdf}}%
    \put(0.26687113,0.38172848){\color[rgb]{0.08627451,0.35294118,0.77254902}\makebox(0,0)[lt]{\lineheight{1.25}\smash{\begin{tabular}[t]{l}$v$\end{tabular}}}}%
    \put(0.331542,0.33538939){\color[rgb]{0.08627451,0.35294118,0.77254902}\makebox(0,0)[lt]{\lineheight{1.25}\smash{\begin{tabular}[t]{l}$e_1, e_2$\end{tabular}}}}%
    \put(0,0){\includegraphics[width=\unitlength,page=3]{split2pert.pdf}}%
    \put(0.28505411,0.53200118){\color[rgb]{0.08627451,0.35294118,0.77254902}\makebox(0,0)[lt]{\lineheight{1.25}\smash{\begin{tabular}[t]{l}$\Gamma$\end{tabular}}}}%
    \put(0,0){\includegraphics[width=\unitlength,page=4]{split2pert.pdf}}%
    \put(0.05902833,0.53620341){\color[rgb]{0.50196078,0.2,0}\makebox(0,0)[lt]{\lineheight{1.25}\smash{\begin{tabular}[t]{l}$B^\dual$\end{tabular}}}}%
    \put(0.70161797,0.52628445){\color[rgb]{0.08627451,0.35294118,0.77254902}\makebox(0,0)[lt]{\lineheight{1.25}\smash{\begin{tabular}[t]{l}$(3,2)$\end{tabular}}}}%
    \put(0.49323796,0.38783512){\color[rgb]{0.08627451,0.35294118,0.77254902}\makebox(0,0)[lt]{\lineheight{1.25}\smash{\begin{tabular}[t]{l}$(3,0)$\end{tabular}}}}%
    \put(0,0){\includegraphics[width=\unitlength,page=5]{split2pert.pdf}}%
    \put(0.57839446,0.3692152){\color[rgb]{0.08627451,0.35294118,0.77254902}\makebox(0,0)[lt]{\lineheight{1.25}\smash{\begin{tabular}[t]{l}$e_2$\end{tabular}}}}%
    \put(0,0){\includegraphics[width=\unitlength,page=6]{split2pert.pdf}}%
    \put(0.62751732,0.36492893){\color[rgb]{0.08627451,0.35294118,0.77254902}\makebox(0,0)[lt]{\lineheight{1.25}\smash{\begin{tabular}[t]{l}$e_1$\end{tabular}}}}%
    \put(0,0){\includegraphics[width=\unitlength,page=7]{split2pert.pdf}}%
    \put(0.61786228,0.55938365){\color[rgb]{0.08627451,0.35294118,0.77254902}\makebox(0,0)[lt]{\lineheight{1.25}\smash{\begin{tabular}[t]{l}$\tGam$\end{tabular}}}}%
    \put(0.92565651,0.63726059){\color[rgb]{0.08627451,0.35294118,0.77254902}\makebox(0,0)[lt]{\lineheight{1.25}\smash{\begin{tabular}[t]{l}$(3,2)$\end{tabular}}}}%
    \put(0.74675831,0.45798509){\color[rgb]{0.08627451,0.35294118,0.77254902}\makebox(0,0)[lt]{\lineheight{1.25}\smash{\begin{tabular}[t]{l}$(3,0)$\end{tabular}}}}%
    \put(0,0){\includegraphics[width=\unitlength,page=8]{split2pert.pdf}}%
    \put(0.80893349,0.43478077){\color[rgb]{0.08627451,0.35294118,0.77254902}\makebox(0,0)[lt]{\lineheight{1.25}\smash{\begin{tabular}[t]{l}$e_2$\end{tabular}}}}%
    \put(0,0){\includegraphics[width=\unitlength,page=9]{split2pert.pdf}}%
    \put(0.86964553,0.45737679){\color[rgb]{0.08627451,0.35294118,0.77254902}\makebox(0,0)[lt]{\lineheight{1.25}\smash{\begin{tabular}[t]{l}$e_1$\end{tabular}}}}%
    \put(0.81148495,0.49706566){\color[rgb]{0.08627451,0.35294118,0.77254902}\makebox(0,0)[lt]{\lineheight{1.25}\smash{\begin{tabular}[t]{l}$(3,1)$\end{tabular}}}}%
    \put(0.91838779,0.51278881){\color[rgb]{0.08627451,0.35294118,0.77254902}\makebox(0,0)[lt]{\lineheight{1.25}\smash{\begin{tabular}[t]{l}$\Gamma_v^\pert$\end{tabular}}}}%
    \put(0.74171133,0.19627041){\color[rgb]{0.08627451,0.35294118,0.77254902}\makebox(0,0)[lt]{\lineheight{1.25}\smash{\begin{tabular}[t]{l}$(3,2)$\end{tabular}}}}%
    \put(0.53333133,0.0578212){\color[rgb]{0.08627451,0.35294118,0.77254902}\makebox(0,0)[lt]{\lineheight{1.25}\smash{\begin{tabular}[t]{l}$(3,0)$\end{tabular}}}}%
    \put(0,0){\includegraphics[width=\unitlength,page=10]{split2pert.pdf}}%
    \put(0.67321705,0.04903765){\color[rgb]{0.08627451,0.35294118,0.77254902}\makebox(0,0)[lt]{\lineheight{1.25}\smash{\begin{tabular}[t]{l}$e_1$\end{tabular}}}}%
    \put(0,0){\includegraphics[width=\unitlength,page=11]{split2pert.pdf}}%
    \put(0.62845783,0.02861498){\color[rgb]{0.08627451,0.35294118,0.77254902}\makebox(0,0)[lt]{\lineheight{1.25}\smash{\begin{tabular}[t]{l}$e_2$\end{tabular}}}}%
    \put(0,0){\includegraphics[width=\unitlength,page=12]{split2pert.pdf}}%
    \put(0.65795564,0.22936956){\color[rgb]{0.08627451,0.35294118,0.77254902}\makebox(0,0)[lt]{\lineheight{1.25}\smash{\begin{tabular}[t]{l}$\tGam$\end{tabular}}}}%
    \put(0,0){\includegraphics[width=\unitlength,page=13]{split2pert.pdf}}%
    \put(0.36175796,0.17061316){\color[rgb]{0,0,0}\makebox(0,0)[lt]{\lineheight{1.25}\smash{\begin{tabular}[t]{l}Split $e_1$, $e_2$\end{tabular}}}}%
  \end{picture}%
\endgroup%

%% file: bunch.pdf_tex
\begingroup%
  \makeatletter%
  \providecommand\color[2][]{%
    \errmessage{(Inkscape) Color is used for the text in Inkscape, but the package 'color.sty' is not loaded}%
    \renewcommand\color[2][]{}%
  }%
  \providecommand\transparent[1]{%
    \errmessage{(Inkscape) Transparency is used (non-zero) for the text in Inkscape, but the package 'transparent.sty' is not loaded}%
    \renewcommand\transparent[1]{}%
  }%
  \providecommand\rotatebox[2]{#2}%
  \newcommand*\fsize{\dimexpr\f@size pt\relax}%
  \newcommand*\lineheight[1]{\fontsize{\fsize}{#1\fsize}\selectfont}%
  \ifx\svgwidth\undefined%
    \setlength{\unitlength}{372.67535456bp}%
    \ifx\svgscale\undefined%
      \relax%
    \else%
      \setlength{\unitlength}{\unitlength * \real{\svgscale}}%
    \fi%
  \else%
    \setlength{\unitlength}{\svgwidth}%
  \fi%
  \global\let\svgwidth\undefined%
  \global\let\svgscale\undefined%
  \makeatother%
  \begin{picture}(1,0.20013844)%
    \lineheight{1}%
    \setlength\tabcolsep{0pt}%
    \put(0,0){\includegraphics[width=\unitlength,page=1]{bunch.pdf}}%
    \put(0.16599075,0.11573235){\color[rgb]{0.50196078,0,0.50196078}\makebox(0,0)[lt]{\lineheight{1.25}\smash{\begin{tabular}[t]{l}$b$\end{tabular}}}}%
    \put(0,0){\includegraphics[width=\unitlength,page=2]{bunch.pdf}}%
    \put(0.82867501,0.11891963){\color[rgb]{0.50196078,0,0.50196078}\makebox(0,0)[lt]{\lineheight{1.25}\smash{\begin{tabular}[t]{l}$b$\end{tabular}}}}%
    \put(0,0){\includegraphics[width=\unitlength,page=3]{bunch.pdf}}%
    \put(0.51987383,0.11542223){\color[rgb]{0.50196078,0,0.50196078}\makebox(0,0)[lt]{\lineheight{1.25}\smash{\begin{tabular}[t]{l}$b$\end{tabular}}}}%
    \put(0,0){\includegraphics[width=\unitlength,page=4]{bunch.pdf}}%
    \put(0.151918,0.05566797){\color[rgb]{0.08627451,0.35294118,0.77254902}\makebox(0,0)[lt]{\lineheight{1.25}\smash{\begin{tabular}[t]{l}$-n \mu_b$\end{tabular}}}}%
    \put(0.05487817,0.18552489){\color[rgb]{0.08627451,0.35294118,0.77254902}\makebox(0,0)[lt]{\lineheight{1.25}\smash{\begin{tabular}[t]{l}$f_1$\end{tabular}}}}%
    \put(0.14451896,0.00511112){\color[rgb]{0.08627451,0.35294118,0.77254902}\makebox(0,0)[lt]{\lineheight{1.25}\smash{\begin{tabular}[t]{l}$f_0$\end{tabular}}}}%
    \put(0.87897525,0.01074834){\color[rgb]{0.08627451,0.35294118,0.77254902}\makebox(0,0)[lt]{\lineheight{1.25}\smash{\begin{tabular}[t]{l}$f_0$\end{tabular}}}}%
    \put(0.91664719,0.17403793){\color[rgb]{0.08627451,0.35294118,0.77254902}\makebox(0,0)[lt]{\lineheight{1.25}\smash{\begin{tabular}[t]{l}$f_1$\end{tabular}}}}%
    \put(0.13868435,0.1023298){\color[rgb]{0.08627451,0.35294118,0.77254902}\makebox(0,0)[lt]{\lineheight{1.25}\smash{\begin{tabular}[t]{l}$e$\end{tabular}}}}%
    \put(0.84774438,0.10589231){\color[rgb]{0.08627451,0.35294118,0.77254902}\makebox(0,0)[lt]{\lineheight{1.25}\smash{\begin{tabular}[t]{l}$e$\end{tabular}}}}%
    \put(0.73739757,0.05668014){\color[rgb]{0.08627451,0.35294118,0.77254902}\makebox(0,0)[lt]{\lineheight{1.25}\smash{\begin{tabular}[t]{l}$(k-n) \mu_b$\end{tabular}}}}%
    \put(0.43630358,0.18335558){\color[rgb]{0.08627451,0.35294118,0.77254902}\makebox(0,0)[lt]{\lineheight{1.25}\smash{\begin{tabular}[t]{l}$f_1$\end{tabular}}}}%
    \put(0.52594437,0.00294181){\color[rgb]{0.08627451,0.35294118,0.77254902}\makebox(0,0)[lt]{\lineheight{1.25}\smash{\begin{tabular}[t]{l}$f_0$\end{tabular}}}}%
    \put(0.53420232,0.08435362){\color[rgb]{0.08627451,0.35294118,0.77254902}\makebox(0,0)[lt]{\lineheight{1.25}\smash{\begin{tabular}[t]{l}$v$\end{tabular}}}}%
    \put(0.02931521,0.08416887){\color[rgb]{0.08627451,0.35294118,0.77254902}\makebox(0,0)[lt]{\lineheight{1.25}\smash{\begin{tabular}[t]{l}$\Gamma_{v,+}^\pert$\end{tabular}}}}%
    \put(0.92887204,0.09139951){\color[rgb]{0.08627451,0.35294118,0.77254902}\makebox(0,0)[lt]{\lineheight{1.25}\smash{\begin{tabular}[t]{l}$\Gamma_{v,-}^\pert$\end{tabular}}}}%
  \end{picture}%
\endgroup%

%% file: 2p.pdf_tex
\begingroup%
  \makeatletter%
  \providecommand\color[2][]{%
    \errmessage{(Inkscape) Color is used for the text in Inkscape, but the package 'color.sty' is not loaded}%
    \renewcommand\color[2][]{}%
  }%
  \providecommand\transparent[1]{%
    \errmessage{(Inkscape) Transparency is used (non-zero) for the text in Inkscape, but the package 'transparent.sty' is not loaded}%
    \renewcommand\transparent[1]{}%
  }%
  \providecommand\rotatebox[2]{#2}%
  \newcommand*\fsize{\dimexpr\f@size pt\relax}%
  \newcommand*\lineheight[1]{\fontsize{\fsize}{#1\fsize}\selectfont}%
  \ifx\svgwidth\undefined%
    \setlength{\unitlength}{491.70280396bp}%
    \ifx\svgscale\undefined%
      \relax%
    \else%
      \setlength{\unitlength}{\unitlength * \real{\svgscale}}%
    \fi%
  \else%
    \setlength{\unitlength}{\svgwidth}%
  \fi%
  \global\let\svgwidth\undefined%
  \global\let\svgscale\undefined%
  \makeatother%
  \begin{picture}(1,0.23454338)%
    \lineheight{1}%
    \setlength\tabcolsep{0pt}%
    \put(0,0){\includegraphics[width=\unitlength,page=1]{2p.pdf}}%
    \put(0.24410511,0.13893562){\color[rgb]{0,0,0}\makebox(0,0)[lt]{\lineheight{1.25}\smash{\begin{tabular}[t]{l}Cut\end{tabular}}}}%
    \put(0,0){\includegraphics[width=\unitlength,page=2]{2p.pdf}}%
    \put(0.89405795,0.19155972){\color[rgb]{1,0,0}\makebox(0,0)[lt]{\lineheight{1.25}\smash{\begin{tabular}[t]{l}$\PP_{ann}$\end{tabular}}}}%
    \put(0.70177944,0.03867501){\color[rgb]{0.66666667,0.26666667,0}\makebox(0,0)[lt]{\lineheight{1.25}\smash{\begin{tabular}[t]{l}$\PP_{in}$\end{tabular}}}}%
  \end{picture}%
\endgroup%

%% file: b8p2-dual.pdf_tex
\begingroup%
  \makeatletter%
  \providecommand\color[2][]{%
    \errmessage{(Inkscape) Color is used for the text in Inkscape, but the package 'color.sty' is not loaded}%
    \renewcommand\color[2][]{}%
  }%
  \providecommand\transparent[1]{%
    \errmessage{(Inkscape) Transparency is used (non-zero) for the text in Inkscape, but the package 'transparent.sty' is not loaded}%
    \renewcommand\transparent[1]{}%
  }%
  \providecommand\rotatebox[2]{#2}%
  \newcommand*\fsize{\dimexpr\f@size pt\relax}%
  \newcommand*\lineheight[1]{\fontsize{\fsize}{#1\fsize}\selectfont}%
  \ifx\svgwidth\undefined%
    \setlength{\unitlength}{495.43235432bp}%
    \ifx\svgscale\undefined%
      \relax%
    \else%
      \setlength{\unitlength}{\unitlength * \real{\svgscale}}%
    \fi%
  \else%
    \setlength{\unitlength}{\svgwidth}%
  \fi%
  \global\let\svgwidth\undefined%
  \global\let\svgscale\undefined%
  \makeatother%
  \begin{picture}(1,0.4338668)%
    \lineheight{1}%
    \setlength\tabcolsep{0pt}%
    \put(0,0){\includegraphics[width=\unitlength,page=1]{b8p2-dual.pdf}}%
    \put(0.04152988,0.16548079){\color[rgb]{0,0,0}\makebox(0,0)[lt]{\lineheight{1.25}\smash{\begin{tabular}[t]{l}{\LARGE $\PP$}\end{tabular}}}}%
    \put(0.15624614,0.24091579){\color[rgb]{0.50196078,0.2,0}\makebox(0,0)[lt]{\lineheight{1.25}\smash{\begin{tabular}[t]{l}$P_0$\end{tabular}}}}%
    \put(0.12877345,0.19477733){\color[rgb]{0.50196078,0.2,0}\makebox(0,0)[lt]{\lineheight{1.25}\smash{\begin{tabular}[t]{l}$P_1$\end{tabular}}}}%
    \put(0.10495849,0.29197439){\color[rgb]{0.50196078,0.2,0}\makebox(0,0)[lt]{\lineheight{1.25}\smash{\begin{tabular}[t]{l}$P_2$\end{tabular}}}}%
    \put(0.12273386,0.33047575){\color[rgb]{0.50196078,0.2,0}\makebox(0,0)[lt]{\lineheight{1.25}\smash{\begin{tabular}[t]{l}$P_3$\end{tabular}}}}%
    \put(0.22005322,0.28870368){\color[rgb]{0.50196078,0.2,0}\makebox(0,0)[lt]{\lineheight{1.25}\smash{\begin{tabular}[t]{l}$P_4$\end{tabular}}}}%
    \put(0.22402381,0.21517584){\color[rgb]{0.50196078,0.2,0}\makebox(0,0)[lt]{\lineheight{1.25}\smash{\begin{tabular}[t]{l}$P_5$\end{tabular}}}}%
    \put(0.19227482,0.17656972){\color[rgb]{0.50196078,0.2,0}\makebox(0,0)[lt]{\lineheight{1.25}\smash{\begin{tabular}[t]{l}$P_6$\end{tabular}}}}%
    \put(0.19481805,0.0872718){\color[rgb]{0.78039216,0.21960784,0.21960784}\makebox(0,0)[lt]{\lineheight{1.25}\smash{\begin{tabular}[t]{l}$Q_6$\end{tabular}}}}%
    \put(0.28528957,0.21586205){\color[rgb]{0.78039216,0.21960784,0.21960784}\makebox(0,0)[lt]{\lineheight{1.25}\smash{\begin{tabular}[t]{l}$Q_5$\end{tabular}}}}%
    \put(0.27790439,0.28944262){\color[rgb]{0.78039216,0.21960784,0.21960784}\makebox(0,0)[lt]{\lineheight{1.25}\smash{\begin{tabular}[t]{l}$Q_4$\end{tabular}}}}%
    \put(0.08462324,0.37924411){\color[rgb]{0.78039216,0.21960784,0.21960784}\makebox(0,0)[lt]{\lineheight{1.25}\smash{\begin{tabular}[t]{l}$Q_3$\end{tabular}}}}%
    \put(-0.00098063,0.34260792){\color[rgb]{0.78039216,0.21960784,0.21960784}\makebox(0,0)[lt]{\lineheight{1.25}\smash{\begin{tabular}[t]{l}$Q_2$\end{tabular}}}}%
    \put(0.09752416,0.10164127){\color[rgb]{0.78039216,0.21960784,0.21960784}\makebox(0,0)[lt]{\lineheight{1.25}\smash{\begin{tabular}[t]{l}$Q_1$\end{tabular}}}}%
    \put(0,0){\includegraphics[width=\unitlength,page=2]{b8p2-dual.pdf}}%
    \put(0.74296053,0.11743808){\makebox(0,0)[lt]{\lineheight{1.25}\smash{\begin{tabular}[t]{l}$Q_5^\dual$\end{tabular}}}}%
    \put(0.51956566,0.18956583){\makebox(0,0)[lt]{\lineheight{1.25}\smash{\begin{tabular}[t]{l}$P_0^\dual$\end{tabular}}}}%
    \put(0.52196973,0.12098877){\makebox(0,0)[lt]{\lineheight{1.25}\smash{\begin{tabular}[t]{l}$Q_2^\dual$\end{tabular}}}}%
    \put(0.66527192,0.20731675){\makebox(0,0)[lt]{\lineheight{1.25}\smash{\begin{tabular}[t]{l}$P_4^\dual$\end{tabular}}}}%
    \put(0.67416734,0.17924537){\makebox(0,0)[lt]{\lineheight{1.25}\smash{\begin{tabular}[t]{l}$P_5^\dual$\end{tabular}}}}%
    \put(0.87610175,0.11216841){\color[rgb]{0,0,0}\makebox(0,0)[lt]{\lineheight{1.25}\smash{\begin{tabular}[t]{l}{\LARGE $B^\dual(\rho)$}\end{tabular}}}}%
    \put(0.70645055,0.08233649){\makebox(0,0)[lt]{\lineheight{1.25}\smash{\begin{tabular}[t]{l}$Q_6^\dual$\end{tabular}}}}%
    \put(0.53043012,0.08440577){\makebox(0,0)[lt]{\lineheight{1.25}\smash{\begin{tabular}[t]{l}$Q_1^\dual$\end{tabular}}}}%
    \put(0,0){\includegraphics[width=\unitlength,page=3]{b8p2-dual.pdf}}%
    \put(0.84460187,0.00317435){\makebox(0,0)[lt]{\lineheight{1.25}\smash{\begin{tabular}[t]{l}$(3,-1)$\end{tabular}}}}%
    \put(0,0){\includegraphics[width=\unitlength,page=4]{b8p2-dual.pdf}}%
    \put(0.80054022,0.08429515){\makebox(0,0)[lt]{\lineheight{1.25}\smash{\begin{tabular}[t]{l}$(0,1)$\end{tabular}}}}%
    \put(0,0){\includegraphics[width=\unitlength,page=5]{b8p2-dual.pdf}}%
    \put(0.73773153,0.02350266){\makebox(0,0)[lt]{\lineheight{1.25}\smash{\begin{tabular}[t]{l}$(1,0)$\end{tabular}}}}%
    \put(0,0){\includegraphics[width=\unitlength,page=6]{b8p2-dual.pdf}}%
    \put(0.48003339,0.02773879){\makebox(0,0)[lt]{\lineheight{1.25}\smash{\begin{tabular}[t]{l}$(1,0)$\end{tabular}}}}%
    \put(0,0){\includegraphics[width=\unitlength,page=7]{b8p2-dual.pdf}}%
    \put(0.4039852,0.09443822){\makebox(0,0)[lt]{\lineheight{1.25}\smash{\begin{tabular}[t]{l}$(-3,-1)$\end{tabular}}}}%
    \put(0,0){\includegraphics[width=\unitlength,page=8]{b8p2-dual.pdf}}%
    \put(0.39937981,0.00468679){\makebox(0,0)[lt]{\lineheight{1.25}\smash{\begin{tabular}[t]{l}$(-6,-1)$\end{tabular}}}}%
    \put(0,0){\includegraphics[width=\unitlength,page=9]{b8p2-dual.pdf}}%
    \put(0.64019575,0.04044254){\color[rgb]{0.4,0.4,0.4}\makebox(0,0)[lt]{\lineheight{1.25}\smash{\begin{tabular}[t]{l}$A_1$\end{tabular}}}}%
    \put(0,0){\includegraphics[width=\unitlength,page=10]{b8p2-dual.pdf}}%
    \put(0.79502376,0.19175696){\color[rgb]{0.4,0.4,0.4}\makebox(0,0)[lt]{\lineheight{1.25}\smash{\begin{tabular}[t]{l}$A_3$\end{tabular}}}}%
    \put(0,0){\includegraphics[width=\unitlength,page=11]{b8p2-dual.pdf}}%
    \put(0.70601579,0.28150315){\makebox(0,0)[lt]{\lineheight{1.25}\smash{\begin{tabular}[t]{l}$Q_4^\dual$\end{tabular}}}}%
    \put(0,0){\includegraphics[width=\unitlength,page=12]{b8p2-dual.pdf}}%
    \put(0.67863142,0.30689089){\makebox(0,0)[lt]{\lineheight{1.25}\smash{\begin{tabular}[t]{l}$Q_3^\dual$\end{tabular}}}}%
    \put(0,0){\includegraphics[width=\unitlength,page=13]{b8p2-dual.pdf}}%
    \put(0.76334764,0.35072238){\makebox(0,0)[lt]{\lineheight{1.25}\smash{\begin{tabular}[t]{l}$(0,1)$\end{tabular}}}}%
    \put(0,0){\includegraphics[width=\unitlength,page=14]{b8p2-dual.pdf}}%
    \put(0.62127511,0.37653354){\makebox(0,0)[lt]{\lineheight{1.25}\smash{\begin{tabular}[t]{l}$(-3,-1)$\end{tabular}}}}%
    \put(0,0){\includegraphics[width=\unitlength,page=15]{b8p2-dual.pdf}}%
    \put(0.76063192,0.39758574){\makebox(0,0)[lt]{\lineheight{1.25}\smash{\begin{tabular}[t]{l}$(3,2)$\end{tabular}}}}%
    \put(0,0){\includegraphics[width=\unitlength,page=16]{b8p2-dual.pdf}}%
    \put(0.54161438,0.23783809){\color[rgb]{0.4,0.4,0.4}\makebox(0,0)[lt]{\lineheight{1.25}\smash{\begin{tabular}[t]{l}$A_2$\end{tabular}}}}%
    \put(0.89095257,0.37568222){\makebox(0,0)[lt]{\lineheight{1.25}\smash{\begin{tabular}[t]{l}$\rho$\end{tabular}}}}%
    \put(0,0){\includegraphics[width=\unitlength,page=17]{b8p2-dual.pdf}}%
    \put(0.91018558,0.30942929){\makebox(0,0)[lt]{\lineheight{1.25}\smash{\begin{tabular}[t]{l}$Q_4^\dual$\end{tabular}}}}%
    \put(0,0){\includegraphics[width=\unitlength,page=18]{b8p2-dual.pdf}}%
    \put(0.49813338,0.33772311){\makebox(0,0)[lt]{\lineheight{1.25}\smash{\begin{tabular}[t]{l}$Q_3^\dual$\end{tabular}}}}%
    \put(0,0){\includegraphics[width=\unitlength,page=19]{b8p2-dual.pdf}}%
    \put(0.44274448,0.31586501){\makebox(0,0)[lt]{\lineheight{1.25}\smash{\begin{tabular}[t]{l}$\rho$\end{tabular}}}}%
  \end{picture}%
\endgroup%

%% file: b8p2-affine.pdf_tex
\begingroup%
  \makeatletter%
  \providecommand\color[2][]{%
    \errmessage{(Inkscape) Color is used for the text in Inkscape, but the package 'color.sty' is not loaded}%
    \renewcommand\color[2][]{}%
  }%
  \providecommand\transparent[1]{%
    \errmessage{(Inkscape) Transparency is used (non-zero) for the text in Inkscape, but the package 'transparent.sty' is not loaded}%
    \renewcommand\transparent[1]{}%
  }%
  \providecommand\rotatebox[2]{#2}%
  \newcommand*\fsize{\dimexpr\f@size pt\relax}%
  \newcommand*\lineheight[1]{\fontsize{\fsize}{#1\fsize}\selectfont}%
  \ifx\svgwidth\undefined%
    \setlength{\unitlength}{191.1349934bp}%
    \ifx\svgscale\undefined%
      \relax%
    \else%
      \setlength{\unitlength}{\unitlength * \real{\svgscale}}%
    \fi%
  \else%
    \setlength{\unitlength}{\svgwidth}%
  \fi%
  \global\let\svgwidth\undefined%
  \global\let\svgscale\undefined%
  \makeatother%
  \begin{picture}(1,0.95059367)%
    \lineheight{1}%
    \setlength\tabcolsep{0pt}%
    \put(0,0){\includegraphics[width=\unitlength,page=1]{b8p2-affine.pdf}}%
    \put(0.56891551,0.13835358){\color[rgb]{0.4,0.4,0.4}\makebox(0,0)[lt]{\lineheight{1.25}\smash{\begin{tabular}[t]{l}$A_1$\end{tabular}}}}%
    \put(0,0){\includegraphics[width=\unitlength,page=2]{b8p2-affine.pdf}}%
    \put(0.39856319,0.47946048){\color[rgb]{0.4,0.4,0.4}\makebox(0,0)[lt]{\lineheight{1.25}\smash{\begin{tabular}[t]{l}$A_2$\end{tabular}}}}%
    \put(0,0){\includegraphics[width=\unitlength,page=3]{b8p2-affine.pdf}}%
    \put(0.85933935,0.40745557){\color[rgb]{0.4,0.4,0.4}\makebox(0,0)[lt]{\lineheight{1.25}\smash{\begin{tabular}[t]{l}$A_3$\end{tabular}}}}%
    \put(0,0){\includegraphics[width=\unitlength,page=4]{b8p2-affine.pdf}}%
    \put(0.86995677,0.91075749){\color[rgb]{0,0,0}\makebox(0,0)[lt]{\lineheight{1.25}\smash{\begin{tabular}[t]{l}$\muout$\end{tabular}}}}%
    \put(0.86755576,0.77933815){\color[rgb]{0.50196078,0.2,0}\makebox(0,0)[lt]{\lineheight{1.25}\smash{\begin{tabular}[t]{l}$(3,2)$\end{tabular}}}}%
    \put(0.9140625,0.17898521){\color[rgb]{0.50196078,0.2,0}\makebox(0,0)[lt]{\lineheight{1.25}\smash{\begin{tabular}[t]{l}$(3,-1)$\end{tabular}}}}%
    \put(-0.00155702,0.20961699){\color[rgb]{0.50196078,0.2,0}\makebox(0,0)[lt]{\lineheight{1.25}\smash{\begin{tabular}[t]{l}$(-6,-1)$\end{tabular}}}}%
  \end{picture}%
\endgroup%

%% file: b4p2dual.pdf_tex
\begingroup%
  \makeatletter%
  \providecommand\color[2][]{%
    \errmessage{(Inkscape) Color is used for the text in Inkscape, but the package 'color.sty' is not loaded}%
    \renewcommand\color[2][]{}%
  }%
  \providecommand\transparent[1]{%
    \errmessage{(Inkscape) Transparency is used (non-zero) for the text in Inkscape, but the package 'transparent.sty' is not loaded}%
    \renewcommand\transparent[1]{}%
  }%
  \providecommand\rotatebox[2]{#2}%
  \newcommand*\fsize{\dimexpr\f@size pt\relax}%
  \newcommand*\lineheight[1]{\fontsize{\fsize}{#1\fsize}\selectfont}%
  \ifx\svgwidth\undefined%
    \setlength{\unitlength}{486.58144644bp}%
    \ifx\svgscale\undefined%
      \relax%
    \else%
      \setlength{\unitlength}{\unitlength * \real{\svgscale}}%
    \fi%
  \else%
    \setlength{\unitlength}{\svgwidth}%
  \fi%
  \global\let\svgwidth\undefined%
  \global\let\svgscale\undefined%
  \makeatother%
  \begin{picture}(1,0.49167466)%
    \lineheight{1}%
    \setlength\tabcolsep{0pt}%
    \put(0,0){\includegraphics[width=\unitlength,page=1]{b4p2dual.pdf}}%
    \put(0.47775726,0.17610487){\color[rgb]{0,0,0}\makebox(0,0)[lt]{\lineheight{1.25}\smash{\begin{tabular}[t]{l}shear\end{tabular}}}}%
    \put(0.65108989,0.01799111){\color[rgb]{0,0,0}\makebox(0,0)[lt]{\lineheight{1.25}\smash{\begin{tabular}[t]{l}shear\end{tabular}}}}%
    \put(0,0){\includegraphics[width=\unitlength,page=2]{b4p2dual.pdf}}%
    \put(0.62004522,0.15843316){\color[rgb]{0,0,0}\makebox(0,0)[lt]{\lineheight{1.25}\smash{\begin{tabular}[t]{l}$P_1^\dual$\end{tabular}}}}%
    \put(0.57786626,0.12589225){\color[rgb]{0,0,0}\makebox(0,0)[lt]{\lineheight{1.25}\smash{\begin{tabular}[t]{l}$Q_1^\dual$\end{tabular}}}}%
    \put(0.64506833,0.12559205){\color[rgb]{0,0,0}\makebox(0,0)[lt]{\lineheight{1.25}\smash{\begin{tabular}[t]{l}$P_{10}^\dual$\end{tabular}}}}%
    \put(0,0){\includegraphics[width=\unitlength,page=3]{b4p2dual.pdf}}%
    \put(0.61477241,0.08127259){\color[rgb]{0,0,0}\makebox(0,0)[lt]{\lineheight{1.25}\smash{\begin{tabular}[t]{l}$Q_{10}^\dual$\end{tabular}}}}%
    \put(0.57602796,0.23669091){\color[rgb]{0,0,0}\makebox(0,0)[lt]{\lineheight{1.25}\smash{\begin{tabular}[t]{l}$Q_2^\dual$\end{tabular}}}}%
    \put(0.5917296,0.28709998){\color[rgb]{0,0,0}\makebox(0,0)[lt]{\lineheight{1.25}\smash{\begin{tabular}[t]{l}$Q_3^\dual$\end{tabular}}}}%
    \put(0.65726513,0.34553911){\color[rgb]{0,0,0}\makebox(0,0)[lt]{\lineheight{1.25}\smash{\begin{tabular}[t]{l}$Q_4^\dual$\end{tabular}}}}%
    \put(0.62303383,0.36274913){\color[rgb]{0,0,0}\makebox(0,0)[lt]{\lineheight{1.25}\smash{\begin{tabular}[t]{l}$R_x^\dual$\end{tabular}}}}%
    \put(0.81892651,0.29067209){\color[rgb]{0,0,0}\makebox(0,0)[lt]{\lineheight{1.25}\smash{\begin{tabular}[t]{l}$R_y^\dual$\end{tabular}}}}%
    \put(0.83581504,0.12743213){\color[rgb]{0,0,0}\makebox(0,0)[lt]{\lineheight{1.25}\smash{\begin{tabular}[t]{l}$R_z^\dual$\end{tabular}}}}%
    \put(0,0){\includegraphics[width=\unitlength,page=4]{b4p2dual.pdf}}%
    \put(0.12548001,0.22924323){\color[rgb]{0.50196078,0.2,0}\makebox(0,0)[lt]{\lineheight{1.25}\smash{\begin{tabular}[t]{l}$P_2$\end{tabular}}}}%
    \put(0.11359393,0.19029406){\color[rgb]{0.50196078,0.2,0}\makebox(0,0)[lt]{\lineheight{1.25}\smash{\begin{tabular}[t]{l}$P_1$\end{tabular}}}}%
    \put(0.17086558,0.13913287){\color[rgb]{0.50196078,0.2,0}\makebox(0,0)[lt]{\lineheight{1.25}\smash{\begin{tabular}[t]{l}$P_{10}$\end{tabular}}}}%
    \put(0.21809272,0.12814109){\color[rgb]{0.50196078,0.2,0}\makebox(0,0)[lt]{\lineheight{1.25}\smash{\begin{tabular}[t]{l}$P_9$\end{tabular}}}}%
    \put(0.236076,0.31688646){\color[rgb]{0.50196078,0.2,0}\makebox(0,0)[lt]{\lineheight{1.25}\smash{\begin{tabular}[t]{l}$P_4$\end{tabular}}}}%
    \put(0.19895514,0.30406863){\color[rgb]{0.50196078,0.2,0}\makebox(0,0)[lt]{\lineheight{1.25}\smash{\begin{tabular}[t]{l}$P_3$\end{tabular}}}}%
    \put(0.01676073,0.24799307){\color[rgb]{0.50196078,0.2,0}\makebox(0,0)[lt]{\lineheight{1.25}\smash{\begin{tabular}[t]{l}$Q_2$\end{tabular}}}}%
    \put(-0.00138118,0.16085644){\color[rgb]{0.50196078,0.2,0}\makebox(0,0)[lt]{\lineheight{1.25}\smash{\begin{tabular}[t]{l}$Q_1$\end{tabular}}}}%
    \put(0.14569212,0.36920744){\color[rgb]{0.50196078,0.2,0}\makebox(0,0)[lt]{\lineheight{1.25}\smash{\begin{tabular}[t]{l}$Q_3$\end{tabular}}}}%
    \put(0.24578728,0.39170278){\color[rgb]{0.50196078,0.2,0}\makebox(0,0)[lt]{\lineheight{1.25}\smash{\begin{tabular}[t]{l}$Q_4$\end{tabular}}}}%
    \put(0.18663994,0.43434533){\color[rgb]{0.78431373,0.21568627,0.21568627}\makebox(0,0)[lt]{\lineheight{1.25}\smash{\begin{tabular}[t]{l}$R_x$\end{tabular}}}}%
    \put(0.40470806,0.40712686){\color[rgb]{0.78431373,0.21568627,0.21568627}\makebox(0,0)[lt]{\lineheight{1.25}\smash{\begin{tabular}[t]{l}$R_y$\end{tabular}}}}%
    \put(0.40571466,0.2139363){\color[rgb]{0.78431373,0.21568627,0.21568627}\makebox(0,0)[lt]{\lineheight{1.25}\smash{\begin{tabular}[t]{l}$R_z$\end{tabular}}}}%
    \put(0,0){\includegraphics[width=\unitlength,page=5]{b4p2dual.pdf}}%
    \put(0.72183186,0.45432779){\color[rgb]{0.50196078,0.2,0}\makebox(0,0)[lt]{\lineheight{1.25}\smash{\begin{tabular}[t]{l}$\mu_{out}$\end{tabular}}}}%
    \put(0.88251896,0.04256046){\color[rgb]{0.50196078,0.2,0}\makebox(0,0)[lt]{\lineheight{1.25}\smash{\begin{tabular}[t]{l}$\mu_{out}$\end{tabular}}}}%
  \end{picture}%
\endgroup%

%% file: ellip.pdf_tex
\begingroup%
  \makeatletter%
  \providecommand\color[2][]{%
    \errmessage{(Inkscape) Color is used for the text in Inkscape, but the package 'color.sty' is not loaded}%
    \renewcommand\color[2][]{}%
  }%
  \providecommand\transparent[1]{%
    \errmessage{(Inkscape) Transparency is used (non-zero) for the text in Inkscape, but the package 'transparent.sty' is not loaded}%
    \renewcommand\transparent[1]{}%
  }%
  \providecommand\rotatebox[2]{#2}%
  \newcommand*\fsize{\dimexpr\f@size pt\relax}%
  \newcommand*\lineheight[1]{\fontsize{\fsize}{#1\fsize}\selectfont}%
  \ifx\svgwidth\undefined%
    \setlength{\unitlength}{369.63926916bp}%
    \ifx\svgscale\undefined%
      \relax%
    \else%
      \setlength{\unitlength}{\unitlength * \real{\svgscale}}%
    \fi%
  \else%
    \setlength{\unitlength}{\svgwidth}%
  \fi%
  \global\let\svgwidth\undefined%
  \global\let\svgscale\undefined%
  \makeatother%
  \begin{picture}(1,0.32676287)%
    \lineheight{1}%
    \setlength\tabcolsep{0pt}%
    \put(0,0){\includegraphics[width=\unitlength,page=1]{ellip.pdf}}%
    \put(0.54232495,0.13368971){\color[rgb]{0,0,0}\makebox(0,0)[lt]{\lineheight{1.25}\smash{\begin{tabular}[t]{l}$R_x^\dual$\end{tabular}}}}%
    \put(0,0){\includegraphics[width=\unitlength,page=2]{ellip.pdf}}%
    \put(0.25236184,0.30576041){\color[rgb]{0,0,0}\makebox(0,0)[lt]{\lineheight{1.25}\smash{\begin{tabular}[t]{l}$x$\end{tabular}}}}%
    \put(0.12163319,0.16358561){\color[rgb]{0,0,0}\makebox(0,0)[lt]{\lineheight{1.25}\smash{\begin{tabular}[t]{l}$\Phi(X)$\end{tabular}}}}%
    \put(0.19882296,0.22431216){\color[rgb]{0.78431373,0.21568627,0.21568627}\makebox(0,0)[lt]{\lineheight{1.25}\smash{\begin{tabular}[t]{l}$R_x$\end{tabular}}}}%
    \put(0.0244837,0.17223713){\color[rgb]{0.50196078,0.2,0}\makebox(0,0)[lt]{\lineheight{1.25}\smash{\begin{tabular}[t]{l}$Q_2$\end{tabular}}}}%
    \put(0.14290704,0.03356402){\color[rgb]{0.50196078,0.2,0}\makebox(0,0)[lt]{\lineheight{1.25}\smash{\begin{tabular}[t]{l}$Q_1$\end{tabular}}}}%
    \put(0,0){\includegraphics[width=\unitlength,page=3]{ellip.pdf}}%
    \put(0.72351926,0.17617){\color[rgb]{0,0,0}\makebox(0,0)[lt]{\lineheight{1.25}\smash{\begin{tabular}[t]{l}$\mu_{out}$\end{tabular}}}}%
    \put(0.89563912,0.01095779){\color[rgb]{0,0,0}\makebox(0,0)[lt]{\lineheight{1.25}\smash{\begin{tabular}[t]{l}$\mu_{out}$\end{tabular}}}}%
    \put(0,0){\includegraphics[width=\unitlength,page=4]{ellip.pdf}}%
    \put(0.85015763,0.17169591){\color[rgb]{0,0,0}\makebox(0,0)[lt]{\lineheight{1.25}\smash{\begin{tabular}[t]{l}$\mu^{(2)}_{out}$\end{tabular}}}}%
    \put(0,0){\includegraphics[width=\unitlength,page=5]{ellip.pdf}}%
    \put(0.89316027,0.08924363){\color[rgb]{0,0,0}\makebox(0,0)[lt]{\lineheight{1.25}\smash{\begin{tabular}[t]{l}$\mu^{(1)}_{out}$\end{tabular}}}}%
    \put(0,0){\includegraphics[width=\unitlength,page=6]{ellip.pdf}}%
    \put(0.43349455,0.18042713){\color[rgb]{0,0,0}\makebox(0,0)[lt]{\lineheight{1.25}\smash{\begin{tabular}[t]{l}$Q_2^\dual$\end{tabular}}}}%
    \put(0.57187202,0.03385728){\color[rgb]{0,0,0}\makebox(0,0)[lt]{\lineheight{1.25}\smash{\begin{tabular}[t]{l}$Q_1^\dual$\end{tabular}}}}%
  \end{picture}%
\endgroup%

%% file: divcurve.pdf_tex
\begingroup%
  \makeatletter%
  \providecommand\color[2][]{%
    \errmessage{(Inkscape) Color is used for the text in Inkscape, but the package 'color.sty' is not loaded}%
    \renewcommand\color[2][]{}%
  }%
  \providecommand\transparent[1]{%
    \errmessage{(Inkscape) Transparency is used (non-zero) for the text in Inkscape, but the package 'transparent.sty' is not loaded}%
    \renewcommand\transparent[1]{}%
  }%
  \providecommand\rotatebox[2]{#2}%
  \newcommand*\fsize{\dimexpr\f@size pt\relax}%
  \newcommand*\lineheight[1]{\fontsize{\fsize}{#1\fsize}\selectfont}%
  \ifx\svgwidth\undefined%
    \setlength{\unitlength}{372.26005608bp}%
    \ifx\svgscale\undefined%
      \relax%
    \else%
      \setlength{\unitlength}{\unitlength * \real{\svgscale}}%
    \fi%
  \else%
    \setlength{\unitlength}{\svgwidth}%
  \fi%
  \global\let\svgwidth\undefined%
  \global\let\svgscale\undefined%
  \makeatother%
  \begin{picture}(1,0.53571076)%
    \lineheight{1}%
    \setlength\tabcolsep{0pt}%
    \put(0,0){\includegraphics[width=\unitlength,page=1]{divcurve.pdf}}%
    \put(0.91288381,0.39324887){\color[rgb]{0,0,1}\makebox(0,0)[lt]{\lineheight{1.25}\smash{\begin{tabular}[t]{l}$(3,2)$\end{tabular}}}}%
    \put(0,0){\includegraphics[width=\unitlength,page=2]{divcurve.pdf}}%
    \put(0.85141785,0.2619827){\makebox(0,0)[lt]{\lineheight{1.25}\smash{\begin{tabular}[t]{l}$v_1$\end{tabular}}}}%
    \put(0.77653325,0.17176757){\makebox(0,0)[lt]{\lineheight{1.25}\smash{\begin{tabular}[t]{l}$v_0$\end{tabular}}}}%
    \put(0.96674283,0.50547569){\makebox(0,0)[lt]{\lineheight{1.25}\smash{\begin{tabular}[t]{l}$v_2$\end{tabular}}}}%
    \put(0.22147208,0.27214858){\makebox(0,0)[lt]{\lineheight{1.25}\smash{\begin{tabular}[t]{l}$u_0$\end{tabular}}}}%
    \put(0.24415935,0.32891117){\makebox(0,0)[lt]{\lineheight{1.25}\smash{\begin{tabular}[t]{l}$u_1$\end{tabular}}}}%
    \put(0.24676808,0.368493){\makebox(0,0)[lt]{\lineheight{1.25}\smash{\begin{tabular}[t]{l}$u_2$\end{tabular}}}}%
  \end{picture}%
\endgroup%

%% file: vertcurve.pdf_tex
\begingroup%
  \makeatletter%
  \providecommand\color[2][]{%
    \errmessage{(Inkscape) Color is used for the text in Inkscape, but the package 'color.sty' is not loaded}%
    \renewcommand\color[2][]{}%
  }%
  \providecommand\transparent[1]{%
    \errmessage{(Inkscape) Transparency is used (non-zero) for the text in Inkscape, but the package 'transparent.sty' is not loaded}%
    \renewcommand\transparent[1]{}%
  }%
  \providecommand\rotatebox[2]{#2}%
  \newcommand*\fsize{\dimexpr\f@size pt\relax}%
  \newcommand*\lineheight[1]{\fontsize{\fsize}{#1\fsize}\selectfont}%
  \ifx\svgwidth\undefined%
    \setlength{\unitlength}{416.5126384bp}%
    \ifx\svgscale\undefined%
      \relax%
    \else%
      \setlength{\unitlength}{\unitlength * \real{\svgscale}}%
    \fi%
  \else%
    \setlength{\unitlength}{\svgwidth}%
  \fi%
  \global\let\svgwidth\undefined%
  \global\let\svgscale\undefined%
  \makeatother%
  \begin{picture}(1,0.64176751)%
    \lineheight{1}%
    \setlength\tabcolsep{0pt}%
    \put(0,0){\includegraphics[width=\unitlength,page=1]{vertcurve.pdf}}%
    \put(0.51244715,0.48508174){\makebox(0,0)[lt]{\lineheight{1.25}\smash{\begin{tabular}[t]{l}$(3,0)$\end{tabular}}}}%
    \put(0.48023735,0.45865557){\makebox(0,0)[lt]{\lineheight{1.25}\smash{\begin{tabular}[t]{l}L\end{tabular}}}}%
    \put(0,0){\includegraphics[width=\unitlength,page=2]{vertcurve.pdf}}%
    \put(0.6669349,0.4243126){\makebox(0,0)[lt]{\lineheight{1.25}\smash{\begin{tabular}[t]{l}$(3,-1)$\end{tabular}}}}%
    \put(0.53073929,0.45504358){\makebox(0,0)[lt]{\lineheight{1.25}\smash{\begin{tabular}[t]{l}$u_0$\end{tabular}}}}%
    \put(0.5956377,0.45738989){\makebox(0,0)[lt]{\lineheight{1.25}\smash{\begin{tabular}[t]{l}$u_1$\end{tabular}}}}%
    \put(0.62999672,0.43273394){\makebox(0,0)[lt]{\lineheight{1.25}\smash{\begin{tabular}[t]{l}$u_3$\end{tabular}}}}%
    \put(0,0){\includegraphics[width=\unitlength,page=3]{vertcurve.pdf}}%
    \put(0.18796302,0.23485891){\makebox(0,0)[lt]{\lineheight{1.25}\smash{\begin{tabular}[t]{l}$u_0$\end{tabular}}}}%
    \put(0,0){\includegraphics[width=\unitlength,page=4]{vertcurve.pdf}}%
    \put(0.27809854,0.20031025){\makebox(0,0)[lt]{\lineheight{1.25}\smash{\begin{tabular}[t]{l}$u_1$\end{tabular}}}}%
    \put(0,0){\includegraphics[width=\unitlength,page=5]{vertcurve.pdf}}%
    \put(0.3085166,0.23776515){\makebox(0,0)[lt]{\lineheight{1.25}\smash{\begin{tabular}[t]{l}$u_2$\end{tabular}}}}%
    \put(0.33460541,0.1920514){\makebox(0,0)[lt]{\lineheight{1.25}\smash{\begin{tabular}[t]{l}$u_3$\end{tabular}}}}%
    \put(0,0){\includegraphics[width=\unitlength,page=6]{vertcurve.pdf}}%
    \put(0.8201058,0.13880847){\makebox(0,0)[lt]{\lineheight{1.25}\smash{\begin{tabular}[t]{l}$v_1$\end{tabular}}}}%
    \put(0.73652843,0.19047504){\makebox(0,0)[lt]{\lineheight{1.25}\smash{\begin{tabular}[t]{l}$v_0$\end{tabular}}}}%
    \put(0.82643209,0.19255709){\makebox(0,0)[lt]{\lineheight{1.25}\smash{\begin{tabular}[t]{l}$v_2$\end{tabular}}}}%
    \put(0.68944598,0.03322339){\color[rgb]{0,0,1}\makebox(0,0)[lt]{\lineheight{1.25}\smash{\begin{tabular}[t]{l}{\tiny (3,0)}\end{tabular}}}}%
    \put(0,0){\includegraphics[width=\unitlength,page=7]{vertcurve.pdf}}%
    \put(0.94428964,0.04459066){\makebox(0,0)[lt]{\lineheight{1.25}\smash{\begin{tabular}[t]{l}$v_3$\end{tabular}}}}%
    \put(0.88718231,0.10779566){\color[rgb]{0,0,1}\makebox(0,0)[lt]{\lineheight{1.25}\smash{\begin{tabular}[t]{l}{\tiny (3,-1)}\end{tabular}}}}%
    \put(0.87371916,0.18919375){\color[rgb]{0,0,1}\makebox(0,0)[lt]{\lineheight{1.25}\smash{\begin{tabular}[t]{l}{\tiny (0,-1)}\end{tabular}}}}%
    \put(0,0){\includegraphics[width=\unitlength,page=8]{vertcurve.pdf}}%
    \put(0.62779008,0.47531118){\makebox(0,0)[lt]{\lineheight{1.25}\smash{\begin{tabular}[t]{l}$u_2$\end{tabular}}}}%
    \put(0,0){\includegraphics[width=\unitlength,page=9]{vertcurve.pdf}}%
  \end{picture}%
\endgroup%

%% file: dp1full.pdf_tex
\begingroup%
  \makeatletter%
  \providecommand\color[2][]{%
    \errmessage{(Inkscape) Color is used for the text in Inkscape, but the package 'color.sty' is not loaded}%
    \renewcommand\color[2][]{}%
  }%
  \providecommand\transparent[1]{%
    \errmessage{(Inkscape) Transparency is used (non-zero) for the text in Inkscape, but the package 'transparent.sty' is not loaded}%
    \renewcommand\transparent[1]{}%
  }%
  \providecommand\rotatebox[2]{#2}%
  \newcommand*\fsize{\dimexpr\f@size pt\relax}%
  \newcommand*\lineheight[1]{\fontsize{\fsize}{#1\fsize}\selectfont}%
  \ifx\svgwidth\undefined%
    \setlength{\unitlength}{431.08131788bp}%
    \ifx\svgscale\undefined%
      \relax%
    \else%
      \setlength{\unitlength}{\unitlength * \real{\svgscale}}%
    \fi%
  \else%
    \setlength{\unitlength}{\svgwidth}%
  \fi%
  \global\let\svgwidth\undefined%
  \global\let\svgscale\undefined%
  \makeatother%
  \begin{picture}(1,0.64929278)%
    \lineheight{1}%
    \setlength\tabcolsep{0pt}%
    \put(0,0){\includegraphics[width=\unitlength,page=1]{dp1full.pdf}}%
    \put(0.2062042,0.23027708){\makebox(0,0)[lt]{\lineheight{1.25}\smash{\begin{tabular}[t]{l}$u_0$\end{tabular}}}}%
    \put(0.26974097,0.18484059){\makebox(0,0)[lt]{\lineheight{1.25}\smash{\begin{tabular}[t]{l}$u_1$\end{tabular}}}}%
    \put(0.34154013,0.22365761){\makebox(0,0)[lt]{\lineheight{1.25}\smash{\begin{tabular}[t]{l}$u_2$\end{tabular}}}}%
    \put(0.23976635,0.2998991){\makebox(0,0)[lt]{\lineheight{1.25}\smash{\begin{tabular}[t]{l}$u_3$\end{tabular}}}}%
    \put(0.08001001,0.34283134){\makebox(0,0)[lt]{\lineheight{1.25}\smash{\begin{tabular}[t]{l}$u_4$\end{tabular}}}}%
    \put(-0.00063909,0.36857781){\makebox(0,0)[lt]{\lineheight{1.25}\smash{\begin{tabular}[t]{l}$u_6$\end{tabular}}}}%
    \put(0.0845073,0.37803677){\makebox(0,0)[lt]{\lineheight{1.25}\smash{\begin{tabular}[t]{l}$u_5$\end{tabular}}}}%
    \put(0.17566693,0.10101327){\makebox(0,0)[lt]{\lineheight{1.25}\smash{\begin{tabular}[t]{l}$u_8$\end{tabular}}}}%
    \put(0.14029611,0.19389448){\makebox(0,0)[lt]{\lineheight{1.25}\smash{\begin{tabular}[t]{l}$u_7$\end{tabular}}}}%
    \put(0,0){\includegraphics[width=\unitlength,page=2]{dp1full.pdf}}%
    \put(0.51273648,0.47183244){\makebox(0,0)[lt]{\lineheight{1.25}\smash{\begin{tabular}[t]{l}$(1,0)$\end{tabular}}}}%
    \put(0.57944935,0.51262229){\makebox(0,0)[lt]{\lineheight{1.25}\smash{\begin{tabular}[t]{l}$(1,1)$\end{tabular}}}}%
    \put(0.48263349,0.2913954){\makebox(0,0)[lt]{\lineheight{1.25}\smash{\begin{tabular}[t]{l}$(-1,0)$\end{tabular}}}}%
    \put(0.33034464,0.63471791){\makebox(0,0)[lt]{\lineheight{1.25}\smash{\begin{tabular}[t]{l}$(2,1)$\end{tabular}}}}%
    \put(0.39643951,0.61603874){\makebox(0,0)[lt]{\lineheight{1.25}\smash{\begin{tabular}[t]{l}$(3,2)$\end{tabular}}}}%
    \put(0.47586552,0.44563063){\makebox(0,0)[lt]{\lineheight{1.25}\smash{\begin{tabular}[t]{l}$L$\end{tabular}}}}%
    \put(0,0){\includegraphics[width=\unitlength,page=3]{dp1full.pdf}}%
    \put(0.95157395,0.59371016){\makebox(0,0)[lt]{\lineheight{1.25}\smash{\begin{tabular}[t]{l}$(3,2)$\end{tabular}}}}%
    \put(0,0){\includegraphics[width=\unitlength,page=4]{dp1full.pdf}}%
    \put(0.62142196,0.20062875){\color[rgb]{0,0,1}\transparent{0.99645102}\rotatebox{63.541492}{\makebox(0,0)[lt]{\lineheight{1.25}\smash{\begin{tabular}[t]{l}$\simeq$\end{tabular}}}}}%
    \put(0,0){\includegraphics[width=\unitlength,page=5]{dp1full.pdf}}%
    \put(0.82702589,0.19765634){\color[rgb]{0,0,1}\transparent{0.99645102}\makebox(0,0)[lt]{\lineheight{1.25}\smash{\begin{tabular}[t]{l}$\simeq$\end{tabular}}}}%
    \put(0.78237829,0.15718744){\makebox(0,0)[lt]{\lineheight{1.25}\smash{\begin{tabular}[t]{l}$v_1$\end{tabular}}}}%
    \put(0.72787132,0.18505956){\makebox(0,0)[lt]{\lineheight{1.25}\smash{\begin{tabular}[t]{l}$v_0$\end{tabular}}}}%
    \put(0.82365988,0.15335653){\makebox(0,0)[lt]{\lineheight{1.25}\smash{\begin{tabular}[t]{l}$v_2$\end{tabular}}}}%
    \put(0.87732894,0.43439205){\makebox(0,0)[lt]{\lineheight{1.25}\smash{\begin{tabular}[t]{l}$v_3$\end{tabular}}}}%
    \put(0.85445796,0.49687444){\makebox(0,0)[lt]{\lineheight{1.25}\smash{\begin{tabular}[t]{l}$v_4$\end{tabular}}}}%
    \put(0.718511,0.29517462){\makebox(0,0)[lt]{\lineheight{1.25}\smash{\begin{tabular}[t]{l}$v_6$\end{tabular}}}}%
    \put(0.58626913,0.13637695){\makebox(0,0)[lt]{\lineheight{1.25}\smash{\begin{tabular}[t]{l}$v_6$\end{tabular}}}}%
    \put(0.64627136,0.10100139){\makebox(0,0)[lt]{\lineheight{1.25}\smash{\begin{tabular}[t]{l}$v_7$\end{tabular}}}}%
    \put(0.89905276,0.56615863){\makebox(0,0)[lt]{\lineheight{1.25}\smash{\begin{tabular}[t]{l}$v_5$\end{tabular}}}}%
    \put(0.71966841,0.09795493){\makebox(0,0)[lt]{\lineheight{1.25}\smash{\begin{tabular}[t]{l}$v_8$\end{tabular}}}}%
    \put(0.6702187,0.12946539){\color[rgb]{0,0,1}\makebox(0,0)[lt]{\lineheight{1.25}\smash{\begin{tabular}[t]{l}{\tiny (-1,0)}\end{tabular}}}}%
    \put(0.76222787,0.18355137){\color[rgb]{0,0,1}\makebox(0,0)[lt]{\lineheight{1.25}\smash{\begin{tabular}[t]{l}{\tiny (1,0)}\end{tabular}}}}%
    \put(0.79558445,0.3704055){\color[rgb]{0,0,1}\rotatebox{59.456533}{\makebox(0,0)[lt]{\lineheight{1.25}\smash{\begin{tabular}[t]{l}{\tiny (2,1)}\end{tabular}}}}}%
    \put(0.86270633,0.34947105){\color[rgb]{0,0,1}\rotatebox{78.60363}{\makebox(0,0)[lt]{\lineheight{1.25}\smash{\begin{tabular}[t]{l}{\tiny (1,1)}\end{tabular}}}}}%
  \end{picture}%
\endgroup%

%% file: corner.pdf_tex
\begingroup%
  \makeatletter%
  \providecommand\color[2][]{%
    \errmessage{(Inkscape) Color is used for the text in Inkscape, but the package 'color.sty' is not loaded}%
    \renewcommand\color[2][]{}%
  }%
  \providecommand\transparent[1]{%
    \errmessage{(Inkscape) Transparency is used (non-zero) for the text in Inkscape, but the package 'transparent.sty' is not loaded}%
    \renewcommand\transparent[1]{}%
  }%
  \providecommand\rotatebox[2]{#2}%
  \newcommand*\fsize{\dimexpr\f@size pt\relax}%
  \newcommand*\lineheight[1]{\fontsize{\fsize}{#1\fsize}\selectfont}%
  \ifx\svgwidth\undefined%
    \setlength{\unitlength}{387.13999625bp}%
    \ifx\svgscale\undefined%
      \relax%
    \else%
      \setlength{\unitlength}{\unitlength * \real{\svgscale}}%
    \fi%
  \else%
    \setlength{\unitlength}{\svgwidth}%
  \fi%
  \global\let\svgwidth\undefined%
  \global\let\svgscale\undefined%
  \makeatother%
  \begin{picture}(1,0.32360934)%
    \lineheight{1}%
    \setlength\tabcolsep{0pt}%
    \put(0,0){\includegraphics[width=\unitlength,page=1]{corner.pdf}}%
    \put(0.14488597,0.11681101){\color[rgb]{0,0,1}\makebox(0,0)[lt]{\lineheight{1.25}\smash{\begin{tabular}[t]{l}$w_{out}$\end{tabular}}}}%
    \put(0.21370588,0.13211172){\color[rgb]{0,0,0}\makebox(0,0)[lt]{\lineheight{1.25}\smash{\begin{tabular}[t]{l}$R_x^\dual$\end{tabular}}}}%
    \put(0.97007719,0.30341189){\color[rgb]{0,0,0}\makebox(0,0)[lt]{\lineheight{1.25}\smash{\begin{tabular}[t]{l}$x$\end{tabular}}}}%
    \put(0,0){\includegraphics[width=\unitlength,page=2]{corner.pdf}}%
    \put(0.58389381,0.11165229){\color[rgb]{0,0,1}\makebox(0,0)[lt]{\lineheight{1.25}\smash{\begin{tabular}[t]{l}$w_{out}$\end{tabular}}}}%
    \put(0,0){\includegraphics[width=\unitlength,page=3]{corner.pdf}}%
    \put(0.57287011,0.15775413){\color[rgb]{0,0,1}\makebox(0,0)[lt]{\lineheight{1.25}\smash{\begin{tabular}[t]{l}$w_0$\end{tabular}}}}%
    \put(0.66186475,0.15231244){\color[rgb]{0,0,1}\makebox(0,0)[lt]{\lineheight{1.25}\smash{\begin{tabular}[t]{l}$w_1$\end{tabular}}}}%
    \put(0.59488543,0.2145585){\color[rgb]{0,0,1}\makebox(0,0)[lt]{\lineheight{1.25}\smash{\begin{tabular}[t]{l}$w_2$\end{tabular}}}}%
    \put(0.92651233,0.26587175){\color[rgb]{0.50196078,0.2,0}\makebox(0,0)[lt]{\lineheight{1.25}\smash{\begin{tabular}[t]{l}$u_{w_2}$\end{tabular}}}}%
    \put(0.92651233,0.19227096){\color[rgb]{0.50196078,0.2,0}\makebox(0,0)[lt]{\lineheight{1.25}\smash{\begin{tabular}[t]{l}$u_{w_1}$\end{tabular}}}}%
    \put(0.83003066,0.17512247){\color[rgb]{0,0,1}\makebox(0,0)[lt]{\lineheight{1.25}\smash{\begin{tabular}[t]{l}$u_{w_0}$, \\$u_{w_{out}}$\end{tabular}}}}%
    \put(0,0){\includegraphics[width=\unitlength,page=4]{corner.pdf}}%
    \put(0.11004223,0.00784117){\color[rgb]{0,0,1}\makebox(0,0)[lt]{\lineheight{1.25}\smash{\begin{tabular}[t]{l}$\Gamma_0$\end{tabular}}}}%
    \put(0,0){\includegraphics[width=\unitlength,page=5]{corner.pdf}}%
    \put(0.29706156,0.13462527){\color[rgb]{0,0,0}\makebox(0,0)[lt]{\lineheight{1.25}\smash{\begin{tabular}[t]{l}Augment\end{tabular}}}}%
    \put(0.1376656,0.17584561){\color[rgb]{0,0,1}\makebox(0,0)[lt]{\lineheight{1.25}\smash{\begin{tabular}[t]{l}$e_{out}$\end{tabular}}}}%
  \end{picture}%
\endgroup%

%% file: ff2ell.pdf_tex
\begingroup%
  \makeatletter%
  \providecommand\color[2][]{%
    \errmessage{(Inkscape) Color is used for the text in Inkscape, but the package 'color.sty' is not loaded}%
    \renewcommand\color[2][]{}%
  }%
  \providecommand\transparent[1]{%
    \errmessage{(Inkscape) Transparency is used (non-zero) for the text in Inkscape, but the package 'transparent.sty' is not loaded}%
    \renewcommand\transparent[1]{}%
  }%
  \providecommand\rotatebox[2]{#2}%
  \newcommand*\fsize{\dimexpr\f@size pt\relax}%
  \newcommand*\lineheight[1]{\fontsize{\fsize}{#1\fsize}\selectfont}%
  \ifx\svgwidth\undefined%
    \setlength{\unitlength}{517.67169128bp}%
    \ifx\svgscale\undefined%
      \relax%
    \else%
      \setlength{\unitlength}{\unitlength * \real{\svgscale}}%
    \fi%
  \else%
    \setlength{\unitlength}{\svgwidth}%
  \fi%
  \global\let\svgwidth\undefined%
  \global\let\svgscale\undefined%
  \makeatother%
  \begin{picture}(1,0.63830248)%
    \lineheight{1}%
    \setlength\tabcolsep{0pt}%
    \put(0,0){\includegraphics[width=\unitlength,page=1]{ff2ell.pdf}}%
    \put(0.25280496,0.07604676){\color[rgb]{0.50196078,0.2,0}\makebox(0,0)[lt]{\lineheight{1.25}\smash{\begin{tabular}[t]{l}$P_1$\end{tabular}}}}%
    \put(0.31153087,0.10884692){\color[rgb]{0.50196078,0.2,0}\makebox(0,0)[lt]{\lineheight{1.25}\smash{\begin{tabular}[t]{l}$P_2$\end{tabular}}}}%
    \put(0.13362769,0.08362765){\color[rgb]{0.50196078,0.2,0}\makebox(0,0)[lt]{\lineheight{1.25}\smash{\begin{tabular}[t]{l}$Q_1$\end{tabular}}}}%
    \put(0.21820072,0.14381954){\color[rgb]{0.50196078,0.2,0}\makebox(0,0)[lt]{\lineheight{1.25}\smash{\begin{tabular}[t]{l}$Q_2$\end{tabular}}}}%
    \put(0.49910249,0.37171688){\color[rgb]{0,0,1}\makebox(0,0)[lt]{\lineheight{1.25}\smash{\begin{tabular}[t]{l}$(-1,1)$\end{tabular}}}}%
    \put(0.0800613,0.13157841){\color[rgb]{0.78431373,0.21568627,0.21568627}\makebox(0,0)[lt]{\lineheight{1.25}\smash{\begin{tabular}[t]{l}$R_1$\end{tabular}}}}%
    \put(0.15136626,0.20173472){\color[rgb]{0.10196078,0.10196078,0.10196078}\makebox(0,0)[lt]{\lineheight{1.25}\smash{\begin{tabular}[t]{l}$-x+2y=1$\end{tabular}}}}%
    \put(-0.00095271,0.07982653){\color[rgb]{0.10196078,0.10196078,0.10196078}\makebox(0,0)[lt]{\lineheight{1.25}\smash{\begin{tabular}[t]{l}$-x+y=1$\end{tabular}}}}%
    \put(0.40770108,0.23706719){\color[rgb]{0.10196078,0.10196078,0.10196078}\makebox(0,0)[lt]{\lineheight{1.25}\smash{\begin{tabular}[t]{l}$x+2y=1$\end{tabular}}}}%
    \put(0.56401644,0.31367303){\color[rgb]{0,0,0}\makebox(0,0)[lt]{\lineheight{1.25}\smash{\begin{tabular}[t]{l}$R_1^\dual$\end{tabular}}}}%
    \put(0.59171049,0.22755592){\color[rgb]{0,0,0}\makebox(0,0)[lt]{\lineheight{1.25}\smash{\begin{tabular}[t]{l}$Q_1^\dual$\end{tabular}}}}%
    \put(0.75407924,0.06880368){\color[rgb]{0,0,0}\makebox(0,0)[lt]{\lineheight{1.25}\smash{\begin{tabular}[t]{l}$P_1^\dual$\end{tabular}}}}%
    \put(0.74435812,0.1142422){\color[rgb]{0,0,0}\makebox(0,0)[lt]{\lineheight{1.25}\smash{\begin{tabular}[t]{l}$P_2^\dual$\end{tabular}}}}%
    \put(0.59249329,0.42249191){\color[rgb]{0,0,0}\makebox(0,0)[lt]{\lineheight{1.25}\smash{\begin{tabular}[t]{l}$Q_2^\dual$\end{tabular}}}}%
    \put(0,0){\includegraphics[width=\unitlength,page=2]{ff2ell.pdf}}%
    \put(0.75668741,0.39042274){\color[rgb]{0,0,0}\makebox(0,0)[lt]{\lineheight{1.25}\smash{\begin{tabular}[t]{l}shear\end{tabular}}}}%
    \put(0.80254068,0.1100724){\color[rgb]{0,0,1}\makebox(0,0)[lt]{\lineheight{1.25}\smash{\begin{tabular}[t]{l}$(0,1)$\end{tabular}}}}%
    \put(0,0){\includegraphics[width=\unitlength,page=3]{ff2ell.pdf}}%
  \end{picture}%
\endgroup%

%% file: moving-ff.pdf_tex
\begingroup%
  \makeatletter%
  \providecommand\color[2][]{%
    \errmessage{(Inkscape) Color is used for the text in Inkscape, but the package 'color.sty' is not loaded}%
    \renewcommand\color[2][]{}%
  }%
  \providecommand\transparent[1]{%
    \errmessage{(Inkscape) Transparency is used (non-zero) for the text in Inkscape, but the package 'transparent.sty' is not loaded}%
    \renewcommand\transparent[1]{}%
  }%
  \providecommand\rotatebox[2]{#2}%
  \newcommand*\fsize{\dimexpr\f@size pt\relax}%
  \newcommand*\lineheight[1]{\fontsize{\fsize}{#1\fsize}\selectfont}%
  \ifx\svgwidth\undefined%
    \setlength{\unitlength}{492.82522356bp}%
    \ifx\svgscale\undefined%
      \relax%
    \else%
      \setlength{\unitlength}{\unitlength * \real{\svgscale}}%
    \fi%
  \else%
    \setlength{\unitlength}{\svgwidth}%
  \fi%
  \global\let\svgwidth\undefined%
  \global\let\svgscale\undefined%
  \makeatother%
  \begin{picture}(1,0.1619813)%
    \lineheight{1}%
    \setlength\tabcolsep{0pt}%
    \put(0,0){\includegraphics[width=\unitlength,page=1]{moving-ff.pdf}}%
    \put(0.02647446,0.13939723){\color[rgb]{0,0,0}\makebox(0,0)[lt]{\lineheight{1.25}\smash{\begin{tabular}[t]{l}$\Phi(X)$\end{tabular}}}}%
    \put(0,0){\includegraphics[width=\unitlength,page=2]{moving-ff.pdf}}%
    \put(0.0914263,0.12212793){\color[rgb]{0.50196078,0,0.50196078}\makebox(0,0)[lt]{\lineheight{1.25}\smash{\begin{tabular}[t]{l}$b_0$, $b_1$\end{tabular}}}}%
    \put(0,0){\includegraphics[width=\unitlength,page=3]{moving-ff.pdf}}%
    \put(0.30435767,0.00106007){\color[rgb]{0,0,0}\makebox(0,0)[lt]{\lineheight{1.25}\smash{\begin{tabular}[t]{l}$A$\end{tabular}}}}%
    \put(0.66478266,0.12981413){\color[rgb]{0.50196078,0,0.50196078}\makebox(0,0)[lt]{\lineheight{1.25}\smash{\begin{tabular}[t]{l}$b_0$\end{tabular}}}}%
    \put(0.66591405,0.10690189){\color[rgb]{0.50196078,0,0.50196078}\makebox(0,0)[lt]{\lineheight{1.25}\smash{\begin{tabular}[t]{l}$b_1$\end{tabular}}}}%
    \put(0.59779354,0.14838146){\color[rgb]{0,0,0}\makebox(0,0)[lt]{\lineheight{1.25}\smash{\begin{tabular}[t]{l}$\Phi(X)$\end{tabular}}}}%
    \put(0.23125471,0.1460318){\color[rgb]{0,0,0}\makebox(0,0)[lt]{\lineheight{1.25}\smash{\begin{tabular}[t]{l}$\A(X)$\end{tabular}}}}%
    \put(0,0){\includegraphics[width=\unitlength,page=4]{moving-ff.pdf}}%
    \put(0.85860018,0.05123514){\color[rgb]{0,0,0}\makebox(0,0)[lt]{\lineheight{1.25}\smash{\begin{tabular}[t]{l}$A'$\end{tabular}}}}%
    \put(0,0){\includegraphics[width=\unitlength,page=5]{moving-ff.pdf}}%
    \put(0.86050356,0.0074833){\color[rgb]{0,0,0}\makebox(0,0)[lt]{\lineheight{1.25}\smash{\begin{tabular}[t]{l}$A$\end{tabular}}}}%
    \put(0.80721605,0.14137445){\color[rgb]{0,0,0}\makebox(0,0)[lt]{\lineheight{1.25}\smash{\begin{tabular}[t]{l}$\A(X)$\end{tabular}}}}%
    \put(0.90065519,0.12272997){\color[rgb]{0.50196078,0,0.50196078}\makebox(0,0)[lt]{\lineheight{1.25}\smash{\begin{tabular}[t]{l}$b_0$\end{tabular}}}}%
    \put(0.81233644,0.07365227){\color[rgb]{0.50196078,0,0.50196078}\makebox(0,0)[lt]{\lineheight{1.25}\smash{\begin{tabular}[t]{l}$b_1$\end{tabular}}}}%
    \put(0.90550762,0.0736001){\color[rgb]{0.50196078,0,0.50196078}\makebox(0,0)[lt]{\lineheight{1.25}\smash{\begin{tabular}[t]{l}$b_1$\end{tabular}}}}%
    \put(0,0){\includegraphics[width=\unitlength,page=6]{moving-ff.pdf}}%
    \put(0.33168681,0.07326135){\color[rgb]{0.50196078,0,0.50196078}\makebox(0,0)[lt]{\lineheight{1.25}\smash{\begin{tabular}[t]{l}$b_0$, $b_1$\end{tabular}}}}%
  \end{picture}%
\endgroup%

%% file: types.pdf_tex
\begingroup%
  \makeatletter%
  \providecommand\color[2][]{%
    \errmessage{(Inkscape) Color is used for the text in Inkscape, but the package 'color.sty' is not loaded}%
    \renewcommand\color[2][]{}%
  }%
  \providecommand\transparent[1]{%
    \errmessage{(Inkscape) Transparency is used (non-zero) for the text in Inkscape, but the package 'transparent.sty' is not loaded}%
    \renewcommand\transparent[1]{}%
  }%
  \providecommand\rotatebox[2]{#2}%
  \newcommand*\fsize{\dimexpr\f@size pt\relax}%
  \newcommand*\lineheight[1]{\fontsize{\fsize}{#1\fsize}\selectfont}%
  \ifx\svgwidth\undefined%
    \setlength{\unitlength}{485.27895326bp}%
    \ifx\svgscale\undefined%
      \relax%
    \else%
      \setlength{\unitlength}{\unitlength * \real{\svgscale}}%
    \fi%
  \else%
    \setlength{\unitlength}{\svgwidth}%
  \fi%
  \global\let\svgwidth\undefined%
  \global\let\svgscale\undefined%
  \makeatother%
  \begin{picture}(1,0.29633853)%
    \lineheight{1}%
    \setlength\tabcolsep{0pt}%
    \put(0,0){\includegraphics[width=\unitlength,page=1]{types.pdf}}%
    \put(0.08971766,0.26483295){\color[rgb]{0,0,1}\makebox(0,0)[lt]{\lineheight{1.25}\smash{\begin{tabular}[t]{l}$\cT_{t_0-\eps}$\end{tabular}}}}%
    \put(0.36678656,0.26638793){\color[rgb]{0,0,1}\makebox(0,0)[lt]{\lineheight{1.25}\smash{\begin{tabular}[t]{l}$\cT_{t_0}$\end{tabular}}}}%
    \put(0.04113205,0.20862243){\color[rgb]{0,0,1}\makebox(0,0)[lt]{\lineheight{1.25}\smash{\begin{tabular}[t]{l}$v_-$\end{tabular}}}}%
    \put(0.14009196,0.20853317){\color[rgb]{0,0,1}\makebox(0,0)[lt]{\lineheight{1.25}\smash{\begin{tabular}[t]{l}$v_+$\end{tabular}}}}%
    \put(0.097017,0.21937632){\color[rgb]{0,0,1}\makebox(0,0)[lt]{\lineheight{1.25}\smash{\begin{tabular}[t]{l}$e$\end{tabular}}}}%
    \put(0.3943806,0.21193409){\color[rgb]{0,0,1}\makebox(0,0)[lt]{\lineheight{1.25}\smash{\begin{tabular}[t]{l}$\ell(e)=0$\end{tabular}}}}%
    \put(0,0){\includegraphics[width=\unitlength,page=2]{types.pdf}}%
    \put(0.54681507,0.21448964){\color[rgb]{0,0,1}\makebox(0,0)[lt]{\lineheight{1.25}\smash{\begin{tabular}[t]{l}$v_-$\end{tabular}}}}%
    \put(0.61963491,0.22757987){\color[rgb]{0,0,1}\makebox(0,0)[lt]{\lineheight{1.25}\smash{\begin{tabular}[t]{l}$e$\end{tabular}}}}%
    \put(0.66227455,0.21925029){\color[rgb]{0,0,1}\makebox(0,0)[lt]{\lineheight{1.25}\smash{\begin{tabular}[t]{l}$v_+$\end{tabular}}}}%
    \put(0.55617014,0.27043418){\color[rgb]{0,0,1}\makebox(0,0)[lt]{\lineheight{1.25}\smash{\begin{tabular}[t]{l}$\cT_{t_0-\eps}$\end{tabular}}}}%
    \put(0,0){\includegraphics[width=\unitlength,page=3]{types.pdf}}%
    \put(0.20335591,0.18548657){\makebox(0,0)[lt]{\lineheight{1.25}\smash{\begin{tabular}[t]{l}Type 1A\end{tabular}}}}%
    \put(0,0){\includegraphics[width=\unitlength,page=4]{types.pdf}}%
    \put(0.86034745,0.27690833){\color[rgb]{0,0,1}\makebox(0,0)[lt]{\lineheight{1.25}\smash{\begin{tabular}[t]{l}$\cT_{t_0}$\end{tabular}}}}%
    \put(0.91347655,0.22309909){\color[rgb]{0,0,1}\makebox(0,0)[lt]{\lineheight{1.25}\smash{\begin{tabular}[t]{l}$\ell(e)=0$\end{tabular}}}}%
    \put(0,0){\includegraphics[width=\unitlength,page=5]{types.pdf}}%
    \put(0.72100432,0.18560462){\makebox(0,0)[lt]{\lineheight{1.25}\smash{\begin{tabular}[t]{l}Type 1B\end{tabular}}}}%
    \put(0,0){\includegraphics[width=\unitlength,page=6]{types.pdf}}%
    \put(0.33046702,0.10198787){\color[rgb]{0,0,1}\makebox(0,0)[lt]{\lineheight{1.25}\smash{\begin{tabular}[t]{l}$\cT_{t_0-\eps}$\end{tabular}}}}%
    \put(0.310042,0.0242683){\color[rgb]{0,0,1}\makebox(0,0)[lt]{\lineheight{1.25}\smash{\begin{tabular}[t]{l}$v_-$\end{tabular}}}}%
    \put(0.38084133,0.04568806){\color[rgb]{0,0,1}\makebox(0,0)[lt]{\lineheight{1.25}\smash{\begin{tabular}[t]{l}$v_+$\end{tabular}}}}%
    \put(0.33776637,0.05653121){\color[rgb]{0,0,1}\makebox(0,0)[lt]{\lineheight{1.25}\smash{\begin{tabular}[t]{l}$e$\end{tabular}}}}%
    \put(0,0){\includegraphics[width=\unitlength,page=7]{types.pdf}}%
    \put(0.56311438,0.10607237){\color[rgb]{0,0,1}\makebox(0,0)[lt]{\lineheight{1.25}\smash{\begin{tabular}[t]{l}$\cT_{t_0}$\end{tabular}}}}%
    \put(0,0){\includegraphics[width=\unitlength,page=8]{types.pdf}}%
    \put(0.456472,0.01338639){\makebox(0,0)[lt]{\lineheight{1.25}\smash{\begin{tabular}[t]{l}Type 1C\end{tabular}}}}%
    \put(0.56199064,0.04696155){\color[rgb]{0,0.50196078,0.50196078}\makebox(0,0)[lt]{\lineheight{1.25}\smash{\begin{tabular}[t]{l}$\lam$\end{tabular}}}}%
    \put(0.28641095,0.05999193){\color[rgb]{0,0.50196078,0.50196078}\makebox(0,0)[lt]{\lineheight{1.25}\smash{\begin{tabular}[t]{l}$\lam$\end{tabular}}}}%
    \put(0,0){\includegraphics[width=\unitlength,page=9]{types.pdf}}%
    \put(0.71347977,0.07022605){\makebox(0,0)[lt]{\lineheight{1.25}\smash{\begin{tabular}[t]{l}Index $0$\end{tabular}}}}%
  \end{picture}%
\endgroup%

%% file: slide-wall.pdf_tex
\begingroup%
  \makeatletter%
  \providecommand\color[2][]{%
    \errmessage{(Inkscape) Color is used for the text in Inkscape, but the package 'color.sty' is not loaded}%
    \renewcommand\color[2][]{}%
  }%
  \providecommand\transparent[1]{%
    \errmessage{(Inkscape) Transparency is used (non-zero) for the text in Inkscape, but the package 'transparent.sty' is not loaded}%
    \renewcommand\transparent[1]{}%
  }%
  \providecommand\rotatebox[2]{#2}%
  \newcommand*\fsize{\dimexpr\f@size pt\relax}%
  \newcommand*\lineheight[1]{\fontsize{\fsize}{#1\fsize}\selectfont}%
  \ifx\svgwidth\undefined%
    \setlength{\unitlength}{476.69178148bp}%
    \ifx\svgscale\undefined%
      \relax%
    \else%
      \setlength{\unitlength}{\unitlength * \real{\svgscale}}%
    \fi%
  \else%
    \setlength{\unitlength}{\svgwidth}%
  \fi%
  \global\let\svgwidth\undefined%
  \global\let\svgscale\undefined%
  \makeatother%
  \begin{picture}(1,0.43314059)%
    \lineheight{1}%
    \setlength\tabcolsep{0pt}%
    \put(0,0){\includegraphics[width=\unitlength,page=1]{slide-wall.pdf}}%
    \put(0.18440719,0.28564929){\color[rgb]{0,0,1}\makebox(0,0)[lt]{\lineheight{1.25}\smash{\begin{tabular}[t]{l}$(-1,2)$\end{tabular}}}}%
    \put(0.05865917,0.39115808){\color[rgb]{0,0,1}\makebox(0,0)[lt]{\lineheight{1.25}\smash{\begin{tabular}[t]{l}$(-3,2)$\end{tabular}}}}%
    \put(0.13026818,0.36000499){\color[rgb]{0,0,1}\makebox(0,0)[lt]{\lineheight{1.25}\smash{\begin{tabular}[t]{l}$(-2,0)$\end{tabular}}}}%
    \put(0,0){\includegraphics[width=\unitlength,page=2]{slide-wall.pdf}}%
    \put(0.27576229,0.31767876){\color[rgb]{0,0,0}\makebox(0,0)[lt]{\lineheight{1.25}\smash{\begin{tabular}[t]{l}Slide $b$\\down\end{tabular}}}}%
    \put(0.18018596,0.34902385){\color[rgb]{0.50196078,0,0.50196078}\makebox(0,0)[lt]{\lineheight{1.25}\smash{\begin{tabular}[t]{l}$b$\end{tabular}}}}%
    \put(0,0){\includegraphics[width=\unitlength,page=3]{slide-wall.pdf}}%
    \put(0.55115734,0.28565211){\color[rgb]{0,0,1}\makebox(0,0)[lt]{\lineheight{1.25}\smash{\begin{tabular}[t]{l}$(-1,2)$\end{tabular}}}}%
    \put(0.53141974,0.28000073){\color[rgb]{0.50196078,0,0.50196078}\makebox(0,0)[lt]{\lineheight{1.25}\smash{\begin{tabular}[t]{l}$b$\end{tabular}}}}%
    \put(0,0){\includegraphics[width=\unitlength,page=4]{slide-wall.pdf}}%
    \put(0.94341287,0.27549708){\color[rgb]{0,0,1}\makebox(0,0)[lt]{\lineheight{1.25}\smash{\begin{tabular}[t]{l}$(-1,2)$\end{tabular}}}}%
    \put(0.89763758,0.27129925){\color[rgb]{0.50196078,0,0.50196078}\makebox(0,0)[lt]{\lineheight{1.25}\smash{\begin{tabular}[t]{l}$b$\end{tabular}}}}%
    \put(0,0){\includegraphics[width=\unitlength,page=5]{slide-wall.pdf}}%
    \put(0.46091112,0.34009558){\color[rgb]{0,0,1}\makebox(0,0)[lt]{\lineheight{1.25}\smash{\begin{tabular}[t]{l}$(-3,2)$\end{tabular}}}}%
    \put(0.79219181,0.30380359){\color[rgb]{0,0,1}\makebox(0,0)[lt]{\lineheight{1.25}\smash{\begin{tabular}[t]{l}$(-3,2)$\end{tabular}}}}%
    \put(0,0){\includegraphics[width=\unitlength,page=6]{slide-wall.pdf}}%
    \put(0.66274038,0.31724758){\color[rgb]{0,0,0}\makebox(0,0)[lt]{\lineheight{1.25}\smash{\begin{tabular}[t]{l}Slide $b$\\down\end{tabular}}}}%
    \put(0.55564824,0.41076974){\color[rgb]{0.50196078,0,0.50196078}\makebox(0,0)[lt]{\lineheight{1.25}\smash{\begin{tabular}[t]{l}Wall\end{tabular}}}}%
    \put(0,0){\includegraphics[width=\unitlength,page=7]{slide-wall.pdf}}%
    \put(0.16744165,0.14064276){\color[rgb]{0,0,1}\makebox(0,0)[lt]{\lineheight{1.25}\smash{\begin{tabular}[t]{l}(-1,2)\end{tabular}}}}%
    \put(0.19471994,0.02741362){\color[rgb]{0,0,1}\makebox(0,0)[lt]{\lineheight{1.25}\smash{\begin{tabular}[t]{l}(0,2)\end{tabular}}}}%
    \put(0,0){\includegraphics[width=\unitlength,page=8]{slide-wall.pdf}}%
    \put(0.26202391,0.13276303){\color[rgb]{0.50196078,0,0.50196078}\makebox(0,0)[lt]{\lineheight{1.25}\smash{\begin{tabular}[t]{l}$b$\end{tabular}}}}%
    \put(0,0){\includegraphics[width=\unitlength,page=9]{slide-wall.pdf}}%
    \put(0.63010098,0.13236824){\color[rgb]{0.50196078,0,0.50196078}\makebox(0,0)[lt]{\lineheight{1.25}\smash{\begin{tabular}[t]{l}$b$\end{tabular}}}}%
    \put(0,0){\includegraphics[width=\unitlength,page=10]{slide-wall.pdf}}%
    \put(0.68246382,0.14076416){\color[rgb]{0,0,1}\makebox(0,0)[lt]{\lineheight{1.25}\smash{\begin{tabular}[t]{l}(1,2)\end{tabular}}}}%
    \put(0,0){\includegraphics[width=\unitlength,page=11]{slide-wall.pdf}}%
    \put(0.2510219,0.08580499){\color[rgb]{0,0,1}\makebox(0,0)[lt]{\lineheight{1.25}\smash{\begin{tabular}[t]{l}(-1,0)\end{tabular}}}}%
    \put(0,0){\includegraphics[width=\unitlength,page=12]{slide-wall.pdf}}%
    \put(0.61379557,0.08946566){\color[rgb]{0,0,1}\makebox(0,0)[lt]{\lineheight{1.25}\smash{\begin{tabular}[t]{l}(1,0)\end{tabular}}}}%
    \put(0.66539423,0.03358914){\color[rgb]{0,0,1}\makebox(0,0)[lt]{\lineheight{1.25}\smash{\begin{tabular}[t]{l}(0,2)\end{tabular}}}}%
    \put(0,0){\includegraphics[width=\unitlength,page=13]{slide-wall.pdf}}%
    \put(0.41300738,0.0723556){\color[rgb]{0,0,0}\makebox(0,0)[lt]{\lineheight{1.25}\smash{\begin{tabular}[t]{l}Slide $b$\\to the left\end{tabular}}}}%
    \put(0.15095963,0.33678103){\color[rgb]{0,0,1}\makebox(0,0)[lt]{\lineheight{1.25}\smash{\begin{tabular}[t]{l}$e$\end{tabular}}}}%
    \put(0.02022287,0.3876998){\color[rgb]{0,0,1}\makebox(0,0)[lt]{\lineheight{1.25}\smash{\begin{tabular}[t]{l}$e_1$\end{tabular}}}}%
    \put(0.15849414,0.26305826){\color[rgb]{0,0,1}\makebox(0,0)[lt]{\lineheight{1.25}\smash{\begin{tabular}[t]{l}$e_0$\end{tabular}}}}%
    \put(0.93167649,0.2440437){\color[rgb]{0,0,1}\makebox(0,0)[lt]{\lineheight{1.25}\smash{\begin{tabular}[t]{l}$e_0$\end{tabular}}}}%
    \put(0.8087363,0.34358553){\color[rgb]{0,0,1}\makebox(0,0)[lt]{\lineheight{1.25}\smash{\begin{tabular}[t]{l}$e_1$\end{tabular}}}}%
    \put(0.17515662,0.18732651){\color[rgb]{0,0,1}\makebox(0,0)[lt]{\lineheight{1.25}\smash{\begin{tabular}[t]{l}$e_1$\end{tabular}}}}%
    \put(0.24607536,0.0022999){\color[rgb]{0,0,1}\makebox(0,0)[lt]{\lineheight{1.25}\smash{\begin{tabular}[t]{l}$e_0$\end{tabular}}}}%
    \put(0.66620179,0.00551242){\color[rgb]{0,0,1}\makebox(0,0)[lt]{\lineheight{1.25}\smash{\begin{tabular}[t]{l}$e_0$\end{tabular}}}}%
    \put(0.69887721,0.17545946){\color[rgb]{0,0,1}\makebox(0,0)[lt]{\lineheight{1.25}\smash{\begin{tabular}[t]{l}$e_1$\end{tabular}}}}%
    \put(0.24067589,0.12228498){\color[rgb]{0,0,1}\makebox(0,0)[lt]{\lineheight{1.25}\smash{\begin{tabular}[t]{l}$e$\end{tabular}}}}%
    \put(0.64500932,0.12218354){\color[rgb]{0,0,1}\makebox(0,0)[lt]{\lineheight{1.25}\smash{\begin{tabular}[t]{l}$e$\end{tabular}}}}%
  \end{picture}%
\endgroup%

%% file: mute1.pdf_tex
\begingroup%
  \makeatletter%
  \providecommand\color[2][]{%
    \errmessage{(Inkscape) Color is used for the text in Inkscape, but the package 'color.sty' is not loaded}%
    \renewcommand\color[2][]{}%
  }%
  \providecommand\transparent[1]{%
    \errmessage{(Inkscape) Transparency is used (non-zero) for the text in Inkscape, but the package 'transparent.sty' is not loaded}%
    \renewcommand\transparent[1]{}%
  }%
  \providecommand\rotatebox[2]{#2}%
  \newcommand*\fsize{\dimexpr\f@size pt\relax}%
  \newcommand*\lineheight[1]{\fontsize{\fsize}{#1\fsize}\selectfont}%
  \ifx\svgwidth\undefined%
    \setlength{\unitlength}{424.56241491bp}%
    \ifx\svgscale\undefined%
      \relax%
    \else%
      \setlength{\unitlength}{\unitlength * \real{\svgscale}}%
    \fi%
  \else%
    \setlength{\unitlength}{\svgwidth}%
  \fi%
  \global\let\svgwidth\undefined%
  \global\let\svgscale\undefined%
  \makeatother%
  \begin{picture}(1,0.23004679)%
    \lineheight{1}%
    \setlength\tabcolsep{0pt}%
    \put(0,0){\includegraphics[width=\unitlength,page=1]{mute1.pdf}}%
    \put(0.54999947,0.08786861){\color[rgb]{0.50196078,0,0.50196078}\makebox(0,0)[lt]{\lineheight{1.25}\smash{\begin{tabular}[t]{l}$b_0$\end{tabular}}}}%
    \put(0,0){\includegraphics[width=\unitlength,page=2]{mute1.pdf}}%
    \put(0.11388062,0.1816581){\makebox(0,0)[lt]{\lineheight{1.25}\smash{\begin{tabular}[t]{l}$t=-\eps$\end{tabular}}}}%
    \put(0.44299455,0.17525991){\makebox(0,0)[lt]{\lineheight{1.25}\smash{\begin{tabular}[t]{l}$t=0$\end{tabular}}}}%
    \put(0.78935947,0.17811289){\makebox(0,0)[lt]{\lineheight{1.25}\smash{\begin{tabular}[t]{l}$t=\eps$\end{tabular}}}}%
    \put(0.49091249,0.09968599){\color[rgb]{0,0,1}\makebox(0,0)[lt]{\lineheight{1.25}\smash{\begin{tabular}[t]{l}$\mu_b$\end{tabular}}}}%
    \put(0,0){\includegraphics[width=\unitlength,page=3]{mute1.pdf}}%
    \put(0.917216,0.08786754){\color[rgb]{0.50196078,0,0.50196078}\makebox(0,0)[lt]{\lineheight{1.25}\smash{\begin{tabular}[t]{l}$b_{\eps}$\end{tabular}}}}%
    \put(0.82054068,0.02473315){\color[rgb]{0,0.50196078,0.50196078}\makebox(0,0)[lt]{\lineheight{1.25}\smash{\begin{tabular}[t]{l}$\lam_{\eps}$\end{tabular}}}}%
    \put(0,0){\includegraphics[width=\unitlength,page=4]{mute1.pdf}}%
    \put(0.76736237,0.12518101){\color[rgb]{0.10196078,0.10196078,0.10196078}\makebox(0,0)[lt]{\lineheight{1.25}\smash{\begin{tabular}[t]{l}$\A_{\eps}$\end{tabular}}}}%
    \put(0.39541895,0.1287273){\color[rgb]{0.10196078,0.10196078,0.10196078}\makebox(0,0)[lt]{\lineheight{1.25}\smash{\begin{tabular}[t]{l}$\A_0$\end{tabular}}}}%
    \put(0.45214243,0.06018639){\color[rgb]{0,0.50196078,0.50196078}\makebox(0,0)[lt]{\lineheight{1.25}\smash{\begin{tabular}[t]{l}$\lam_0$\end{tabular}}}}%
    \put(0,0){\includegraphics[width=\unitlength,page=5]{mute1.pdf}}%
    \put(0.20949334,0.08786861){\color[rgb]{0.50196078,0,0.50196078}\makebox(0,0)[lt]{\lineheight{1.25}\smash{\begin{tabular}[t]{l}$b_{-\eps}$\end{tabular}}}}%
    \put(0,0){\includegraphics[width=\unitlength,page=6]{mute1.pdf}}%
    \put(0.09657817,0.09563859){\color[rgb]{0,0.50196078,0.50196078}\makebox(0,0)[lt]{\lineheight{1.25}\smash{\begin{tabular}[t]{l}$\lam_{-\eps}$\end{tabular}}}}%
    \put(0.02567373,0.02946117){\color[rgb]{0.10196078,0.10196078,0.10196078}\makebox(0,0)[lt]{\lineheight{1.25}\smash{\begin{tabular}[t]{l}$\A_{-\eps}$\end{tabular}}}}%
    \put(0,0){\includegraphics[width=\unitlength,page=7]{mute1.pdf}}%
    \put(0.27524503,0.08550513){\color[rgb]{0.50196078,0,0.50196078}\makebox(0,0)[lt]{\lineheight{1.25}\smash{\begin{tabular}[t]{l}$\nu$\end{tabular}}}}%
  \end{picture}%
\endgroup%

%% file: mute2.pdf_tex
\begingroup%
  \makeatletter%
  \providecommand\color[2][]{%
    \errmessage{(Inkscape) Color is used for the text in Inkscape, but the package 'color.sty' is not loaded}%
    \renewcommand\color[2][]{}%
  }%
  \providecommand\transparent[1]{%
    \errmessage{(Inkscape) Transparency is used (non-zero) for the text in Inkscape, but the package 'transparent.sty' is not loaded}%
    \renewcommand\transparent[1]{}%
  }%
  \providecommand\rotatebox[2]{#2}%
  \newcommand*\fsize{\dimexpr\f@size pt\relax}%
  \newcommand*\lineheight[1]{\fontsize{\fsize}{#1\fsize}\selectfont}%
  \ifx\svgwidth\undefined%
    \setlength{\unitlength}{294.13738058bp}%
    \ifx\svgscale\undefined%
      \relax%
    \else%
      \setlength{\unitlength}{\unitlength * \real{\svgscale}}%
    \fi%
  \else%
    \setlength{\unitlength}{\svgwidth}%
  \fi%
  \global\let\svgwidth\undefined%
  \global\let\svgscale\undefined%
  \makeatother%
  \begin{picture}(1,0.33205154)%
    \lineheight{1}%
    \setlength\tabcolsep{0pt}%
    \put(0,0){\includegraphics[width=\unitlength,page=1]{mute2.pdf}}%
    \put(0.16437703,0.26220657){\makebox(0,0)[lt]{\lineheight{1.25}\smash{\begin{tabular}[t]{l}$t=-\eps$\end{tabular}}}}%
    \put(0.7357994,0.25297131){\makebox(0,0)[lt]{\lineheight{1.25}\smash{\begin{tabular}[t]{l}$t=0$\end{tabular}}}}%
    \put(0.75382807,0.15562248){\color[rgb]{0,0.50196078,0.50196078}\makebox(0,0)[lt]{\lineheight{1.25}\smash{\begin{tabular}[t]{l}$\lam_0$\end{tabular}}}}%
    \put(0,0){\includegraphics[width=\unitlength,page=2]{mute2.pdf}}%
    \put(0.17989687,0.19974398){\color[rgb]{0,0.50196078,0.50196078}\makebox(0,0)[lt]{\lineheight{1.25}\smash{\begin{tabular}[t]{l}$\lam_{-\eps}$\end{tabular}}}}%
    \put(0,0){\includegraphics[width=\unitlength,page=3]{mute2.pdf}}%
    \put(0.17386619,0.1116957){\color[rgb]{0,0,1}\makebox(0,0)[lt]{\lineheight{1.25}\smash{\begin{tabular}[t]{l}$v_+$\end{tabular}}}}%
    \put(0.10994071,0.17320885){\color[rgb]{0,0,1}\makebox(0,0)[lt]{\lineheight{1.25}\smash{\begin{tabular}[t]{l}$v_-$\end{tabular}}}}%
    \put(0.20522582,0.15752903){\color[rgb]{0,0,1}\makebox(0,0)[lt]{\lineheight{1.25}\smash{\begin{tabular}[t]{l}$e$\end{tabular}}}}%
    \put(0.7093928,0.12134484){\color[rgb]{0,0,1}\makebox(0,0)[lt]{\lineheight{1.25}\smash{\begin{tabular}[t]{l}$v$\end{tabular}}}}%
  \end{picture}%
\endgroup%

%% file: mute3.pdf_tex
\begingroup%
  \makeatletter%
  \providecommand\color[2][]{%
    \errmessage{(Inkscape) Color is used for the text in Inkscape, but the package 'color.sty' is not loaded}%
    \renewcommand\color[2][]{}%
  }%
  \providecommand\transparent[1]{%
    \errmessage{(Inkscape) Transparency is used (non-zero) for the text in Inkscape, but the package 'transparent.sty' is not loaded}%
    \renewcommand\transparent[1]{}%
  }%
  \providecommand\rotatebox[2]{#2}%
  \newcommand*\fsize{\dimexpr\f@size pt\relax}%
  \newcommand*\lineheight[1]{\fontsize{\fsize}{#1\fsize}\selectfont}%
  \ifx\svgwidth\undefined%
    \setlength{\unitlength}{421.6964357bp}%
    \ifx\svgscale\undefined%
      \relax%
    \else%
      \setlength{\unitlength}{\unitlength * \real{\svgscale}}%
    \fi%
  \else%
    \setlength{\unitlength}{\svgwidth}%
  \fi%
  \global\let\svgwidth\undefined%
  \global\let\svgscale\undefined%
  \makeatother%
  \begin{picture}(1,0.17499185)%
    \lineheight{1}%
    \setlength\tabcolsep{0pt}%
    \put(0,0){\includegraphics[width=\unitlength,page=1]{mute3.pdf}}%
    \put(0.55373744,0.08846579){\color[rgb]{0.50196078,0,0.50196078}\makebox(0,0)[lt]{\lineheight{1.25}\smash{\begin{tabular}[t]{l}$b_0$\end{tabular}}}}%
    \put(0,0){\includegraphics[width=\unitlength,page=2]{mute3.pdf}}%
    \put(0.49424888,0.10036349){\color[rgb]{0,0,1}\makebox(0,0)[lt]{\lineheight{1.25}\smash{\begin{tabular}[t]{l}$\mu_b$\end{tabular}}}}%
    \put(0,0){\includegraphics[width=\unitlength,page=3]{mute3.pdf}}%
    \put(0.82611733,0.02490124){\color[rgb]{0,0.50196078,0.50196078}\makebox(0,0)[lt]{\lineheight{1.25}\smash{\begin{tabular}[t]{l}$\lam_{\eps}$\end{tabular}}}}%
    \put(0.77257761,0.12603178){\color[rgb]{0.10196078,0.10196078,0.10196078}\makebox(0,0)[lt]{\lineheight{1.25}\smash{\begin{tabular}[t]{l}$\A_{\eps}$\end{tabular}}}}%
    \put(0.39810634,0.12960217){\color[rgb]{0.10196078,0.10196078,0.10196078}\makebox(0,0)[lt]{\lineheight{1.25}\smash{\begin{tabular}[t]{l}$\A_0$\end{tabular}}}}%
    \put(0.45521533,0.06059544){\color[rgb]{0,0.50196078,0.50196078}\makebox(0,0)[lt]{\lineheight{1.25}\smash{\begin{tabular}[t]{l}$\lam_0$\end{tabular}}}}%
    \put(0,0){\includegraphics[width=\unitlength,page=4]{mute3.pdf}}%
    \put(0.21091712,0.08846579){\color[rgb]{0.50196078,0,0.50196078}\makebox(0,0)[lt]{\lineheight{1.25}\smash{\begin{tabular}[t]{l}$b_{-\eps}$\end{tabular}}}}%
    \put(0.09723455,0.09628857){\color[rgb]{0,0.50196078,0.50196078}\makebox(0,0)[lt]{\lineheight{1.25}\smash{\begin{tabular}[t]{l}$\lam_{-\eps}$\end{tabular}}}}%
    \put(0.02584822,0.0296614){\color[rgb]{0.10196078,0.10196078,0.10196078}\makebox(0,0)[lt]{\lineheight{1.25}\smash{\begin{tabular}[t]{l}$\A_{-\eps}$\end{tabular}}}}%
    \put(0,0){\includegraphics[width=\unitlength,page=5]{mute3.pdf}}%
    \put(0.27711568,0.08608625){\color[rgb]{0.50196078,0,0.50196078}\makebox(0,0)[lt]{\lineheight{1.25}\smash{\begin{tabular}[t]{l}$\nu$\end{tabular}}}}%
    \put(0,0){\includegraphics[width=\unitlength,page=6]{mute3.pdf}}%
    \put(0.14779972,0.15637201){\color[rgb]{0,0,1}\makebox(0,0)[lt]{\lineheight{1.25}\smash{\begin{tabular}[t]{l}$(k,l)$\end{tabular}}}}%
    \put(0.91103772,0.13079196){\color[rgb]{0,0,1}\makebox(0,0)[lt]{\lineheight{1.25}\smash{\begin{tabular}[t]{l}$(k,l)$\end{tabular}}}}%
    \put(0.75279818,0.08320112){\color[rgb]{0,0,1}\makebox(0,0)[lt]{\lineheight{1.25}\smash{\begin{tabular}[t]{l}$(k+k_1,l)$\end{tabular}}}}%
    \put(0.91936607,0.08736532){\color[rgb]{0,0,1}\makebox(0,0)[lt]{\lineheight{1.25}\smash{\begin{tabular}[t]{l}$(-k_1,0)$\end{tabular}}}}%
  \end{picture}%
\endgroup%

%% file: signs.pdf_tex
\begingroup%
  \makeatletter%
  \providecommand\color[2][]{%
    \errmessage{(Inkscape) Color is used for the text in Inkscape, but the package 'color.sty' is not loaded}%
    \renewcommand\color[2][]{}%
  }%
  \providecommand\transparent[1]{%
    \errmessage{(Inkscape) Transparency is used (non-zero) for the text in Inkscape, but the package 'transparent.sty' is not loaded}%
    \renewcommand\transparent[1]{}%
  }%
  \providecommand\rotatebox[2]{#2}%
  \newcommand*\fsize{\dimexpr\f@size pt\relax}%
  \newcommand*\lineheight[1]{\fontsize{\fsize}{#1\fsize}\selectfont}%
  \ifx\svgwidth\undefined%
    \setlength{\unitlength}{254.95734168bp}%
    \ifx\svgscale\undefined%
      \relax%
    \else%
      \setlength{\unitlength}{\unitlength * \real{\svgscale}}%
    \fi%
  \else%
    \setlength{\unitlength}{\svgwidth}%
  \fi%
  \global\let\svgwidth\undefined%
  \global\let\svgscale\undefined%
  \makeatother%
  \begin{picture}(1,0.58221076)%
    \lineheight{1}%
    \setlength\tabcolsep{0pt}%
    \put(0,0){\includegraphics[width=\unitlength,page=1]{signs.pdf}}%
    \put(0.02208883,0.56911908){\color[rgb]{0,0,0}\makebox(0,0)[lt]{\lineheight{1.25}\smash{\begin{tabular}[t]{l}$a$\end{tabular}}}}%
    \put(0.42679778,0.57064575){\color[rgb]{0,0,0}\makebox(0,0)[lt]{\lineheight{1.25}\smash{\begin{tabular}[t]{l}$a$\end{tabular}}}}%
    \put(0.86248995,0.00295063){\color[rgb]{0,0,0}\makebox(0,0)[lt]{\lineheight{1.25}\smash{\begin{tabular}[t]{l}$b$\end{tabular}}}}%
    \put(0.86558757,0.57059164){\color[rgb]{0,0,0}\makebox(0,0)[lt]{\lineheight{1.25}\smash{\begin{tabular}[t]{l}$a$\end{tabular}}}}%
    \put(0.70763179,0.28859291){\color[rgb]{0,0,0}\makebox(0,0)[lt]{\lineheight{1.25}\smash{\begin{tabular}[t]{l}$b$\end{tabular}}}}%
    \put(0.29822063,0.27962815){\color[rgb]{0,0,0}\makebox(0,0)[lt]{\lineheight{1.25}\smash{\begin{tabular}[t]{l}$b$\end{tabular}}}}%
    \put(0,0){\includegraphics[width=\unitlength,page=2]{signs.pdf}}%
    \put(0.65022265,0.4581347){\color[rgb]{0.50196078,0,0.50196078}\makebox(0,0)[lt]{\lineheight{1.25}\smash{\begin{tabular}[t]{l}$\nu$\end{tabular}}}}%
  \end{picture}%
\endgroup%

%% file: Tsing.pdf_tex
\begingroup%
  \makeatletter%
  \providecommand\color[2][]{%
    \errmessage{(Inkscape) Color is used for the text in Inkscape, but the package 'color.sty' is not loaded}%
    \renewcommand\color[2][]{}%
  }%
  \providecommand\transparent[1]{%
    \errmessage{(Inkscape) Transparency is used (non-zero) for the text in Inkscape, but the package 'transparent.sty' is not loaded}%
    \renewcommand\transparent[1]{}%
  }%
  \providecommand\rotatebox[2]{#2}%
  \newcommand*\fsize{\dimexpr\f@size pt\relax}%
  \newcommand*\lineheight[1]{\fontsize{\fsize}{#1\fsize}\selectfont}%
  \ifx\svgwidth\undefined%
    \setlength{\unitlength}{161.79304762bp}%
    \ifx\svgscale\undefined%
      \relax%
    \else%
      \setlength{\unitlength}{\unitlength * \real{\svgscale}}%
    \fi%
  \else%
    \setlength{\unitlength}{\svgwidth}%
  \fi%
  \global\let\svgwidth\undefined%
  \global\let\svgscale\undefined%
  \makeatother%
  \begin{picture}(1,0.4238764)%
    \lineheight{1}%
    \setlength\tabcolsep{0pt}%
    \put(0,0){\includegraphics[width=\unitlength,page=1]{Tsing.pdf}}%
    \put(-0.00223303,0.38737897){\color[rgb]{0,0,0}\makebox(0,0)[lt]{\lineheight{1.25}\smash{\begin{tabular}[t]{l}$(0,1)$\end{tabular}}}}%
    \put(0.74539855,0.38485798){\color[rgb]{0,0,0}\makebox(0,0)[lt]{\lineheight{1.25}\smash{\begin{tabular}[t]{l}$(dp^2,dpq-1)$\end{tabular}}}}%
    \put(0.34870246,0.28108259){\color[rgb]{0.50196078,0,0.50196078}\makebox(0,0)[lt]{\lineheight{1.25}\smash{\begin{tabular}[t]{l}$(p,q)$\end{tabular}}}}%
  \end{picture}%
\endgroup%

%% file: cutTsing.pdf_tex
\begingroup%
  \makeatletter%
  \providecommand\color[2][]{%
    \errmessage{(Inkscape) Color is used for the text in Inkscape, but the package 'color.sty' is not loaded}%
    \renewcommand\color[2][]{}%
  }%
  \providecommand\transparent[1]{%
    \errmessage{(Inkscape) Transparency is used (non-zero) for the text in Inkscape, but the package 'transparent.sty' is not loaded}%
    \renewcommand\transparent[1]{}%
  }%
  \providecommand\rotatebox[2]{#2}%
  \newcommand*\fsize{\dimexpr\f@size pt\relax}%
  \newcommand*\lineheight[1]{\fontsize{\fsize}{#1\fsize}\selectfont}%
  \ifx\svgwidth\undefined%
    \setlength{\unitlength}{443.91533458bp}%
    \ifx\svgscale\undefined%
      \relax%
    \else%
      \setlength{\unitlength}{\unitlength * \real{\svgscale}}%
    \fi%
  \else%
    \setlength{\unitlength}{\svgwidth}%
  \fi%
  \global\let\svgwidth\undefined%
  \global\let\svgscale\undefined%
  \makeatother%
  \begin{picture}(1,0.28938372)%
    \lineheight{1}%
    \setlength\tabcolsep{0pt}%
    \put(0,0){\includegraphics[width=\unitlength,page=1]{cutTsing.pdf}}%
    \put(0.05141974,0.16325535){\color[rgb]{0.50196078,0,0.50196078}\makebox(0,0)[lt]{\lineheight{1.25}\smash{\begin{tabular}[t]{l}$b$\end{tabular}}}}%
    \put(0.39961182,0.16438566){\color[rgb]{0.10196078,0.10196078,0.10196078}\makebox(0,0)[lt]{\lineheight{1.25}\smash{\begin{tabular}[t]{l}$P_b^\dual$\end{tabular}}}}%
    \put(0.33052722,0.16625244){\color[rgb]{0.10196078,0.10196078,0.10196078}\makebox(0,0)[lt]{\lineheight{1.25}\smash{\begin{tabular}[t]{l}$Q^\dual$\end{tabular}}}}%
    \put(0.44658553,0.09492512){\color[rgb]{0.10196078,0.10196078,0.10196078}\makebox(0,0)[lt]{\lineheight{1.25}\smash{\begin{tabular}[t]{l}$Q^\dual$\end{tabular}}}}%
    \put(0,0){\includegraphics[width=\unitlength,page=2]{cutTsing.pdf}}%
    \put(0.29915332,0.12406922){\color[rgb]{0.10196078,0.10196078,0.10196078}\makebox(0,0)[lt]{\lineheight{1.25}\smash{\begin{tabular}[t]{l}shear\end{tabular}}}}%
    \put(0.45520032,0.24657062){\color[rgb]{0.10196078,0.10196078,0.10196078}\makebox(0,0)[lt]{\lineheight{1.25}\smash{\begin{tabular}[t]{l}$\lam$\end{tabular}}}}%
    \put(0,0){\includegraphics[width=\unitlength,page=3]{cutTsing.pdf}}%
    \put(0.29922993,0.2686099){\color[rgb]{0.10196078,0.10196078,0.10196078}\makebox(0,0)[lt]{\lineheight{1.25}\smash{\begin{tabular}[t]{l}$\rho$\end{tabular}}}}%
    \put(0,0){\includegraphics[width=\unitlength,page=4]{cutTsing.pdf}}%
    \put(0.59647789,0.05612849){\color[rgb]{0.10196078,0.10196078,0.10196078}\makebox(0,0)[lt]{\lineheight{1.25}\smash{\begin{tabular}[t]{l}$\rho$\end{tabular}}}}%
    \put(0,0){\includegraphics[width=\unitlength,page=5]{cutTsing.pdf}}%
    \put(0.84653695,0.24766552){\color[rgb]{0.10196078,0.10196078,0.10196078}\makebox(0,0)[lt]{\lineheight{1.25}\smash{\begin{tabular}[t]{l}$\lam$\end{tabular}}}}%
    \put(0.0991809,0.22735686){\color[rgb]{0.10196078,0.10196078,0.10196078}\makebox(0,0)[lt]{\lineheight{1.25}\smash{\begin{tabular}[t]{l}$\Phi(L)$\end{tabular}}}}%
    \put(0,0){\includegraphics[width=\unitlength,page=6]{cutTsing.pdf}}%
    \put(0.79094845,0.16548056){\color[rgb]{0.10196078,0.10196078,0.10196078}\makebox(0,0)[lt]{\lineheight{1.25}\smash{\begin{tabular}[t]{l}$P_b^\dual$\end{tabular}}}}%
    \put(0,0){\includegraphics[width=\unitlength,page=7]{cutTsing.pdf}}%
    \put(0.72704392,0.21186763){\color[rgb]{0.10196078,0.10196078,0.10196078}\makebox(0,0)[lt]{\lineheight{1.25}\smash{\begin{tabular}[t]{l}$\mu_{\on{out}}$\end{tabular}}}}%
    \put(0.88048769,0.10477661){\color[rgb]{0.10196078,0.10196078,0.10196078}\makebox(0,0)[lt]{\lineheight{1.25}\smash{\begin{tabular}[t]{l}$\mu_{\on{out}}$\end{tabular}}}}%
    \put(0.39526132,0.01800332){\color[rgb]{0.10196078,0.10196078,0.10196078}\makebox(0,0)[lt]{\lineheight{1.25}\smash{\begin{tabular}[t]{l}$B^\dual_\rho$\end{tabular}}}}%
    \put(0.78446512,0.03158951){\color[rgb]{0.10196078,0.10196078,0.10196078}\makebox(0,0)[lt]{\lineheight{1.25}\smash{\begin{tabular}[t]{l}$\A(X)$\end{tabular}}}}%
    \put(0,0){\includegraphics[width=\unitlength,page=8]{cutTsing.pdf}}%
    \put(0.69042597,0.1259793){\color[rgb]{0.10196078,0.10196078,0.10196078}\makebox(0,0)[lt]{\lineheight{1.25}\smash{\begin{tabular}[t]{l}shear\end{tabular}}}}%
  \end{picture}%
\endgroup%

%% file: curveT.pdf_tex
\begingroup%
  \makeatletter%
  \providecommand\color[2][]{%
    \errmessage{(Inkscape) Color is used for the text in Inkscape, but the package 'color.sty' is not loaded}%
    \renewcommand\color[2][]{}%
  }%
  \providecommand\transparent[1]{%
    \errmessage{(Inkscape) Transparency is used (non-zero) for the text in Inkscape, but the package 'transparent.sty' is not loaded}%
    \renewcommand\transparent[1]{}%
  }%
  \providecommand\rotatebox[2]{#2}%
  \newcommand*\fsize{\dimexpr\f@size pt\relax}%
  \newcommand*\lineheight[1]{\fontsize{\fsize}{#1\fsize}\selectfont}%
  \ifx\svgwidth\undefined%
    \setlength{\unitlength}{376.38976649bp}%
    \ifx\svgscale\undefined%
      \relax%
    \else%
      \setlength{\unitlength}{\unitlength * \real{\svgscale}}%
    \fi%
  \else%
    \setlength{\unitlength}{\svgwidth}%
  \fi%
  \global\let\svgwidth\undefined%
  \global\let\svgscale\undefined%
  \makeatother%
  \begin{picture}(1,0.3145114)%
    \lineheight{1}%
    \setlength\tabcolsep{0pt}%
    \put(0,0){\includegraphics[width=\unitlength,page=1]{curveT.pdf}}%
    \put(0.29646265,0.13614319){\color[rgb]{0.10196078,0.10196078,0.10196078}\makebox(0,0)[lt]{\lineheight{1.25}\smash{\begin{tabular}[t]{l}$\lam$\end{tabular}}}}%
    \put(0,0){\includegraphics[width=\unitlength,page=2]{curveT.pdf}}%
    \put(0.17087275,0.19389446){\color[rgb]{0.10196078,0.10196078,0.10196078}\makebox(0,0)[lt]{\lineheight{1.25}\smash{\begin{tabular}[t]{l}$P_b^\dual$\end{tabular}}}}%
    \put(0,0){\includegraphics[width=\unitlength,page=3]{curveT.pdf}}%
    \put(0.23940337,0.20987611){\color[rgb]{0.50196078,0,0.50196078}\makebox(0,0)[lt]{\lineheight{1.25}\smash{\begin{tabular}[t]{l}$\mu_b$\end{tabular}}}}%
    \put(0,0){\includegraphics[width=\unitlength,page=4]{curveT.pdf}}%
    \put(0.81900514,0.29001069){\color[rgb]{0.10196078,0.10196078,0.10196078}\makebox(0,0)[lt]{\lineheight{1.25}\smash{\begin{tabular}[t]{l}$\lam$\end{tabular}}}}%
    \put(0,0){\includegraphics[width=\unitlength,page=5]{curveT.pdf}}%
    \put(0.75344388,0.19308147){\color[rgb]{0.10196078,0.10196078,0.10196078}\makebox(0,0)[lt]{\lineheight{1.25}\smash{\begin{tabular}[t]{l}$P_b^\dual$\end{tabular}}}}%
    \put(0,0){\includegraphics[width=\unitlength,page=6]{curveT.pdf}}%
    \put(0.50960922,0.18230588){\color[rgb]{0.50196078,0,0.50196078}\makebox(0,0)[lt]{\lineheight{1.25}\smash{\begin{tabular}[t]{l}$\nu$\end{tabular}}}}%
    \put(0,0){\includegraphics[width=\unitlength,page=7]{curveT.pdf}}%
    \put(0.42228003,0.13051171){\color[rgb]{0.10196078,0.10196078,0.10196078}\makebox(0,0)[lt]{\lineheight{1.25}\smash{\begin{tabular}[t]{l}Cross wall in \\direction $\nu=(-p,-q)$\end{tabular}}}}%
    \put(0.03219993,0.13447149){\color[rgb]{0.10196078,0.10196078,0.10196078}\makebox(0,0)[lt]{\lineheight{1.25}\smash{\begin{tabular}[t]{l}$(\A_{-\eps}, \lam_{-\eps})$\end{tabular}}}}%
    \put(0.86947698,0.2093335){\color[rgb]{0.10196078,0.10196078,0.10196078}\makebox(0,0)[lt]{\lineheight{1.25}\smash{\begin{tabular}[t]{l}$(\A_\eps, \lam_\eps)=\A(X)$\end{tabular}}}}%
  \end{picture}%
\endgroup%

%% file: p2new.pdf_tex
\begingroup%
  \makeatletter%
  \providecommand\color[2][]{%
    \errmessage{(Inkscape) Color is used for the text in Inkscape, but the package 'color.sty' is not loaded}%
    \renewcommand\color[2][]{}%
  }%
  \providecommand\transparent[1]{%
    \errmessage{(Inkscape) Transparency is used (non-zero) for the text in Inkscape, but the package 'transparent.sty' is not loaded}%
    \renewcommand\transparent[1]{}%
  }%
  \providecommand\rotatebox[2]{#2}%
  \newcommand*\fsize{\dimexpr\f@size pt\relax}%
  \newcommand*\lineheight[1]{\fontsize{\fsize}{#1\fsize}\selectfont}%
  \ifx\svgwidth\undefined%
    \setlength{\unitlength}{306.57094574bp}%
    \ifx\svgscale\undefined%
      \relax%
    \else%
      \setlength{\unitlength}{\unitlength * \real{\svgscale}}%
    \fi%
  \else%
    \setlength{\unitlength}{\svgwidth}%
  \fi%
  \global\let\svgwidth\undefined%
  \global\let\svgscale\undefined%
  \makeatother%
  \begin{picture}(1,0.31663711)%
    \lineheight{1}%
    \setlength\tabcolsep{0pt}%
    \put(0,0){\includegraphics[width=\unitlength,page=1]{p2new.pdf}}%
    \put(0.19799503,0.20689505){\color[rgb]{0,0,1}\makebox(0,0)[lt]{\lineheight{1.25}\smash{\begin{tabular}[t]{l}1\end{tabular}}}}%
    \put(0.54049316,0.20963504){\color[rgb]{0,0,1}\makebox(0,0)[lt]{\lineheight{1.25}\smash{\begin{tabular}[t]{l}1\end{tabular}}}}%
    \put(0.88162139,0.20963504){\color[rgb]{0,0,1}\makebox(0,0)[lt]{\lineheight{1.25}\smash{\begin{tabular}[t]{l}1\end{tabular}}}}%
  \end{picture}%
\endgroup%

%% file: p1p1_torus.pdf_tex
\begingroup%
  \makeatletter%
  \providecommand\color[2][]{%
    \errmessage{(Inkscape) Color is used for the text in Inkscape, but the package 'color.sty' is not loaded}%
    \renewcommand\color[2][]{}%
  }%
  \providecommand\transparent[1]{%
    \errmessage{(Inkscape) Transparency is used (non-zero) for the text in Inkscape, but the package 'transparent.sty' is not loaded}%
    \renewcommand\transparent[1]{}%
  }%
  \providecommand\rotatebox[2]{#2}%
  \newcommand*\fsize{\dimexpr\f@size pt\relax}%
  \newcommand*\lineheight[1]{\fontsize{\fsize}{#1\fsize}\selectfont}%
  \ifx\svgwidth\undefined%
    \setlength{\unitlength}{503.79208374bp}%
    \ifx\svgscale\undefined%
      \relax%
    \else%
      \setlength{\unitlength}{\unitlength * \real{\svgscale}}%
    \fi%
  \else%
    \setlength{\unitlength}{\svgwidth}%
  \fi%
  \global\let\svgwidth\undefined%
  \global\let\svgscale\undefined%
  \makeatother%
  \begin{picture}(1,0.30521516)%
    \lineheight{1}%
    \setlength\tabcolsep{0pt}%
    \put(0,0){\includegraphics[width=\unitlength,page=1]{p1p1_torus.pdf}}%
    \put(0.78950446,0.18778393){\makebox(0,0)[lt]{\lineheight{1.25}\smash{\begin{tabular}[t]{l}w = 2,-2\end{tabular}}}}%
    \put(0.14482776,0.19590664){\makebox(0,0)[lt]{\lineheight{1.25}\smash{\begin{tabular}[t]{l}w = 2,0,0,-2\end{tabular}}}}%
    \put(0,0){\includegraphics[width=\unitlength,page=2]{p1p1_torus.pdf}}%
  \end{picture}%
\endgroup%

%% file: b3p2_tori.pdf_tex
\begingroup%
  \makeatletter%
  \providecommand\color[2][]{%
    \errmessage{(Inkscape) Color is used for the text in Inkscape, but the package 'color.sty' is not loaded}%
    \renewcommand\color[2][]{}%
  }%
  \providecommand\transparent[1]{%
    \errmessage{(Inkscape) Transparency is used (non-zero) for the text in Inkscape, but the package 'transparent.sty' is not loaded}%
    \renewcommand\transparent[1]{}%
  }%
  \providecommand\rotatebox[2]{#2}%
  \newcommand*\fsize{\dimexpr\f@size pt\relax}%
  \newcommand*\lineheight[1]{\fontsize{\fsize}{#1\fsize}\selectfont}%
  \ifx\svgwidth\undefined%
    \setlength{\unitlength}{161.48814011bp}%
    \ifx\svgscale\undefined%
      \relax%
    \else%
      \setlength{\unitlength}{\unitlength * \real{\svgscale}}%
    \fi%
  \else%
    \setlength{\unitlength}{\svgwidth}%
  \fi%
  \global\let\svgwidth\undefined%
  \global\let\svgscale\undefined%
  \makeatother%
  \begin{picture}(1,0.95601284)%
    \lineheight{1}%
    \setlength\tabcolsep{0pt}%
    \put(0,0){\includegraphics[width=\unitlength,page=1]{b3p2_tori.pdf}}%
    \put(0.56789079,0.87781381){\makebox(0,0)[lt]{\lineheight{1.25}\smash{\begin{tabular}[t]{l}w = 6,-2,-3\end{tabular}}}}%
  \end{picture}%
\endgroup%

%% file: b6p2_tori.pdf_tex
\begingroup%
  \makeatletter%
  \providecommand\color[2][]{%
    \errmessage{(Inkscape) Color is used for the text in Inkscape, but the package 'color.sty' is not loaded}%
    \renewcommand\color[2][]{}%
  }%
  \providecommand\transparent[1]{%
    \errmessage{(Inkscape) Transparency is used (non-zero) for the text in Inkscape, but the package 'transparent.sty' is not loaded}%
    \renewcommand\transparent[1]{}%
  }%
  \providecommand\rotatebox[2]{#2}%
  \newcommand*\fsize{\dimexpr\f@size pt\relax}%
  \newcommand*\lineheight[1]{\fontsize{\fsize}{#1\fsize}\selectfont}%
  \ifx\svgwidth\undefined%
    \setlength{\unitlength}{318.07230377bp}%
    \ifx\svgscale\undefined%
      \relax%
    \else%
      \setlength{\unitlength}{\unitlength * \real{\svgscale}}%
    \fi%
  \else%
    \setlength{\unitlength}{\svgwidth}%
  \fi%
  \global\let\svgwidth\undefined%
  \global\let\svgscale\undefined%
  \makeatother%
  \begin{picture}(1,1.00336711)%
    \lineheight{1}%
    \setlength\tabcolsep{0pt}%
    \put(0,0){\includegraphics[width=\unitlength,page=1]{b6p2_tori.pdf}}%
    \put(0.1191901,0.81942003){\makebox(0,0)[lt]{\lineheight{1.25}\smash{\begin{tabular}[t]{l}w=-6,21\end{tabular}}}}%
    \put(0,0){\includegraphics[width=\unitlength,page=2]{b6p2_tori.pdf}}%
  \end{picture}%
\endgroup%

%% file: b7p2_torus.pdf_tex
\begingroup%
  \makeatletter%
  \providecommand\color[2][]{%
    \errmessage{(Inkscape) Color is used for the text in Inkscape, but the package 'color.sty' is not loaded}%
    \renewcommand\color[2][]{}%
  }%
  \providecommand\transparent[1]{%
    \errmessage{(Inkscape) Transparency is used (non-zero) for the text in Inkscape, but the package 'transparent.sty' is not loaded}%
    \renewcommand\transparent[1]{}%
  }%
  \providecommand\rotatebox[2]{#2}%
  \newcommand*\fsize{\dimexpr\f@size pt\relax}%
  \newcommand*\lineheight[1]{\fontsize{\fsize}{#1\fsize}\selectfont}%
  \ifx\svgwidth\undefined%
    \setlength{\unitlength}{698.18847656bp}%
    \ifx\svgscale\undefined%
      \relax%
    \else%
      \setlength{\unitlength}{\unitlength * \real{\svgscale}}%
    \fi%
  \else%
    \setlength{\unitlength}{\svgwidth}%
  \fi%
  \global\let\svgwidth\undefined%
  \global\let\svgscale\undefined%
  \makeatother%
  \begin{picture}(1,0.42854022)%
    \lineheight{1}%
    \setlength\tabcolsep{0pt}%
    \put(0,0){\includegraphics[width=\unitlength,page=1]{b7p2_torus.pdf}}%
    \put(0.122527,0.23043661){\color[rgb]{0,0,1}\makebox(0,0)[lt]{\lineheight{1.25}\smash{\begin{tabular}[t]{l}(2,0)\end{tabular}}}}%
    \put(0.79575185,0.31011903){\makebox(0,0)[lt]{\lineheight{1.25}\smash{\begin{tabular}[t]{l}1 4 6 4 1 \\2    8 0 8 2\\ 1 4 6 4 1 \end{tabular}}}}%
    \put(0.80274706,0.14386966){\color[rgb]{0,0,1}\makebox(0,0)[lt]{\lineheight{1.25}\smash{\begin{tabular}[t]{l}w = -12,52\end{tabular}}}}%
    \put(0,0){\includegraphics[width=\unitlength,page=2]{b7p2_torus.pdf}}%
    \put(0.64194702,0.2492587){\color[rgb]{0,0,1}\transparent{0.84773701}\makebox(0,0)[lt]{\lineheight{1.25}\smash{\begin{tabular}[t]{l}slope (1,0)\end{tabular}}}}%
    \put(0.70371066,0.15975692){\color[rgb]{0,0,1}\transparent{0.84773701}\makebox(0,0)[lt]{\lineheight{1.25}\smash{\begin{tabular}[t]{l}slope (1,-1)\end{tabular}}}}%
    \put(0.67708204,0.09318533){\color[rgb]{0,0,1}\transparent{0.84773701}\makebox(0,0)[lt]{\lineheight{1.25}\smash{\begin{tabular}[t]{l}slope (2,-1)\end{tabular}}}}%
    \put(0,0){\includegraphics[width=\unitlength,page=3]{b7p2_torus.pdf}}%
    \put(0.79589085,0.35627528){\color[rgb]{0.50196078,0,0.50196078}\makebox(0,0)[lt]{\lineheight{1.25}\smash{\begin{tabular}[t]{l}Coefficients of $W_L$:\end{tabular}}}}%
    \put(0.80272667,0.17658949){\color[rgb]{0.50196078,0,0.50196078}\makebox(0,0)[lt]{\lineheight{1.25}\smash{\begin{tabular}[t]{l}Critical values:\end{tabular}}}}%
  \end{picture}%
\endgroup%

%% file: b8p2_torus.pdf_tex
\begingroup%
  \makeatletter%
  \providecommand\color[2][]{%
    \errmessage{(Inkscape) Color is used for the text in Inkscape, but the package 'color.sty' is not loaded}%
    \renewcommand\color[2][]{}%
  }%
  \providecommand\transparent[1]{%
    \errmessage{(Inkscape) Transparency is used (non-zero) for the text in Inkscape, but the package 'transparent.sty' is not loaded}%
    \renewcommand\transparent[1]{}%
  }%
  \providecommand\rotatebox[2]{#2}%
  \newcommand*\fsize{\dimexpr\f@size pt\relax}%
  \newcommand*\lineheight[1]{\fontsize{\fsize}{#1\fsize}\selectfont}%
  \ifx\svgwidth\undefined%
    \setlength{\unitlength}{340.81295013bp}%
    \ifx\svgscale\undefined%
      \relax%
    \else%
      \setlength{\unitlength}{\unitlength * \real{\svgscale}}%
    \fi%
  \else%
    \setlength{\unitlength}{\svgwidth}%
  \fi%
  \global\let\svgwidth\undefined%
  \global\let\svgscale\undefined%
  \makeatother%
  \begin{picture}(1,1.18721198)%
    \lineheight{1}%
    \setlength\tabcolsep{0pt}%
    \put(0,0){\includegraphics[width=\unitlength,page=1]{b8p2_torus.pdf}}%
    \put(0.25995223,0.59773169){\color[rgb]{0,0,1}\makebox(0,0)[lt]{\lineheight{1.25}\smash{\begin{tabular}[t]{l}(3,0)\end{tabular}}}}%
    \put(0.3541352,0.62014755){\color[rgb]{0,0,1}\makebox(0,0)[lt]{\lineheight{1.25}\smash{\begin{tabular}[t]{l}(0,2)\end{tabular}}}}%
    \put(0.22362582,0.67863378){\color[rgb]{0,0,1}\makebox(0,0)[lt]{\lineheight{1.25}\smash{\begin{tabular}[t]{l}(3,2)\end{tabular}}}}%
    \put(0.26812025,0.54702993){\color[rgb]{0,0,1}\makebox(0,0)[lt]{\lineheight{1.25}\smash{\begin{tabular}[t]{l}(3,-1)\end{tabular}}}}%
    \put(0.19871788,0.55005127){\color[rgb]{0,0,1}\makebox(0,0)[lt]{\lineheight{1.25}\smash{\begin{tabular}[t]{l}(2,-1)\end{tabular}}}}%
    \put(0,0){\includegraphics[width=\unitlength,page=2]{b8p2_torus.pdf}}%
    \put(0.78590062,0.58012672){\color[rgb]{0,0,1}\makebox(0,0)[lt]{\lineheight{1.25}\smash{\begin{tabular}[t]{l}(3,0)\end{tabular}}}}%
    \put(0.83549786,0.59635999){\color[rgb]{0,0,1}\makebox(0,0)[lt]{\lineheight{1.25}\smash{\begin{tabular}[t]{l}(0,1)\end{tabular}}}}%
    \put(0.87701295,0.6046193){\color[rgb]{0,0,1}\makebox(0,0)[lt]{\lineheight{1.25}\smash{\begin{tabular}[t]{l}(0,1)\end{tabular}}}}%
    \put(0.31706998,0.73881485){\color[rgb]{0,0,1}\transparent{0.84773701}\makebox(0,0)[lt]{\lineheight{1.25}\smash{\begin{tabular}[t]{l}1\end{tabular}}}}%
    \put(0.39099815,0.64024386){\color[rgb]{0,0,1}\transparent{0.84773701}\makebox(0,0)[lt]{\lineheight{1.25}\smash{\begin{tabular}[t]{l}9/2\end{tabular}}}}%
    \put(0.39017668,0.5055302){\color[rgb]{0,0,1}\transparent{0.84773701}\makebox(0,0)[lt]{\lineheight{1.25}\smash{\begin{tabular}[t]{l}1\end{tabular}}}}%
    \put(0.36717681,0.44885193){\color[rgb]{0,0,1}\transparent{0.84773701}\makebox(0,0)[lt]{\lineheight{1.25}\smash{\begin{tabular}[t]{l}9\end{tabular}}}}%
    \put(0.90438854,0.63942246){\color[rgb]{0,0,1}\transparent{0.84773701}\makebox(0,0)[lt]{\lineheight{1.25}\smash{\begin{tabular}[t]{l}-3/2\end{tabular}}}}%
    \put(0,0){\includegraphics[width=\unitlength,page=3]{b8p2_torus.pdf}}%
    \put(0.67943633,1.09611437){\color[rgb]{0,0,1}\makebox(0,0)[lt]{\lineheight{1.25}\smash{\begin{tabular}[t]{l}Critical values\\w = -60, 372\end{tabular}}}}%
  \end{picture}%
\endgroup%

%% file: b6p2_exc.pdf_tex
\begingroup%
  \makeatletter%
  \providecommand\color[2][]{%
    \errmessage{(Inkscape) Color is used for the text in Inkscape, but the package 'color.sty' is not loaded}%
    \renewcommand\color[2][]{}%
  }%
  \providecommand\transparent[1]{%
    \errmessage{(Inkscape) Transparency is used (non-zero) for the text in Inkscape, but the package 'transparent.sty' is not loaded}%
    \renewcommand\transparent[1]{}%
  }%
  \providecommand\rotatebox[2]{#2}%
  \newcommand*\fsize{\dimexpr\f@size pt\relax}%
  \newcommand*\lineheight[1]{\fontsize{\fsize}{#1\fsize}\selectfont}%
  \ifx\svgwidth\undefined%
    \setlength{\unitlength}{841.15631104bp}%
    \ifx\svgscale\undefined%
      \relax%
    \else%
      \setlength{\unitlength}{\unitlength * \real{\svgscale}}%
    \fi%
  \else%
    \setlength{\unitlength}{\svgwidth}%
  \fi%
  \global\let\svgwidth\undefined%
  \global\let\svgscale\undefined%
  \makeatother%
  \begin{picture}(1,0.37941018)%
    \lineheight{1}%
    \setlength\tabcolsep{0pt}%
    \put(0,0){\includegraphics[width=\unitlength,page=1]{b6p2_exc.pdf}}%
    \put(0.268451,0.127456){\color[rgb]{0,0,1}\transparent{0.76363599}\makebox(0,0)[lt]{\lineheight{1.25}\smash{\begin{tabular}[t]{l}E(X) = (3)(3)(3) = 27\end{tabular}}}}%
    \put(0,0){\includegraphics[width=\unitlength,page=2]{b6p2_exc.pdf}}%
  \end{picture}%
\endgroup%

%% file: dp.bib
@article {cps:trop,
    AUTHOR = {Carl, Michael and Pumperla, Max and Siebert, Bernd},
     TITLE = {A tropical view on {L}andau-{G}inzburg models},
   JOURNAL = {Acta Math. Sin. (Engl. Ser.)},
  FJOURNAL = {Acta Mathematica Sinica (English Series)},
    VOLUME = {40},
      YEAR = {2024},
    NUMBER = {1},
     PAGES = {329--382},
      ISSN = {1439-8516},
   MRCLASS = {14J33 (14J32 14T20 32Q25 53D37)},
  MRNUMBER = {4685231},
       DOI = {10.1007/s10114-024-2295-y},
       URL = {https://doi-org.proxy.libraries.rutgers.edu/10.1007/s10114-024-2295-y},
}

@article {ns,
    AUTHOR = {Nishinou, Takeo and Siebert, Bernd},
     TITLE = {Toric degenerations of toric varieties and tropical curves},
   JOURNAL = {Duke Math. J.},
  FJOURNAL = {Duke Mathematical Journal},
    VOLUME = {135},
      YEAR = {2006},
    NUMBER = {1},
     PAGES = {1--51},
      ISSN = {0012-7094},
   MRCLASS = {14N10 (14M25 14N35)},
  MRNUMBER = {2259922},
MRREVIEWER = {Diego Matessi},
       DOI = {10.1215/S0012-7094-06-13511-1}
}

@article {pt:wall,
    AUTHOR = {Pascaleff, James and Tonkonog, Dmitry},
     TITLE = {The wall-crossing formula and {L}agrangian mutations},
   JOURNAL = {Adv. Math.},
  FJOURNAL = {Advances in Mathematics},
    VOLUME = {361},
      YEAR = {2020},
     PAGES = {106850, 67},
      ISSN = {0001-8708},
   MRCLASS = {53D12 (14J33 53D40)},
  MRNUMBER = {4043009},
MRREVIEWER = {Jun Zhang},
       DOI = {10.1016/j.aim.2019.106850},
       URL = {https://doi-org.proxy.libraries.rutgers.edu/10.1016/j.aim.2019.106850},
}

@article {flips,
    AUTHOR = {Charest, Fran\c{c}ois and Woodward, Chris T.},
     TITLE = {Floer cohomology and flips},
   JOURNAL = {Mem. Amer. Math. Soc.},
  FJOURNAL = {Memoirs of the American Mathematical Society},
    VOLUME = {279},
      YEAR = {2022},
    NUMBER = {1372},
     PAGES = {v+166},
      ISSN = {0065-9266},
      ISBN = {978-1-4704-5310-7; 978-1-4704-7226-9},
   MRCLASS = {53D40},
  MRNUMBER = {4464438},
       DOI = {10.1090/memo/1372},
       URL = {https://doi-org.proxy.libraries.rutgers.edu/10.1090/memo/1372},
}

@book {ms:jh,
    AUTHOR = {McDuff, Dusa and Salamon, Dietmar},
     TITLE = {{$J$}-holomorphic curves and symplectic topology},
    SERIES = {American Mathematical Society Colloquium Publications},
    VOLUME = {52},
   EDITION = {Second},
 PUBLISHER = {American Mathematical Society, Providence, RI},
      YEAR = {2012},
     PAGES = {xiv+726},
      ISBN = {978-0-8218-8746-2},
   MRCLASS = {53D45 (32Q65 53D35)},
  MRNUMBER = {2954391},
MRREVIEWER = {Mark Alan Branson},
}

@article {gi:eq,
    AUTHOR = {Givental, Alexander B.},
     TITLE = {Equivariant {G}romov-{W}itten invariants},
   JOURNAL = {Internat. Math. Res. Notices},
  FJOURNAL = {International Mathematics Research Notices},
      YEAR = {1996},
    NUMBER = {13},
     PAGES = {613--663},
      ISSN = {1073-7928},
   MRCLASS = {14D07 (14D05 14J32 14N10 32G20)},
  MRNUMBER = {1408320},
MRREVIEWER = {Claire Voisin},
       DOI = {10.1155/S1073792896000414},
       URL = {https://doi-org.proxy.libraries.rutgers.edu/10.1155/S1073792896000414},
}

@incollection {mcduff:isotopy,
    AUTHOR = {McDuff, Dusa},
     TITLE = {From symplectic deformation to isotopy},
 BOOKTITLE = {Topics in symplectic {$4$}-manifolds ({I}rvine, {CA}, 1996)},
    SERIES = {First Int. Press Lect. Ser., I},
     PAGES = {85--99},
 PUBLISHER = {Int. Press, Cambridge, MA},
      YEAR = {1998},
   MRCLASS = {57R15 (57R52 57R57)},
  MRNUMBER = {1635697},
MRREVIEWER = {Paolo Lisca},
}

@article {chooh:fano,
    AUTHOR = {Cho, Cheol-Hyun and Oh, Yong-Geun},
     TITLE = {Floer cohomology and disc instantons of {L}agrangian torus
              fibers in {F}ano toric manifolds},
   JOURNAL = {Asian J. Math.},
  FJOURNAL = {Asian Journal of Mathematics},
    VOLUME = {10},
      YEAR = {2006},
    NUMBER = {4},
     PAGES = {773--814},
      ISSN = {1093-6106},
   MRCLASS = {53D40 (14J45 14J81 14M25 53D12 81T40)},
  MRNUMBER = {2282365},
MRREVIEWER = {Chien-Hao Liu},
       DOI = {10.4310/AJM.2006.v10.n4.a10},
       URL = {https://doi-org.proxy.libraries.rutgers.edu/10.4310/AJM.2006.v10.n4.a10},
}

@article {cm:trans,
    AUTHOR = {Cieliebak, Kai and Mohnke, Klaus},
     TITLE = {Symplectic hypersurfaces and transversality in
              {G}romov-{W}itten theory},
   JOURNAL = {J. Symplectic Geom.},
  FJOURNAL = {The Journal of Symplectic Geometry},
    VOLUME = {5},
      YEAR = {2007},
    NUMBER = {3},
     PAGES = {281--356},
      ISSN = {1527-5256},
   MRCLASS = {53D45},
  MRNUMBER = {2399678},
MRREVIEWER = {Michael J. Usher},
       URL = {http://projecteuclid.org.proxy.libraries.rutgers.edu/euclid.jsg/1210083200},
}

@article {lutz:dp,
    AUTHOR = {Lutz, Wendelin},
     TITLE = {Mirrors to del {P}ezzo surfaces and the classification of
              {$T$}-polygons},
   JOURNAL = {SIGMA Symmetry Integrability Geom. Methods Appl.},
  FJOURNAL = {SIGMA. Symmetry, Integrability and Geometry. Methods and
              Applications},
    VOLUME = {20},
      YEAR = {2024},
     PAGES = {Paper No. 095, 20},
   MRCLASS = {14J33 (14E07)},
  MRNUMBER = {4843384},
       DOI = {10.3842/SIGMA.2024.095},
       URL = {https://doi.org/10.3842/SIGMA.2024.095},
}

@book {gui:mom,
    AUTHOR = {Guillemin, Victor},
     TITLE = {Moment maps and combinatorial invariants of {H}amiltonian
              {$T^n$}-spaces},
    SERIES = {Progress in Mathematics},
    VOLUME = {122},
 PUBLISHER = {Birkh\"{a}user Boston, Inc., Boston, MA},
      YEAR = {1994},
     PAGES = {viii+150},
      ISBN = {0-8176-3770-2},
   MRCLASS = {58F06 (14C40 35L10 58F05)},
  MRNUMBER = {1301331},
MRREVIEWER = {Alejandro Uribe},
       DOI = {10.1007/978-1-4612-0269-1},
       URL = {https://doi-org.proxy.libraries.rutgers.edu/10.1007/978-1-4612-0269-1},
}

@article {we:remove,
    AUTHOR = {Weinstein, Alan},
     TITLE = {Removing intersections of {L}agrangian immersions},
   JOURNAL = {Illinois J. Math.},
  FJOURNAL = {Illinois Journal of Mathematics},
    VOLUME = {27},
      YEAR = {1983},
    NUMBER = {3},
     PAGES = {484--500},
      ISSN = {0019-2082},
   MRCLASS = {58F05 (57R42)},
  MRNUMBER = {698310},
MRREVIEWER = {I. Vaisman},
}

@article{vw:split,
      title={Splitting the diagonal for broken maps}, 
      author={Sushmita Venugopalan and Chris Woodward},
       Journal={arXiv: 2504.15583},
      year={2025}
}

@misc{horivafa,
      title={Mirror Symmetry}, 
      author={Kentaro Hori and Cumrun Vafa},
      year={2000},
      eprint={hep-th/0002222},
      archivePrefix={arXiv},
      primaryClass={hep-th},
      url={https://arxiv.org/abs/hep-th/0002222}, 
}

@article{lau:syz,
      title={SYZ mirror symmetry for del Pezzo surfaces and affine structures}, 
      author={Siu-Cheong Lau and Tsung-Ju Lee and Yu-Shen Lin},
            Journal={arXiv: 2206.01681},
      year={2024}
}

@book {evans:lec,
    AUTHOR = {Evans, Jonny},
     TITLE = {Lectures on {L}agrangian torus fibrations},
    SERIES = {London Mathematical Society Student Texts},
    VOLUME = {105},
 PUBLISHER = {Cambridge University Press, Cambridge},
      YEAR = {2023},
     PAGES = {xiii+225},
      ISBN = {978-1-009-37262-6; 978-1-009-37263-3},
   MRCLASS = {53D12 (53D20 57S12)},
  MRNUMBER = {4605050},
       DOI = {10.1017/9781009372671},
       URL = {https://doi.org/10.1017/9781009372671},
}

@article {pasc,
    AUTHOR = {Pascaleff, James and Tonkonog, Dmitry},
     TITLE = {The wall-crossing formula and {L}agrangian mutations},
   JOURNAL = {Adv. Math.},
  FJOURNAL = {Advances in Mathematics},
    VOLUME = {361},
      YEAR = {2020},
     PAGES = {106850, 67},
      ISSN = {0001-8708,1090-2082},
   MRCLASS = {53D12 (14J33 53D40)},
  MRNUMBER = {4043009},
MRREVIEWER = {Jun\ Zhang},
       DOI = {10.1016/j.aim.2019.106850},
       URL = {https://doi.org/10.1016/j.aim.2019.106850},
}

@article {rittersmith,
    AUTHOR = {Ritter, Alexander F. and Smith, Ivan},
     TITLE = {The monotone wrapped {F}ukaya category and the open-closed
              string map},
   JOURNAL = {Selecta Math. (N.S.)},
  FJOURNAL = {Selecta Mathematica. New Series},
    VOLUME = {23},
      YEAR = {2017},
    NUMBER = {1},
     PAGES = {533--642},
      ISSN = {1022-1824},
   MRCLASS = {53D37 (14J33 14N35 53D40 53D45)},
  MRNUMBER = {3595902},
MRREVIEWER = {Lino Amorim},
       DOI = {10.1007/s00029-016-0255-9},
       URL = {https://doi.org/10.1007/s00029-016-0255-9},
}

@article{vw:trop,
  author = {Venugopalan, Sushmita and Woodward, Chris},
  title = {Tropical {F}ukaya Algebras},
  Journal={arXiv: 2004.14314},
  year={2020}
}

@misc{vw:tl,
  author = {Woodward, Chris},
  title = {Disks bounding tropical Lagrangians in almost toric manifolds},
note = {In preparation}
}

@inproceedings {giv:icm,
    AUTHOR = {Givental, Alexander B.},
     TITLE = {Homological geometry and mirror symmetry},
 BOOKTITLE = {Proceedings of the {I}nternational {C}ongress of
              {M}athematicians, {V}ol. 1, 2 ({Z}\"{u}rich, 1994)},
     PAGES = {472--480},
 PUBLISHER = {Birkh\"{a}user, Basel},
      YEAR = {1995},
   MRCLASS = {58D10 (14J40 14N10 32G20)},
  MRNUMBER = {1403947},
MRREVIEWER = {Bruce Hunt},
}

@article {mikhalkin,
    AUTHOR = {Mikhalkin, Grigory},
     TITLE = {Enumerative tropical algebraic geometry in {$\Bbb R^2$}},
   JOURNAL = {J. Amer. Math. Soc.},
  FJOURNAL = {Journal of the American Mathematical Society},
    VOLUME = {18},
      YEAR = {2005},
    NUMBER = {2},
     PAGES = {313--377},
      ISSN = {0894-0347},
   MRCLASS = {14N10 (05A99 14N35 52B70)},
  MRNUMBER = {2137980},
MRREVIEWER = {Charles D. Cadman},
       DOI = {10.1090/S0894-0347-05-00477-7},
       URL = {https://doi.org/10.1090/S0894-0347-05-00477-7},
}

@article {akhtar,
    AUTHOR = {Akhtar, Mohammad and Coates, Tom and Corti, Alessio and
              Heuberger, Liana and Kasprzyk, Alexander and Oneto, Alessandro
              and Petracci, Andrea and Prince, Thomas and Tveiten, Ketil},
     TITLE = {Mirror symmetry and the classification of orbifold del {P}ezzo
              surfaces},
   JOURNAL = {Proc. Amer. Math. Soc.},
  FJOURNAL = {Proceedings of the American Mathematical Society},
    VOLUME = {144},
      YEAR = {2016},
    NUMBER = {2},
     PAGES = {513--527},
      ISSN = {0002-9939},
   MRCLASS = {14J26 (52B20)},
  MRNUMBER = {3430830},
MRREVIEWER = {Makiko Mase},
       DOI = {10.1090/proc/12876},
       URL = {https://doi-org.proxy.libraries.rutgers.edu/10.1090/proc/12876},
}

@book {manin:cubic,
    AUTHOR = {Manin, Yu. I.},
     TITLE = {Cubic forms},
    SERIES = {North-Holland Mathematical Library},
    VOLUME = {4},
   EDITION = {Second},
      NOTE = {Algebra, geometry, arithmetic,
              Translated from the Russian by M. Hazewinkel},
 PUBLISHER = {North-Holland Publishing Co., Amsterdam},
      YEAR = {1986},
     PAGES = {x+326},
      ISBN = {0-444-87823-8},
   MRCLASS = {11Gxx (14Gxx 14J20)},
  MRNUMBER = {833513},
}

@article {vwx,
    AUTHOR = {Venugopalan, Sushmita and Woodward, Chris T. and Xu, Guangbo},
     TITLE = {Fukaya categories of blowups},
   JOURNAL = {J. Inst. Math. Jussieu},
  FJOURNAL = {Journal of the Institute of Mathematics of Jussieu. JIMJ.
              Journal de l'Institut de Math\'{e}matiques de Jussieu},
    VOLUME = {25},
      YEAR = {2026},
    NUMBER = {1},
     PAGES = {85--214},
      ISSN = {1474-7480,1475-3030},
   MRCLASS = {53D37 (53D40)},
  MRNUMBER = {5018881},
       DOI = {10.1017/S1474748025101138},
       URL = {https://doi-org.proxy.libraries.rutgers.edu/10.1017/S1474748025101138},
}

@article {serg:adj,
    AUTHOR = {Serganova, Vera V. and Skorobogatov, Alexei N.},
     TITLE = {Adjoint representation of {${\rm E}_8$} and del {P}ezzo
              surfaces of degree 1},
   JOURNAL = {Ann. Inst. Fourier (Grenoble)},
  FJOURNAL = {Universit\'{e} de Grenoble. Annales de l'Institut Fourier},
    VOLUME = {61},
      YEAR = {2011},
    NUMBER = {6},
     PAGES = {2337--2360 (2012)},
      ISSN = {0373-0956,1777-5310},
   MRCLASS = {14J26 (14M17 22E46)},
  MRNUMBER = {2976314},
MRREVIEWER = {Ulrich\ Derenthal},
       DOI = {10.5802/aif.2676},
       URL = {https://doi-org.proxy.libraries.rutgers.edu/10.5802/aif.2676},
}

@incollection {ohtaono,
    AUTHOR = {Ohta, Hiroshi and Ono, Kaoru},
     TITLE = {Symplectic {$4$}-manifolds with {$b^+_2=1$}},
 BOOKTITLE = {Geometry and physics ({A}arhus, 1995)},
    SERIES = {Lecture Notes in Pure and Appl. Math.},
    VOLUME = {184},
     PAGES = {237--244},
 PUBLISHER = {Dekker, New York},
      YEAR = {1997},
      ISBN = {0-8247-9791-4},
   MRCLASS = {57R15 (57R57)},
  MRNUMBER = {1423170},
MRREVIEWER = {Paolo\ Lisca},
}

@article {le:stacks,
    AUTHOR = {Lerman, Eugene},
     TITLE = {Orbifolds as stacks?},
   JOURNAL = {Enseign. Math. (2)},
  FJOURNAL = {L'Enseignement Math\'{e}matique. Revue Internationale. 2e
              S\'{e}rie},
    VOLUME = {56},
      YEAR = {2010},
    NUMBER = {3-4},
     PAGES = {315--363},
      ISSN = {0013-8584},
   MRCLASS = {18D05 (22A22)},
  MRNUMBER = {2778793},
MRREVIEWER = {Chenchang\ Zhu},
       DOI = {10.4171/LEM/56-3-4},
       URL = {https://doi-org.proxy.libraries.rutgers.edu/10.4171/LEM/56-3-4},
}

@article {gu:ka,
    AUTHOR = {Guillemin, Victor},
     TITLE = {Kaehler structures on toric varieties},
   JOURNAL = {J. Differential Geom.},
  FJOURNAL = {Journal of Differential Geometry},
    VOLUME = {40},
      YEAR = {1994},
    NUMBER = {2},
     PAGES = {285--309},
      ISSN = {0022-040X,1945-743X},
   MRCLASS = {32J27 (14M25 32C17 32M05 58F05)},
  MRNUMBER = {1293656},
MRREVIEWER = {Ivailo\ M.\ Mladenov},
       URL = {http://projecteuclid.org.proxy.libraries.rutgers.edu/euclid.jdg/1214455538},
}

@article {cm:dt,
    AUTHOR = {Cheung, Man-Wai and Mandel, Travis},
     TITLE = {Donaldson-{T}homas invariants from tropical disks},
   JOURNAL = {Selecta Math. (N.S.)},
  FJOURNAL = {Selecta Mathematica. New Series},
    VOLUME = {26},
      YEAR = {2020},
    NUMBER = {4},
     PAGES = {Paper No. 57, 46},
      ISSN = {1022-1824,1420-9020},
   MRCLASS = {14N35 (13F60 14N10 14T20)},
  MRNUMBER = {4131036},
MRREVIEWER = {Joaquim\ Ro\'{e}},
       DOI = {10.1007/s00029-020-00580-8},
       URL = {https://doi.org/10.1007/s00029-020-00580-8},
}

@article{vianna:dp,
      title={Infinitely many monotone {L}agrangian tori in del {P}ezzo surfaces}, 
      author={Renato Vianna},
      Journal={arXiv: 1602.03356},
      year={2016}
}

@article {vianna:inf,
    AUTHOR = {Vianna, Renato Ferreira de Velloso},
     TITLE = {Infinitely many exotic monotone {L}agrangian tori in
              {$\Bbb{CP}^2$}},
   JOURNAL = {J. Topol.},
  FJOURNAL = {Journal of Topology},
    VOLUME = {9},
      YEAR = {2016},
    NUMBER = {2},
     PAGES = {535--551},
      ISSN = {1753-8416},
   MRCLASS = {53D12 (53D37 53D42)},
  MRNUMBER = {3509972},
MRREVIEWER = {Stefan Nemirovski},
       DOI = {10.1112/jtopol/jtw002},
       URL = {https://doi-org.proxy.libraries.rutgers.edu/10.1112/jtopol/jtw002},
}

@article{lin:trop,
      title={On the Tropical Discs Counting on Elliptic K3 Surfaces with General Singular Fibres}, 
      author={Yu-Shen Lin},
      Journal={arXiv: 1710.10625},
      year={2019}
}

@article {lin:open,
    AUTHOR = {Lin, Yu-Shen},
     TITLE = {Open {G}romov-{W}itten invariants on elliptic {K}3 surfaces
              and wall-crossing},
   JOURNAL = {Comm. Math. Phys.},
  FJOURNAL = {Communications in Mathematical Physics},
    VOLUME = {349},
      YEAR = {2017},
    NUMBER = {1},
     PAGES = {109--164},
      ISSN = {0010-3616},
   MRCLASS = {14N35 (14J27 14J28 32Q65 53D45)},
  MRNUMBER = {3592747},
MRREVIEWER = {Feng Qu},
       DOI = {10.1007/s00220-016-2754-0},
       URL = {https://doi-org.proxy.libraries.rutgers.edu/10.1007/s00220-016-2754-0},
}

@article {leungsym,
    AUTHOR = {Leung, Naichung Conan and Symington, Margaret},
     TITLE = {Almost toric symplectic four-manifolds},
   JOURNAL = {J. Symplectic Geom.},
  FJOURNAL = {The Journal of Symplectic Geometry},
    VOLUME = {8},
      YEAR = {2010},
    NUMBER = {2},
     PAGES = {143--187},
      ISSN = {1527-5256,1540-2347},
   MRCLASS = {53D35 (53D12 57R17)},
  MRNUMBER = {2670163},
MRREVIEWER = {Vicente\ Mu\~{n}oz},
       URL = {http://projecteuclid.org/euclid.jsg/1279199213},
}

@article {cw:traj,
    AUTHOR = {Charest, Fran\c{c}ois and Woodward, Chris},
     TITLE = {Floer trajectories and stabilizing divisors},
   JOURNAL = {J. Fixed Point Theory Appl.},
  FJOURNAL = {Journal of Fixed Point Theory and Applications},
    VOLUME = {19},
      YEAR = {2017},
    NUMBER = {2},
     PAGES = {1165--1236},
      ISSN = {1661-7738},
   MRCLASS = {53D40},
  MRNUMBER = {3659006},
MRREVIEWER = {Jelena Kati\'{c}},
       DOI = {10.1007/s11784-016-0379-8},
       URL = {https://doi-org.proxy.libraries.rutgers.edu/10.1007/s11784-016-0379-8},
}

@article {graef:trop,
    AUTHOR = {Graefnitz, Tim},
     TITLE = {Tropical correspondence for smooth del {P}ezzo log
              {C}alabi-{Y}au pairs},
   JOURNAL = {J. Algebraic Geom.},
  FJOURNAL = {Journal of Algebraic Geometry},
    VOLUME = {31},
      YEAR = {2022},
    NUMBER = {4},
     PAGES = {687--749},
      ISSN = {1056-3911,1534-7486},
   MRCLASS = {14N35 (14T20)},
  MRNUMBER = {4484551},
MRREVIEWER = {Benjamin\ Gammage},
}

@article{testa:thesis,
      title={The irreducibility of the spaces of rational curves on del {P}ezzo surfaces}, 
      author={Damiano Testa},
      Journal={arXiv: math/0609355},
      year={2006}
}

@article {fooo:toric,
    AUTHOR = {Fukaya, Kenji and Oh, Yong-Geun and Ohta, Hiroshi and Ono,
              Kaoru},
     TITLE = {Lagrangian {F}loer theory and mirror symmetry on compact toric
              manifolds},
   JOURNAL = {Ast\'{e}risque},
  FJOURNAL = {Ast\'{e}risque},
    NUMBER = {376},
      YEAR = {2016},
     PAGES = {vi+340},
      ISSN = {0303-1179},
      ISBN = {978-2-85629-825-1},
   MRCLASS = {53D40 (14M25 53D37 53D45)},
  MRNUMBER = {3460884},
MRREVIEWER = {Christopher T. Woodward},
}

@article {coates:max,
    AUTHOR = {Coates, Tom and Kasprzyk, Alexander M. and Pitton, Giuseppe
              and Tveiten, Ketil},
     TITLE = {Maximally mutable {L}aurent polynomials},
   JOURNAL = {Proc. A.},
  FJOURNAL = {Proceedings A},
    VOLUME = {477},
      YEAR = {2021},
    NUMBER = {2254},
     PAGES = {Paper No. 20210584, 21},
      ISSN = {1364-5021},
   MRCLASS = {14J33 (14J45)},
  MRNUMBER = {4340449},
MRREVIEWER = {Hoil Kim},
}

@article{li2023toric,
      title={Almost toric presentations of symplectic log Calabi-Yau pairs}, 
      author={Tian-Jun Li and Jie Min and Shengzhen Ning},
      Journal={arXiv: 2303.09964},
      year={2023}
}

@book {gross:trop,
    AUTHOR = {Gross, Mark},
     TITLE = {Tropical geometry and mirror symmetry},
    SERIES = {CBMS Regional Conference Series in Mathematics},
    VOLUME = {114},
 PUBLISHER = {Published for the Conference Board of the Mathematical
              Sciences, Washington, DC; by the American Mathematical
              Society, Providence, RI},
      YEAR = {2011},
     PAGES = {xvi+317},
      ISBN = {978-0-8218-5232-3},
   MRCLASS = {14T05 (14J32 14J33 52B20)},
  MRNUMBER = {2722115},
MRREVIEWER = {Hsian-Hua Tseng},
       DOI = {10.1090/cbms/114},
       URL = {https://doi-org.proxy.libraries.rutgers.edu/10.1090/cbms/114},
}

@article {gr:corresp,
    AUTHOR = {Gr\"{a}fnitz, Tim},
     TITLE = {Theta functions, broken lines and 2-marked log
              {G}romov-{W}itten invariants},
   JOURNAL = {Manuscripta Math.},
  FJOURNAL = {Manuscripta Mathematica},
    VOLUME = {176},
      YEAR = {2025},
    NUMBER = {4},
     PAGES = {Paper No. 41, 34},
      ISSN = {0025-2611},
   MRCLASS = {14N35 (05E14 14A21 14J33 14K25 14N10 14T20)},
  MRNUMBER = {4921699},
       DOI = {10.1007/s00229-025-01640-z},
       URL = {https://doi-org.proxy.libraries.rutgers.edu/10.1007/s00229-025-01640-z},
}

@article {grz:proper,
    AUTHOR = {Gr\"{a}fnitz, Tim and Ruddat, Helge and Zaslow, Eric},
     TITLE = {The proper {L}andau-{G}inzburg potential is the open mirror
              map},
   JOURNAL = {Adv. Math.},
  FJOURNAL = {Advances in Mathematics},
    VOLUME = {447},
      YEAR = {2024},
     PAGES = {Paper No. 109639, 69},
      ISSN = {0001-8708},
   MRCLASS = {14J33 (13F60 14J45 14N10 14T20 53D37)},
  MRNUMBER = {4739248},
MRREVIEWER = {Hoil Kim},
       DOI = {10.1016/j.aim.2024.109639},
       URL = {https://doi-org.proxy.libraries.rutgers.edu/10.1016/j.aim.2024.109639},
}

@incollection {sym,
    AUTHOR = {Symington, Margaret},
     TITLE = {Four dimensions from two in symplectic topology},
 BOOKTITLE = {Topology and geometry of manifolds ({A}thens, {GA}, 2001)},
    SERIES = {Proc. Sympos. Pure Math.},
    VOLUME = {71},
     PAGES = {153--208},
 PUBLISHER = {Amer. Math. Soc., Providence, RI},
      YEAR = {2003},
   MRCLASS = {53D35 (53D20 55R55 57R17)},
  MRNUMBER = {2024634},
MRREVIEWER = {Vicente Mu\~{n}oz},
       DOI = {10.1090/pspum/071/2024634},
       URL = {https://doi-org.proxy.libraries.rutgers.edu/10.1090/pspum/071/2024634},
}

@article{parker:blowups,
    AUTHOR = {Parker, Brett},
     TITLE = {Tropical enumeration of curves in blowups of {$\Bbb{C} P^2$}},
   JOURNAL = {J. Differential Geom.},
  FJOURNAL = {Journal of Differential Geometry},
    VOLUME = {129},
      YEAR = {2025},
    NUMBER = {1},
     PAGES = {165--223},
      ISSN = {0022-040X},
   MRCLASS = {14N35 (14T20 53D45)},
  MRNUMBER = {4856594},
       DOI = {10.4310/jdg/1736262182},
       URL = {https://doi-org.proxy.libraries.rutgers.edu/10.4310/jdg/1736262182},
}

@article {oh:lag,
    AUTHOR = {Oh, Yong-Geun},
     TITLE = {Floer cohomology of {L}agrangian intersections and
              pseudo-holomorphic disks. {I}},
   JOURNAL = {Comm. Pure Appl. Math.},
  FJOURNAL = {Communications on Pure and Applied Mathematics},
    VOLUME = {46},
      YEAR = {1993},
    NUMBER = {7},
     PAGES = {949--993},
      ISSN = {0010-3640,1097-0312},
   MRCLASS = {58E05 (57R99 58F05)},
  MRNUMBER = {1223659},
MRREVIEWER = {Wilhelm\ Klingenberg},
       DOI = {10.1002/cpa.3160460702},
       URL = {https://doi.org/10.1002/cpa.3160460702},
}

@article{bardwellevans,
    AUTHOR = {Bardwell-Evans, Sam and Cheung, Man-Wai Mandy and Hong, Hansol
              and Lin, Yu-Shen},
     TITLE = {Scattering diagrams from holomorphic discs in log
              {C}alabi-{Y}au surfaces},
   JOURNAL = {J. Differential Geom.},
  FJOURNAL = {Journal of Differential Geometry},
    VOLUME = {130},
      YEAR = {2025},
    NUMBER = {1},
     PAGES = {--},
      ISSN = {0022-040X},
   MRCLASS = {14J33 (13F60 14T20)},
  MRNUMBER = {4904497},
       DOI = {10.4310/jdg/1747064192},
       URL = {https://doi-org.proxy.libraries.rutgers.edu/10.4310/jdg/1747064192}}

@article {barrott:explicit,
    AUTHOR = {Barrott, Lawrence Jack},
     TITLE = {Explicit equations for mirror families to log {C}alabi-{Y}au
              surfaces},
   JOURNAL = {Bull. Korean Math. Soc.},
  FJOURNAL = {Bulletin of the Korean Mathematical Society},
    VOLUME = {57},
      YEAR = {2020},
    NUMBER = {1},
     PAGES = {139--165},
      ISSN = {1015-8634,2234-3016},
   MRCLASS = {14J33 (14J26)},
  MRNUMBER = {4060189},
MRREVIEWER = {Shengtian\ Zhou},
       DOI = {10.4134/BKMS.b190128},
       URL = {https://doi.org/10.4134/BKMS.b190128},
}

@article {duist:global,
    AUTHOR = {Duistermaat, J. J.},
     TITLE = {On global action-angle coordinates},
   JOURNAL = {Comm. Pure Appl. Math.},
  FJOURNAL = {Communications on Pure and Applied Mathematics},
    VOLUME = {33},
      YEAR = {1980},
    NUMBER = {6},
     PAGES = {687--706},
      ISSN = {0010-3640},
   MRCLASS = {58F05 (58G15 70G10)},
  MRNUMBER = {596430},
MRREVIEWER = {Helmut R\"{u}ssmann},
       DOI = {10.1002/cpa.3160330602},
       URL = {https://doi-org.proxy.libraries.rutgers.edu/10.1002/cpa.3160330602},
}

@article {sheridan:hypersurface,
    AUTHOR = {Sheridan, Nick},
     TITLE = {On the {F}ukaya category of a {F}ano hypersurface in
              projective space},
   JOURNAL = {Publ. Math. Inst. Hautes \'{E}tudes Sci.},
  FJOURNAL = {Publications Math\'{e}matiques. Institut de Hautes \'{E}tudes
              Scientifiques},
    VOLUME = {124},
      YEAR = {2016},
     PAGES = {165--317},
      ISSN = {0073-8301},
   MRCLASS = {53D37 (14F05 14J33 14N35)},
  MRNUMBER = {3578916},
MRREVIEWER = {Christian Lehn},
       DOI = {10.1007/s10240-016-0082-8},
       URL = {https://doi-org.proxy.libraries.rutgers.edu/10.1007/s10240-016-0082-8},
}

@article {bryan:lgw,
    AUTHOR = {Bryan, Jim and Pandharipande, Rahul},
     TITLE = {The local {G}romov-{W}itten theory of curves},
      NOTE = {With an appendix by Bryan, C. Faber, A. Okounkov and
              Pandharipande},
   JOURNAL = {J. Amer. Math. Soc.},
  FJOURNAL = {Journal of the American Mathematical Society},
    VOLUME = {21},
      YEAR = {2008},
    NUMBER = {1},
     PAGES = {101--136},
      ISSN = {0894-0347},
   MRCLASS = {14N35 (57R56)},
  MRNUMBER = {2350052},
MRREVIEWER = {Sergiy Koshkin},
       DOI = {10.1090/S0894-0347-06-00545-5},
       URL = {https://doi-org.proxy.libraries.rutgers.edu/10.1090/S0894-0347-06-00545-5},
}
